\documentclass[a4paper]{amsart}
\usepackage{amsthm,amssymb,amsmath,amscd}
\usepackage{graphicx}
\usepackage{color}
\usepackage[all]{xy}
\usepackage{tikz}  
\usepackage{mathptmx}
\usetikzlibrary{knots}
\usepackage{txfonts}
\usepackage{here}
\usepackage[hang,small,bf]{caption}
\usepackage[subrefformat=parens]{subcaption}
\usepackage[subrefformat=parens]{subcaption}

\theoremstyle{definition}
\newtheorem{Definition}[subsection]{Definition}
\newtheorem{Proposition}[subsection]{Proposition}
\newtheorem{Corollary}[subsection]{Corollary}
\newtheorem{Theorem}[subsection]{Theorem}
\newtheorem{Lemma}[subsection]{Lemma}

\newtheorem{Remark}[subsection]{Remark}

\allowdisplaybreaks[2]

\usepackage{comment}
\usepackage{geometry}

\usepackage{breqn}

\makeatletter

	\@addtoreset{equation}{section}

	\@addtoreset{figure}{section}

	\@addtoreset{table}{section}
\makeatother
\makeatother

\begin{document}

  \tikzset{->-/.style 2 args={
    postaction={decorate},
    decoration={markings, mark=at position #1 with {\arrow[thick, #2]{>}}}
    },
    ->-/.default={0.5}{}
}
\tikzset{-<-/.style 2 args={
    postaction={decorate},
    decoration={markings, mark=at position #1 with {\arrow[thick, #2]{<}}}
    },
    -<-/.default={0.5}{}
}
\tikzset{
    overarc/.style={
        white, double=black, double distance=0.4pt, line width=2.4pt
    }
}

\begin{abstract}
The colored $\mathfrak{sl}_{3}$  Jones polynomial $J_{(n_{1}, n_{2})}^{\mathfrak{sl}_{3}}(L;q)$ are given by a link and an $(n_{1}, n_{2})$-irreducible representation of $\mathfrak{sl}_{3}$. In general, it is hard to calculate $J_{(n_{1}, n_{2})}^{\mathfrak{sl}_{3}}(L;q)$ for an oriented link $L$. However, we calculate the one-row $\mathfrak{sl}_{3}$ colored Jones polynomials $J_{(n, 0)}^{\mathfrak{sl}_{3}}(P(\alpha,\beta,\gamma);q)$  for  three-parameter families of oriented pretzel links $P(\alpha,\beta,\gamma)$ by using Kuperberg's linear skein theory by setting $n_{2}=0$.  Furthermore, we show  the existence of the tails of  $J_{(n, 0)}^{\mathfrak{sl}_{3}}(P(2\alpha +1, 2\beta+1,2\gamma);q)$ for the alternating pretzel knots $P(2\alpha +1, 2\beta+1,2\gamma)$.
\end{abstract} 

\title{THE  ONE-ROW COLORED $\mathfrak{sl}_{3}$ JONES POLYNOMIALS FOR PRETZEL LINKS}
\date{}
\author{KOTARO KAWASOE}
\maketitle

\section{Introduction}
\label{sec:Introduction}
The quantum invariant of a link is obtained by Lie algebra $\mathfrak{g}$ and its irreducible representation. In particular, if $\mathfrak{g}=\mathfrak{sl}_{2}$ and $V=\mathbb{C}^{2}$, we call the quantum invariant of a link Jones polynomial \cite{Jon85}. Kauffman \cite{Kau87}  reformulated the Jones polynomial of a link $L$ by using Kauffman bracket skein relation. Furthermore, as a generalization of the Jones polynomial, we constitute the colored $\mathfrak{sl}_{2}$ Jones polynomial $J_{N}^{\mathfrak{sl}_{2}}(L;q)$ by an $(N+1)$-dimensional irreducible representation of $\mathfrak{sl}_{2}$. For a link $L$, the graphical techniques with Kauffman bracket and Jones-Wenzel projector \cite{Wen87} for a link diagram give $J_{N}^{\mathfrak{sl}_{2}}(L;q)$. Jones-Wenzel projector is an element of  $A_{1}$ web space and plays a crucial role in  $J_{N}^{\mathfrak{sl}_{2}}(L;q)$. Linear skein theory is this method of calculating $J_{N}^{\mathfrak{sl}_{2}}(L;q)$. There are many examples of $J_{N}^{\mathfrak{sl}_{2}}(L;q)$ calculated by using linear skein theory. We can see one example in \cite{Lic97}. The explicit formulae in $J_{N}^{\mathfrak{sl}_{2}}(L;q)$ are helpful when we consider the property that there exists the limit of the colored Jones polynomial called the tail, and conjectures related to quantum invariants of links such as the volume conjecture \cite{Kas97}, the Jones slope conjecture \cite{Gar11}, Etc.\\
\quad As regards the tail of links, first, Dasbash and Lin \cite{Das06} showed that the first two coefficients and the last two coefficients of  $J_{N+1}^{\mathfrak{sl}_{2}}(K;q)$ do not depend on $N$ for alternating knots K. They also showed that the third and the third to last coefficients of  $J_{N}^{\mathfrak{sl}_{2}}(K;q)$ on $N$ for alternating knots do not depend on $N$ provided $N\ge3$. This result led them to predict that the coefficients of $J_{N}^{\mathfrak{sl}_{2}}(K;q)$ up to the $k$-th do not depend on $N$ provided  $N \ge k$; for alternating knots $K$, there exists a q series $T^{\mathfrak{sl}_{2}}(K;q)$ such that
\begin{align*}
T^{\mathfrak{sl}_{2}}(K;q)-\hat{J}_{N}^{\mathfrak{sl}_{2}}(K;q)\in q^{N+1}\mathbb{Z}[[q]]
\end{align*}
where $\hat{J}_{N}^{\mathfrak{sl}_{2}}(L;q)$ is normalization with the minimum degree of $J_{N}^{\mathfrak{sl}_{2}}(L;q)$ and $T^{\mathfrak{sl}_{2}}(L;q)$ is called the tails of $J_{N}^{\mathfrak{sl}_{2}}(K;q)$.
Armond \cite{Arm13} showed the existence of the tails of the colored Jones $\mathfrak{sl}_{3}$  polynomials for adequate knots containing alternating knots. Garoufalidis and L$\hat{\mbox{e}}$ \cite{GL15} gave proof of the existence of more general stability of coefficients of $J_{N}^{\mathfrak{sl}_{2}}(K;q)$  for alternating knots. 
  Armond and Dasbach \cite{Arm11} gave explicit tails for $(2,2m+1)$-torus knots and $(2,2m)$-torus links, Elhamdadi, Hajij \cite{EH17}, and Beirne \cite{Bei19} gave explicit tails for pretzel knots, Osburn, Kcilthy and Beirne \cite{KO17} \cite{BO17}  gave explicit tails for knots with small crossing numbers.\\
\quad Meanwhile, the colored $\mathfrak{sl}_{3}$ Jones polynomial $J_{(n_{1},n_{2})}^{\mathfrak{sl}_{3}}(L;q)$ is obtained by a link and an $(n_{1},n_{2})$-irreducible representation of $\mathfrak{sl}_{3}$. However, due to the complexity of the calculation, the explicit formulae for the colored  $\mathfrak{sl}_{3}$ Jones polynomial only exists for trefoil knot\cite{Law03}, $(2,2m+1)$-torus knots\cite{Gar17}\cite{Gar13}, two bridge links\cite{Yua17}, and $(2,2m)$-torus links\cite{Yua17}\cite{Yua18}\cite{Yua211}. For trefoil knot and $(2,2m+1)$-torus knots, representation theory methods gave the colored $\mathfrak{sl}_{3}$ Jones polynomials.  In the colored $\mathfrak{sl}_{3}$ Jones polynomial, the $A_{2}$ bracket is the same role as the Kauffman bracket, and the $A_{2}$ clasp is the same role as the Jones-Wenzel projector. For $(2,2m)$-torus links and two bridge links, graphical techniques with $A_{2}$ bracket and the $A_{2}$ claps gave the one-row colored $\mathfrak{sl}_{3}$ Jones polynomials obtained by an $(n,0)$-irreducible representation of $\mathfrak{sl}_{3}$. Kuperburg's linear skein theory \cite{Kup96} is these graphical techniques. Besides, we can consider the tail concerning $J_{(n_{1}, n_{2})}^{\mathfrak{sl}_{3}}(L;q)$. Garoufalidis and Voung \cite{Gar17} showed stability of coefficients of $J_{(n_{1}, n_{2})}^{\mathfrak{sl}_{3}}(L;q)$ for torus knots. Yuasa \cite{Yua17} \cite{Yua18} \cite{Yua211} \cite{Yua212}gave the explicit tails of $J_{(n,0)}^{\mathfrak{sl}_{3}}(L;q)$ for $(2,2m)$-torus links and proved that the existence of the tails of the one-row colored $\mathfrak{sl}_{3}$ Jones polynomials for minus adequate oriented links.    \\
\quad In this paper, we will give explicitly the one-row $\mathfrak{sl}_{3}$ colored Jones polynomials for all three-parameter family of pretzel links $P(\alpha,\beta,\gamma)$ except for those where $\alpha,\beta,\gamma$ are all odd. These pretzel links include $8_5$ and $8_{19}$. $8_5$ and $8_{19}$ are not $(2,2m)$-torus link and three-bridge knot. Furthermore, we prove the exsistace of the tail of  alternating pretzel knots $P(2\alpha +1,2\beta+1,2\gamma)$. For instance, the one-row colored $\mathfrak{sl}_{3}$ Jones polynomial for $P(3,3,2)=8_{5}$ multiplied by $q^{3n^{2}+8n}$ is 
\begin{align*}
&n=1:\quad1-q+q^{2}-2q^{4}+q^{5}-2q^{6}+q^{7}+q^{8}+q^{10}\\
&n=2:\quad1-q+q^{3}-2q^{4}+2q^{6}-2q^{7}-2q^{8}+4q^{9}+q^{10}+\cdots\\
&n=3:\quad1-q+q^{4}+q^{5}+q^{7}-4q^{8}-q^{9}+8q^{10}+\cdots\\
&n=4:\quad1-q-2q^{4}+2q^{5}+q^{6}-2q^{8}-4q^{9}+4q^{10}\cdots\\
&n=5:\quad1-q-2q^{4}+q^{5}+2q^{6}+q^{7}-3q^{8}-4q^{q}+q^{10}+\cdots\\
&n=6:\quad1-q-2q^{4}+q^{5}+q^{6}+2q^{7}-2q^{8}-2q^{9}+3q^{10}+\cdots\\
&n=7:\quad1-q-2q^{4}+q^{5}+q^{6}+q^{7}-q^{8}-q^{9}+2q^{10}+\cdots
\end{align*}
In Appendix, we also give the one-row $\mathfrak{sl}_{3}$ colored Jones polynomials for four and five parameter pretzel knots $8_{10},8_{15},8_{20}$, and $  8_{21}$. As a result, the one-row $\mathfrak{sl}_{3}$ colored Jones polynomials are determined for all knots with eight or fewer crossings except $8_{16},8_{17},8_{18}$. \\
\quad This paper is organized as follows. In section \ref{sec:Preliminaries}, we introduce the Kurperberg's linear skein theory and the formulae for the $A_{2}$ clasp given by Ohtsuki and Yamada \cite{Oht97}, Kim \cite{Kim06} \cite{Kim07} and Yuasa \cite{Yua17} \cite{Yua182}. In section \ref{sec:The formulas of the clasped },
we derive a formula for $m$ times-half twists where two strands with the same directions for Kurperberg's web space. In section \ref{sec:Computing}, we give explicit formulae for the one-row $\mathfrak{sl}_{3}$ colored Jones polynomial of three-parameter family of pretzel links $P(\alpha,\beta,\gamma)$. Moreover, for alternating pretzel knots $P(2\alpha+1,2\beta +1,2\gamma)$, we give the proof that the tail of $J^{\mathfrak{sl}_{3}}_{(n,0)}(P(2\alpha+1,2\beta +1,2\gamma);q)$ exists. In Appendix, we give  the one-row $\mathfrak{sl}_{3}$ colored Jones Polynomials for $8_{10},8_{15},8_{20}$ and $  8_{21}$.

\section{Preliminaries}
\label{sec:Preliminaries}

We use the following q-integer notations.
\begin{align*}
&\{n\}_{q}=\{q^{\frac{n}{2}}-q^{-\frac{n}{2}}\}_{q},\quad \{n\}_{q}!=\{n\}_{q}\{n-1\}_{q}\cdots\{1\}_{q},\quad[n]_{q}=\frac{\{n\}}{\{1\}},\quad[n]_{q}!=[n]_{q}[n]_{q}\cdots[1]_{q}
\end{align*}
where $n$ is a non-negative integer. 
A q-Pochhammer synbol is difined by
\begin{align*}
(q)_{k}=\prod_{i=1}^{k}(1-q^{i}).
\end{align*}
A q-binomial coeffcient is defined by
\begin{align*}
\begin{bmatrix}
n  \\
k \\
\end{bmatrix}_{q}=\frac{[n]_{q}!}{[k]_{q}![n-k]_{q}!}.
\end{align*}
\begin{Lemma}[\cite{Oht97}]
\label{quantumIntger}
For any integer a, b and c,
\begin{align*}
&(1) \quad[a]_{q}[b]_{q}=\sum_{i=1}^{a}[a+b-(2i-1)]_{q}=\sum_{i=1}^{b}[a+b-(2i-1)]_{q},\\
&(2) \quad[a]_{q}[b]_{q}-[a-c]_{q}[b-c]_{q}=[a+b-c]_{q}[c]_{q},\\
&(3) \quad[a]_{q}[b-c]_{q}+[c]_{q}[a-c]_{q}=[b]_{q}[a-c]_{q}.
\end{align*}
\end{Lemma}

We first define the $A_{2}$ web spaces. Let $D$ be a disk with signed marked point $(P,\epsilon)$ on its boundary where $P$ is a finite set such that $P\subset \partial D$ and $\epsilon:P \rightarrow \{+,-\}$ is a map. A bipartite uni-trivalent graph on D is an immersion of a directed graph such that every vertex is either univalent or trivalent and is divided into sink or source.
\begin{align*}
\mbox{sink}:\begin{minipage}{1\unitlength}\includegraphics[scale=0.07]{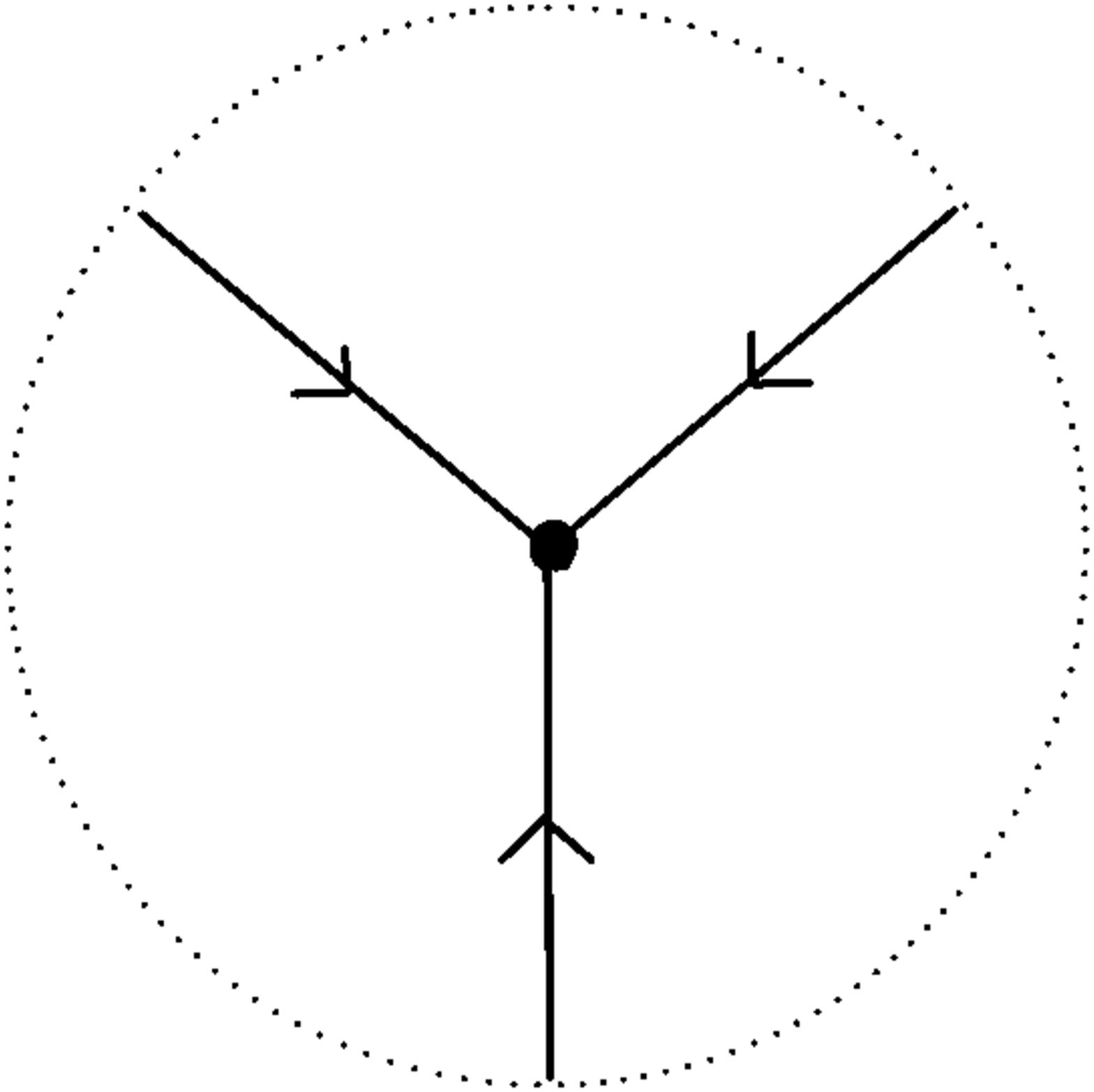}\put(-24,20){\scriptsize$+$\normalsize}\end{minipage}\hspace{55pt}\mbox{or}\hspace{10pt}\begin{minipage}{1\unitlength}\includegraphics[scale=0.07]{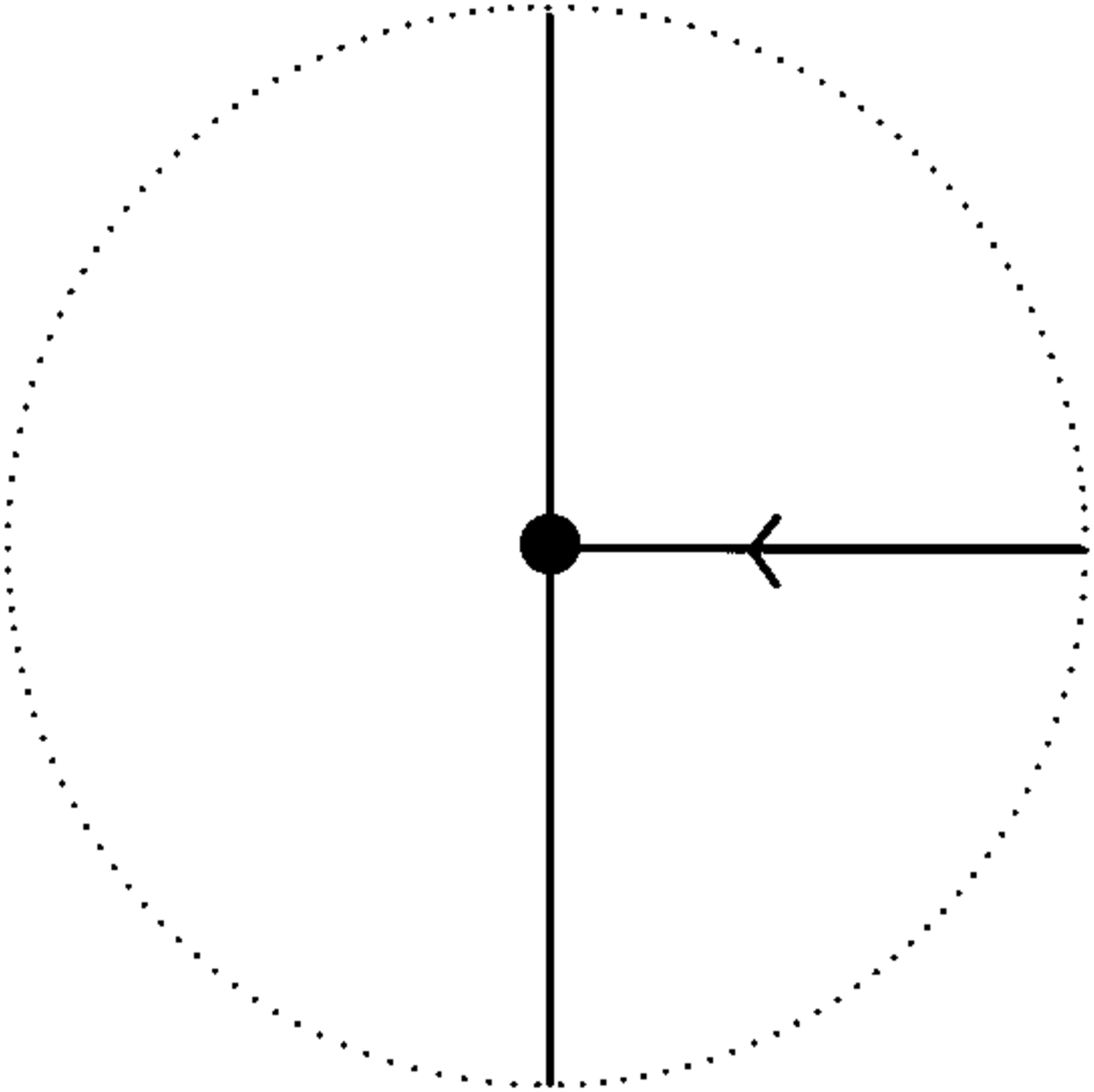}\put(-24,20){\scriptsize$+$\normalsize}\end{minipage}\hspace{60pt}\mbox{source}:\begin{minipage}{1\unitlength}\includegraphics[scale=0.07]{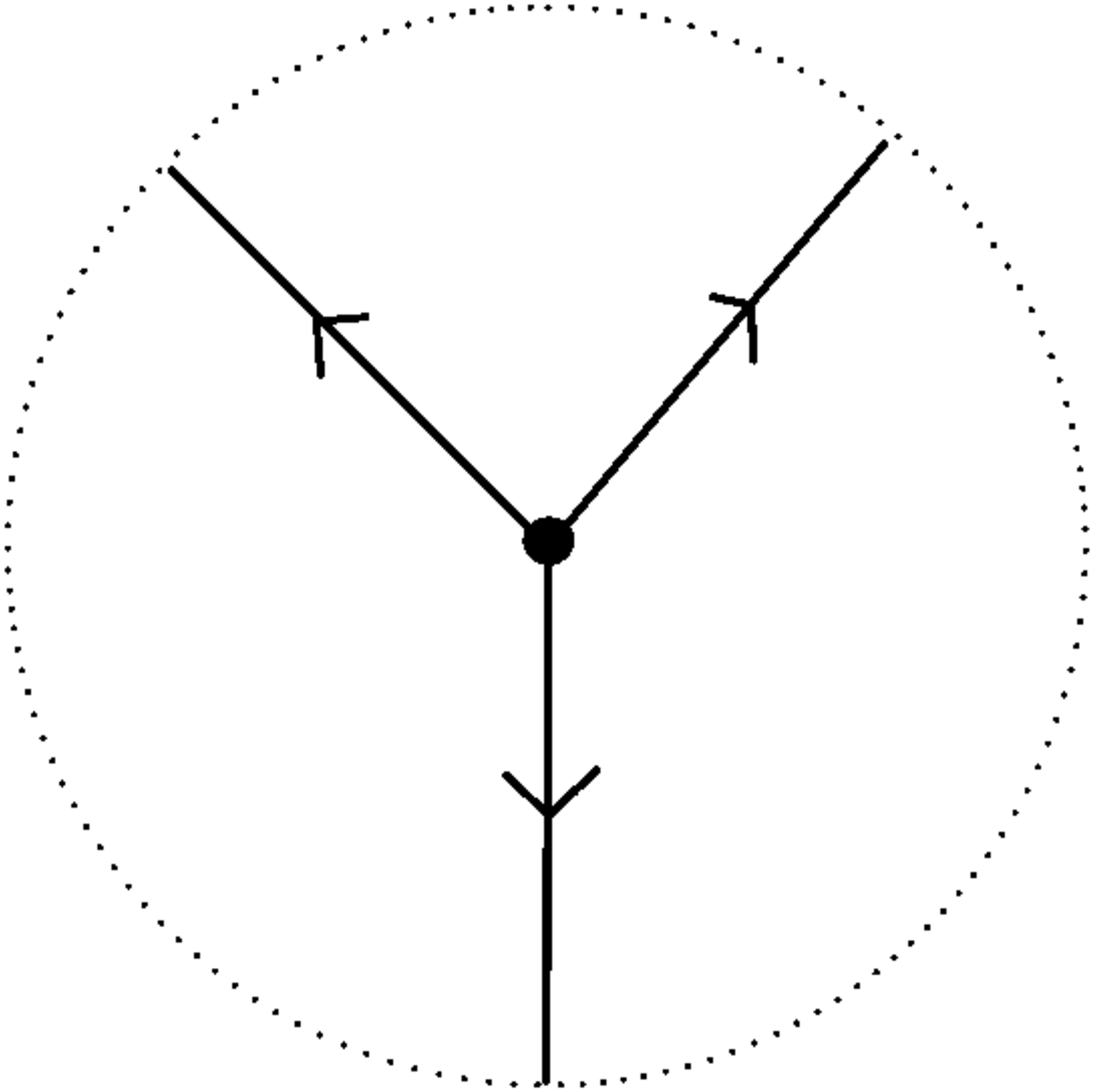}\put(-24,20){\scriptsize$-$\normalsize}\end{minipage}\hspace{55pt}\mbox{or}\hspace{10pt}\begin{minipage}{1\unitlength}\includegraphics[scale=0.07]{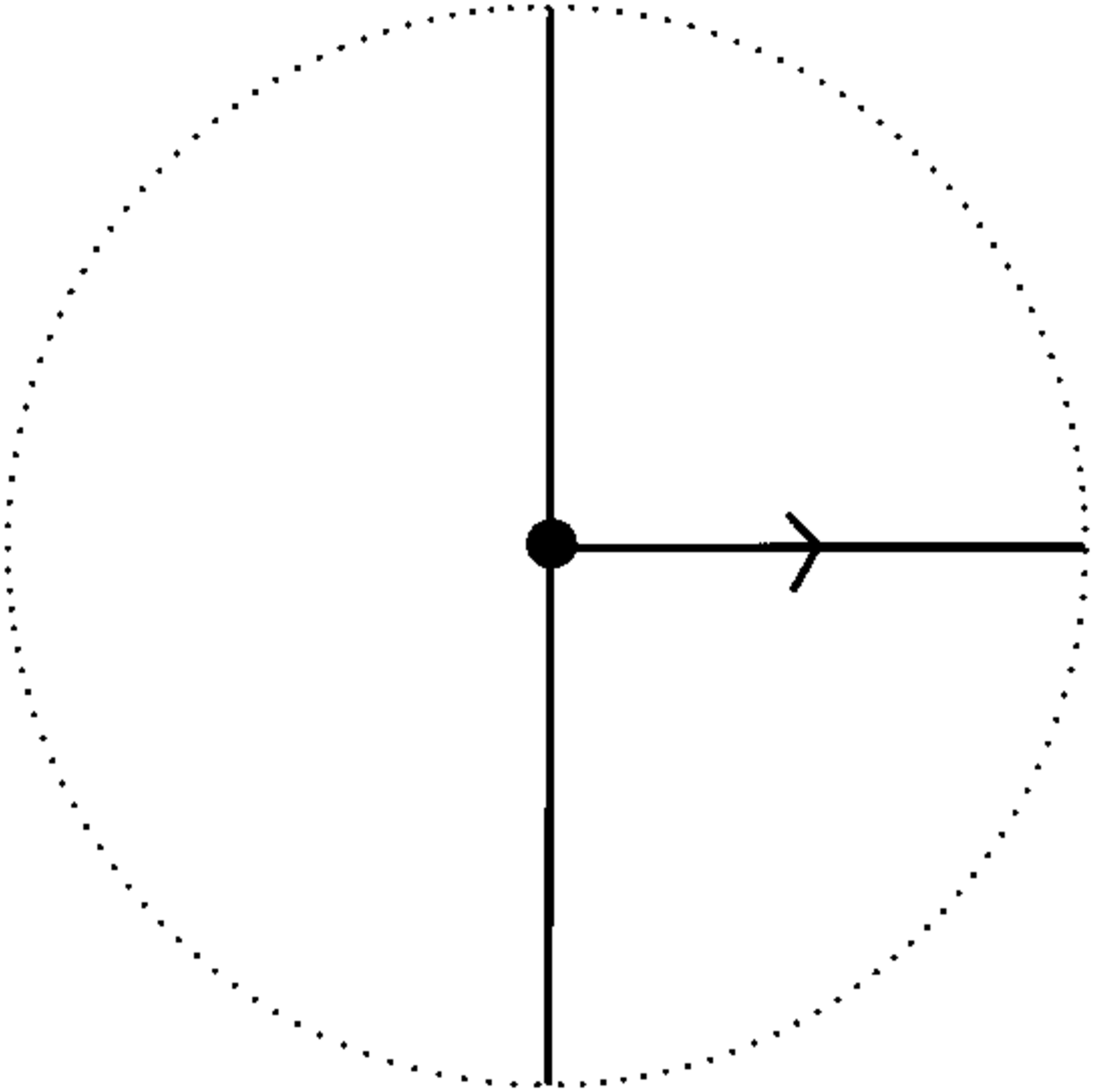}\put(-24,20){\scriptsize$-$\normalsize}\end{minipage}\hspace{60pt}
\end{align*}
A tangled bipartite uni-trivalent graph diagram on $D$ is a bipartite uni-trivalent graph graph on $D$ satisfying:
\begin{align*}
&(1)\quad\mbox{the set of univalent vertices coincide with $P$},\\
&(2)\quad \mbox{every crossing point is a transverse double point of two edges with under or over crossing data.}
\end{align*}
For tangled uni-trivalent graph diagrams $G$ and $G'$, we call $G$ regular isotopic to $G'$ on $D$ if $G$ is related to $G'$ by a finite sequence of boundary-fixing isotopies and Reidemeister moves $(R1)-(R4)$ with some directions of edges.
\begin{align*}
&(R1)\quad\begin{minipage}{1\unitlength}\includegraphics[scale=0.07]{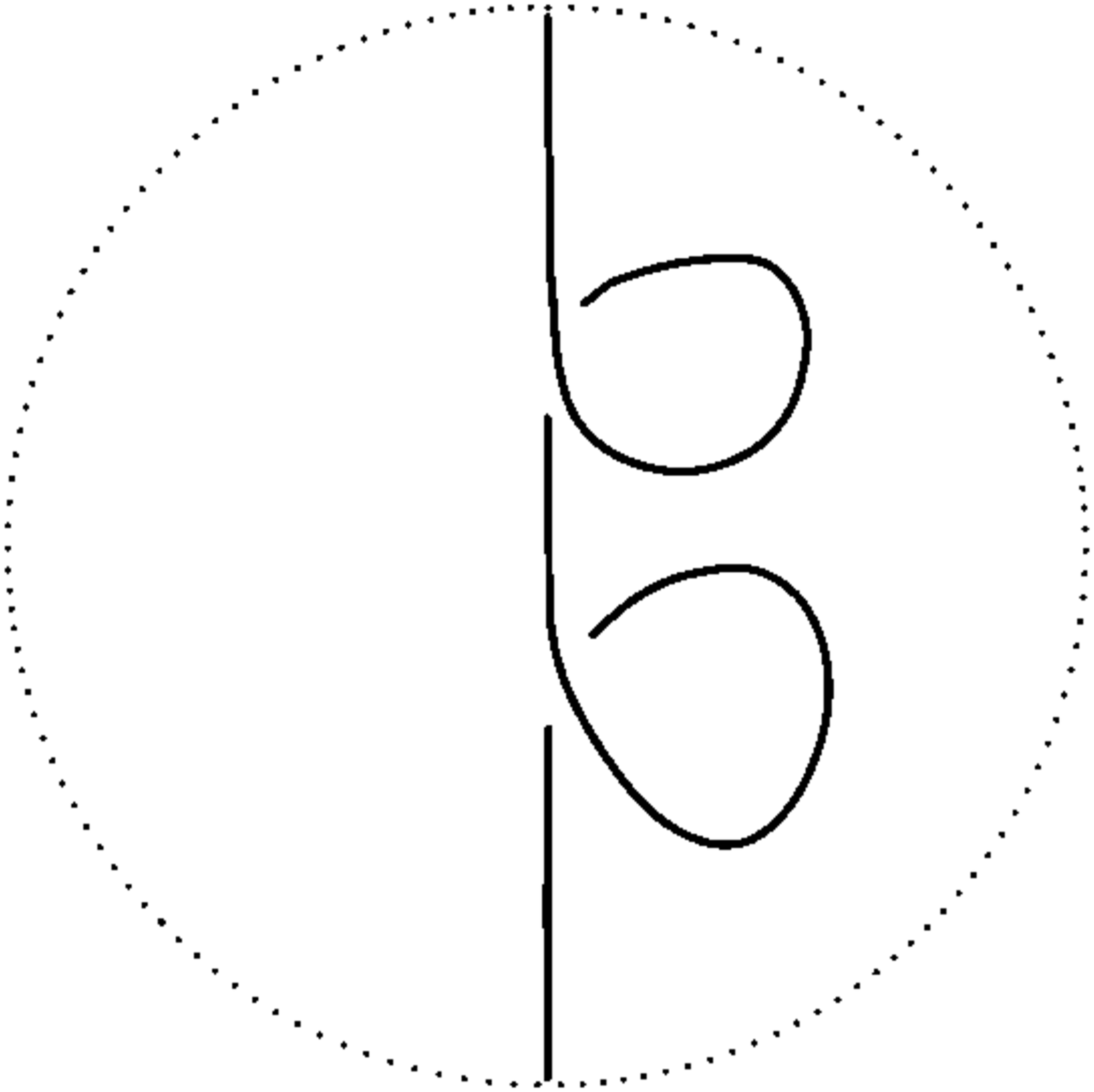}\end{minipage}\hspace{52pt}= \begin{minipage}{1\unitlength}\includegraphics[scale=0.07]{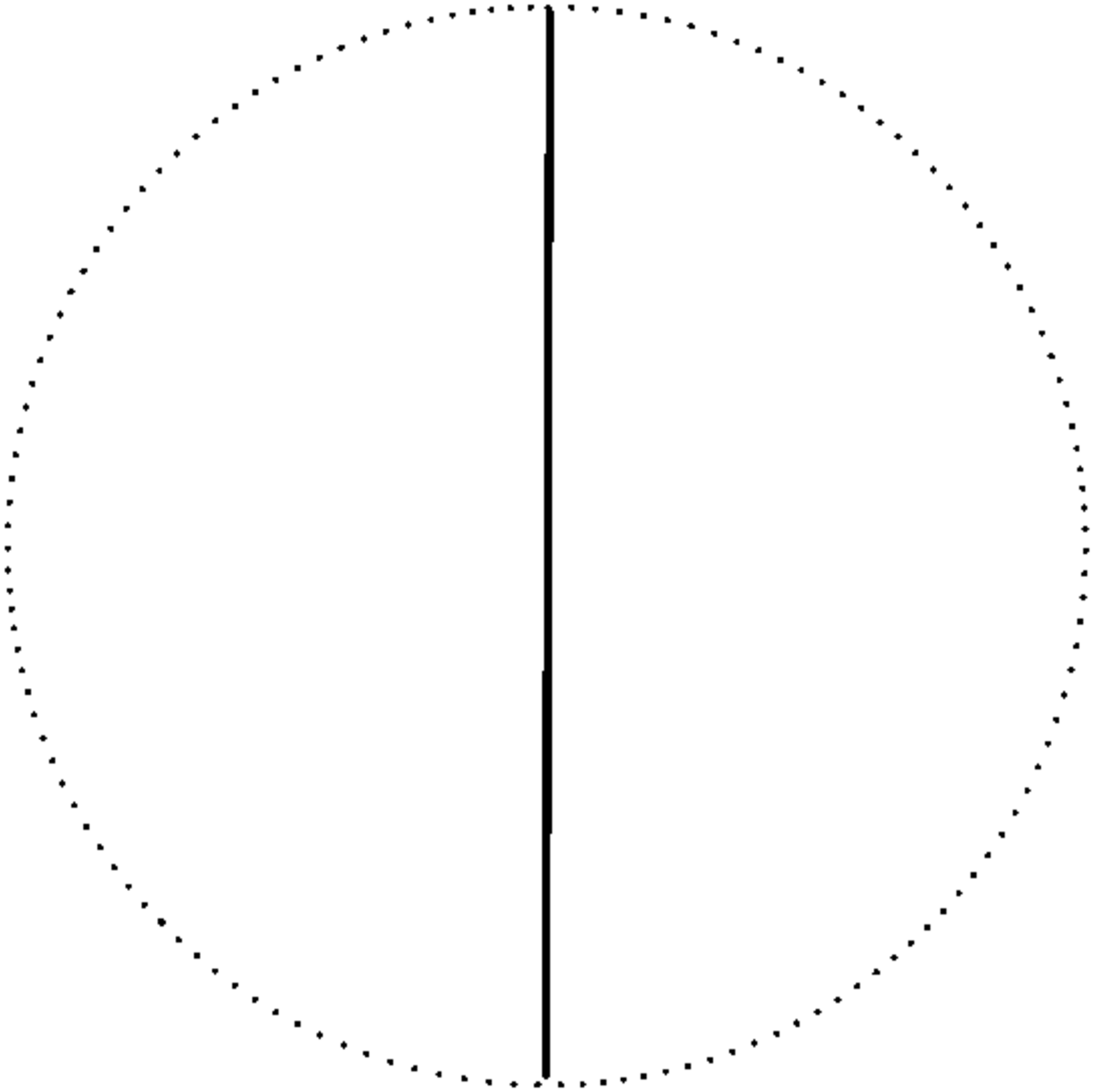}\end{minipage}\hspace{52pt}\quad(R2)\quad\begin{minipage}{1\unitlength}\includegraphics[scale=0.07]{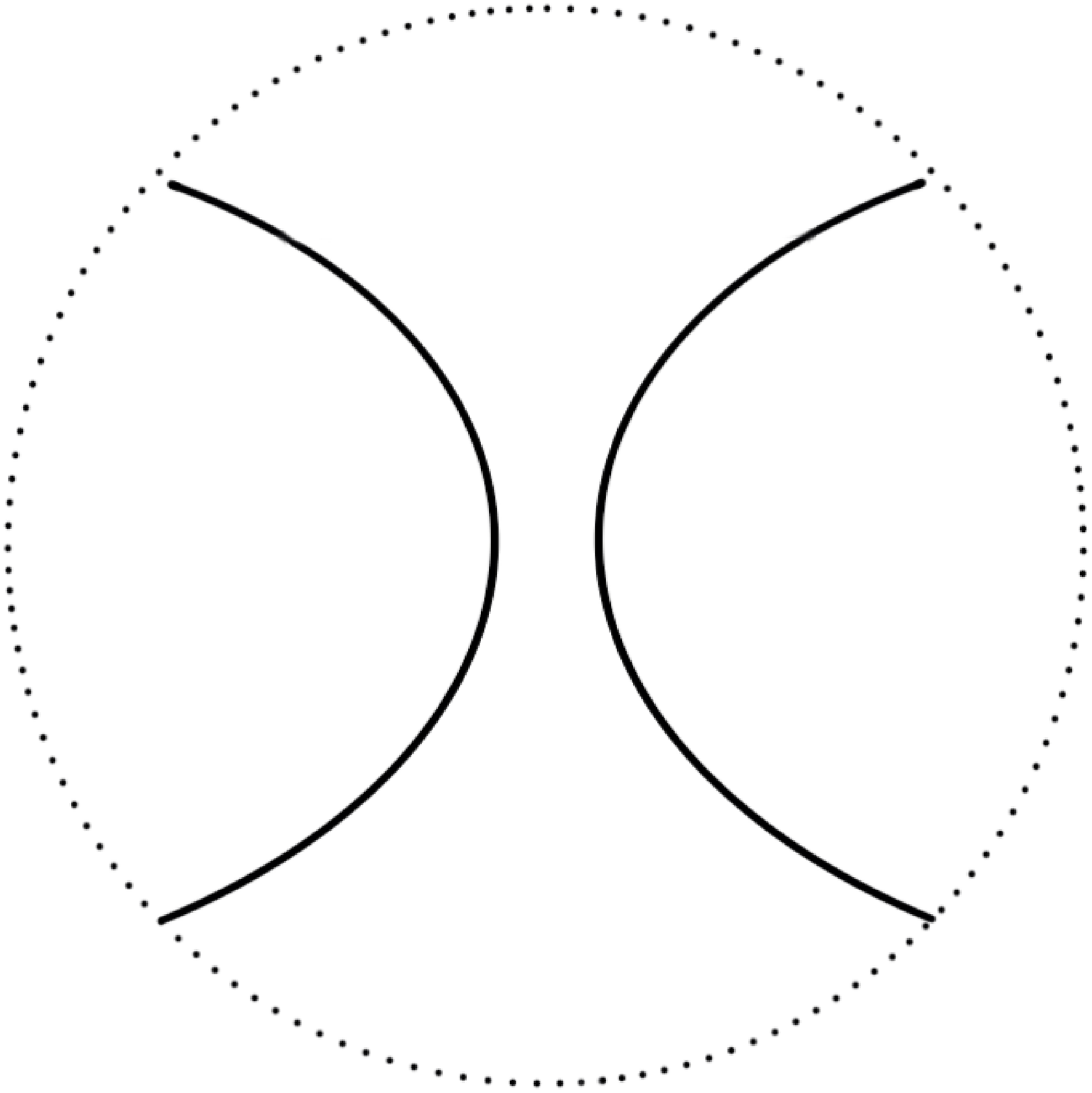}\end{minipage}\hspace{53pt}= \begin{minipage}{1\unitlength}\includegraphics[scale=0.07]{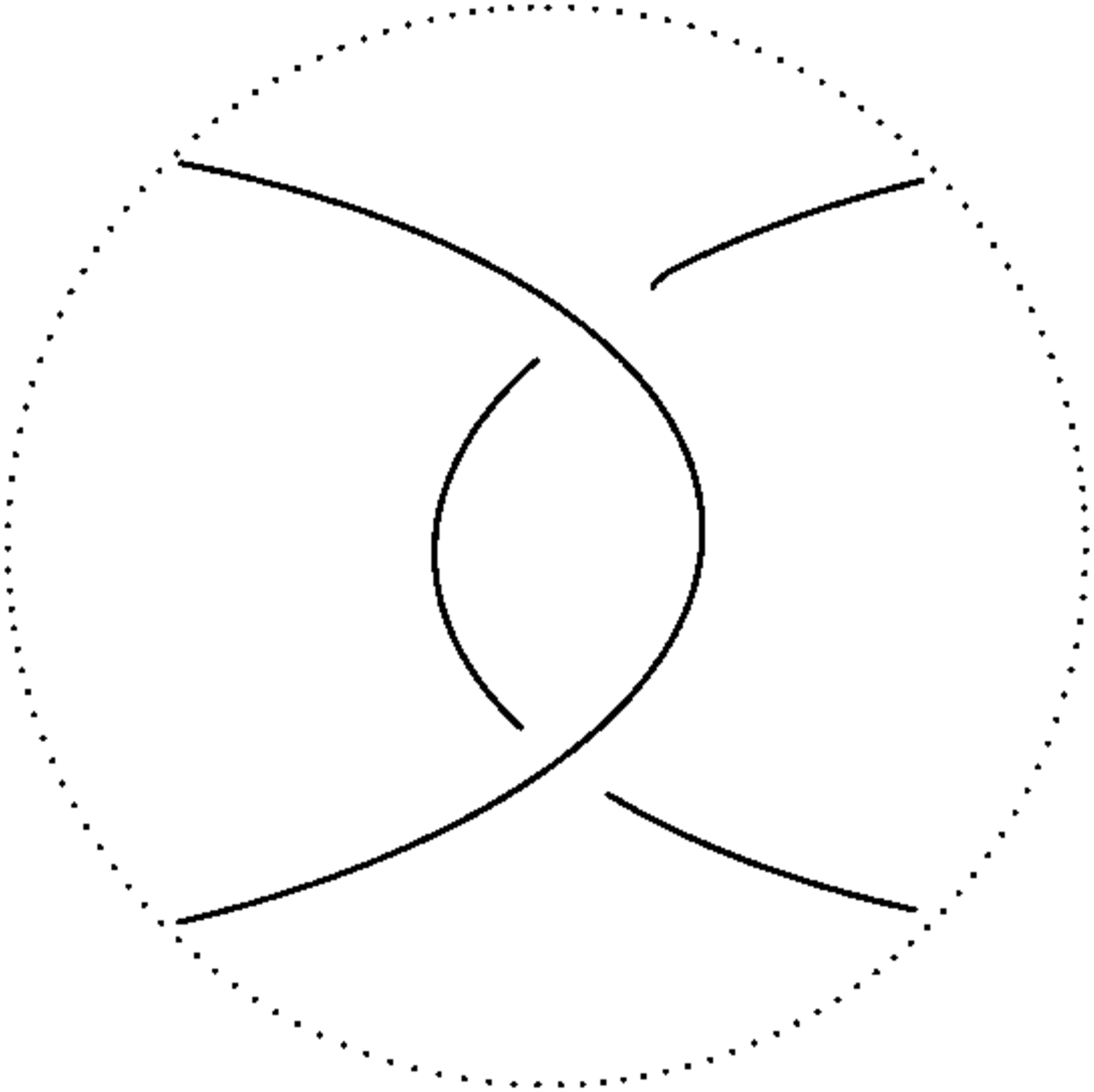}\end{minipage}\hspace{52pt}\\
 &(R3)\quad  \begin{minipage}{1\unitlength}\includegraphics[scale=0.07]{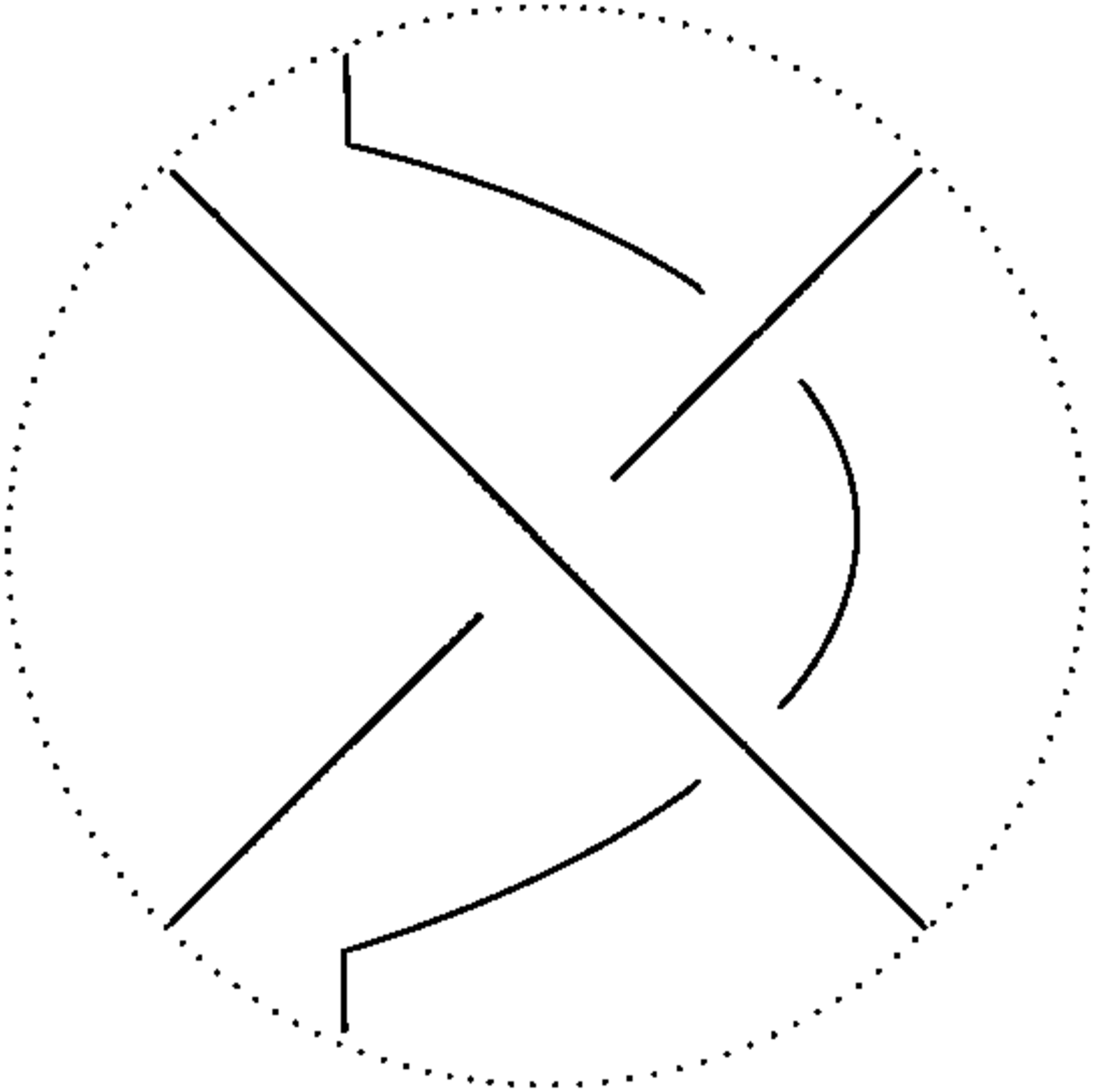}\end{minipage}\hspace{52pt}= \begin{minipage}{1\unitlength}\includegraphics[scale=0.07]{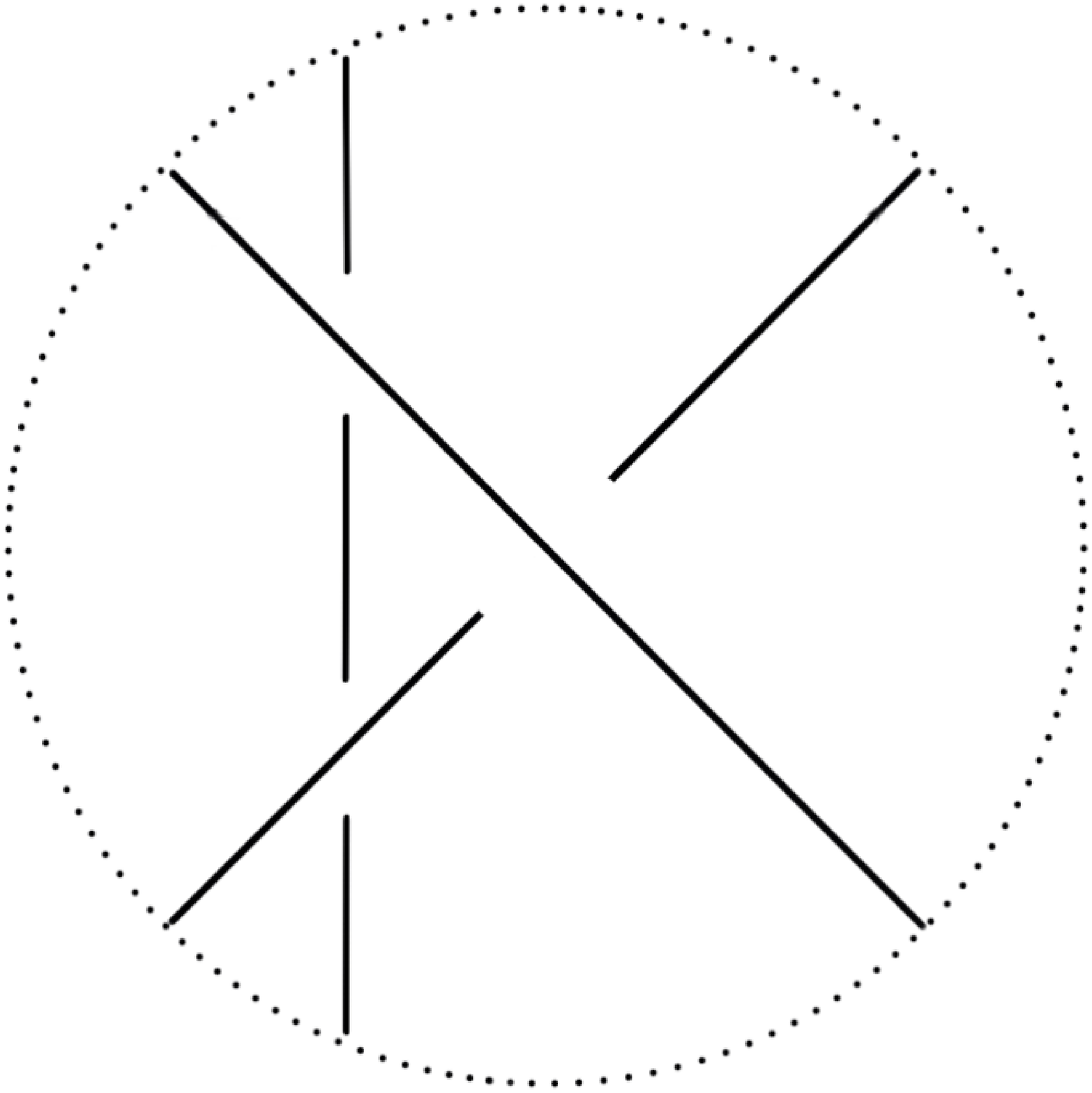}\end{minipage}\hspace{52pt}\\
 &(R4)\quad \begin{minipage}{1\unitlength}\includegraphics[scale=0.07]{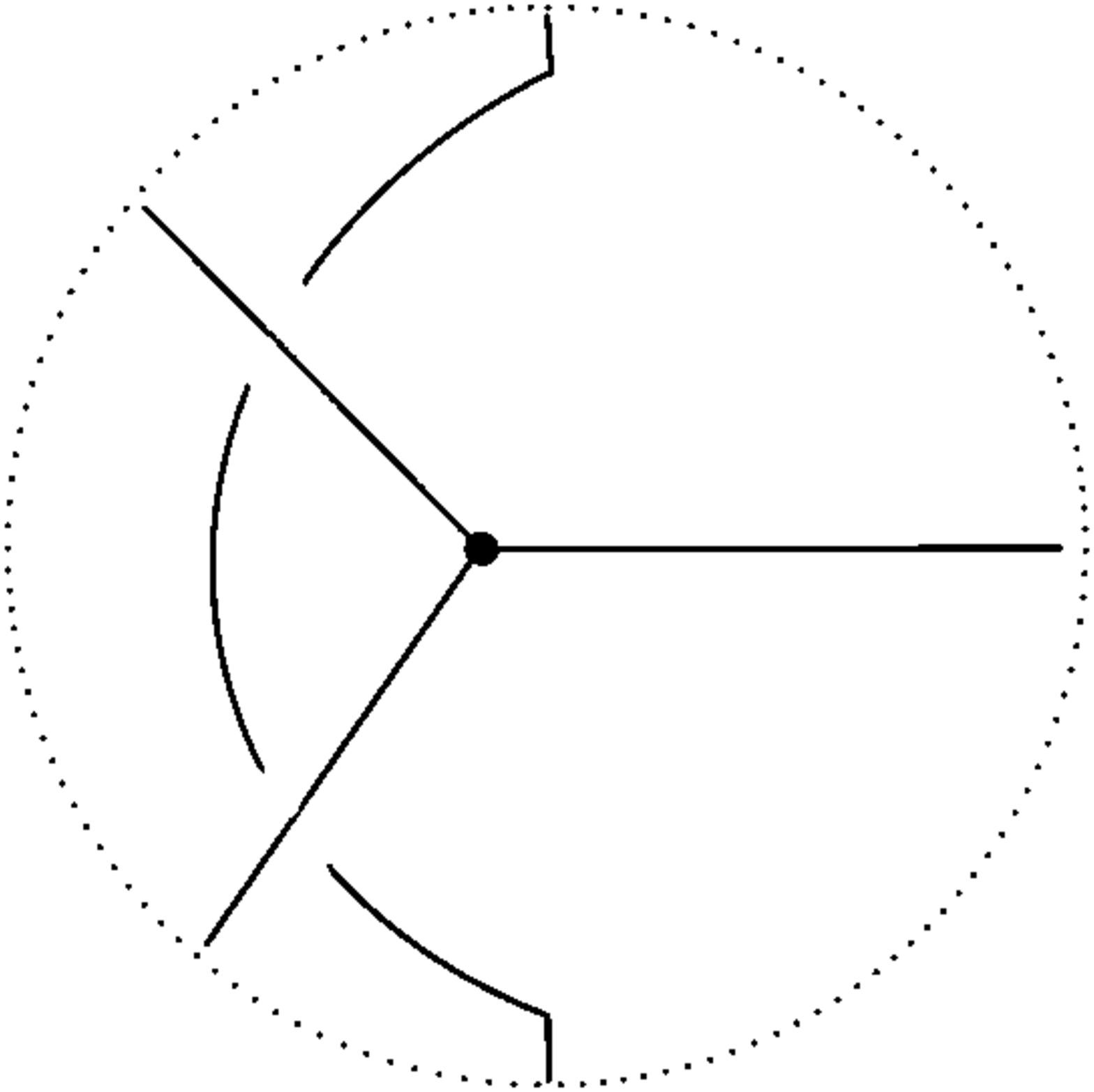}\end{minipage}\hspace{52pt}=  \begin{minipage}{1\unitlength}\includegraphics[scale=0.07]{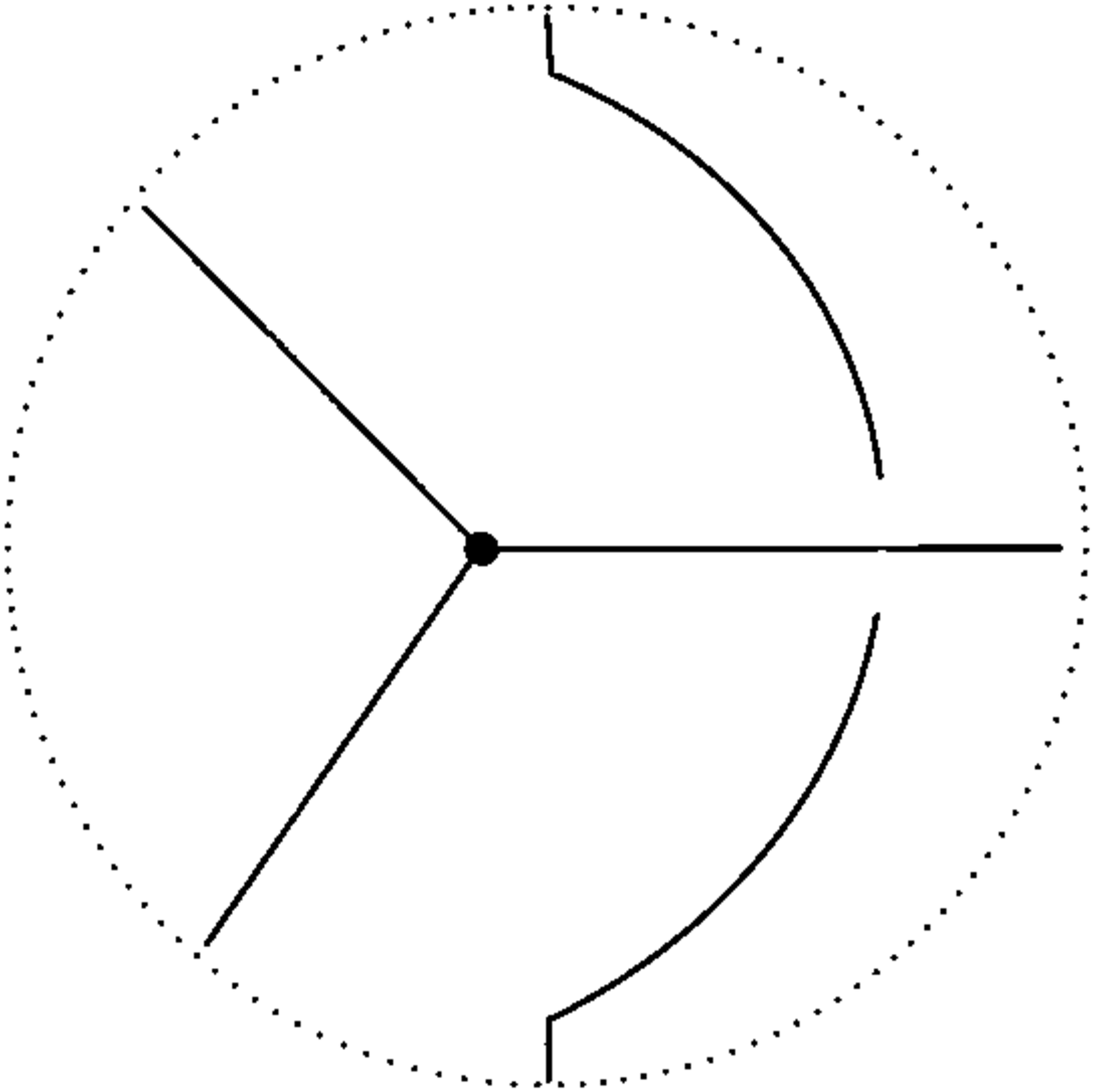}\end{minipage}\hspace{52pt}, \quad \begin{minipage}{1\unitlength}\includegraphics[scale=0.07]{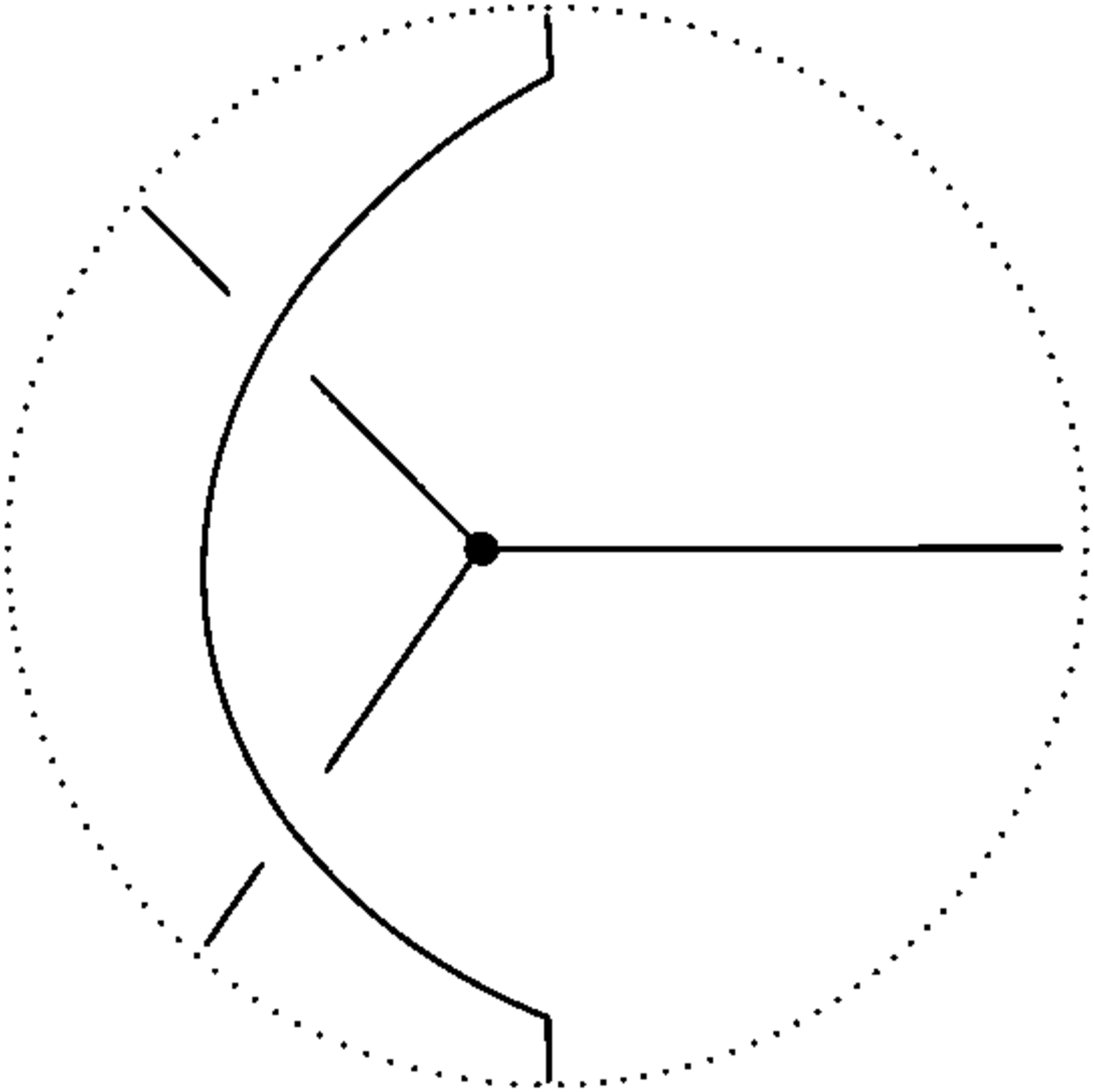}\end{minipage}\hspace{52pt}=  \begin{minipage}{1\unitlength}\includegraphics[scale=0.07]{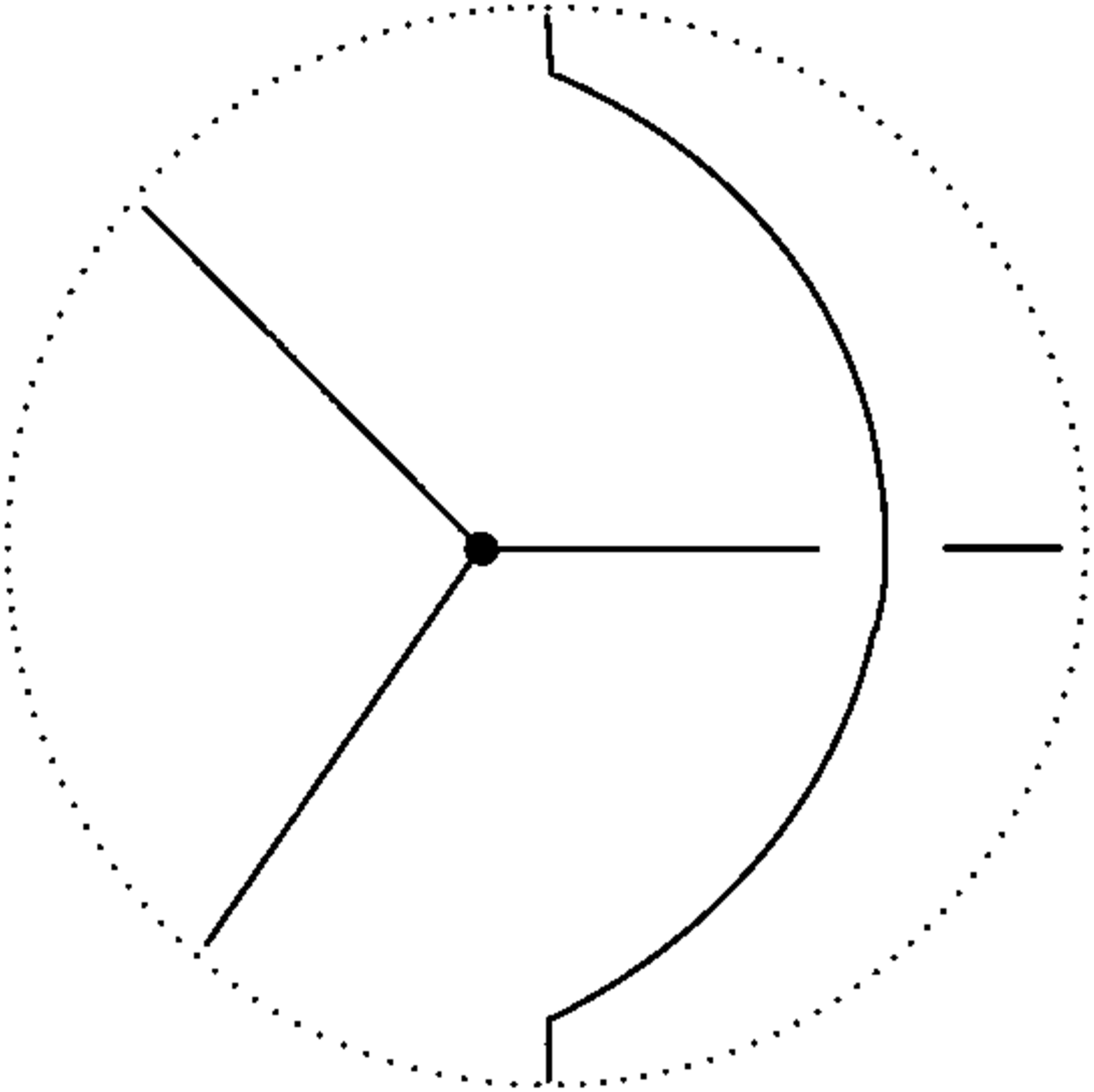}\end{minipage}\hspace{52pt}
\end{align*}
The tangled uni-trivalent graphs on $D$ are regular isotopy classes of tangled trivalent graph diagrams on D. We denote $T(P,\epsilon)$ the set of tangled uni-trivalent graphs on $D$. The $A_{2}$ basis web is the boundary-fixing isotopy class of bipartite uni-trivalent graphs with crossing on $D$ without internal $0,2,4$-gons. Let $B(P,\epsilon)$ be the set of $A_{2}$ basis web. The $A_{2}$ web space $W(P,\epsilon)$ is a $\mathbb{Q}(q^{\frac{1}{6}})$-vector space spanned by $B(P,\epsilon)$. An element in $W(P,\epsilon)$ is called web.

\begin{Definition}[the $A_{2}$ bracket \cite{Kup96}]
We difine a $\mathbb{Q}(q^{\frac{1}{6}})$ linear map $\langle\cdot\rangle_{3}:\mathbb{Q}(q^{\frac{1}{6}})T(P,\epsilon)\rightarrow W(P,\epsilon)$ by the following:
    \begin{align*}
 &\Biggl\langle\begin{minipage}{1\unitlength}\includegraphics[scale=0.07]{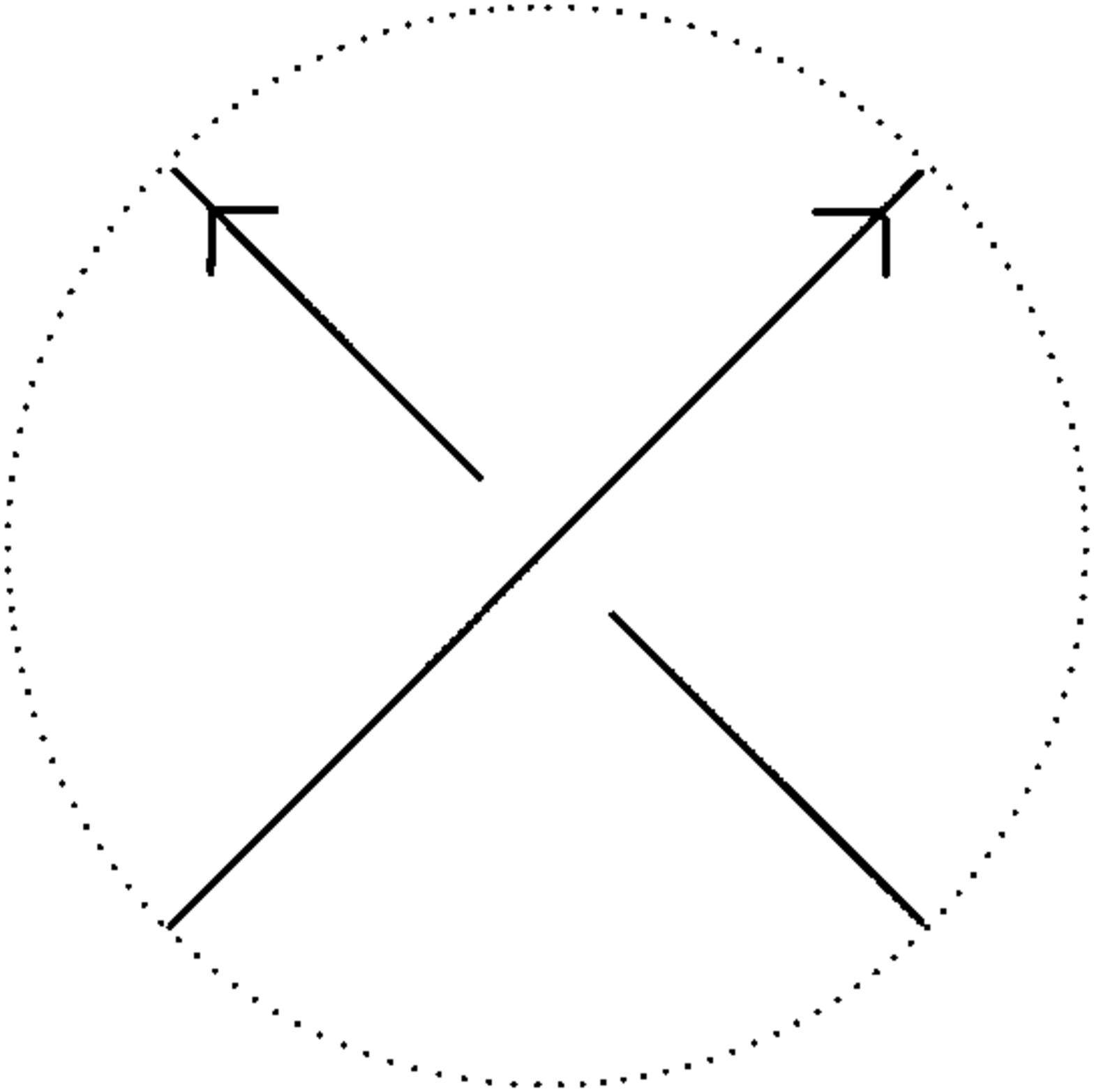}\end{minipage}\hspace{52pt}\Biggr\rangle_{3}=q^{\frac{1}{3}}\Biggl\langle \begin{minipage}{1\unitlength}\includegraphics[scale=0.07]{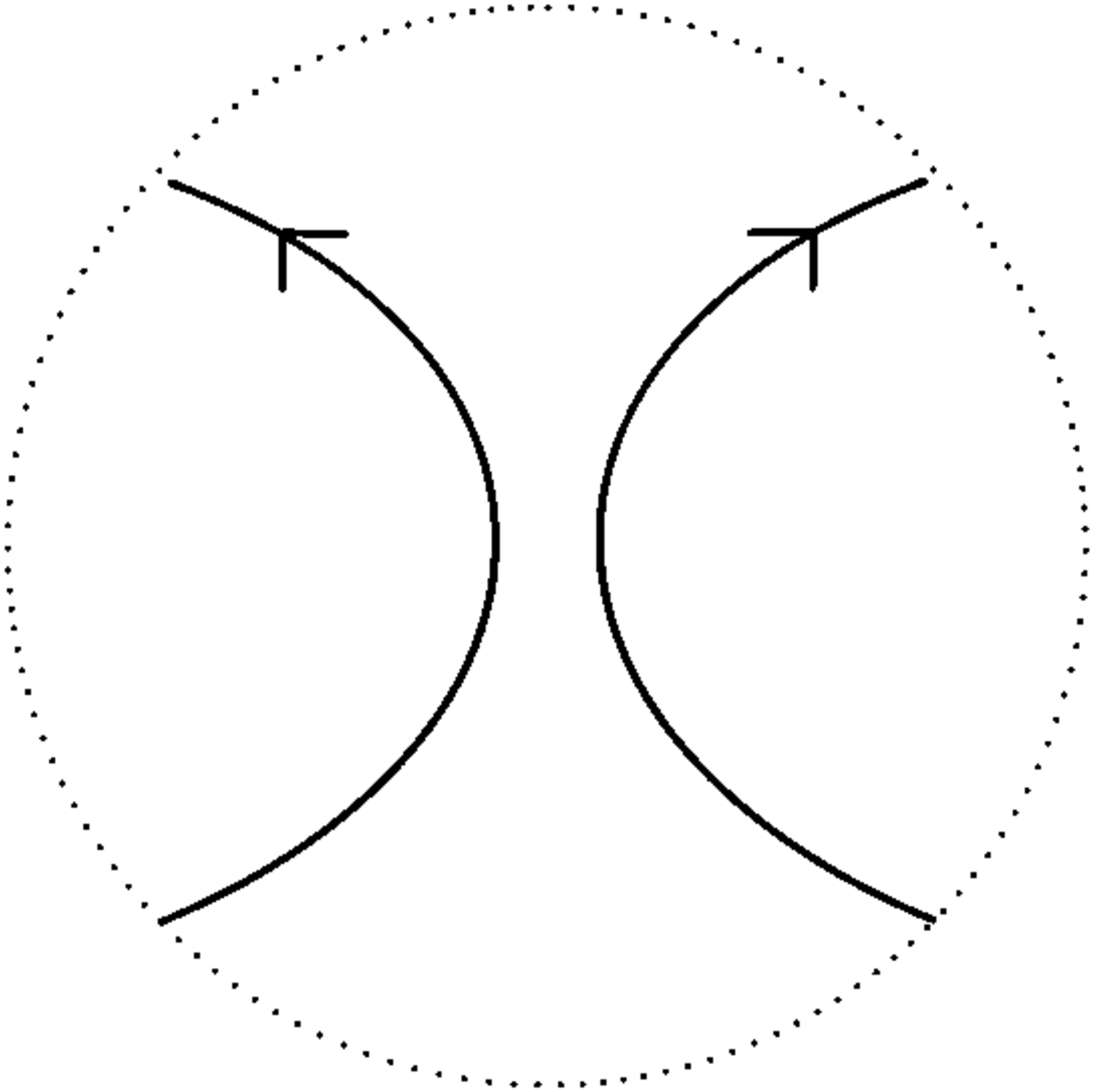}\end{minipage}\hspace{52pt}\Biggr\rangle_{3}-q^{-\frac{1}{6}}\Biggl\langle \begin{minipage}{1\unitlength}\includegraphics[scale=0.07]{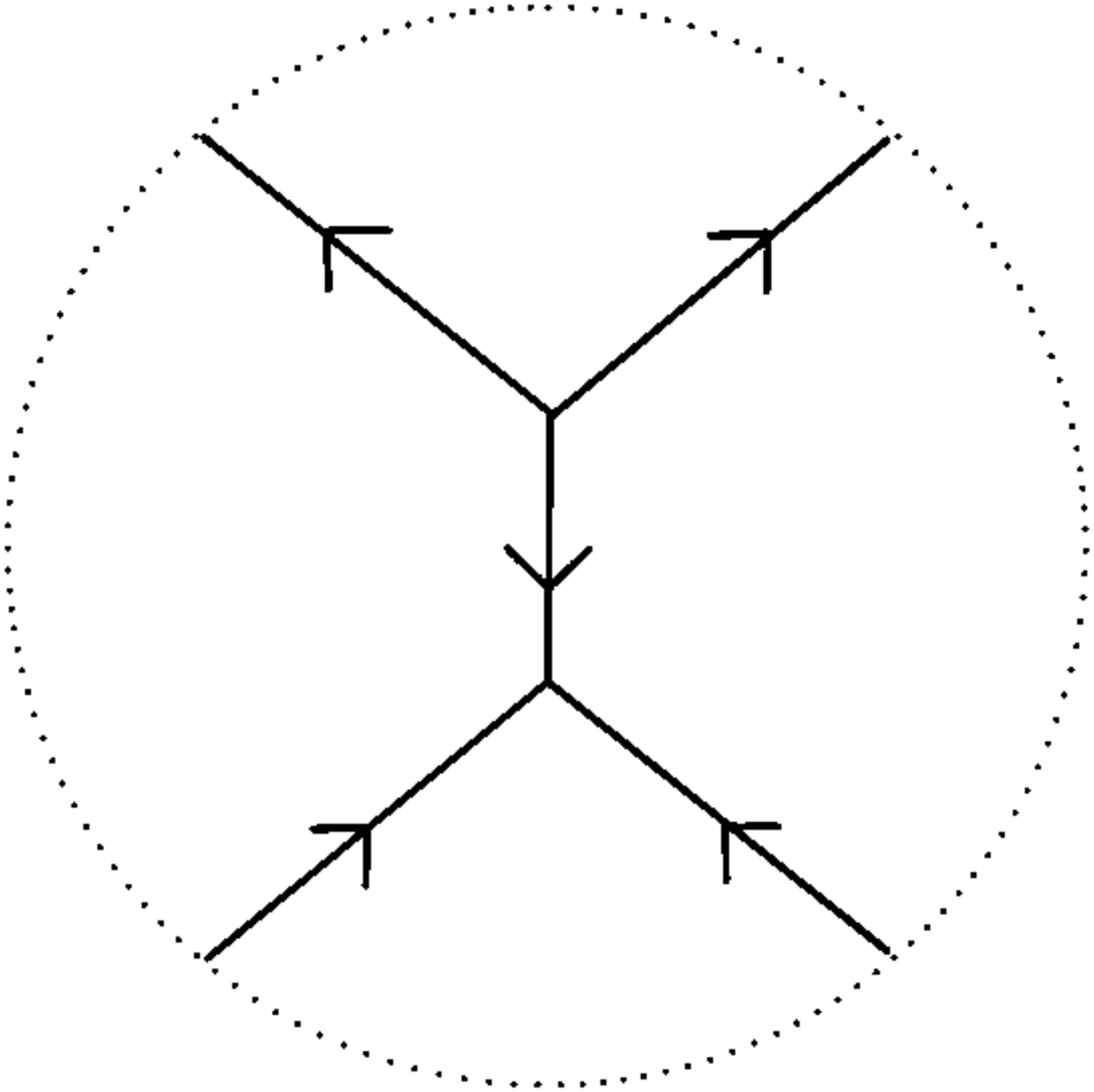}\end{minipage}\hspace{53pt}\Biggr\rangle_{3},\\
 &\Biggl\langle  \begin{minipage}{1\unitlength}\includegraphics[scale=0.07]{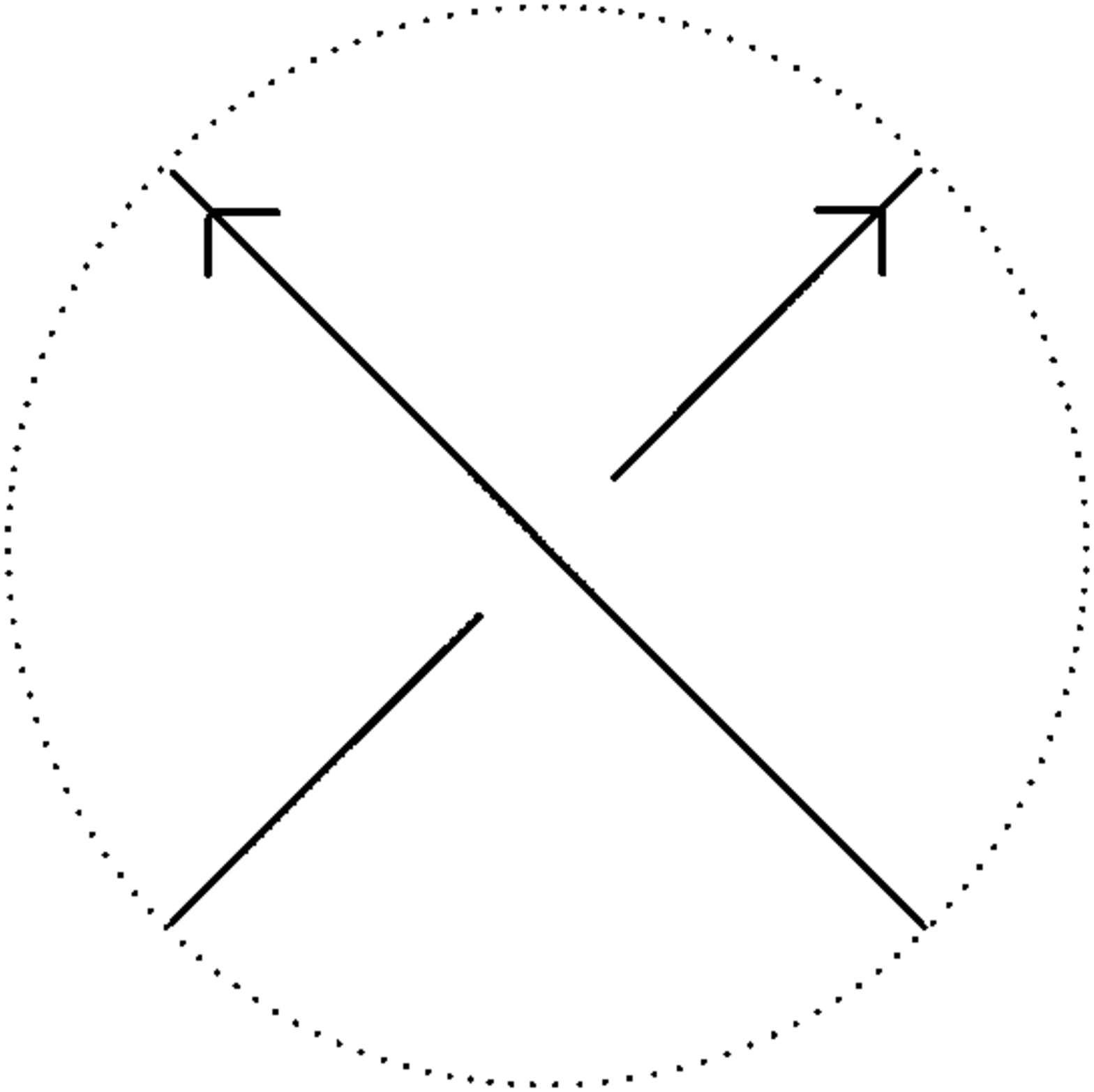}\end{minipage}\hspace{52pt}\Biggr\rangle_{3}=q^{-\frac{1}{3}}\Biggl\langle \begin{minipage}{1\unitlength}\includegraphics[scale=0.07]{pic/Smooth.eps}\end{minipage}\hspace{52pt}\Biggr\rangle_{3}-q^{\frac{1}{6}}\Biggl\langle \begin{minipage}{1\unitlength}\includegraphics[scale=0.07]{pic/6_valent.eps}\end{minipage}\hspace{52pt}\Biggr\rangle_{3},\\
 &\Biggl\langle \begin{minipage}{1\unitlength}\includegraphics[scale=0.07]{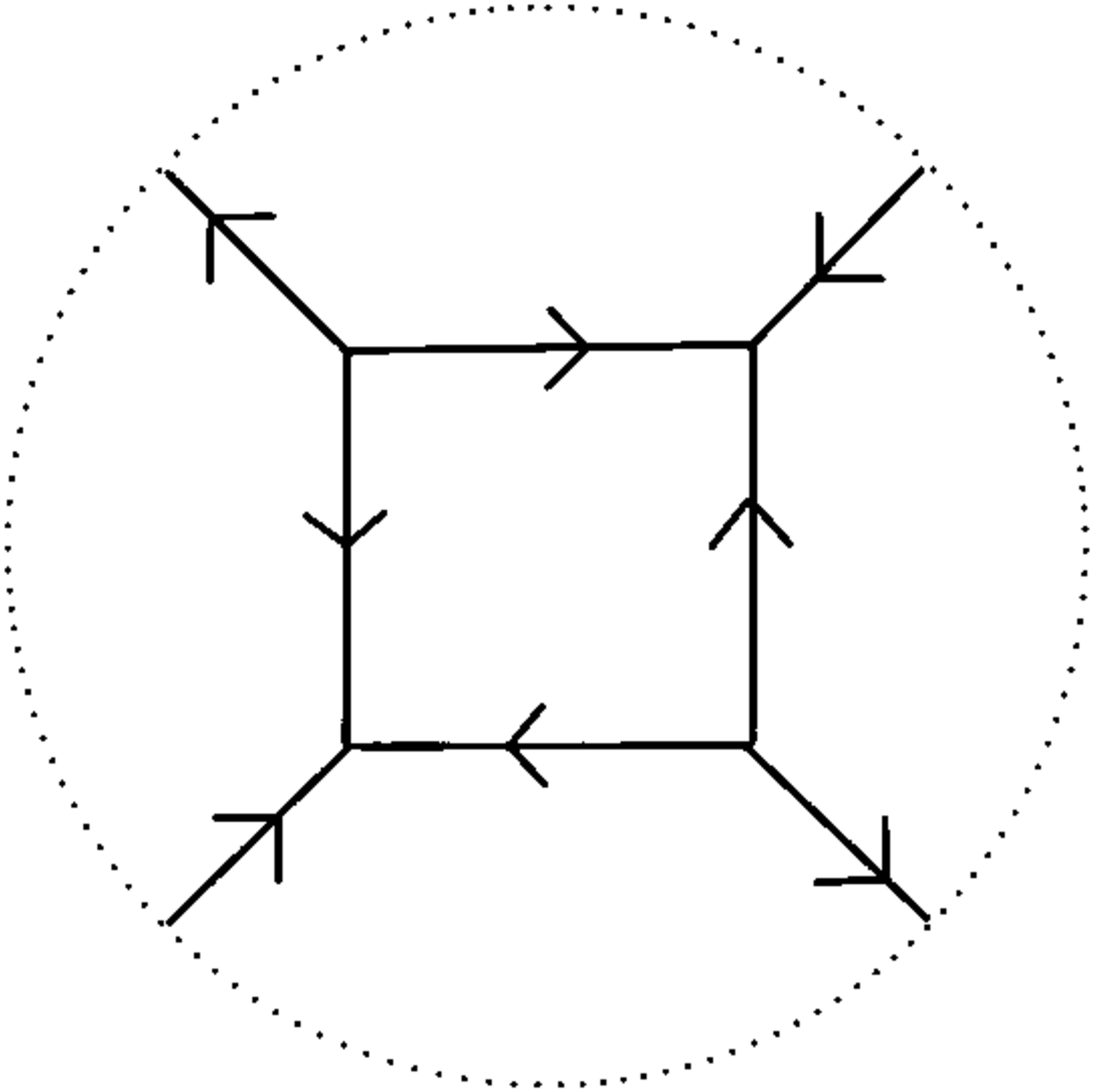}\end{minipage}\hspace{52pt}\Biggr\rangle_{3}=\Biggl\langle  \begin{minipage}{1\unitlength}\includegraphics[scale=0.07]{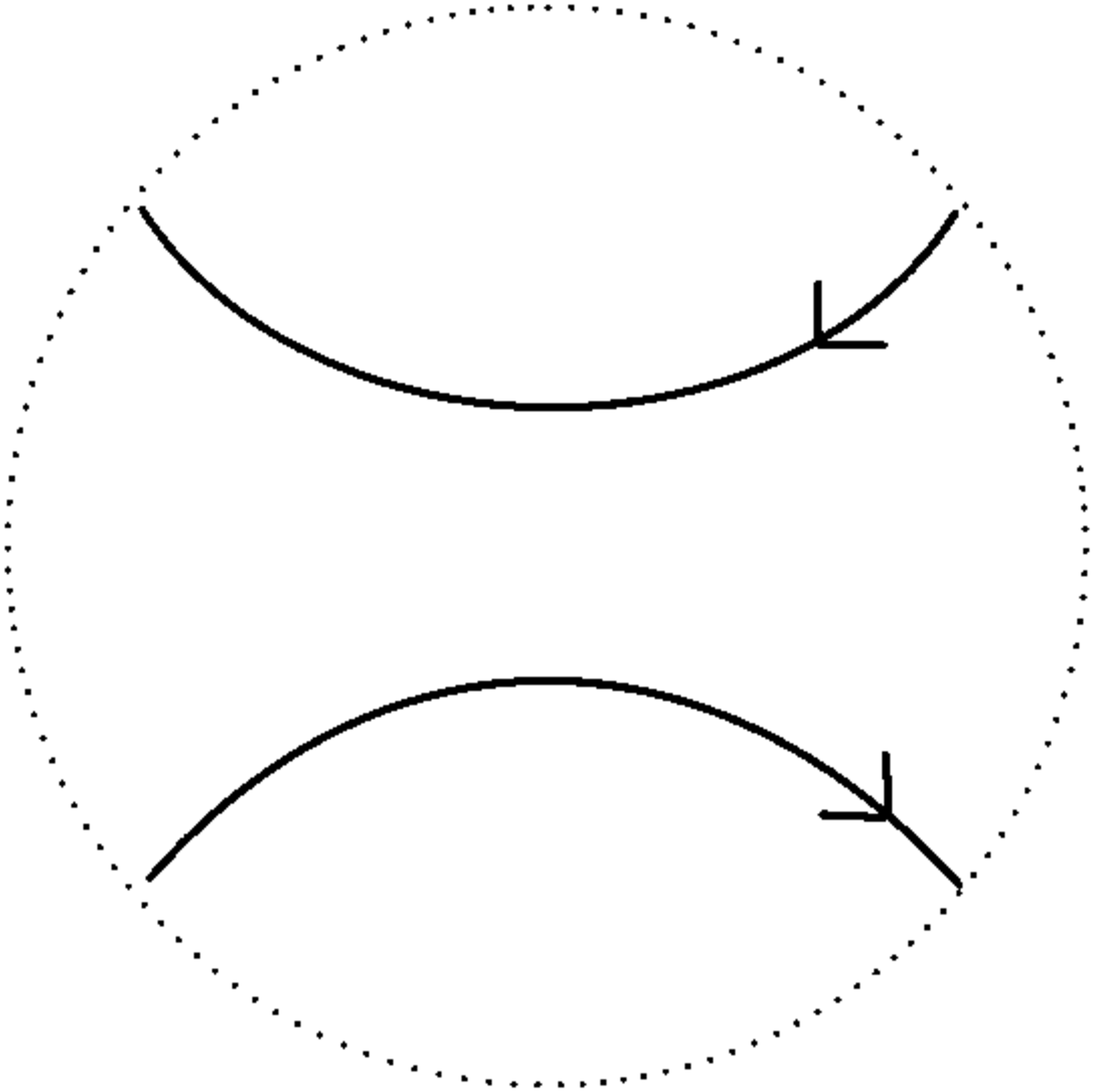}\end{minipage}\hspace{52pt}\Biggr\rangle_{3}+\Biggl\langle \begin{minipage}{1\unitlength}\includegraphics[scale=0.07]{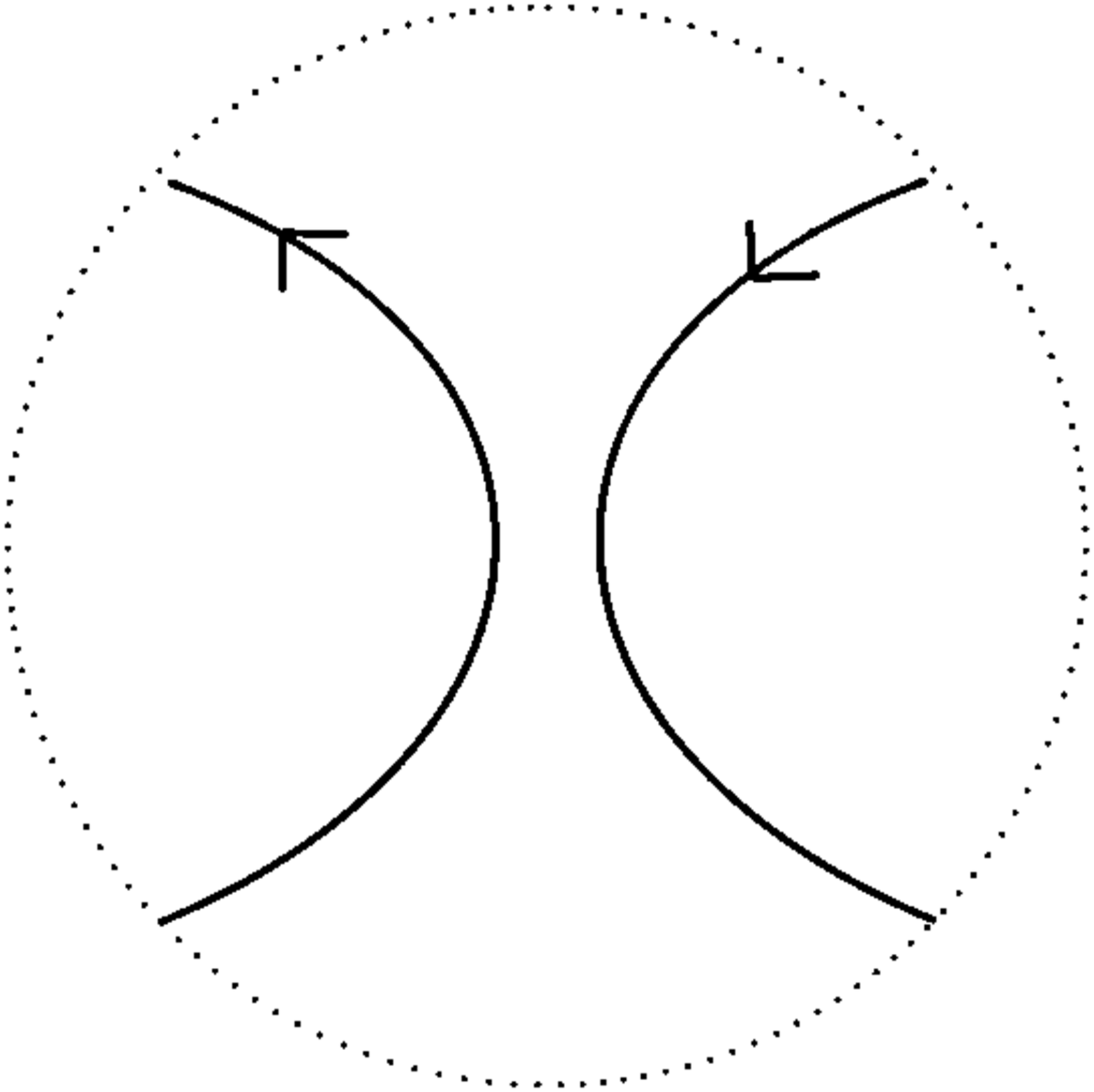}\end{minipage}\hspace{52pt}\Biggr\rangle_{3}\\
 &\Biggl\langle \begin{minipage}{1\unitlength}\includegraphics[scale=0.07]{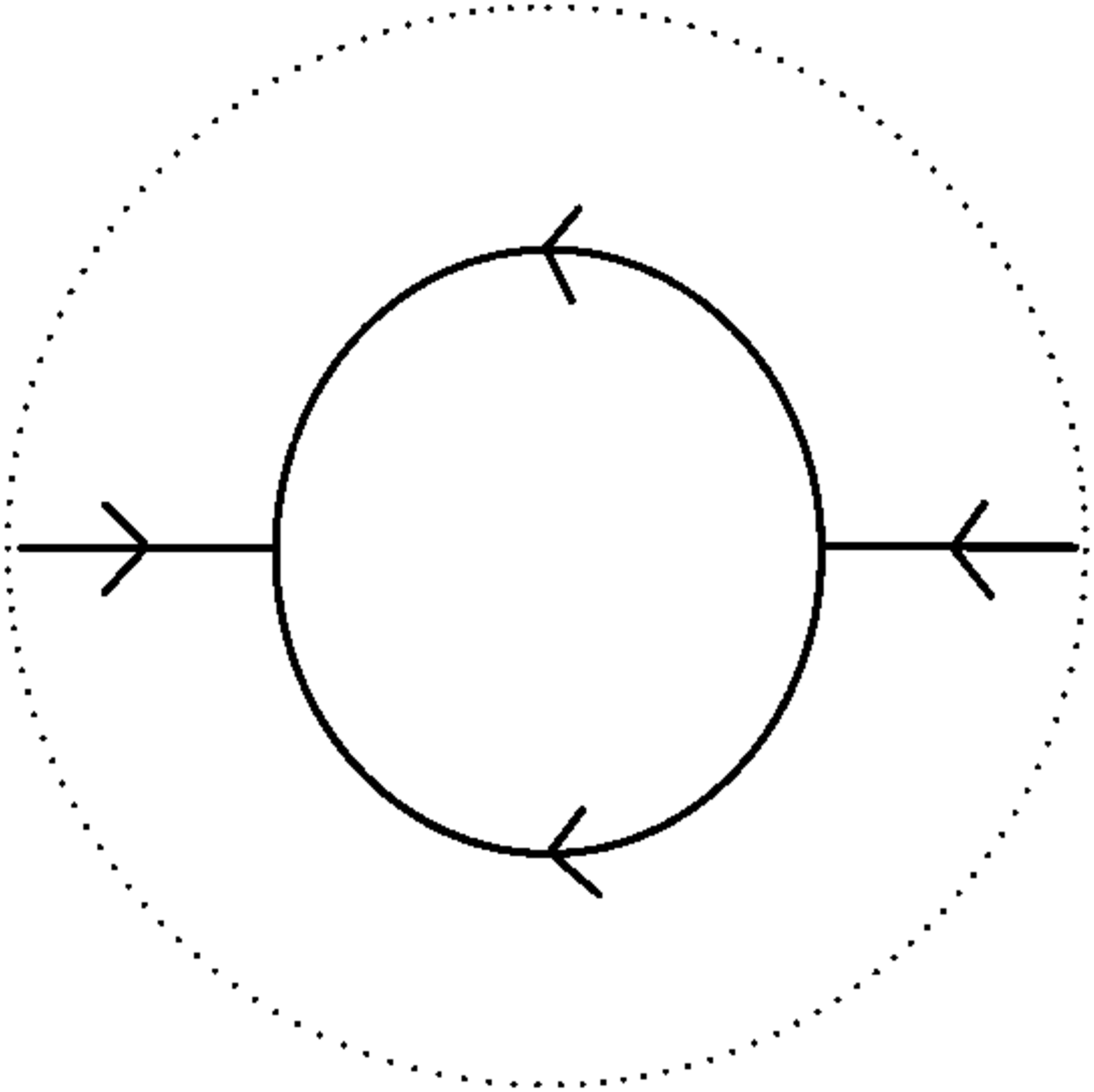}\end{minipage}\hspace{52pt}\Biggr\rangle_{3}=[2]_{q}\Biggl\langle \begin{minipage}{1\unitlength}\includegraphics[scale=0.07]{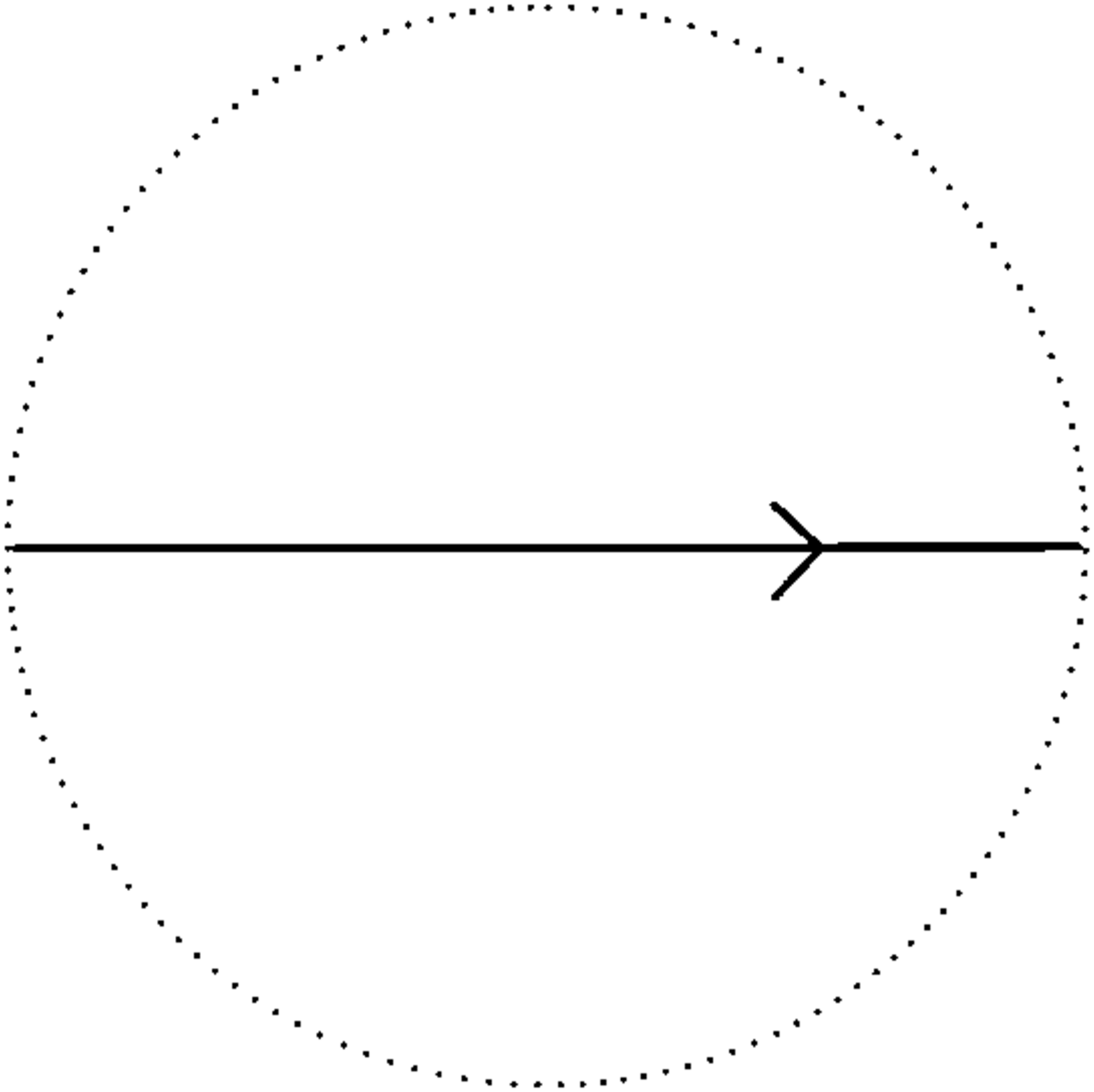}\end{minipage}\hspace{52pt}\Biggr\rangle_{3},\\
 &\Biggl\langle G\cup \begin{minipage}{1\unitlength}\includegraphics[scale=0.07]{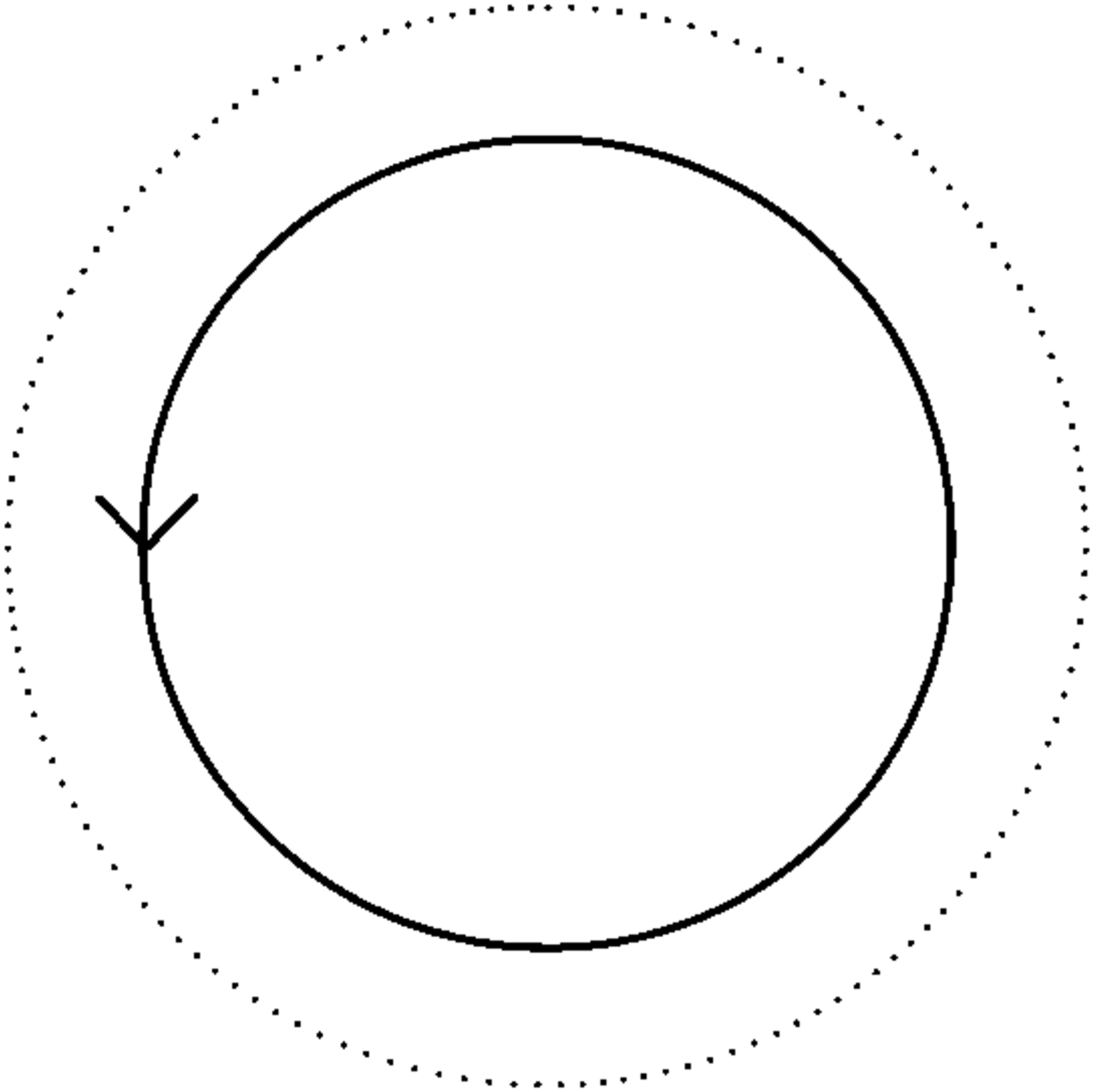}\end{minipage}\hspace{52pt}\Biggr\rangle_{3}=[3]_{q}G
\end{align*}
where $G$ is a bipartite uni-tri valent graph.
\end{Definition}
We can easily show that this linear map is invariant under the regular isotopy. We next introduce $A_{2}$ clasps of type $(n,0)$ according to \cite{Kup96} and \cite{Oht97}. Let $P$ be the set $\{ p_{1},p_{2},...,p_{2n} \}$ where $p_{1},p_{2},...,p_{2n}$ are elements of $P$ aligned counterclockwise on $\partial D$ from $p_{1}$ in order of decreasing  subscript of $p_{i}$, and $\epsilon$ is defined by
\begin{equation}
\epsilon(p_{i}) = 
\begin{cases}
    + & (1\le i \le n) \\
    -  & (n+1\le 2n) 
\end{cases}.
\end{equation}
We denote $W(P;\epsilon)$ as $W_{n^{+}+n^{-}}$.
\begin{Definition}
Let $n$ be a positive intger. We define the $A_{2}$ clasps of type $(n,0)$ by
\begin{align*}
&\Biggl\langle \scriptsize\begin{minipage}{1\unitlength}\includegraphics[scale=0.07]{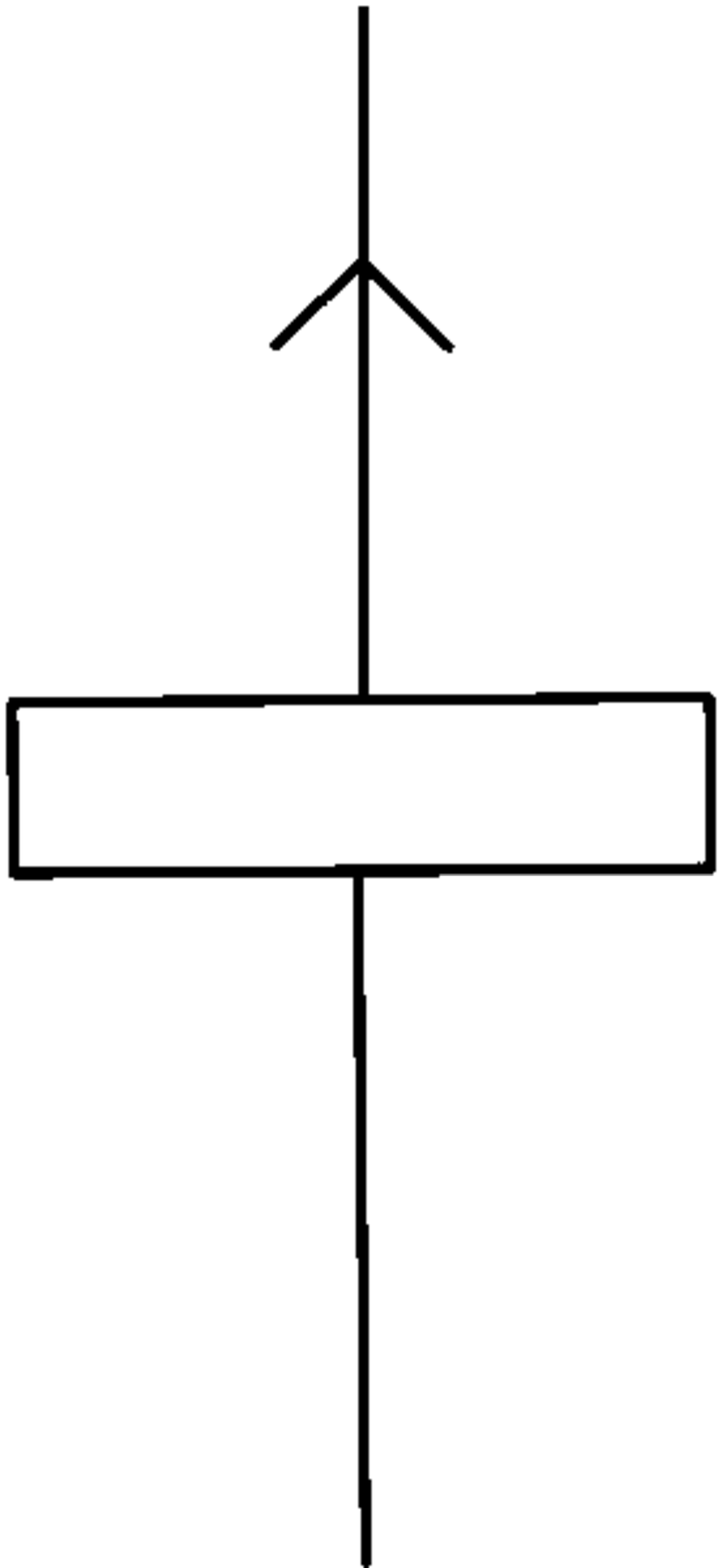}\put(-20,45){$1$}\end{minipage}\hspace{48pt}\normalsize\Biggr\rangle_{3}=\Biggl\langle \scriptsize\begin{minipage}{1\unitlength}\includegraphics[scale=0.07]{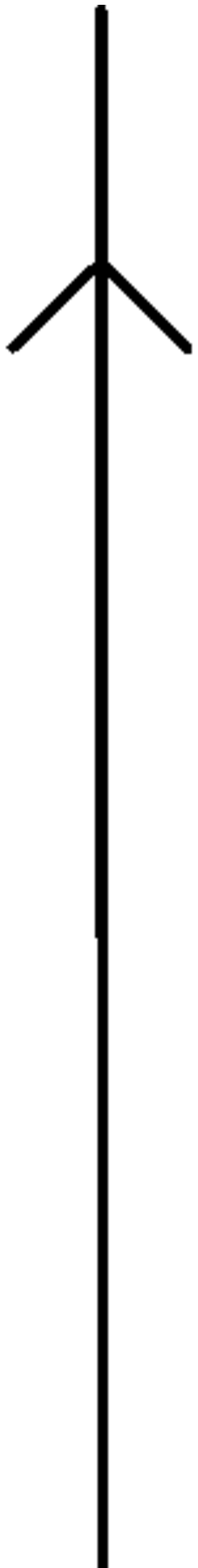}\put(-20,45){$1$}\end{minipage}\hspace{48pt}\normalsize\Biggr\rangle_{3}\in W_{1^{+}+1^{-}},\\
&\Biggl\langle \scriptsize\begin{minipage}{1\unitlength}\includegraphics[scale=0.07]{pic/JonesWenzel1.eps}\put(-20,45){$n$}\end{minipage}\hspace{48pt}\normalsize\Biggr\rangle_{3}=\Biggl\langle \begin{minipage}{1\unitlength}\includegraphics[scale=0.07]{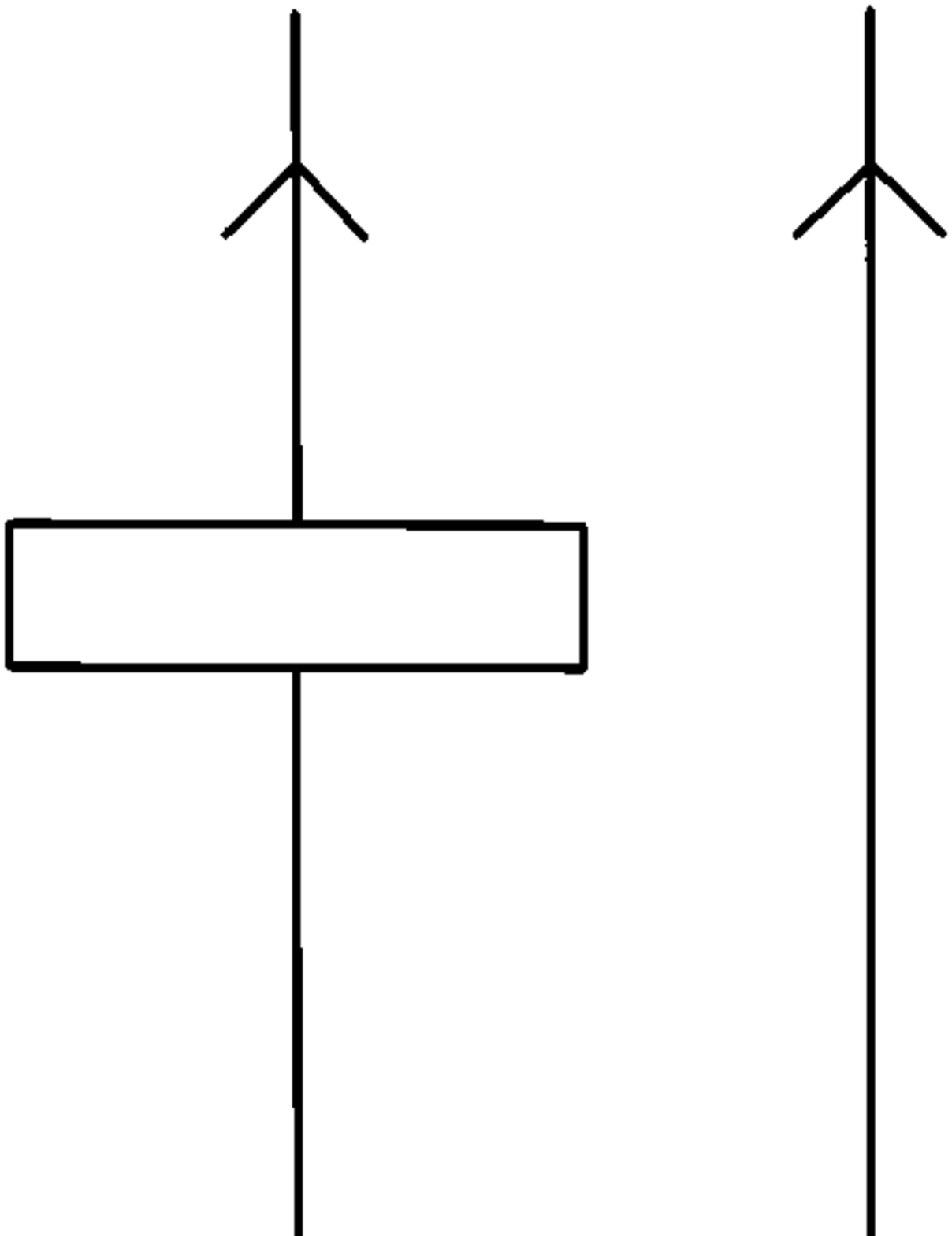}\put(-30,45){\scriptsize$n-1$\normalsize}\put(-5,45){\scriptsize$1$\normalsize}\end{minipage}\hspace{48pt}\Biggr\rangle_{3}-\frac{[n-1]_{q}}{[n]_{q}}\Biggl\langle \hspace{12pt}\begin{minipage}{1\unitlength}\includegraphics[scale=0.07]{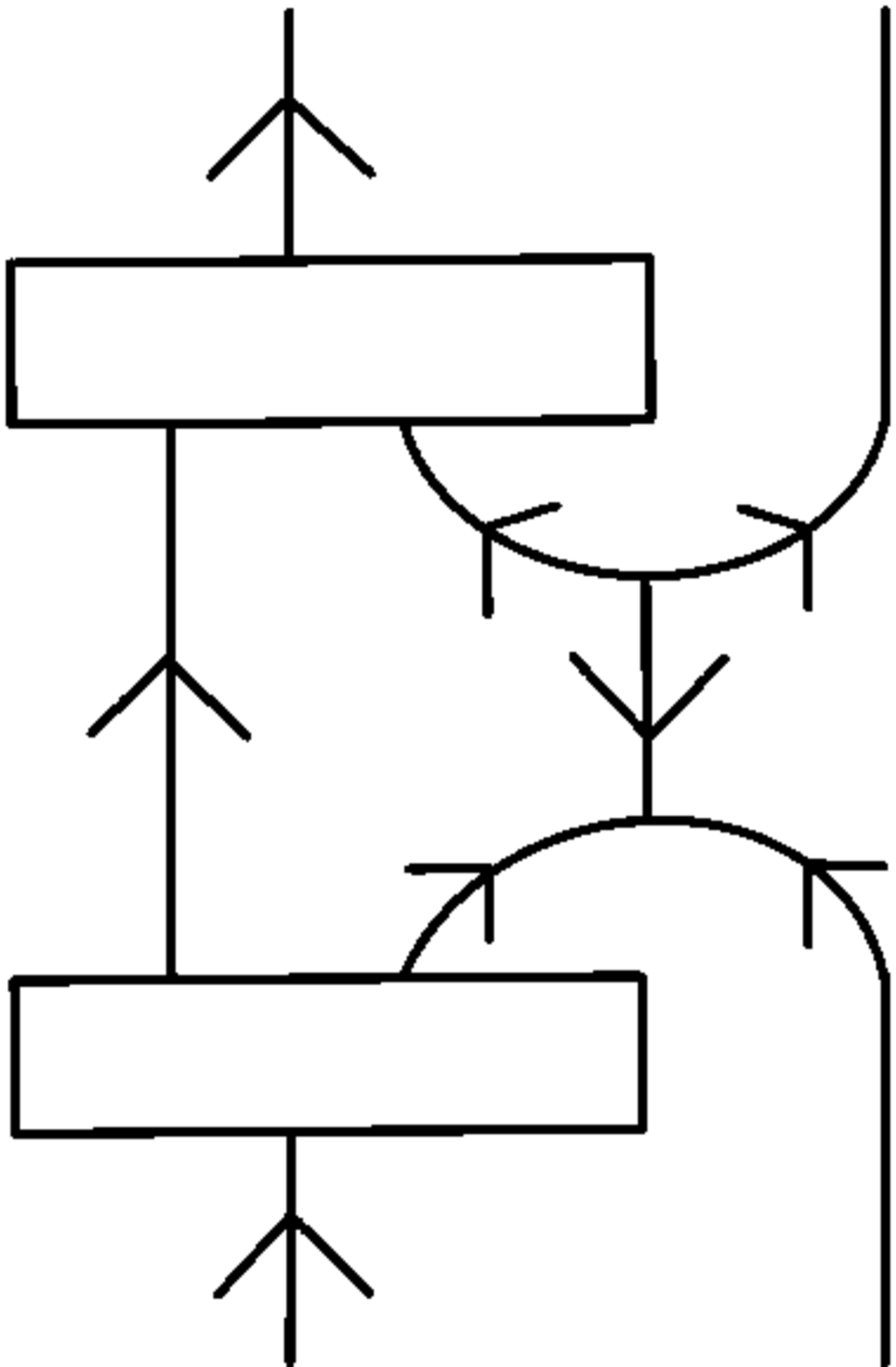}\put(-25,45){\scriptsize$n$\normalsize}\put(-25,0){\scriptsize$n$\normalsize}\put(-49,26){\scriptsize$n-2$\normalsize}\put(-2,45){\scriptsize$1$\normalsize}\put(-2,0){\scriptsize$1$\normalsize}\end{minipage}\hspace{48pt}\Biggr\rangle_{3}\in W_{n^{+}+n^{-}}.
\end{align*}
\end{Definition}

The following properties hold for $A_{2}$ clasps.
\begin{Lemma}[\cite{Kup96}]
For any posintve integer $n$, we have
\begin{align}
\label{al:double1}
&\Biggl\langle \scriptsize\begin{minipage}{1\unitlength}\includegraphics[scale=0.07]{pic/JonesWenzel1.eps}\put(-15,45){\scriptsize$n$\normalsize}\end{minipage}\hspace{48pt}\normalsize\Biggr\rangle_{3}=\Biggl\langle \begin{minipage}{1\unitlength}\includegraphics[scale=0.07]{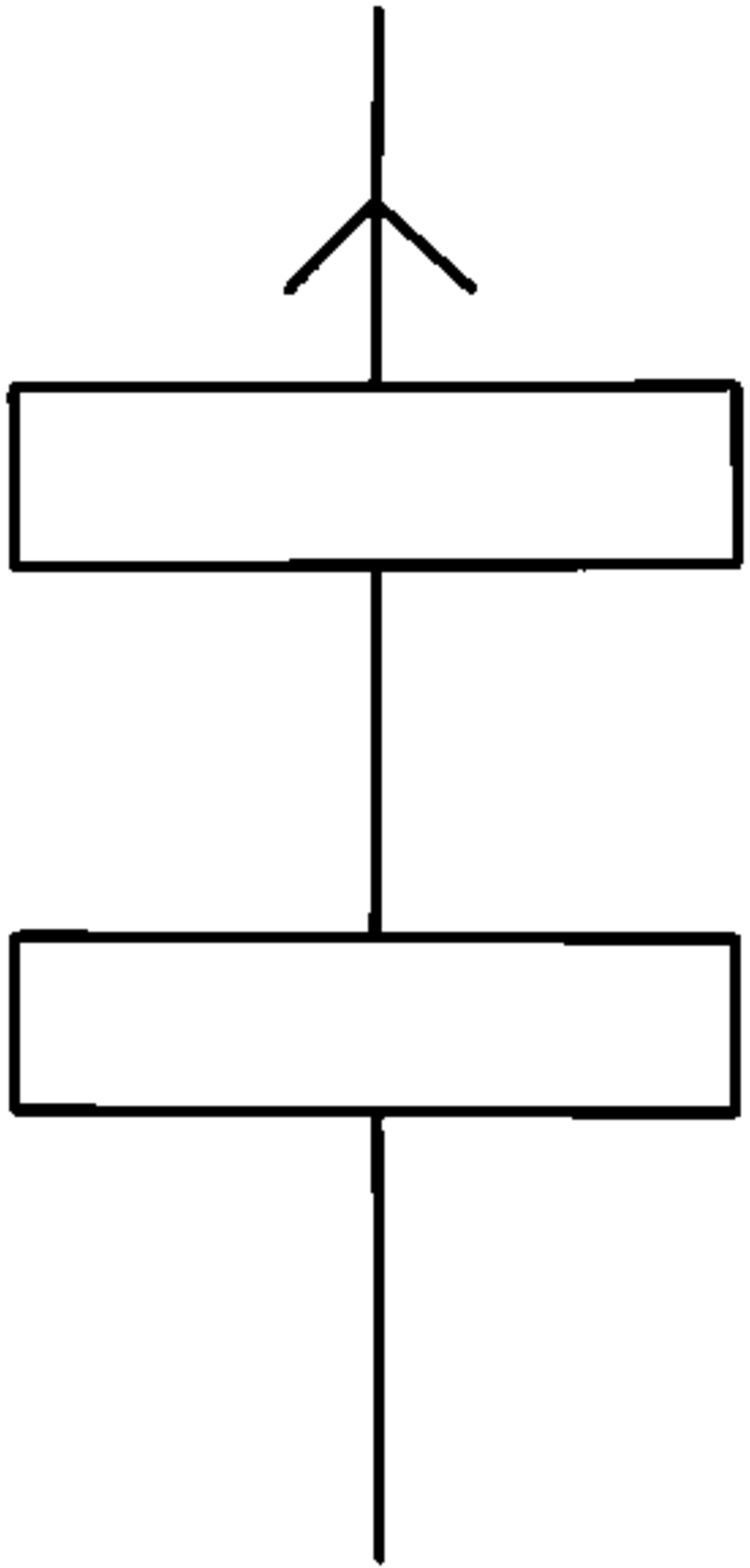}\put(-15,45){\scriptsize$n$\normalsize}\end{minipage}\hspace{48pt}\Biggr\rangle_{3},\\
\label{al:double2}
&\Biggl\langle\begin{minipage}{1\unitlength}\includegraphics[scale=0.07]{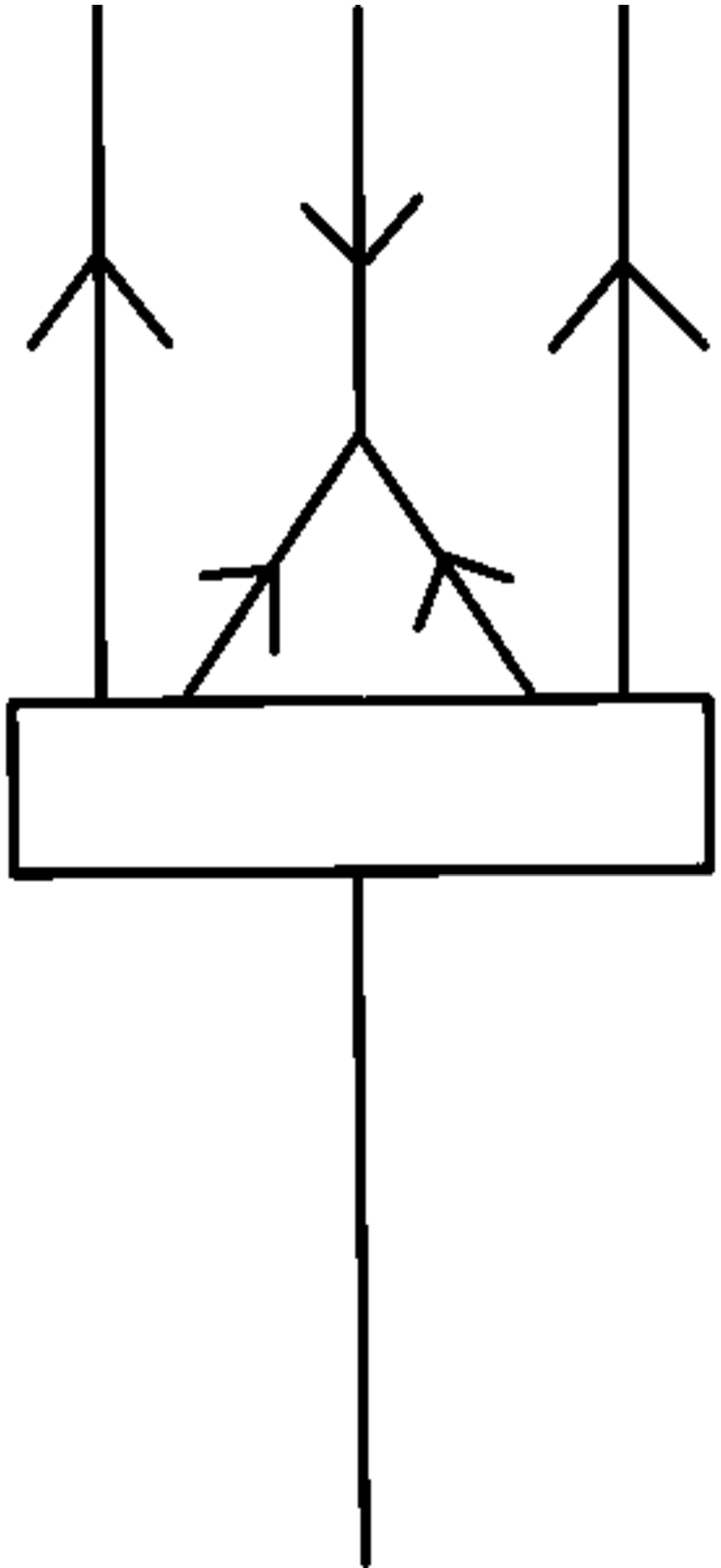}\put(-61,38){\scriptsize$n-k-2$\normalsize}\put(-33,42){\scriptsize$1$\normalsize}\put(-10,38){\scriptsize$k$\normalsize}\end{minipage}\hspace{59pt}\Biggr\rangle_{3}=0 \quad (k=0,1,..,n-2).
\end{align}
Here, a strand labeled by the number $n$ implies the $n$ parallel copies of the strand.
\end{Lemma}

We also introduce the $A_{2}$ clasp of type $(n_{1},n_{2})$ according to \cite{Oht97}.
\begin{Definition}[the $A_{2}$ clasp of type $(n_{1},n_{2})$ \cite{Oht97}]
Let $n_{1}$ and $n_{2}$ be non-negative integers. We define the $A_{2}$ clasp of type $(n_{1},n_{2})$ by
\begin{align*}
\Biggl\langle \scriptsize\begin{minipage}{1\unitlength}\includegraphics[scale=0.07]{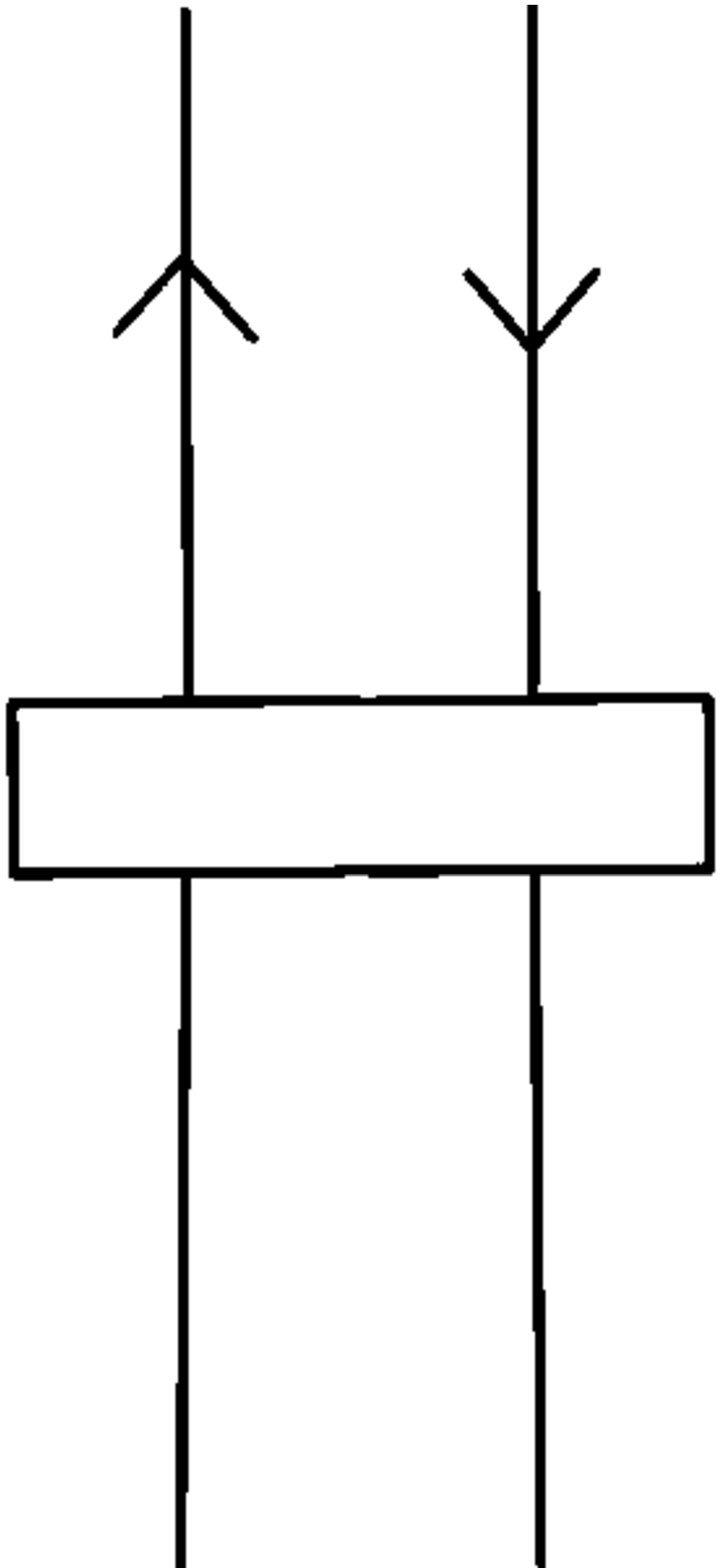}\put(-30,45){$n_{1}$}\put(-15,45){$n_{2}$}\end{minipage}\hspace{45pt}\normalsize\Biggr\rangle_{3}=\sum_{i=0}^{\min\{n_{1}, n_{2}\}}(-1)^{i}\frac{\begin{bmatrix}
n_{1}  \\
i \\
\end{bmatrix}_{q}\begin{bmatrix}
n_{2}  \\
i \\
\end{bmatrix}_{q}}{\begin{bmatrix}
n_{1}+n_{2}+1  \\
i \\
\end{bmatrix}_{q}}\Biggl\langle \scriptsize\begin{minipage}{1\unitlength}\includegraphics[scale=0.07]{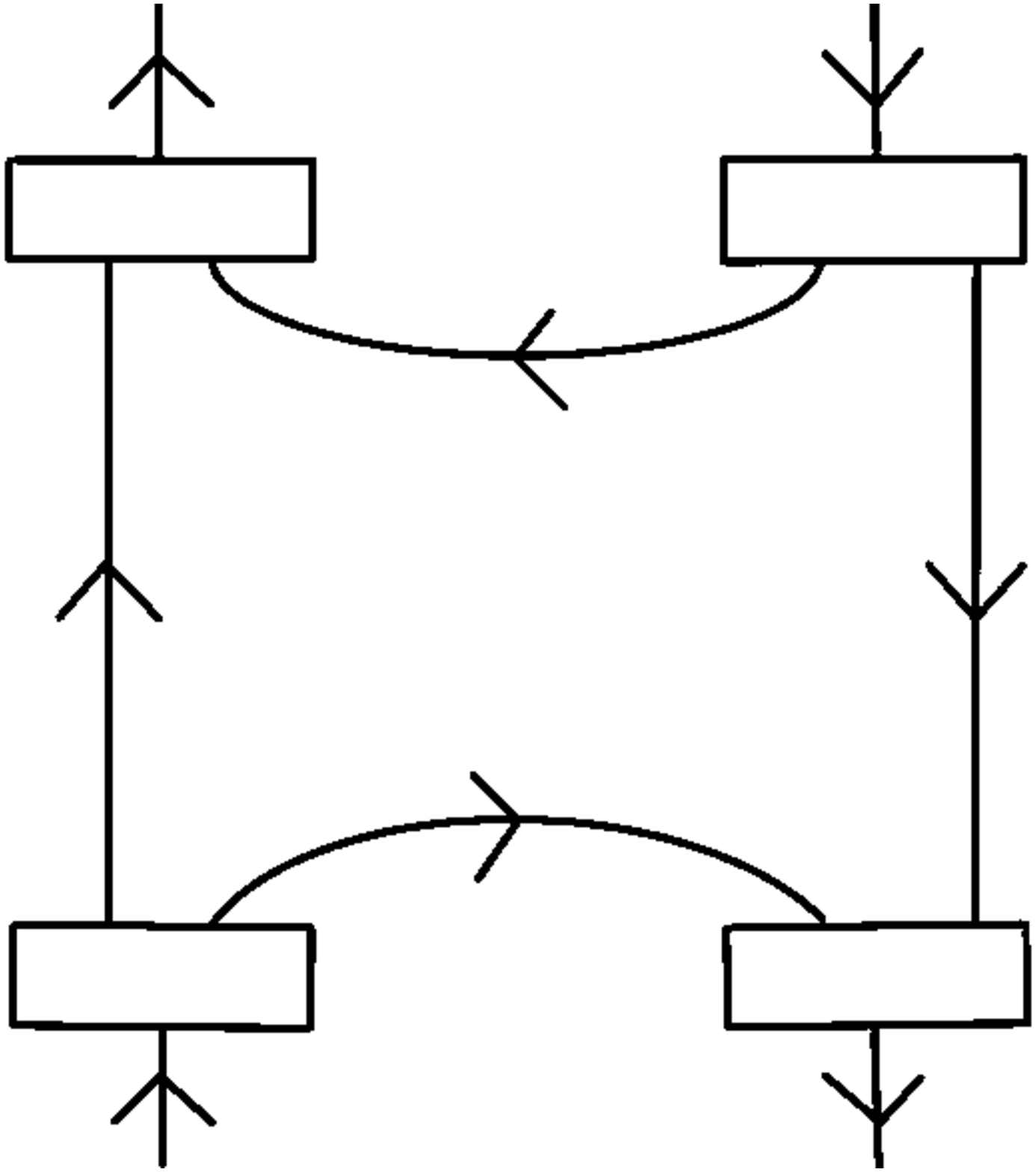}\put(-45,55){\scriptsize$n_{1}$\normalsize}\put(-10,55){\scriptsize$n_{2}$\normalsize}\put(-23,30){\scriptsize$i$\normalsize}\put(-32,17){\scriptsize$i$\normalsize}\end{minipage}\hspace{60pt}\normalsize\Biggr\rangle_{3}\in W_{n_{1}^{+}+n_{2}^{-}+n_{1}^{-}+n_{2}^{+}}.
\end{align*}
\end{Definition}

Ohtsuki, Yamda\cite{Oht97} and Yuasa\cite{Yua17} gave formulae for $A_{2}$ clasps of type $(n,0)$.
\begin{Lemma}[\cite{Oht97}\cite{Yua17}]
For $k=0, 1, ...,n$, we have
\label{Lem:double}
\begin{align}
\label{al:double3}
&\Biggl\langle \scriptsize\begin{minipage}{1\unitlength}\includegraphics[scale=0.07]{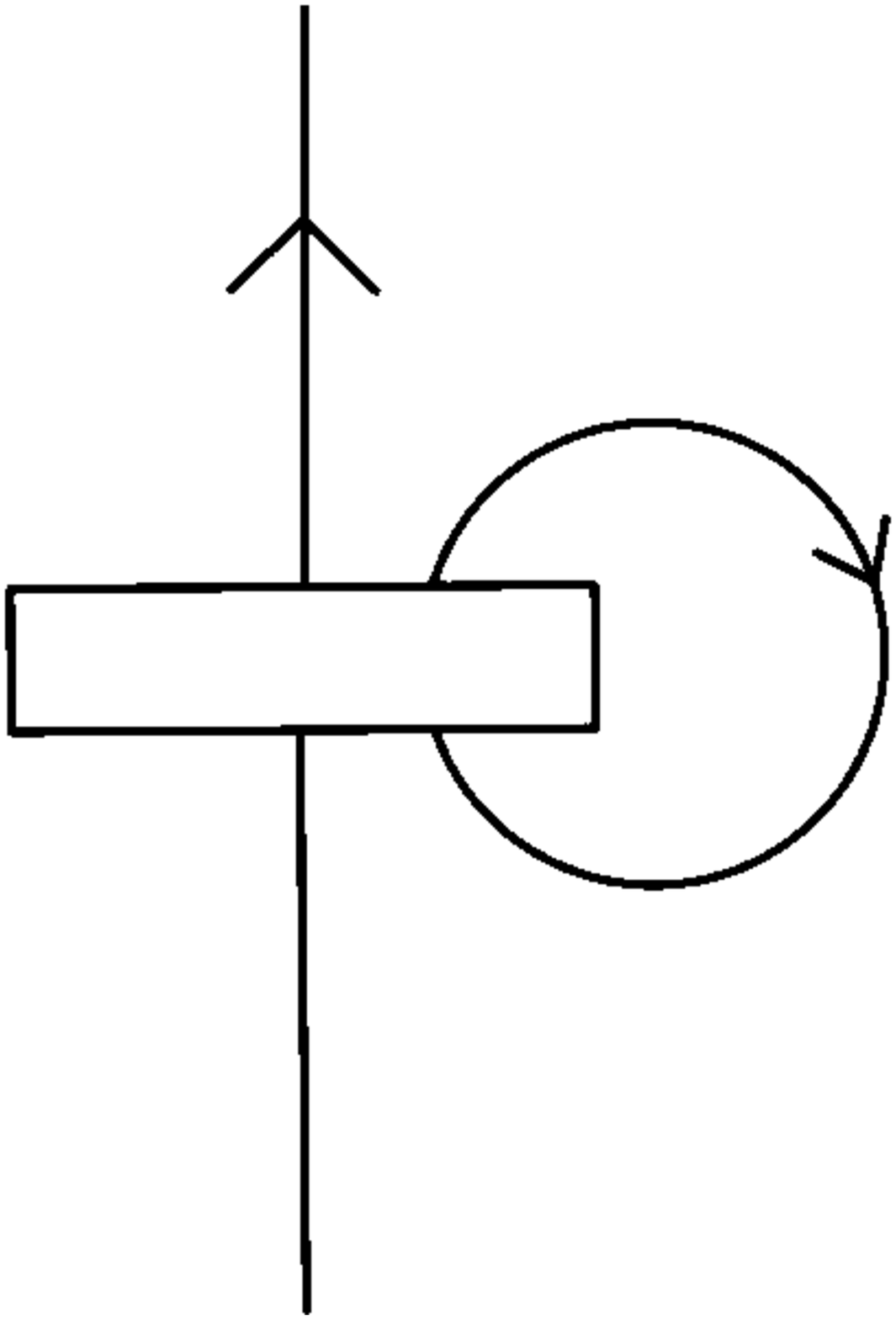}\put(-18,45){\scriptsize$n-k$\normalsize}\put(-8,38){\scriptsize$k$\normalsize}\end{minipage}\hspace{55pt}\normalsize\Biggr\rangle_{3}=\frac{[n+1]_{q}[n+2]_{q}}{[n-k+1]_{q}[n-k+2]_{q}}\Biggl\langle \begin{minipage}{1\unitlength}\includegraphics[scale=0.07]{pic/JonesWenzel1.eps}\put(-18,45){\scriptsize$n-k$\normalsize}\end{minipage}\hspace{48pt}\Biggr\rangle_{3}\\
\label{al:double4}
&\Biggl\langle \scriptsize\begin{minipage}{1\unitlength}\includegraphics[scale=0.07]{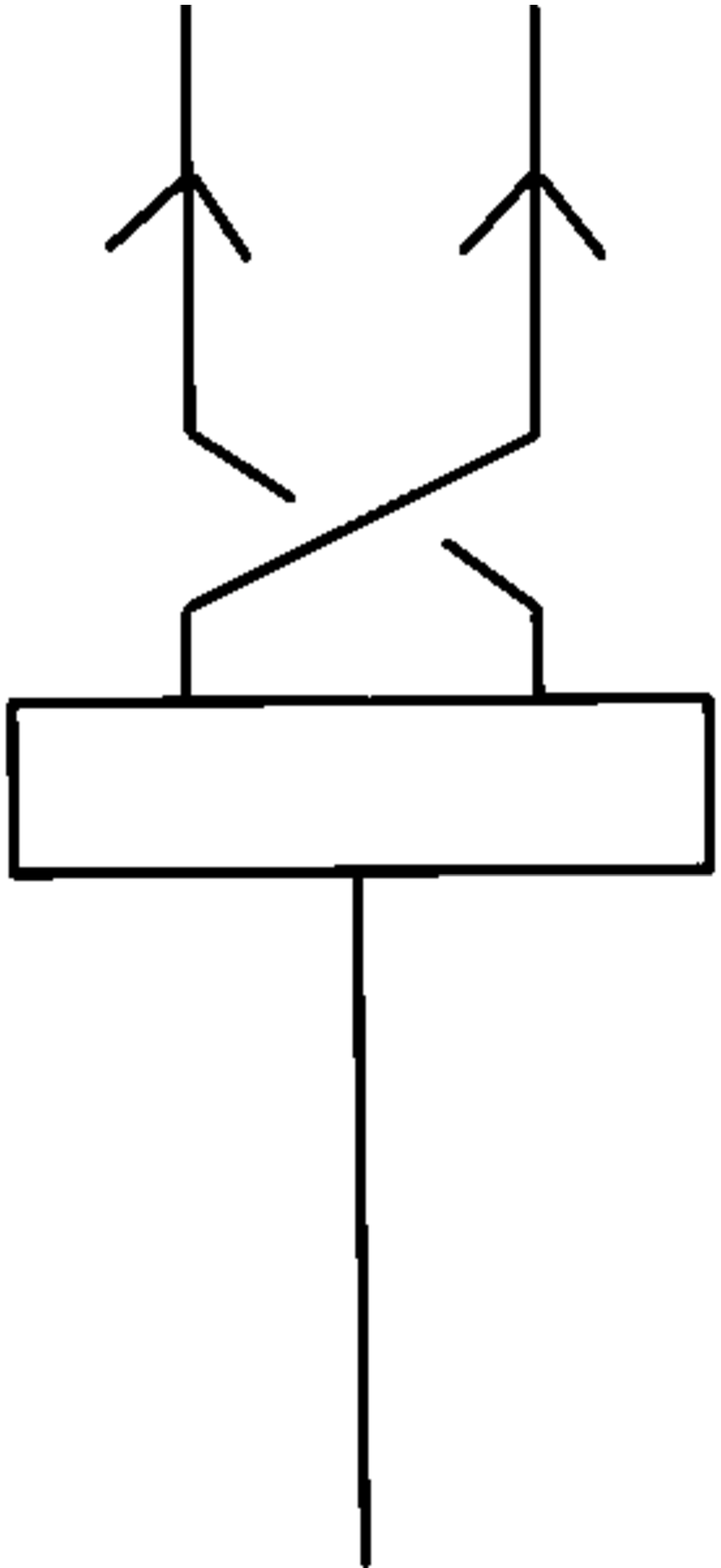}\put(-30,45){\scriptsize$k$\normalsize}\put(-15,45){\scriptsize$n-k$\normalsize}\end{minipage}\hspace{48pt}\normalsize\Biggr\rangle_{3}=q^{\frac{k(n-k)}{3}}\Biggl\langle \begin{minipage}{1\unitlength}\includegraphics[scale=0.07]{pic/JonesWenzel1.eps}\put(-15,45){\scriptsize$n$\normalsize}\end{minipage}\hspace{48pt}\Biggr\rangle_{3}, \quad\Biggl\langle \scriptsize\begin{minipage}{1\unitlength}\includegraphics[scale=0.07]{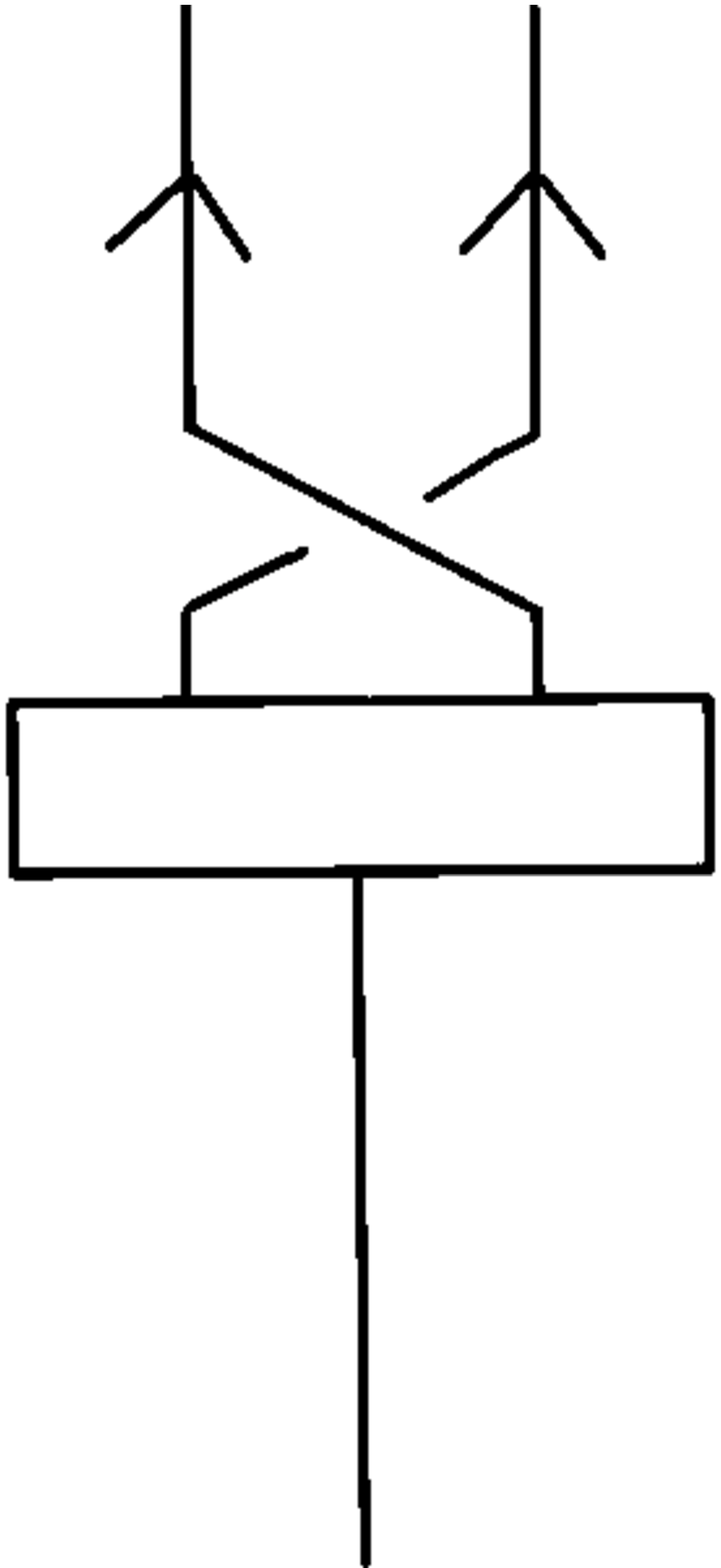}\put(-30,45){\scriptsize$k$\normalsize}\put(-15,45){\scriptsize$n-k$\normalsize}\end{minipage}\hspace{48pt}\normalsize\Biggr\rangle_{3}=q^{-\frac{k(n-k)}{3}}\Biggl\langle \begin{minipage}{1\unitlength}\includegraphics[scale=0.07]{pic/JonesWenzel1.eps}\put(-20,45){\scriptsize$n$\normalsize}\end{minipage}\hspace{48pt}\Biggr\rangle_{3}\\
\label{al:double5}
&\Biggl\langle \scriptsize\begin{minipage}{1\unitlength}\includegraphics[scale=0.07]{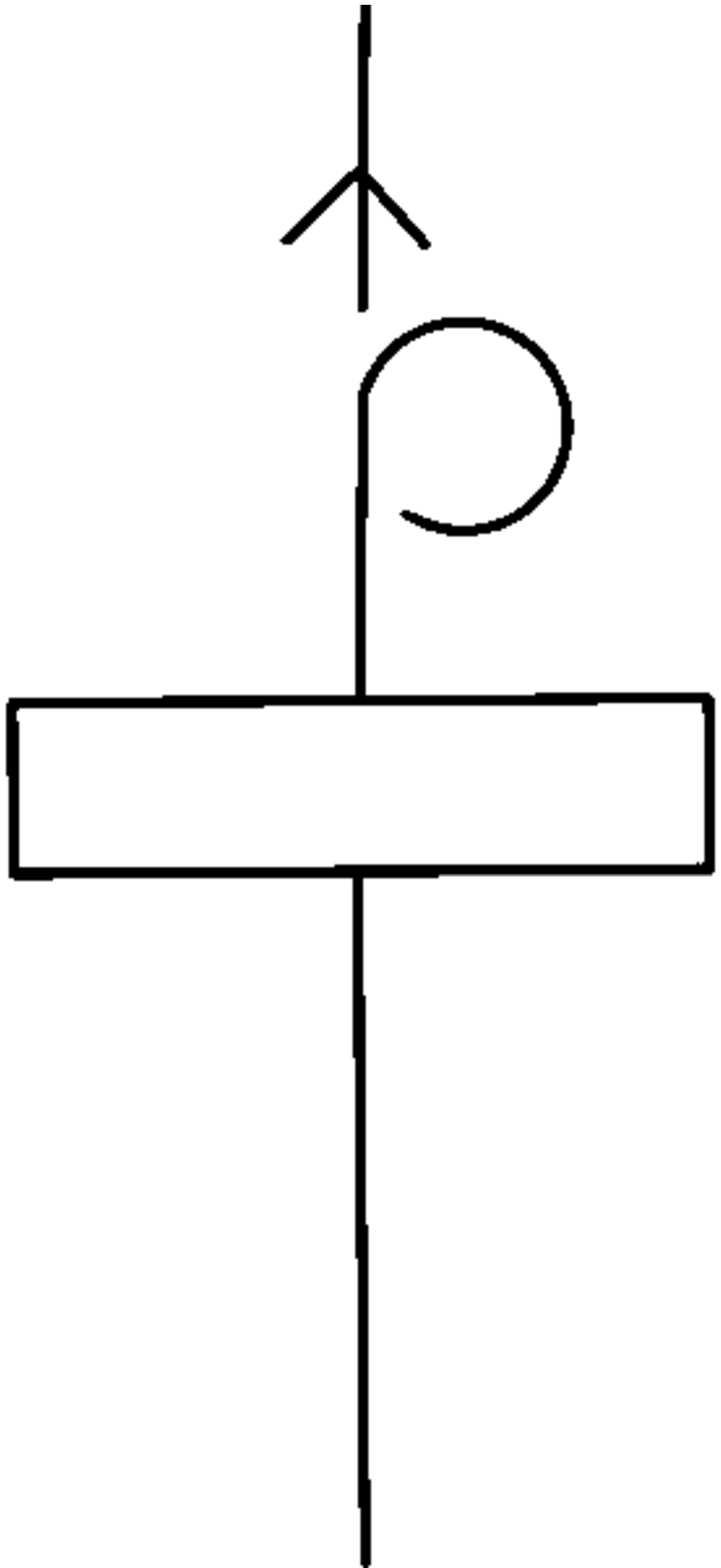}\put(-20,45){\scriptsize$n$\normalsize}\end{minipage}\hspace{48pt}\normalsize\Biggr\rangle_{3}=q^{\frac{n^{2}+3n}{3}}\Biggl\langle \begin{minipage}{1\unitlength}\includegraphics[scale=0.07]{pic/JonesWenzel1.eps}\put(-20,45){\scriptsize$n$\normalsize}\end{minipage}\hspace{48pt}\Biggr\rangle_{3}, \quad\Biggl\langle \scriptsize\begin{minipage}{1\unitlength}\includegraphics[scale=0.07]{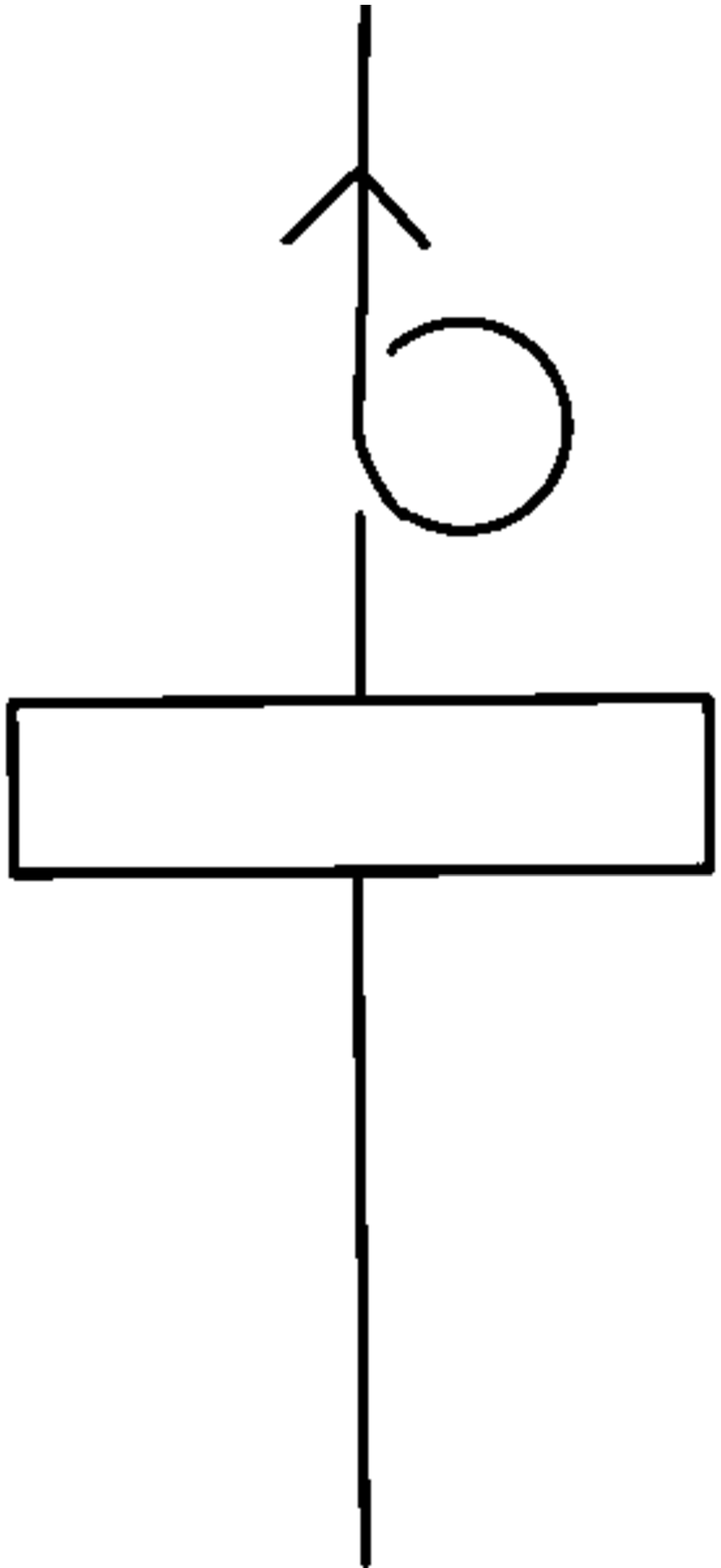}\put(-20,45){\scriptsize$n$\normalsize}\end{minipage}\hspace{48pt}\normalsize\Biggr\rangle_{3}=q^{-\frac{n^{2}+3n}{3}}\Biggl\langle \begin{minipage}{1\unitlength}\includegraphics[scale=0.07]{pic/JonesWenzel1.eps}\put(-20,45){\scriptsize$n$\normalsize}\end{minipage}\hspace{48pt}\Biggr\rangle_{3}\\
\label{al:double6}
&\Delta(n,0)=\Biggl\langle\begin{minipage}{1\unitlength}\includegraphics[scale=0.07]{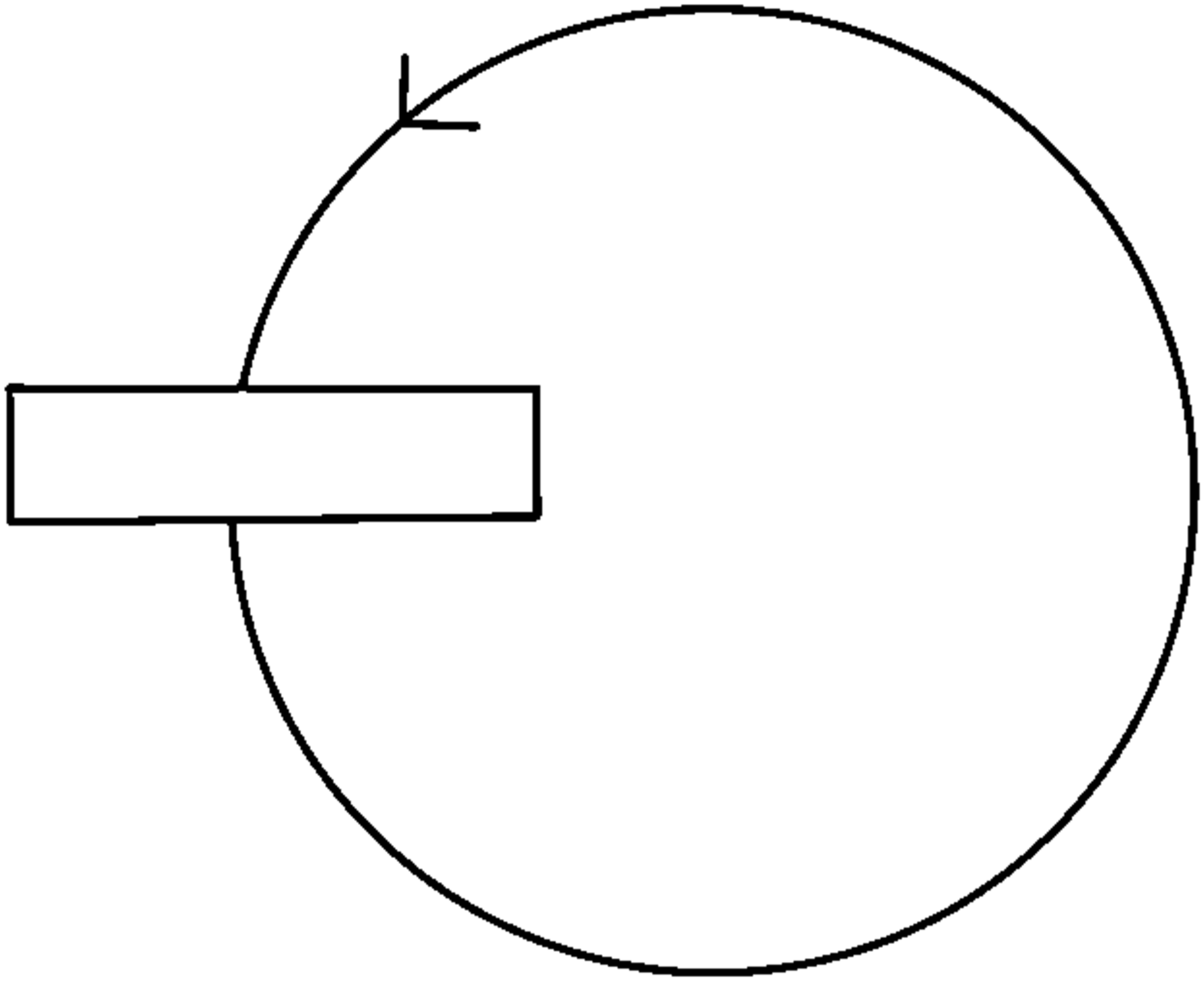}\put(-48,38){\scriptsize$n$\normalsize}\end{minipage}\hspace{59pt}\Biggr\rangle_{3}=\frac{[n+1]_{q}[n+2]_{q}}{[2]_{q}}\\
\label{al:double7}
&\Biggr\langle\begin{minipage}{1\unitlength}\includegraphics[scale=0.1]{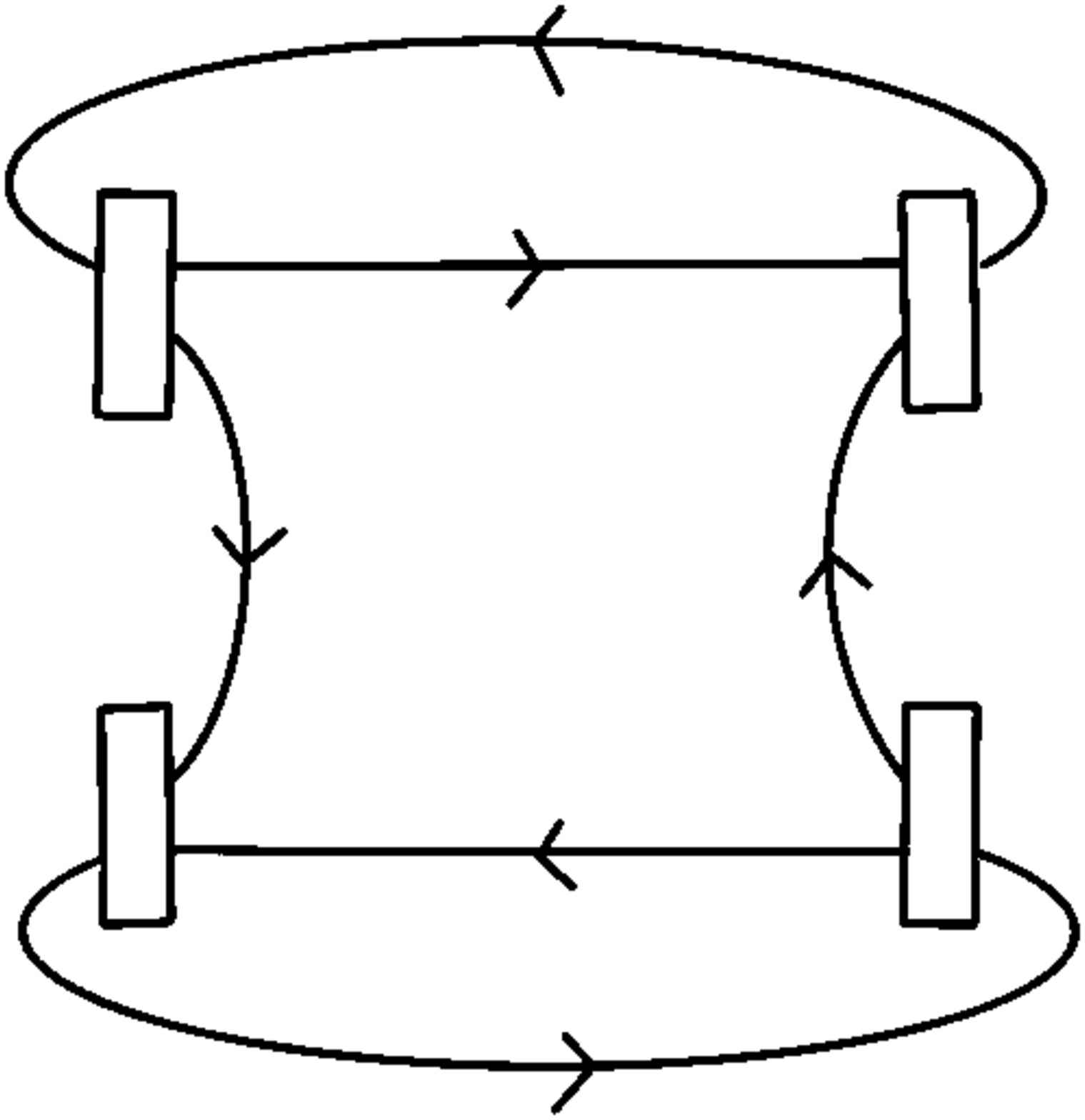}\put(-75,55){\scriptsize$n$\normalsize}\put(-65,35){\scriptsize$k$\normalsize}\put(-15,35){\scriptsize$k$\normalsize}\put(-40,58){\scriptsize$n-k$\normalsize}\put(-40,15){\scriptsize$n-k$\normalsize}\end{minipage}\hspace{70pt} \Biggr\rangle_{3}=q^{-(n-k)}\frac{(1-q^{n+1})(1-q^{n+2})}{(1-q^{k+1})(1-q^{k+2})}\Delta(n,0)
\end{align}
\end{Lemma}

We use the following two types of web, stair-step and triangle, which we can see in \cite{Kim06}, \cite{Kim07}, \cite{Yua17}, Etc.
\begin{Definition}
\label{stairStep}
For positive integers $n$ and $m$, stair-step\begin{minipage}{1\unitlength}\includegraphics[scale=0.08]{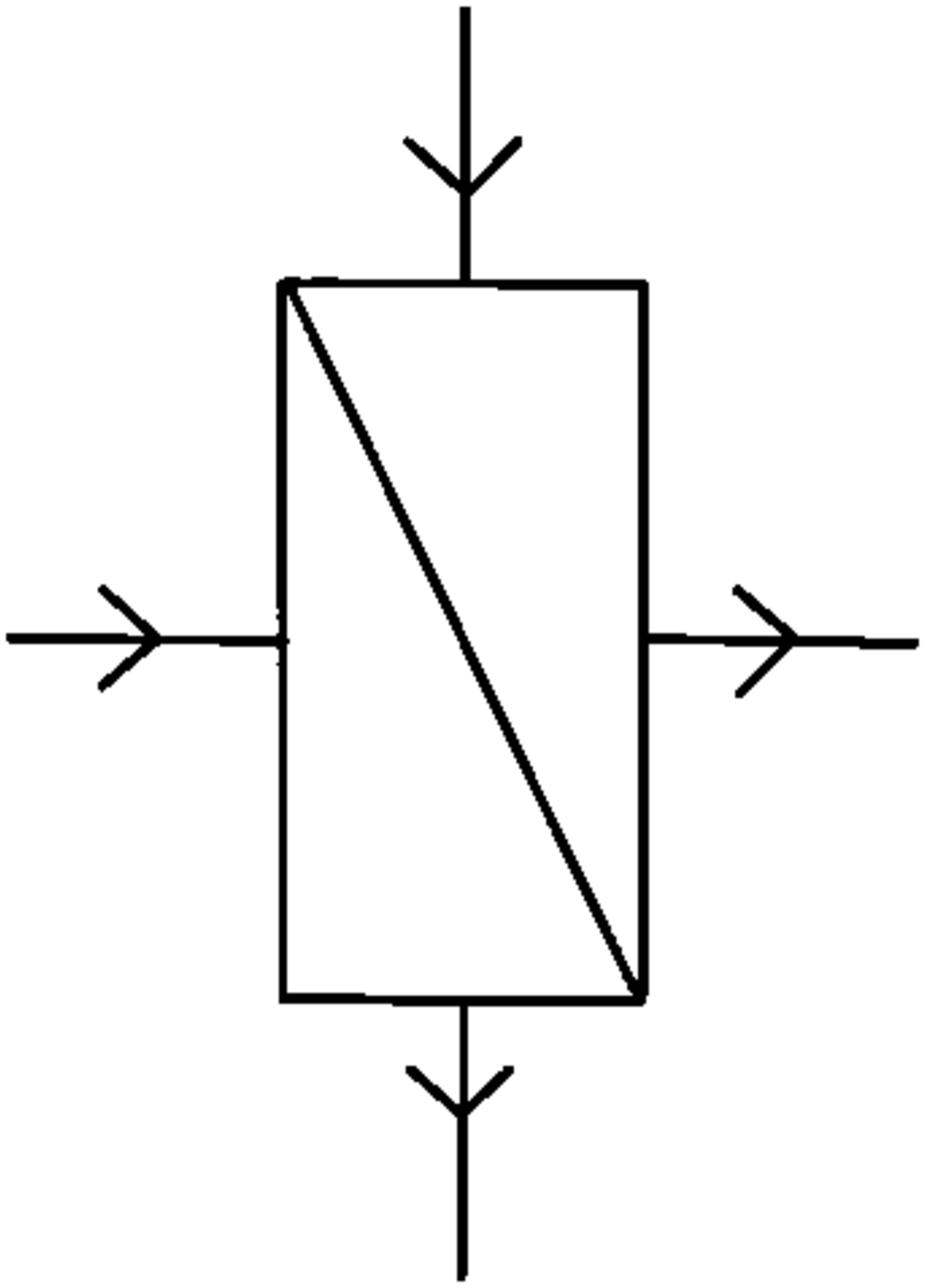}\put(-45,35){\scriptsize$n$\normalsize}\put(-32,58){\scriptsize$m$\normalsize}\end{minipage}\hspace{50pt} is defined by 
\begin{align*}
&\begin{minipage}{1\unitlength}\includegraphics[scale=0.08]{pic/yon.eps}\put(-45,35){\scriptsize$n$\normalsize}\put(-32,58){\scriptsize$1$\normalsize}\end{minipage}\hspace{60pt}=\hspace{20pt}\begin{minipage}{1\unitlength}\includegraphics[scale=0.08]{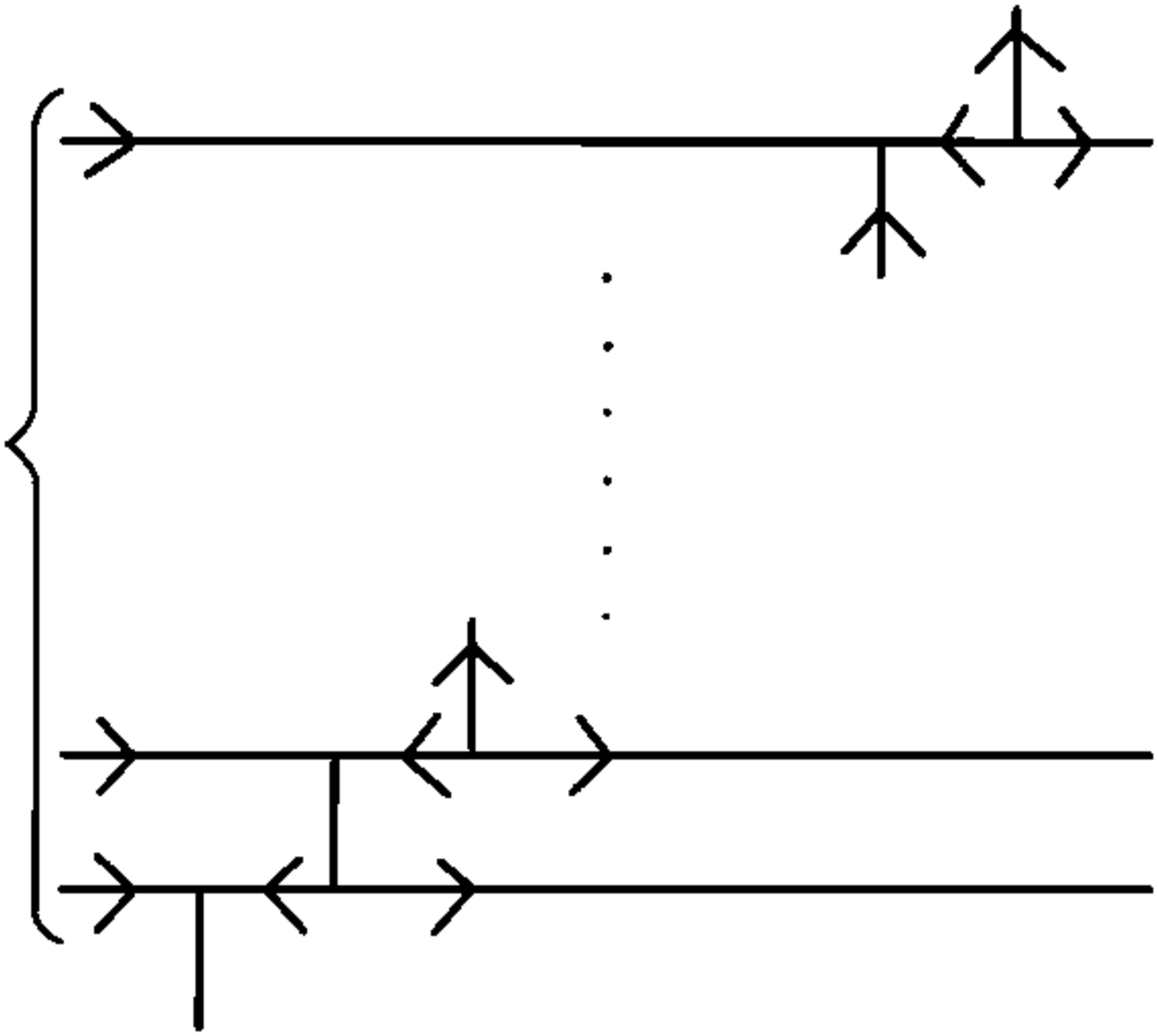}\put(-65,30){\scriptsize$n$\normalsize}\put(0,45){\scriptsize$1$\normalsize}\put(-60,52){\scriptsize$1$\normalsize}\end{minipage},\\
&\begin{minipage}{1\unitlength}\includegraphics[scale=0.08]{pic/yon.eps}\put(-45,35){\scriptsize$n$\normalsize}\put(-32,58){\scriptsize$m$\normalsize}\end{minipage}\hspace{60pt}=\hspace{10pt}\begin{minipage}{1\unitlength}\includegraphics[scale=0.08]{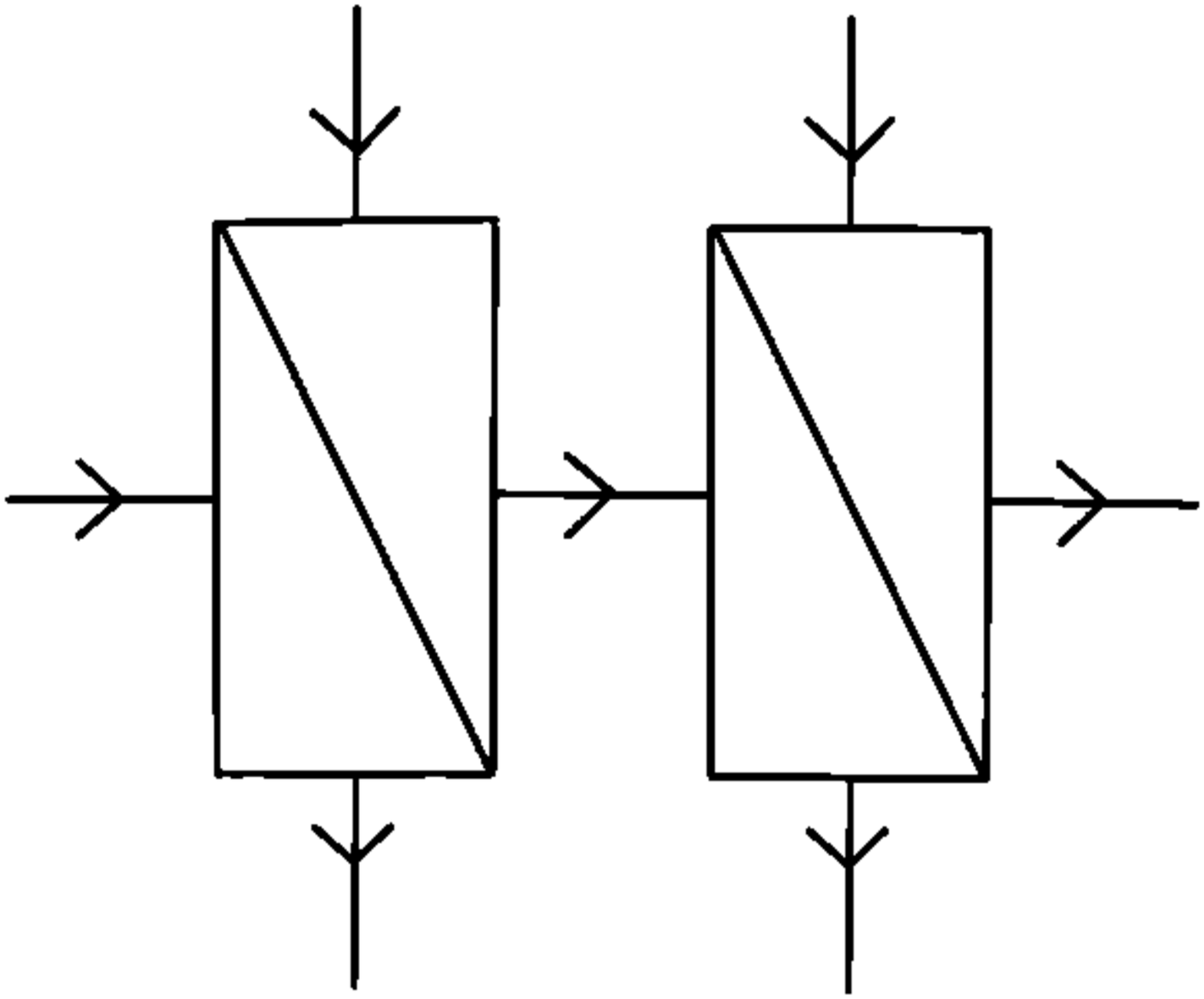}\put(-21,60){\scriptsize$1$\normalsize}\put(-52,60){\scriptsize$m-1$\normalsize}\put(-60,36){\scriptsize$n$\normalsize}\end{minipage}\hspace{80pt}\quad(m>1).
\end{align*}
For positive integers $n$, triangle \begin{minipage}{1\unitlength}\includegraphics[scale=0.06]{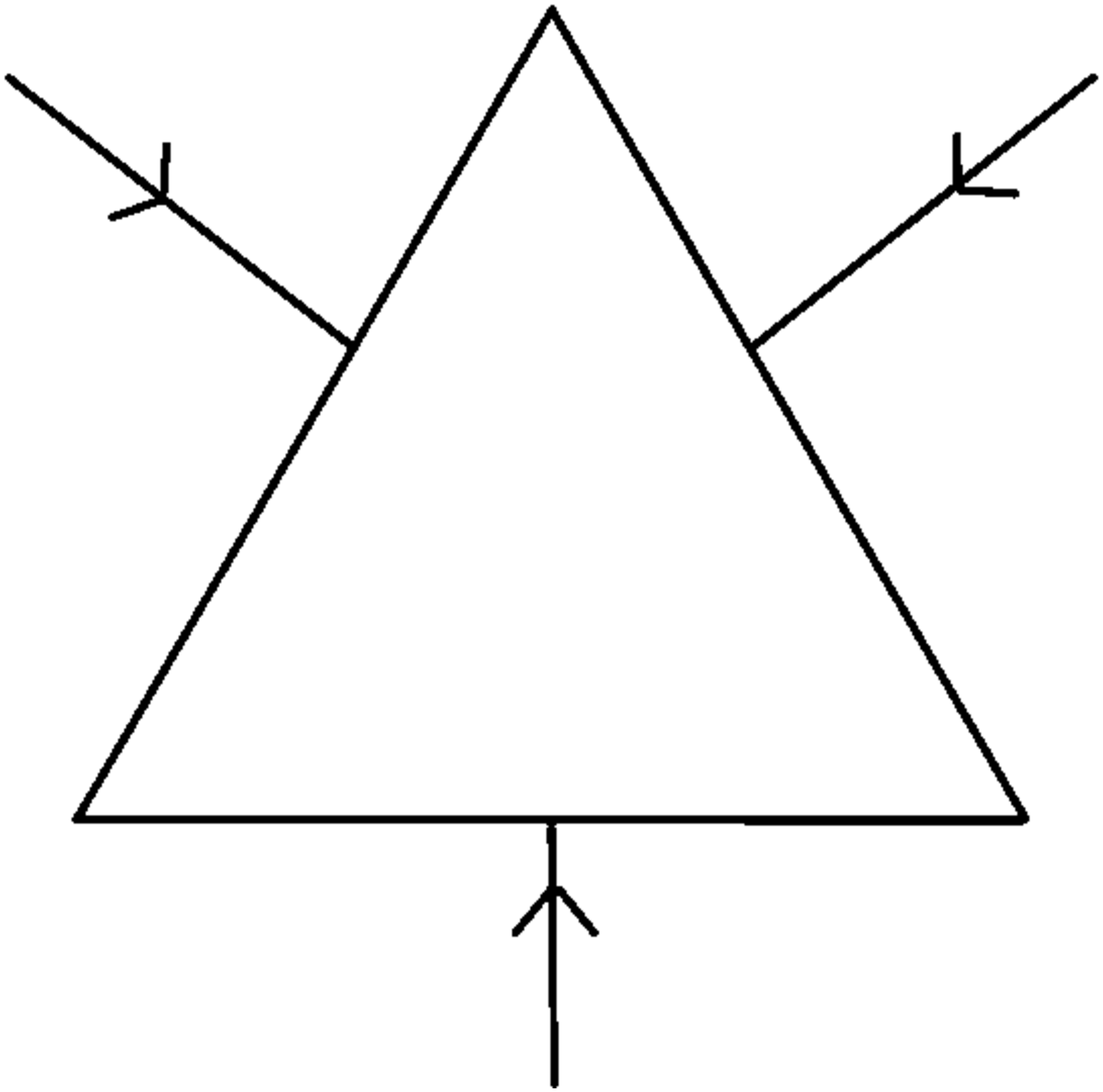}\put(0,40){\scriptsize$n$\normalsize}\put(-50,40){\scriptsize$n$\normalsize}\put(-40,2){\scriptsize$n$\normalsize}\end{minipage}\hspace{50pt} is defined by
\begin{align*}
&\begin{minipage}{1\unitlength}\includegraphics[scale=0.08]{pic/triangle.eps}\put(0,45){\scriptsize$1$\normalsize}\put(-68,45){\scriptsize$1$\normalsize}\put(-40,2){\scriptsize$1$\normalsize}\end{minipage}\hspace{80pt}=\hspace{20pt}\begin{minipage}{1\unitlength}\includegraphics[scale=0.08]{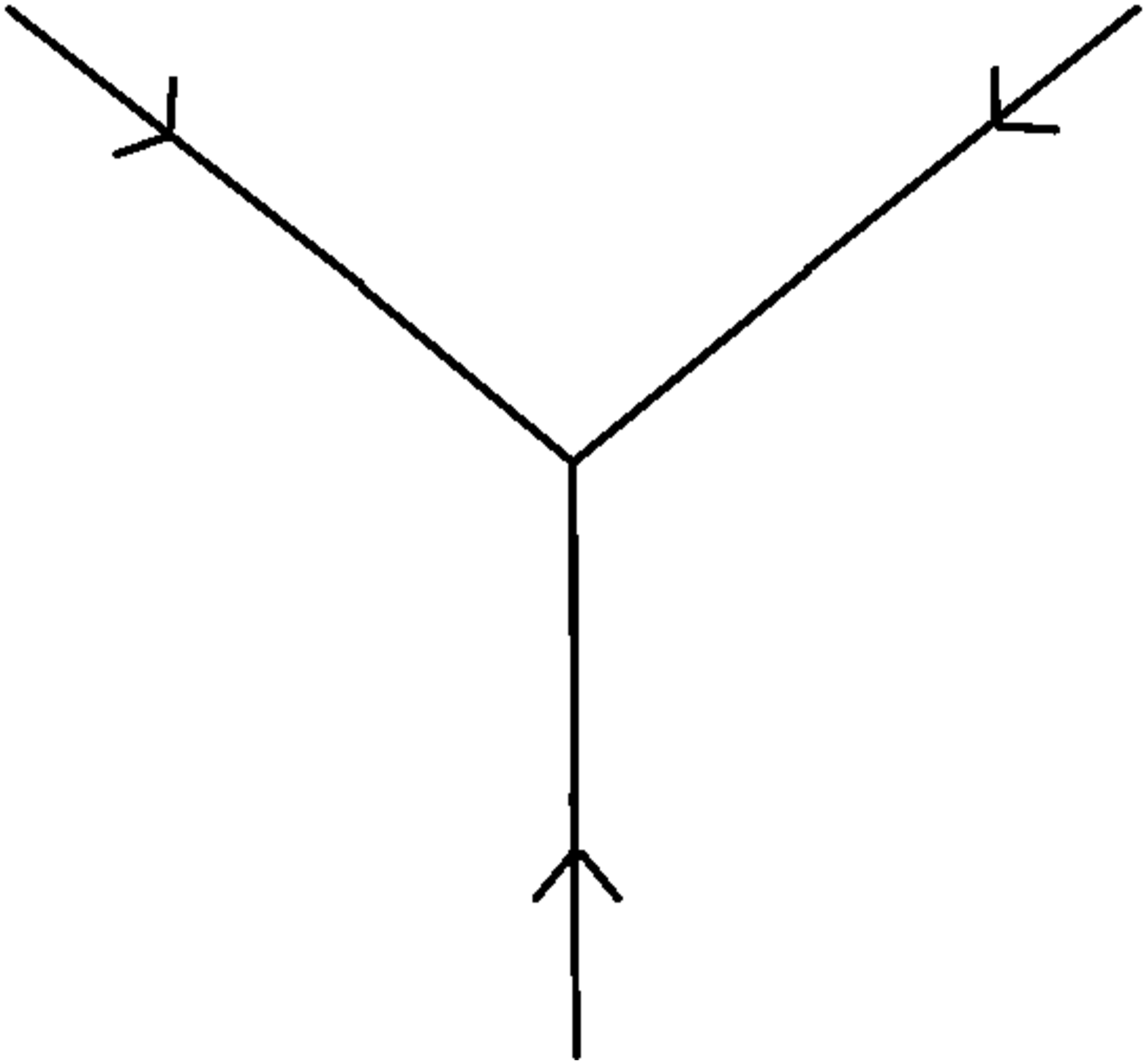}\put(0,45){\scriptsize$1$\normalsize}\put(-68,45){\scriptsize$1$\normalsize}\put(-40,2){\scriptsize$1$\normalsize}\end{minipage}\hspace{80pt},\\
& \begin{minipage}{1\unitlength}\includegraphics[scale=0.08]{pic/triangle.eps}\put(0,45){\scriptsize$n$\normalsize}\put(-68,45){\scriptsize$n$\normalsize}\put(-40,2){\scriptsize$n$\normalsize}\end{minipage}\hspace{80pt}=\hspace{20pt} \begin{minipage}{1\unitlength}\includegraphics[scale=0.08]{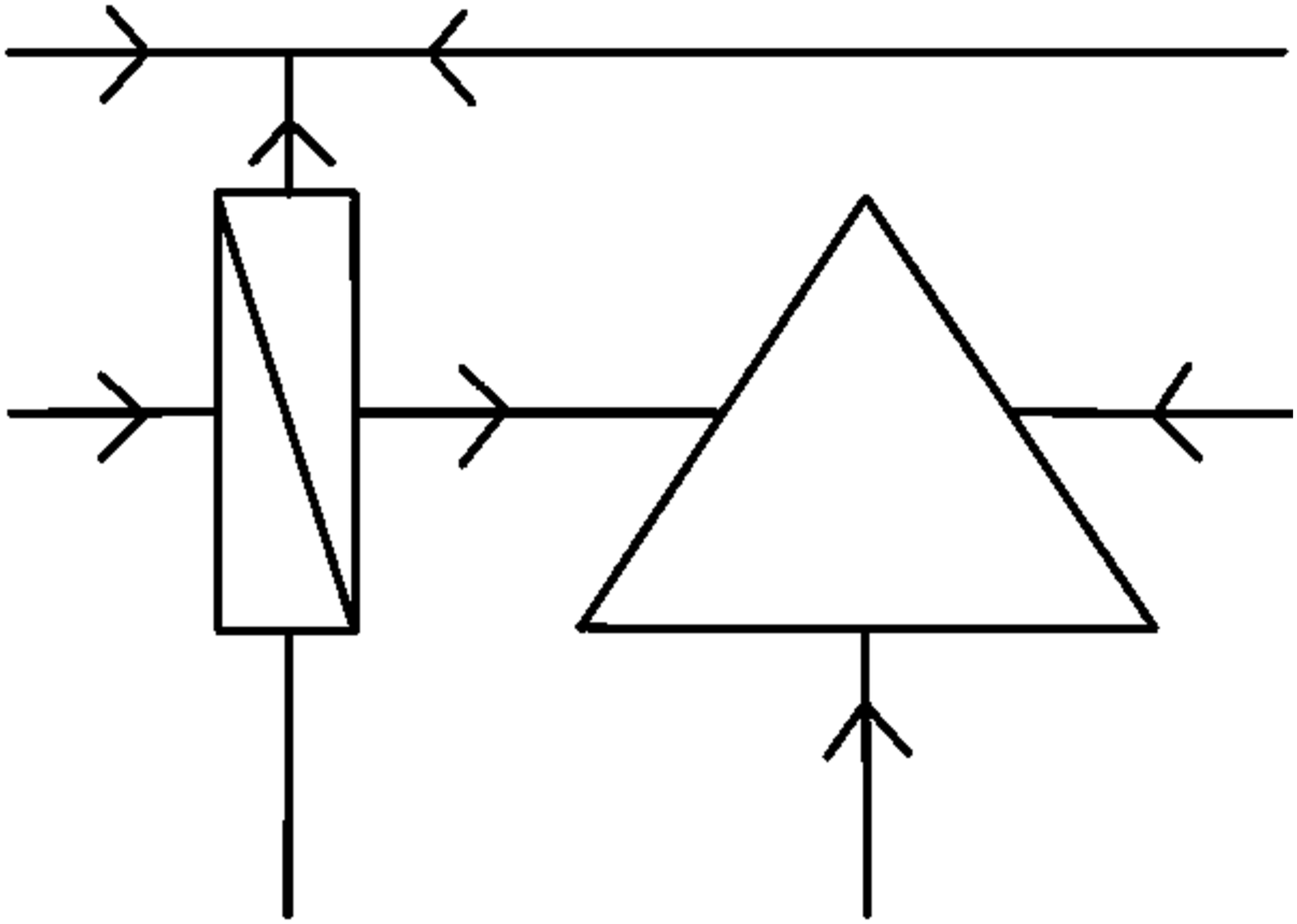}\put(-68,30){\scriptsize$n\hspace{-2pt}-\hspace{-2pt}1$\normalsize}\put(0,45){\scriptsize$1$\normalsize}\put(-68,45){\scriptsize$1$\normalsize}\put(-42,2){\scriptsize$1$\normalsize}\end{minipage}\hspace{70pt}(n>1).
\end{align*}
The opposite direction is defined in the same way.
\end{Definition}

The following formulae help calculation of the  $\mathfrak{sl}_{3}$ colored Jones polynomial. Yuasa \cite{Yua17}, \cite{Yua182}, \cite{Yua212} and Kim \cite{Kim06} gave them. We sometimes omit directions of edges of $A_{2}$ webs.

\begin{Lemma}[\cite{Yua17} \cite{Yua182} \cite{Yua212} \cite{Kim06}]
\label{Lem:delta}
For positive integer $m$ and $n$, we have
\begin{align}
\label{al:delta1}
&\Biggl\langle\hspace{10pt}\begin{minipage}{1\unitlength}\includegraphics[scale=0.08]{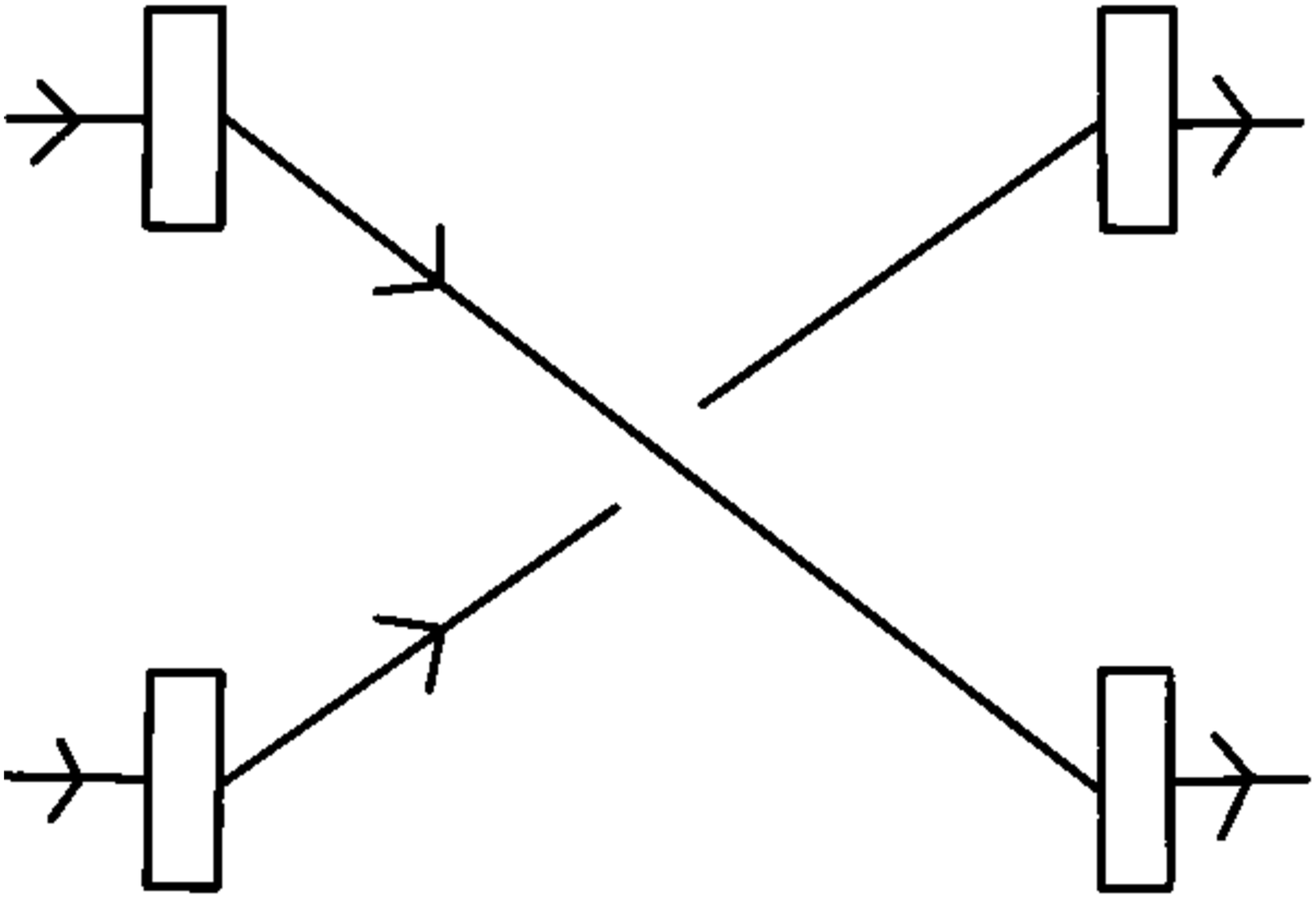}\put(2,15){\scriptsize$n$\normalsize}\put(-65,15){\scriptsize$n$\normalsize}\put(2,45){\scriptsize$n$\normalsize}\put(-65,45){\scriptsize$n$\normalsize}\end{minipage}\hspace{65pt}\Biggr\rangle_{3}=\sum_{k=0}^{n}(-1)^{k}q^{\frac{2n^{2}-6nk+3k^{2}}{6}}\frac{(q)_{n}}{(q)_{k}(q)_{n-k}}\Biggl\langle\hspace{10pt}\begin{minipage}{1\unitlength}\includegraphics[scale=0.08]{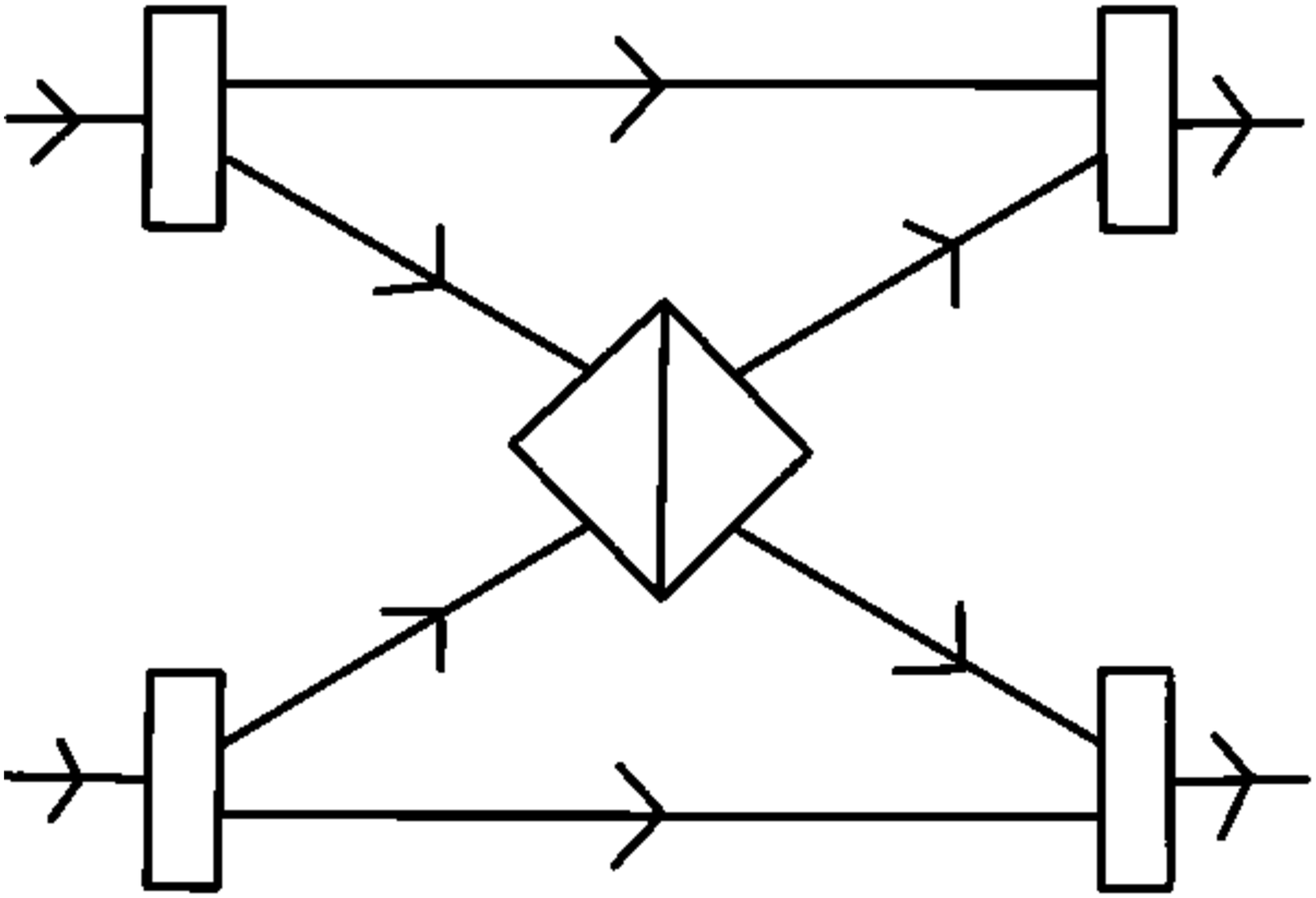}\put(2,15){\scriptsize$n$\normalsize}\put(-45,25){\scriptsize$k$\normalsize}\put(-20,25){\scriptsize$k$\normalsize}\put(-65,15){\scriptsize$n$\normalsize}\put(2,45){\scriptsize$n$\normalsize}\put(-65,45){\scriptsize$n$\normalsize}\put(-40,52){\scriptsize$n-k$\normalsize}\put(-40,2){\scriptsize$n-k$\normalsize}\end{minipage}\hspace{65pt}\Biggr\rangle_{3},\\
\label{al:delta2}
&\Biggl\langle\hspace{10pt}\begin{minipage}{1\unitlength}\includegraphics[scale=0.08]{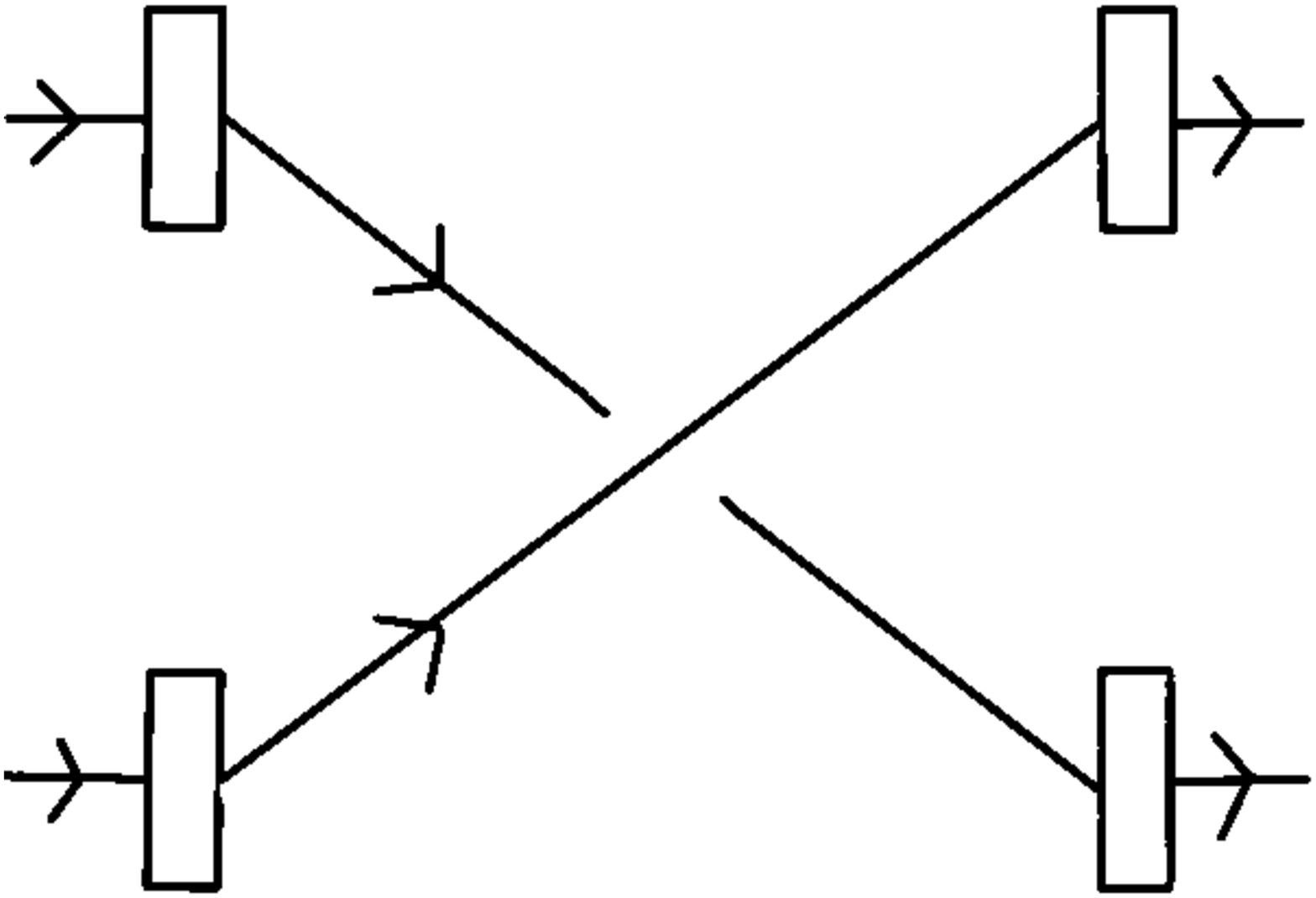}\put(2,15){\scriptsize$n$\normalsize}\put(-65,15){\scriptsize$n$\normalsize}\put(2,45){\scriptsize$n$\normalsize}\put(-65,45){\scriptsize$n$\normalsize}\end{minipage}\hspace{65pt}\Biggr\rangle_{3}=\sum_{k=0}^{n}(-1)^{k}q^{\frac{-2n^{2}+3k^{2}}{6}}\frac{(q)_{n}}{(q)_{k}(q)_{n-k}}\Biggl\langle\hspace{10pt}\begin{minipage}{1\unitlength}\includegraphics[scale=0.08]{pic/delta.eps}\put(2,15){\scriptsize$n$\normalsize}\put(-45,25){\scriptsize$k$\normalsize}\put(-20,25){\scriptsize$k$\normalsize}\put(-65,15){\scriptsize$n$\normalsize}\put(2,45){\scriptsize$n$\normalsize}\put(-65,45){\scriptsize$n$\normalsize}\put(-40,52){\scriptsize$n-k$\normalsize}\put(-40,2){\scriptsize$n-k$\normalsize}\end{minipage}\hspace{65pt}\Biggr\rangle_{3},\\
\label{al:delta3}
&\Biggl\langle\hspace{10pt}\begin{minipage}{1\unitlength}\includegraphics[scale=0.08]{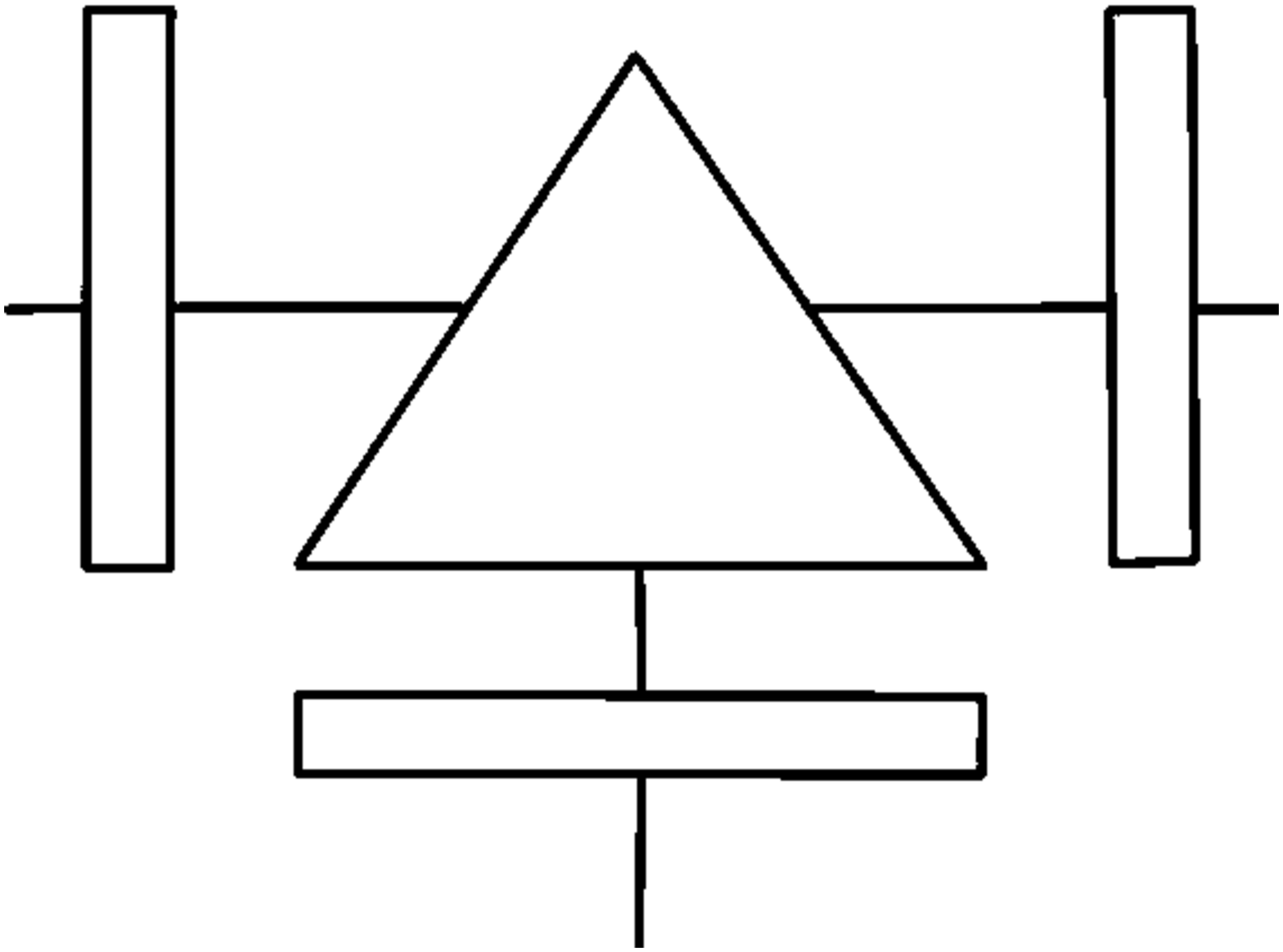}\put(2,28){\scriptsize$n$\normalsize}\put(-65,28){\scriptsize$n$\normalsize}\put(-33,-5){\scriptsize$n$\normalsize}\end{minipage}\hspace{65pt}\Biggr\rangle_{3}=\Biggl\langle\hspace{10pt}\begin{minipage}{1\unitlength}\includegraphics[scale=0.08]{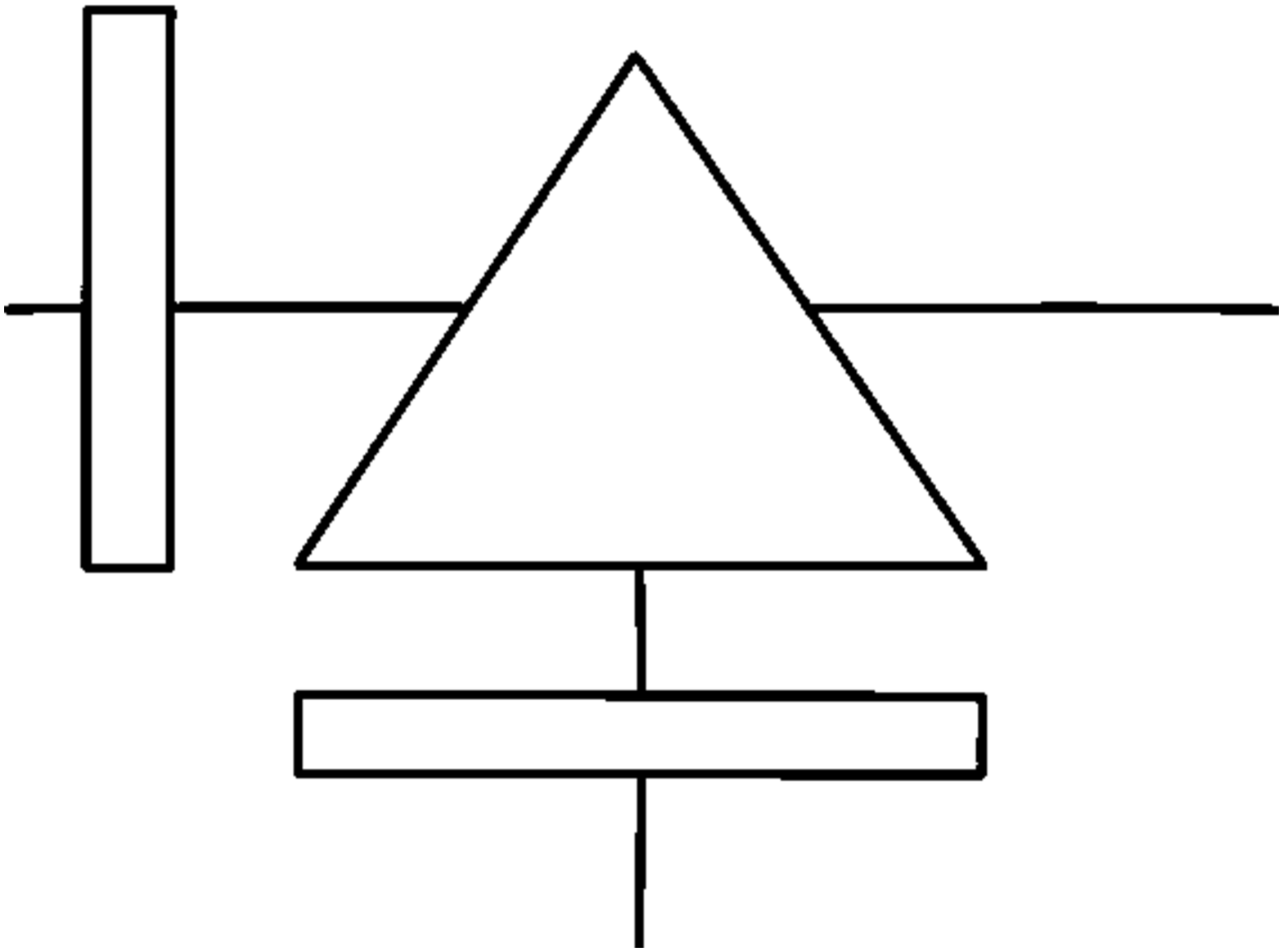}\put(2,28){\scriptsize$n$\normalsize}\put(-65,28){\scriptsize$n$\normalsize}\put(-33,-5){\scriptsize$n$\normalsize}\end{minipage}\hspace{65pt}\Biggr\rangle_{3},\\
\label{al:delta12}
&\Biggl\langle\begin{minipage}{1\unitlength}\includegraphics[scale=0.08]{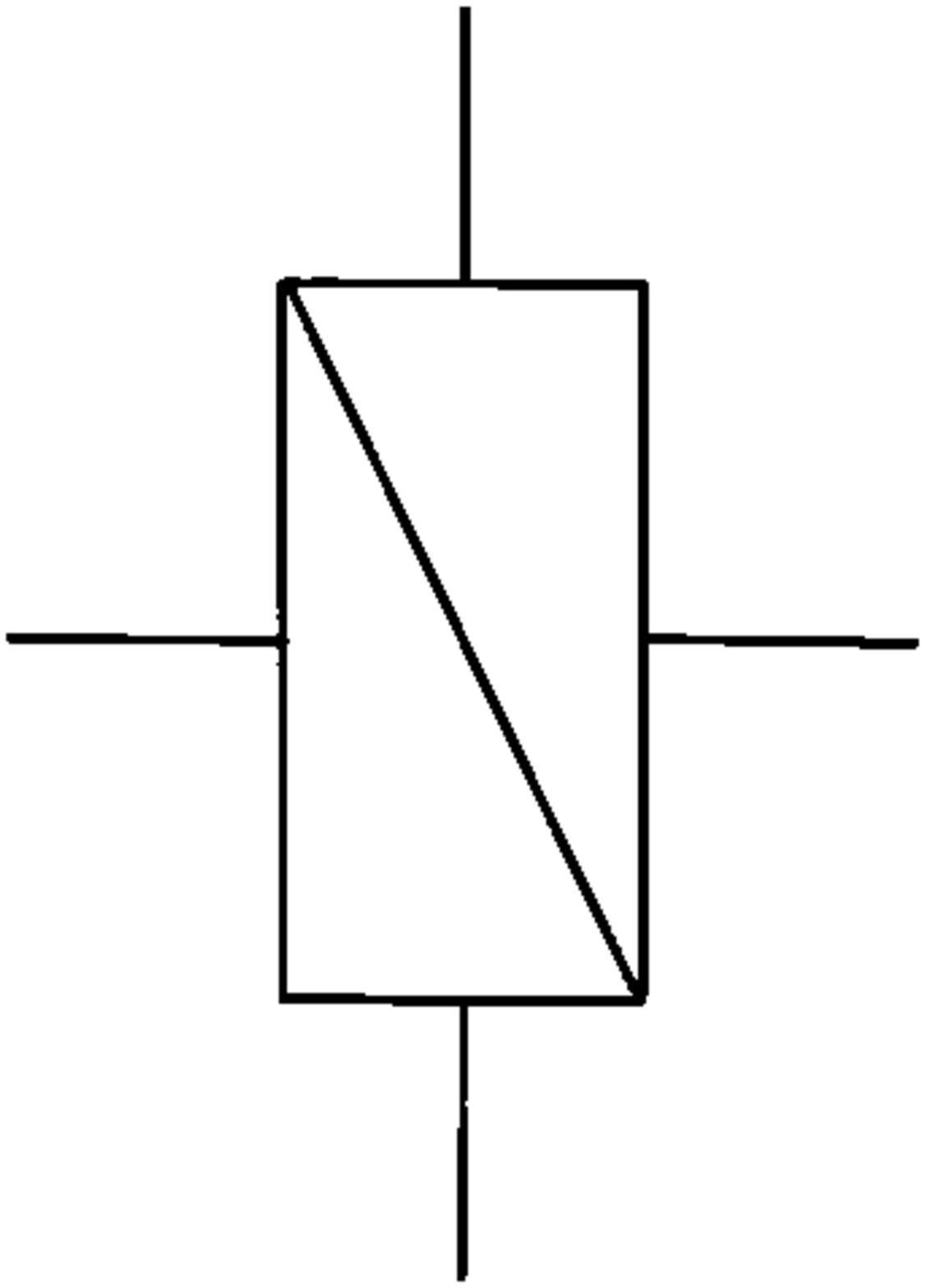}\put(-12,28){\scriptsize$n$\normalsize}\put(-52,28){\scriptsize$n$\normalsize}\put(-33,0){\scriptsize$n$\normalsize}\put(-35,50){\scriptsize$n$\normalsize}\end{minipage}\hspace{55pt}\Biggr\rangle_{3}=\Biggl\langle\hspace{10pt}\begin{minipage}{1\unitlength}\includegraphics[scale=0.08]{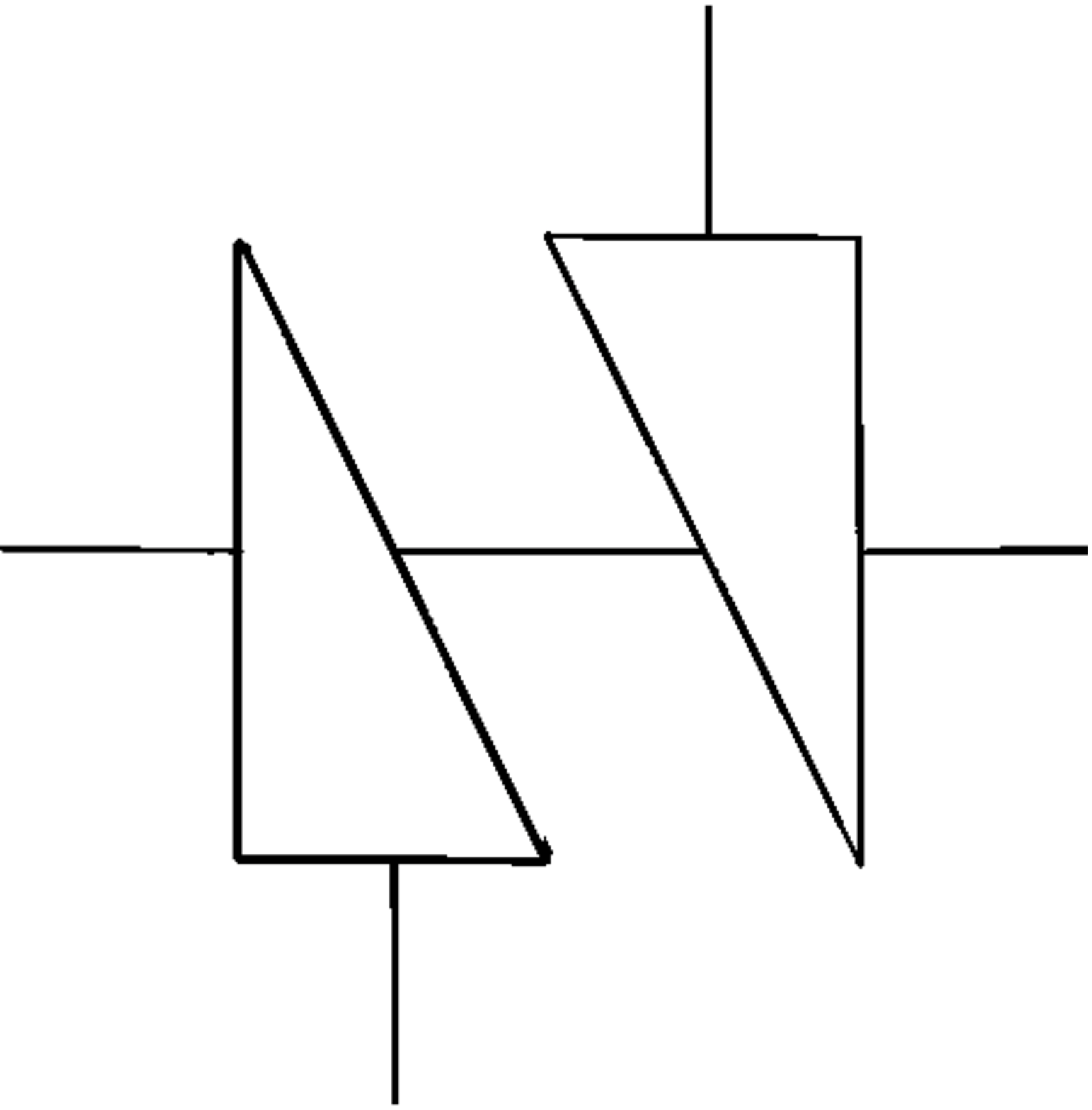}\put(2,28){\scriptsize$n$\normalsize}\put(-52,28){\scriptsize$n$\normalsize}\put(-33,0){\scriptsize$n$\normalsize}\put(-22,50){\scriptsize$n$\normalsize}\end{minipage}\hspace{60pt}\Biggr\rangle_{3},\\
\label{al:delta13}
&\Biggl\langle\hspace{10pt}\begin{minipage}{1\unitlength}\includegraphics[scale=0.08]{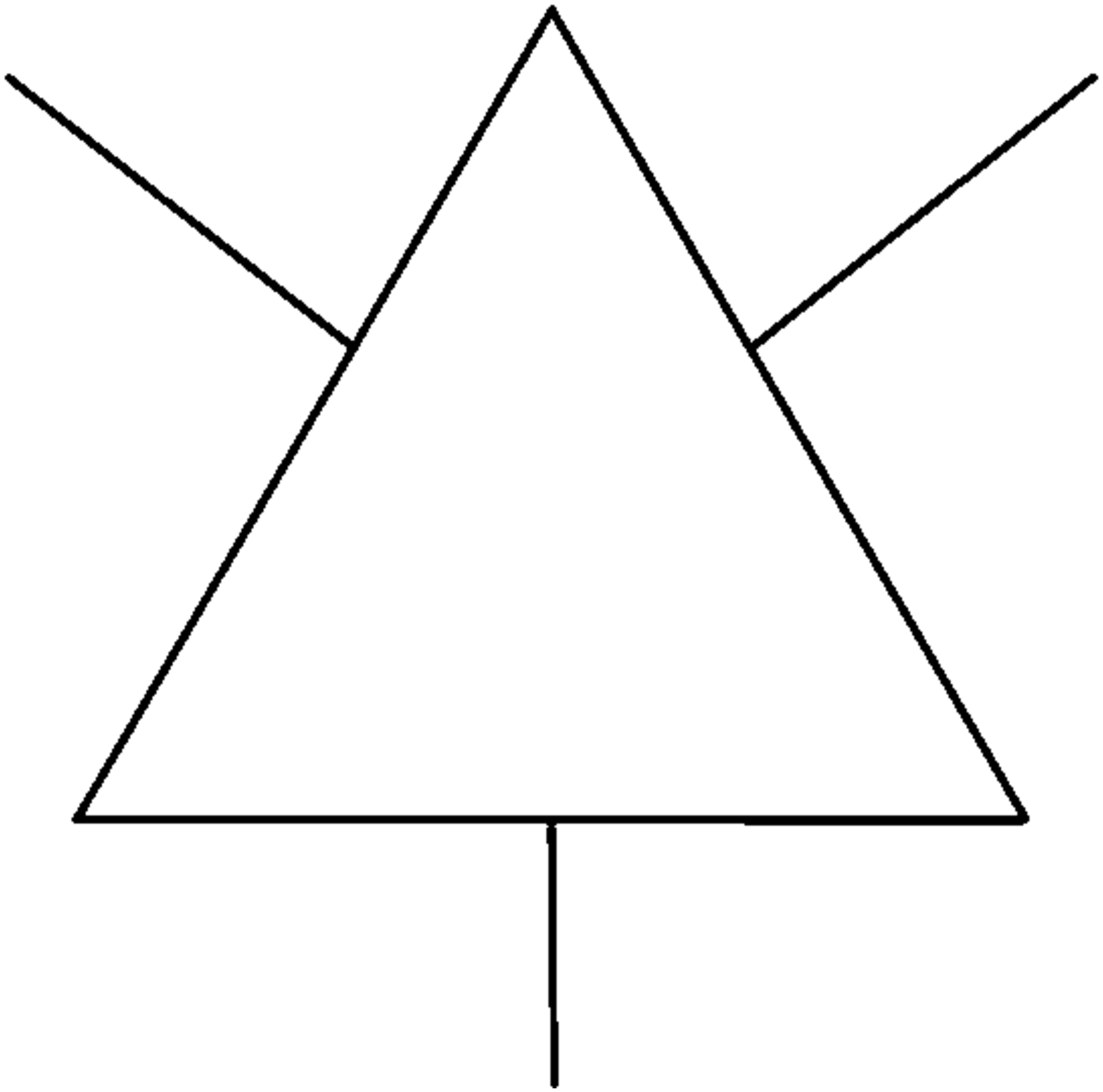}\put(2,50){\scriptsize$n$\normalsize}\put(-65,50){\scriptsize$n$\normalsize}\put(-33,-5){\scriptsize$n$\normalsize}\end{minipage}\hspace{65pt}\Biggr\rangle_{3}=\Biggl\langle\hspace{10pt}\begin{minipage}{1\unitlength}\includegraphics[scale=0.08]{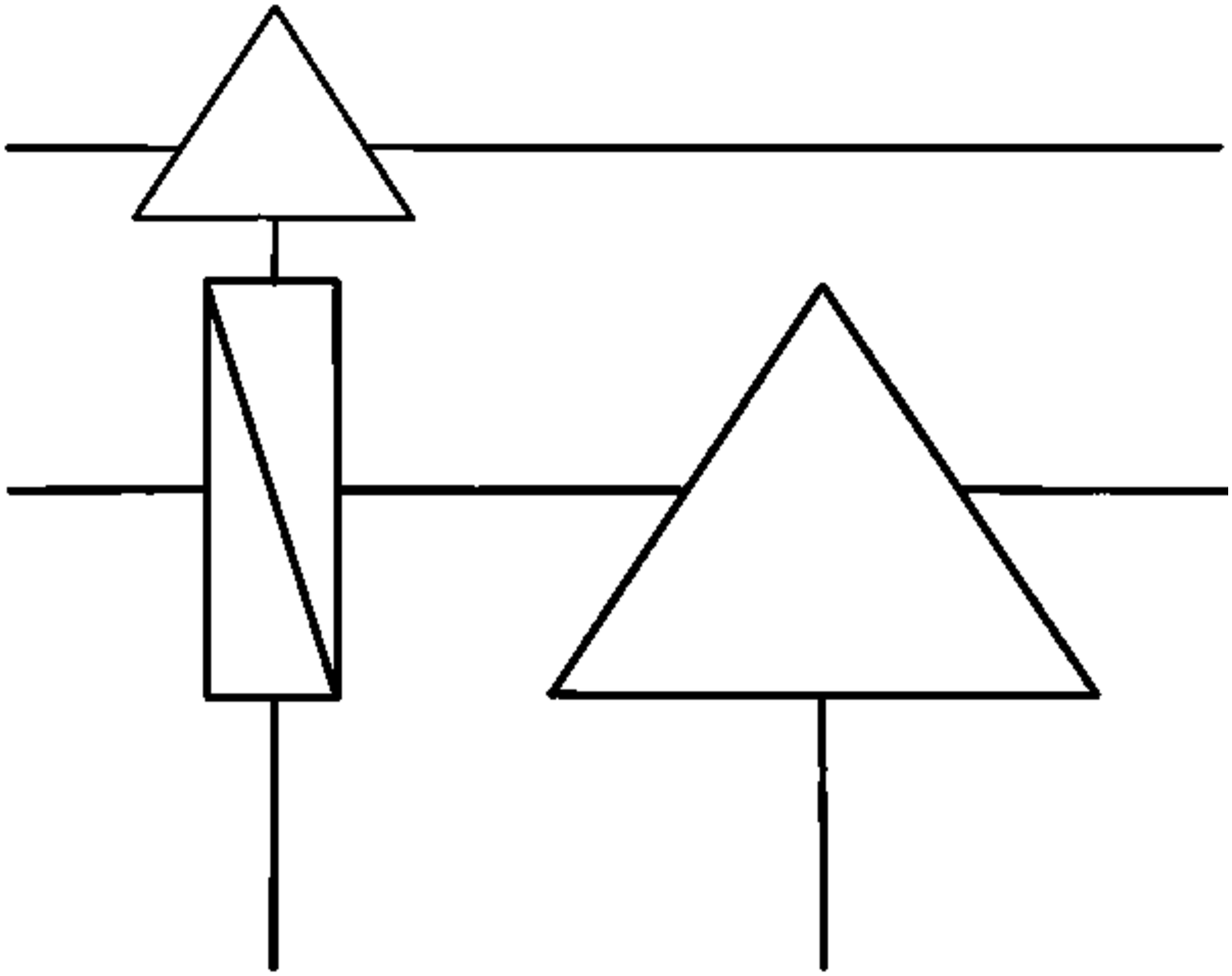}\put(2,28){\scriptsize$n-k$\normalsize}\put(-65,28){\scriptsize$n-k$\normalsize}\put(-65,45){\scriptsize$k$\normalsize}\put(-40,48){\scriptsize$k$\normalsize}\end{minipage}\hspace{75pt}\Biggr\rangle_{3},\\
\label{al:delta4}
&\Biggl\langle\hspace{10pt}\begin{minipage}{1\unitlength}\includegraphics[scale=0.08]{pic/DeltaBanish1.eps}\put(2,28){\scriptsize$n$\normalsize}\put(-65,28){\scriptsize$n$\normalsize}\put(-33,-5){\scriptsize$n$\normalsize}\end{minipage}\hspace{65pt}\Biggr\rangle_{3}=\Biggl\langle\hspace{10pt}\begin{minipage}{1\unitlength}\includegraphics[scale=0.08]{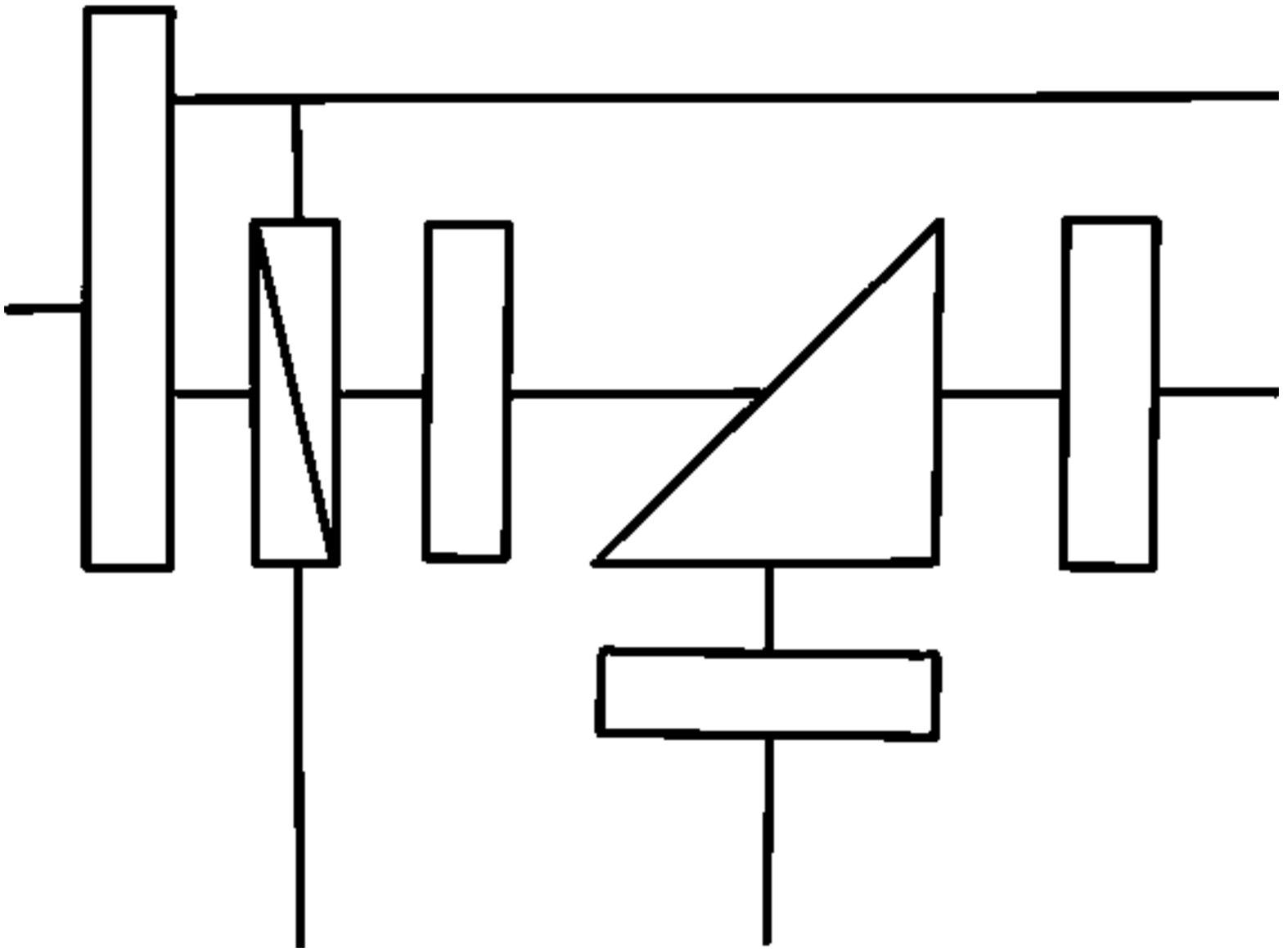}\put(0,38){\scriptsize$1$\normalsize}\put(0,28){\scriptsize$n-1$\normalsize}\put(-65,28){\scriptsize$n$\normalsize}\put(-31,-5){\scriptsize$n-1$\normalsize}\put(-47,-4){\scriptsize$1$\normalsize}\end{minipage}\hspace{70pt}\Biggr\rangle_{3}-\frac{[n-1]_{q}}{[n]_{q}}\Biggl\langle\hspace{10pt}\begin{minipage}{1\unitlength}\includegraphics[scale=0.08]{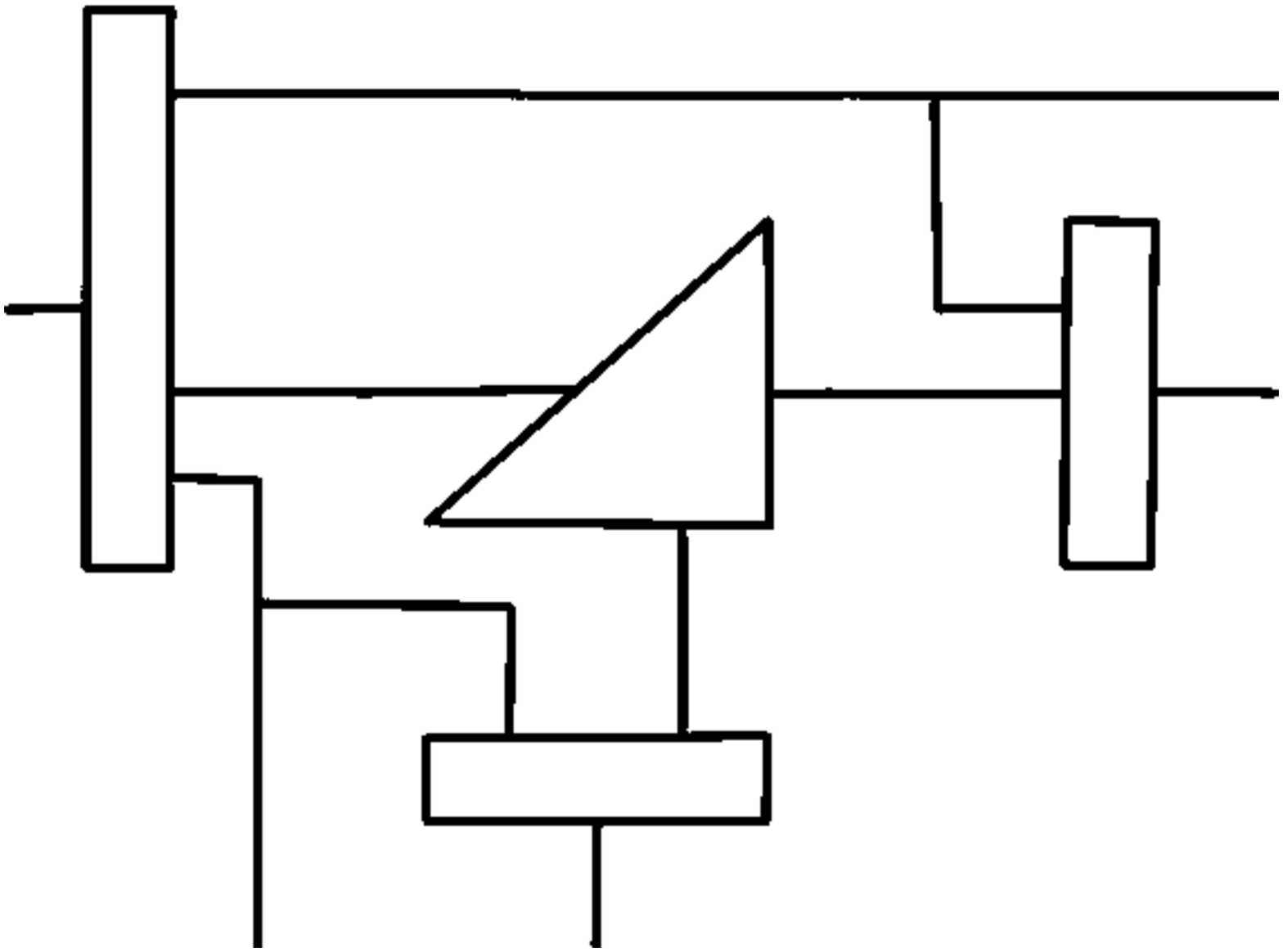}\put(0,28){\scriptsize$n-1$\normalsize}\put(0,38){\scriptsize$1$\normalsize}\put(-65,28){\scriptsize$n$\normalsize}\put(-35,-5){\scriptsize$n-1$\normalsize}\put(-49,-4){\scriptsize$1$\normalsize}\end{minipage}\hspace{70pt}\Biggr\rangle_{3},\\
\label{al:delta5}
&\Biggl\langle\hspace{20pt}\begin{minipage}{1\unitlength}\includegraphics[scale=0.08]{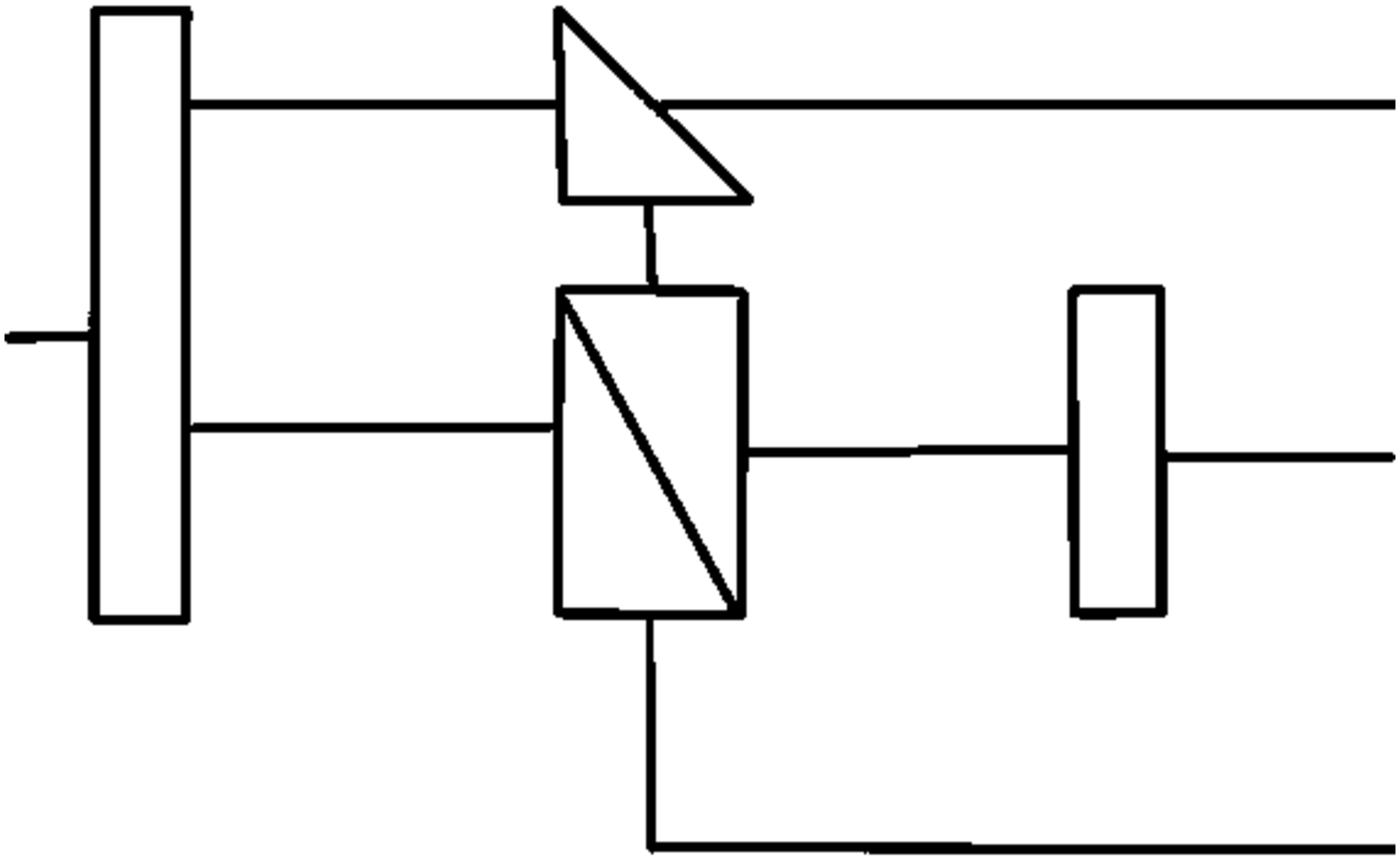}\put(2,42){\scriptsize$m$\normalsize}\put(2,28){\scriptsize$n$\normalsize}\put(-75,28){\scriptsize$n+m$\normalsize}\put(-33,0){\scriptsize$m$\normalsize}\end{minipage}\hspace{65pt}\Biggr\rangle_{3}=\Biggl\langle\hspace{20pt}\begin{minipage}{1\unitlength}\includegraphics[scale=0.08]{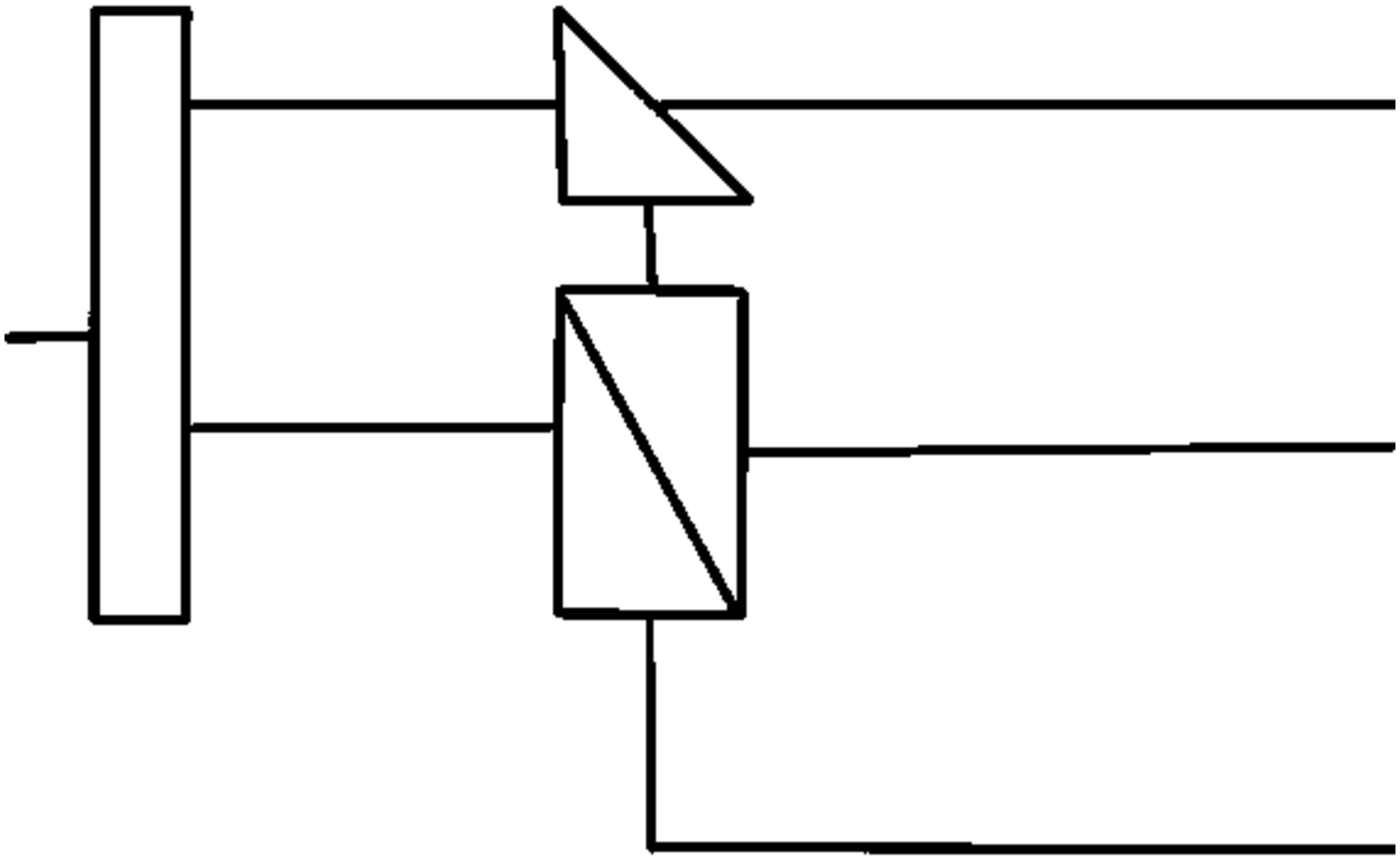}\put(2,42){\scriptsize$m$\normalsize}\put(2,28){\scriptsize$n$\normalsize}\put(-75,28){\scriptsize$n+m$\normalsize}\put(-33,0){\scriptsize$m$\normalsize}\end{minipage}\hspace{65pt}\Biggr\rangle_{3},\\
&\Biggl\langle\hspace{10pt}\begin{minipage}{1\unitlength}\includegraphics[scale=0.08]{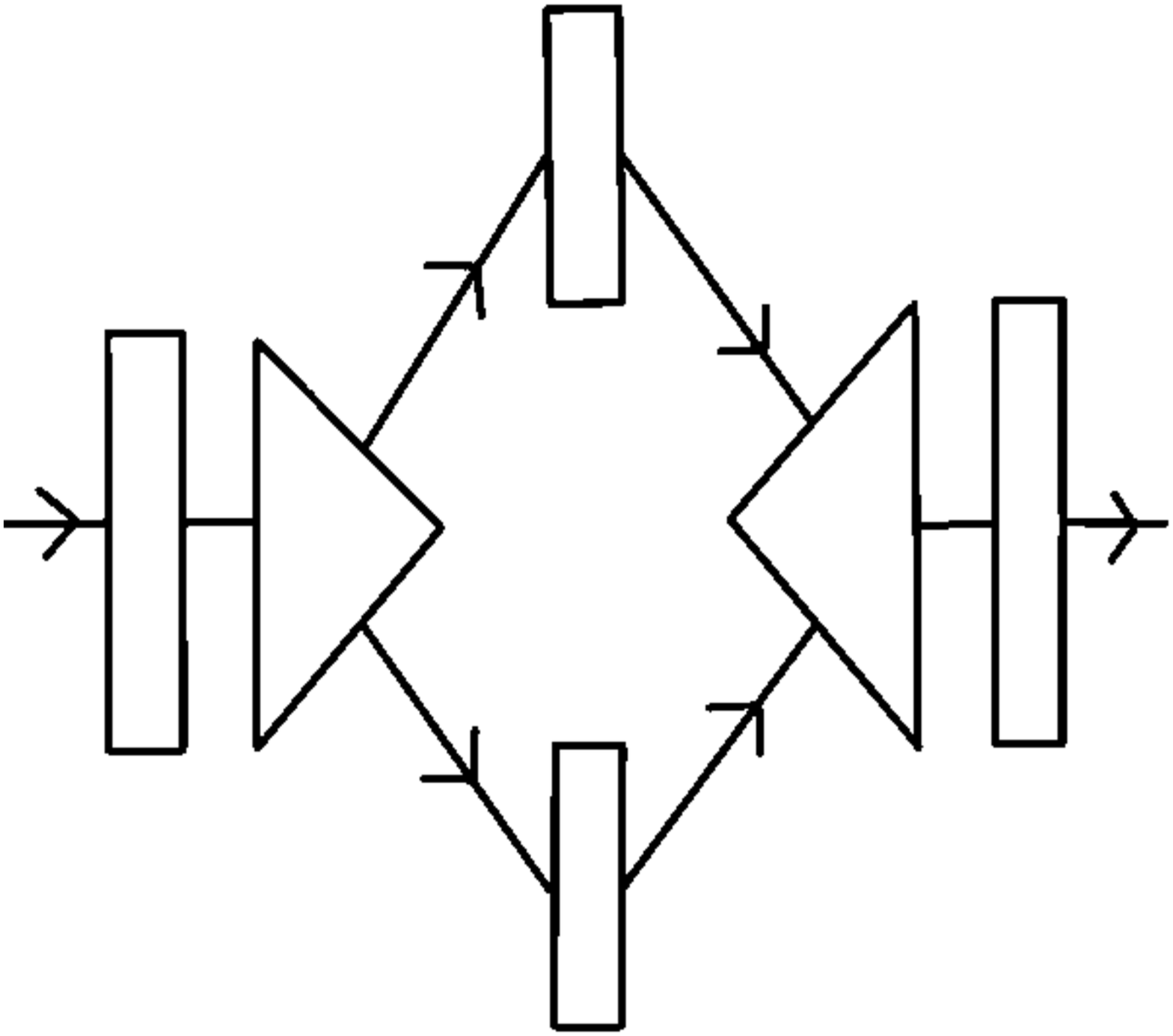}\put(2,28){\scriptsize$n$\normalsize}\put(-65,28){\scriptsize$n$\normalsize}\end{minipage}\hspace{65pt}\Biggr\rangle_{3}=[n+1]_{q}\sum_{k=0}^{n}\Biggl\langle\hspace{10pt}\begin{minipage}{1\unitlength}\includegraphics[scale=0.08]{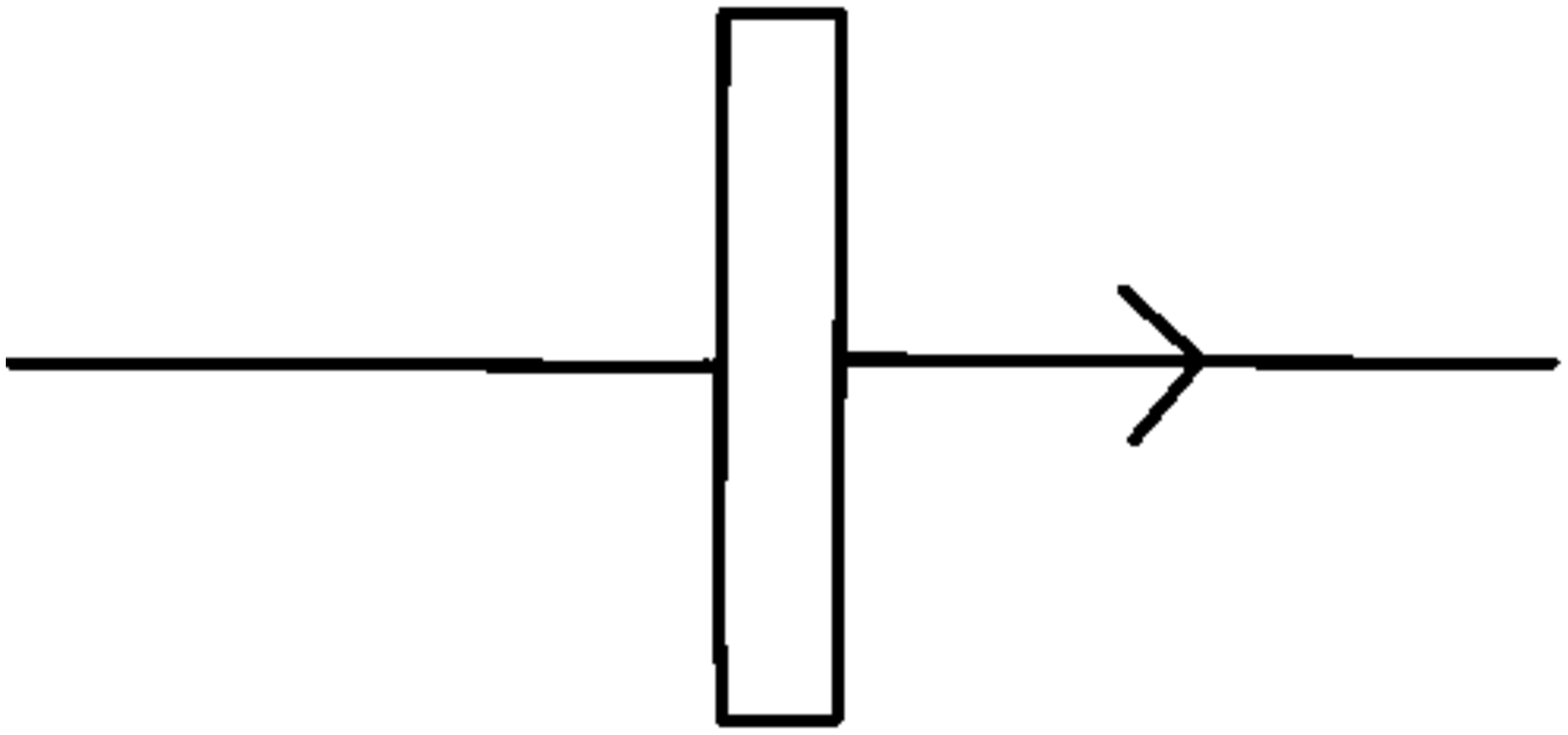}\put(-10,28){\scriptsize$k$\normalsize}\end{minipage}\hspace{65pt}\Biggr\rangle_{3},\\
\label{al:delta6}
&\Biggl\langle\hspace{10pt}\begin{minipage}{1\unitlength}\includegraphics[scale=0.08]{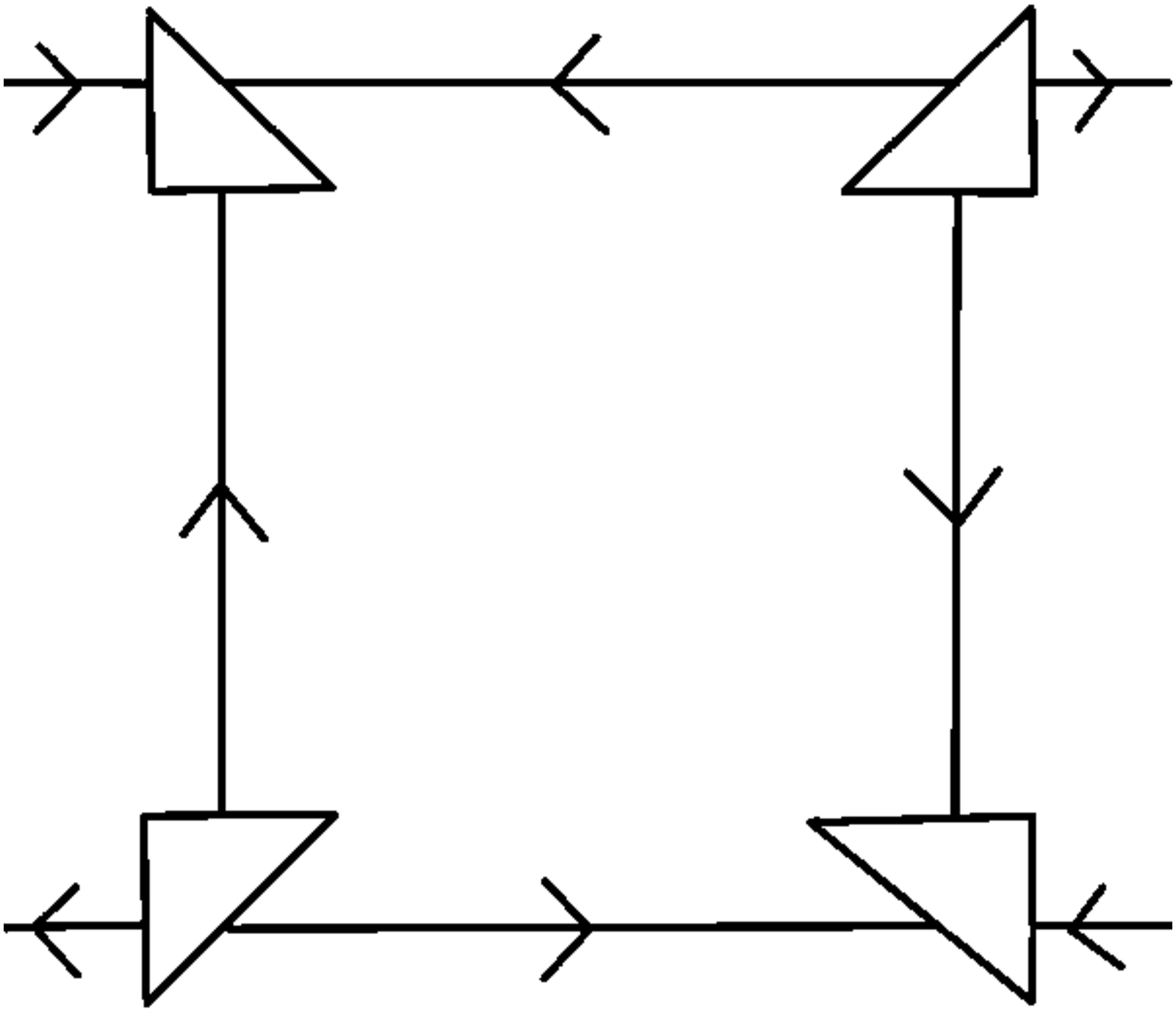}\put(2,56){\scriptsize$n$\normalsize}\put(2,12){\scriptsize$n$\normalsize}\put(-30,58){\scriptsize$n$\normalsize}\put(-30,15){\scriptsize$n$\normalsize}\put(-65,58){\scriptsize$n$\normalsize}\put(-65,10){\scriptsize$n$\normalsize}\end{minipage}\hspace{65pt}\Biggr\rangle_{3}=\sum_{k=0}^{n}\Biggl\langle\hspace{10pt}\begin{minipage}{1\unitlength}\includegraphics[scale=0.08]{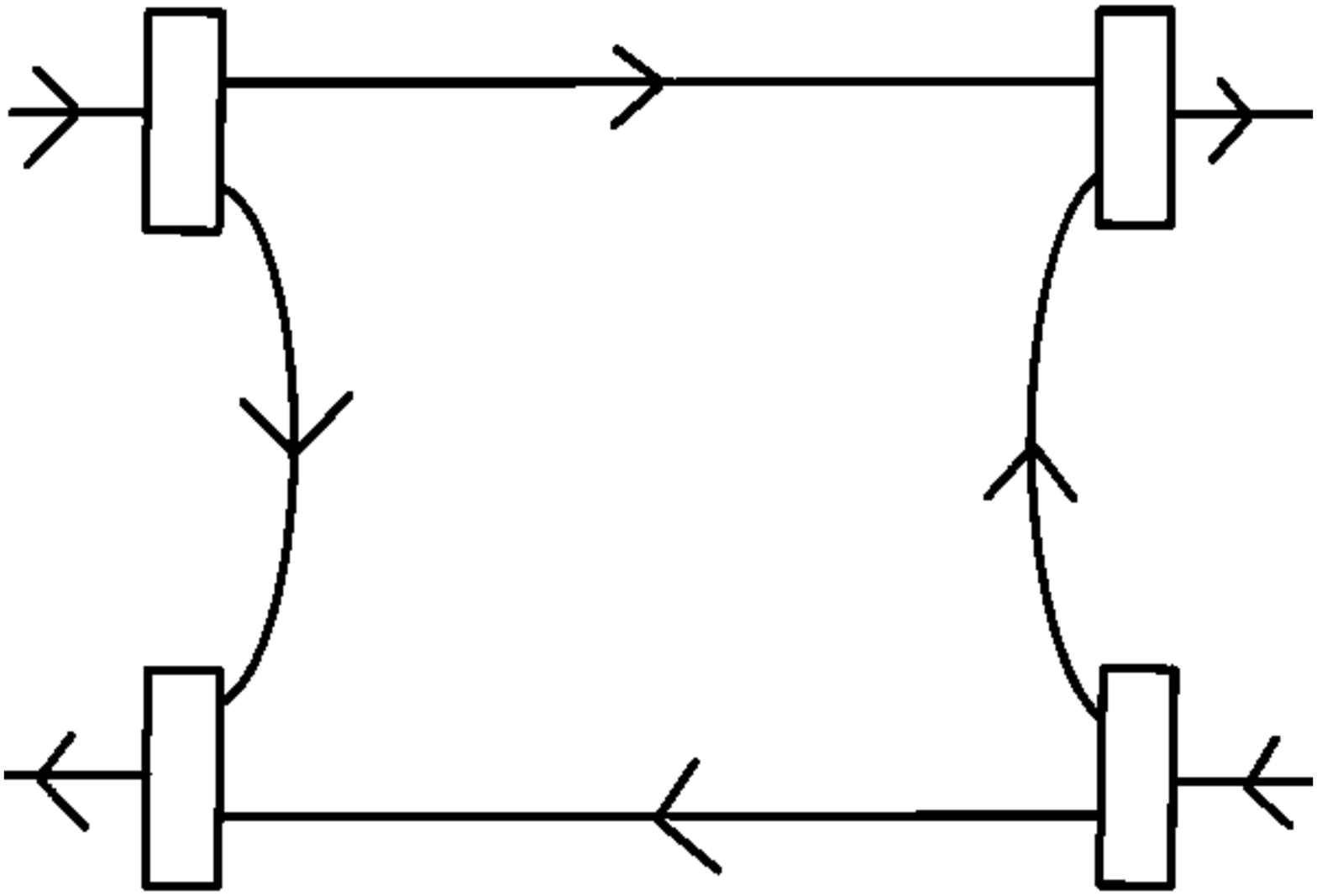}\put(-10,32){\scriptsize$k$\normalsize}\put(-55,30){\scriptsize$k$\normalsize}\put(-35,52){\scriptsize$n-k$\normalsize}\put(-35,14){\scriptsize$n-k$\normalsize}\end{minipage}\hspace{75pt}\Biggr\rangle_{3},\\
\label{al:delta10}
&\Biggr\langle\begin{minipage}{1\unitlength}\includegraphics[scale=0.09]{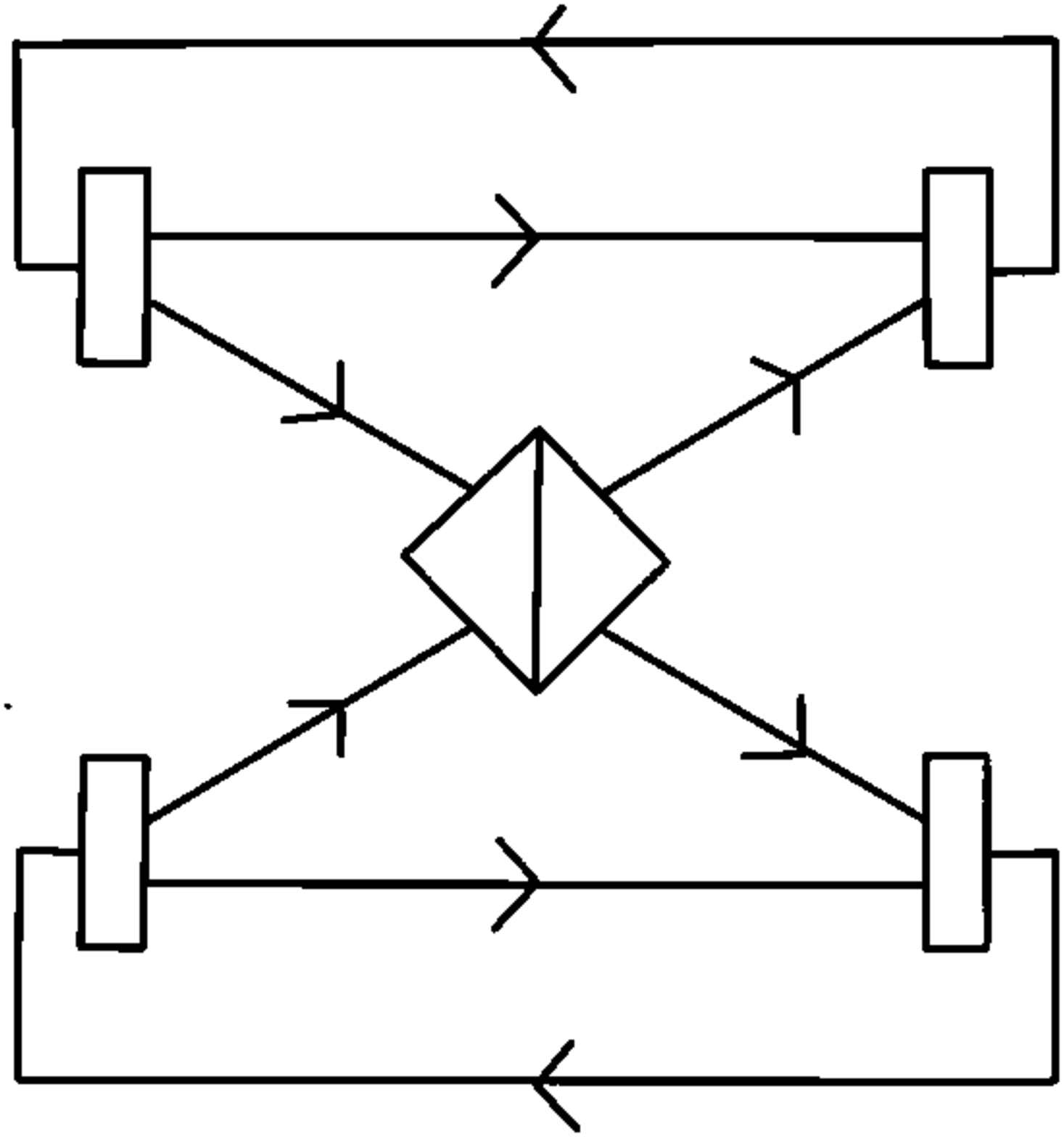}\put(-73,55){\scriptsize$n$\normalsize}\put(-65,38){\scriptsize$n-k$\normalsize}\put(-65,25){\scriptsize$n-k$\normalsize}\put(-45,55){\scriptsize$k$\normalsize}\put(-45,8){\scriptsize$k$\normalsize}\put(-73,10){\scriptsize$n$\normalsize}\end{minipage}\hspace{70pt} \Biggr\rangle_{3}=\frac{[n+1]_{q}[n+2]_{q}}{[n-k+2]_{q}}\Delta(n,0)\quad(0\le k\le n),\\
\label{al:delta20}
&\Biggl\langle \scriptsize\begin{minipage}{1\unitlength}\includegraphics[scale=0.08]{pic/JonesWenzel1.eps}\put(-20,45){$n$}\end{minipage}\hspace{48pt}\normalsize\Biggr\rangle_{3}=\sum_{i=0}^{n-1}(-1)^{i}\frac{[n-i]}{[n]_{q}}\Biggl\langle \scriptsize\begin{minipage}{1\unitlength}\includegraphics[scale=0.08]{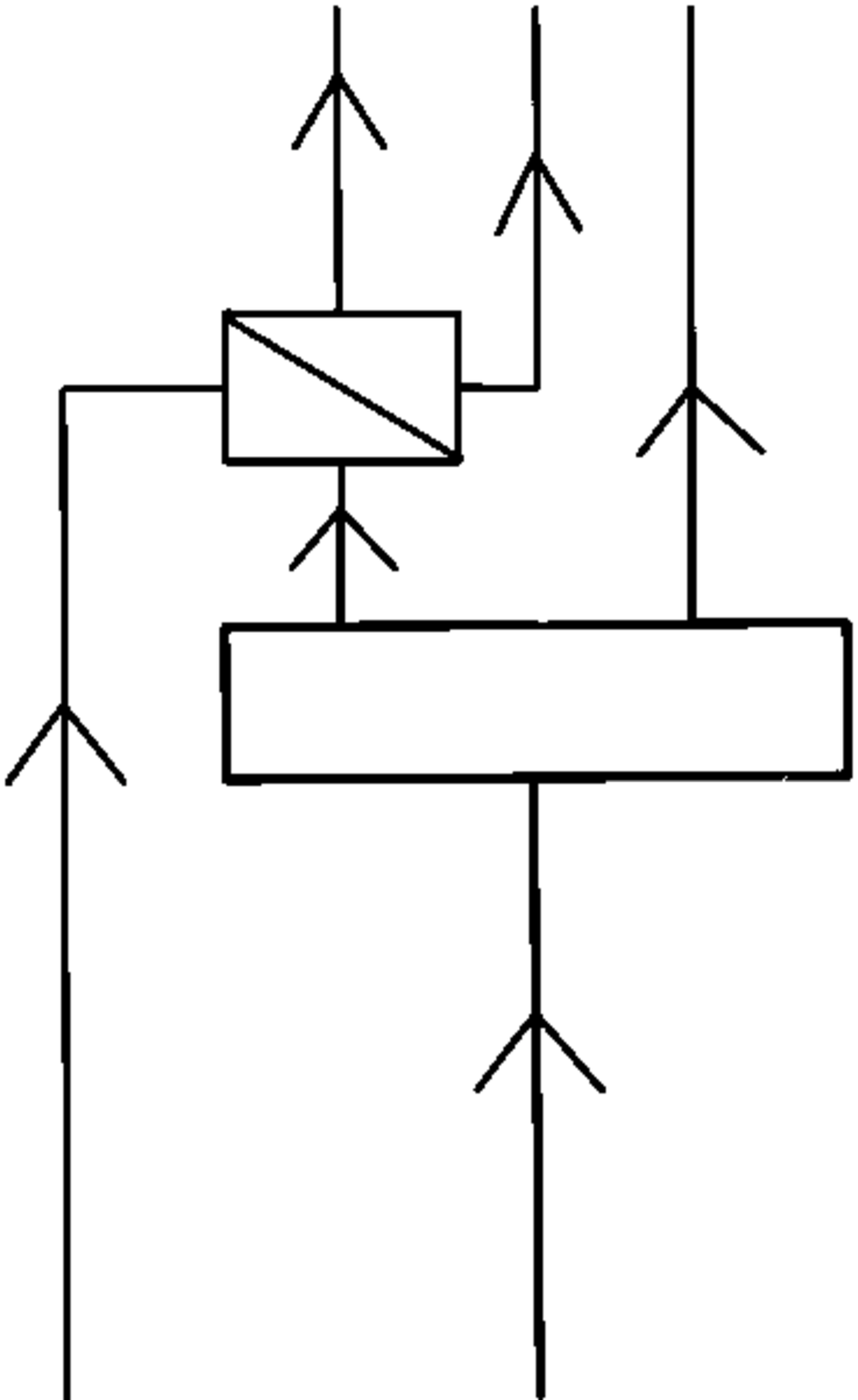}\put(-38,60){$i$}\put(-30,60){$1$}\put(-20,45){$n-i-1$}\end{minipage}\hspace{48pt}\Biggr\rangle_{3}\in W_{1^{+}+1^{-}},
\end{align}

\begin{align}
\label{al:delta7}
&\Biggl\langle\hspace{5pt}\begin{minipage}{1\unitlength}\includegraphics[scale=0.08]{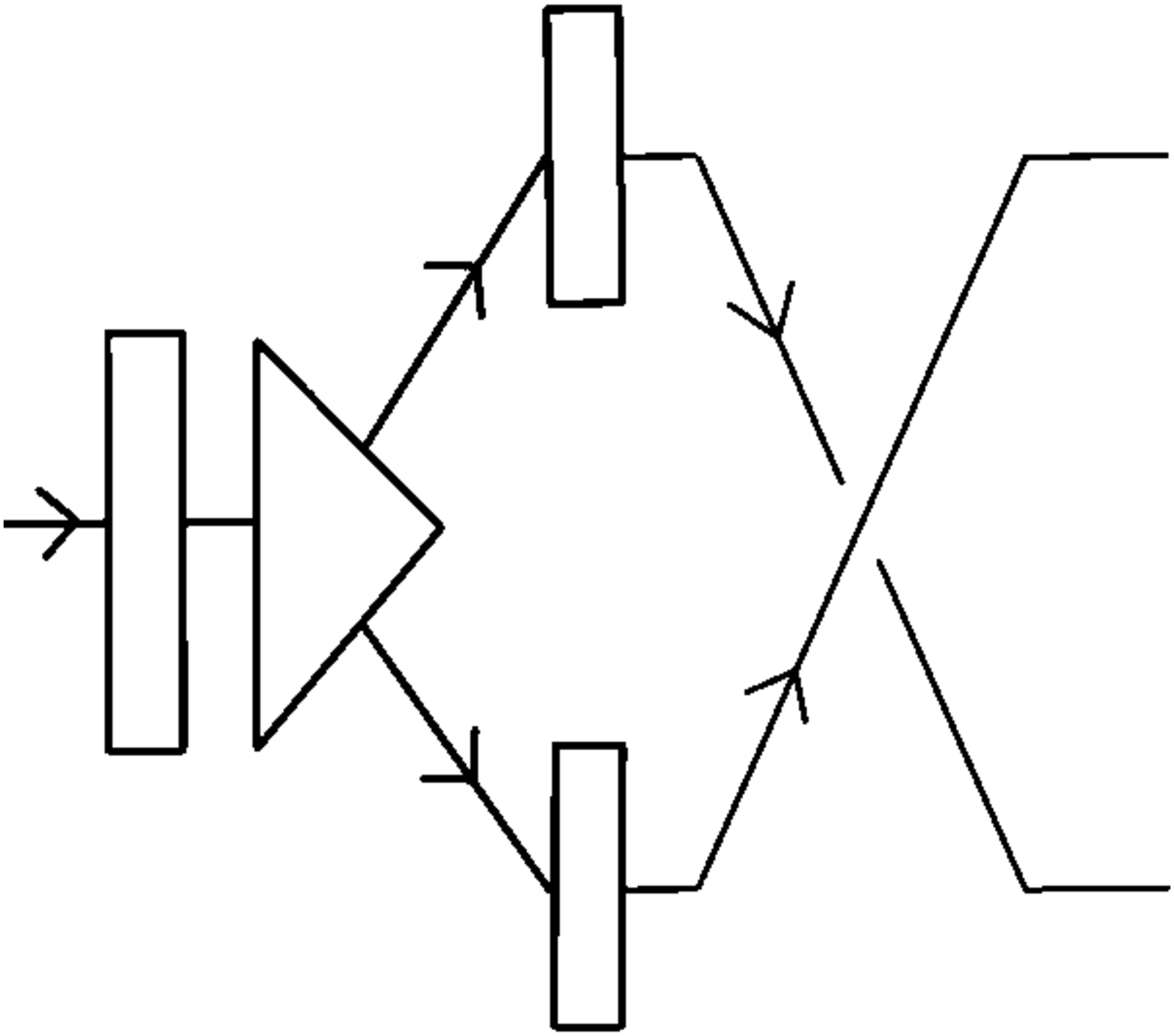}\put(-65,28){\scriptsize$n$\normalsize}\end{minipage}\hspace{65pt}\Biggr\rangle_{3}=(-1)^{n}q^{-\frac{n^{2}+3n}{6}}\Biggl\langle\hspace{10pt}\begin{minipage}{1\unitlength}\includegraphics[scale=0.08]{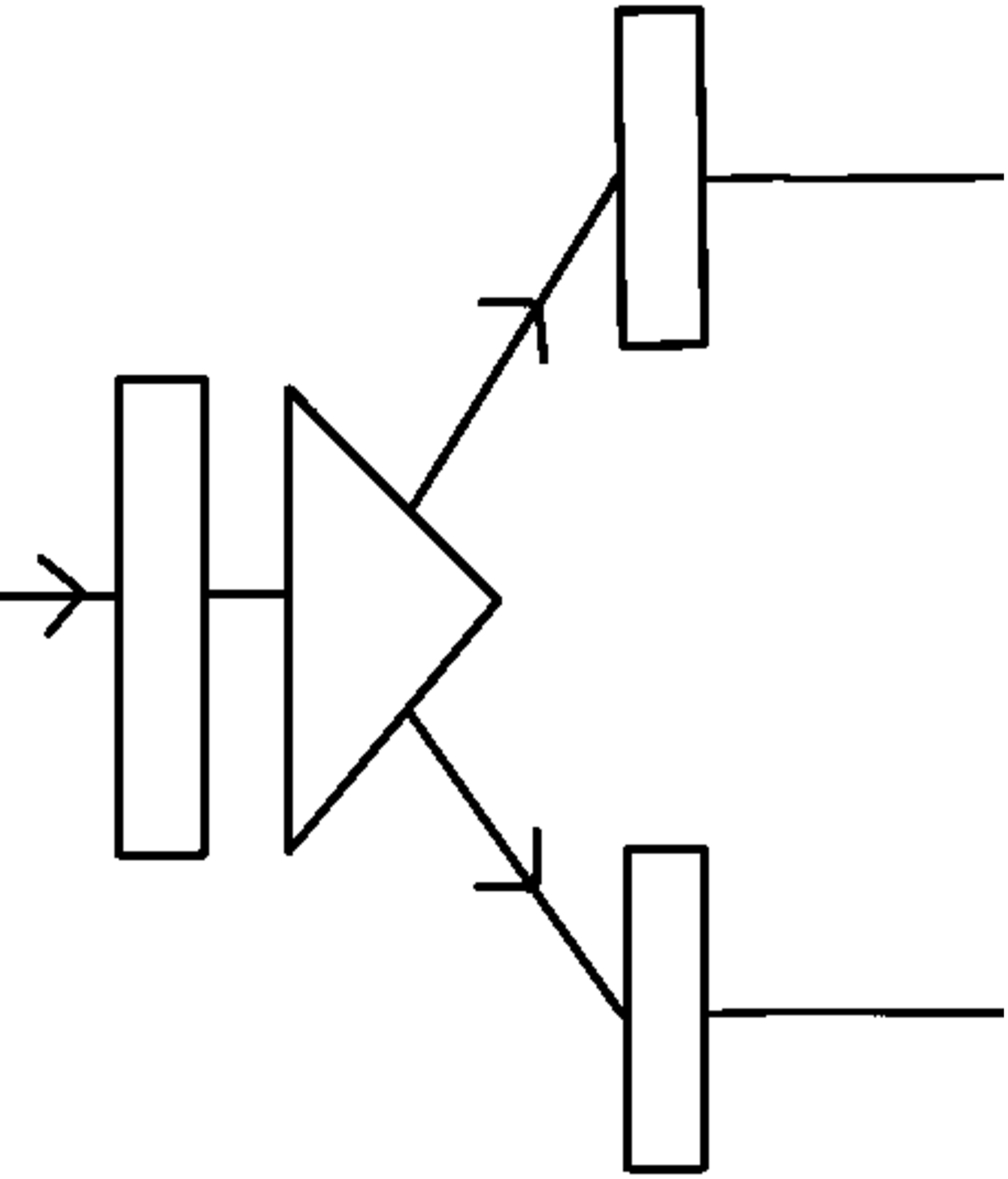}\put(-55,30){\scriptsize$n$\normalsize}\put(-30,50){\scriptsize$n$\normalsize}\put(-30,12){\scriptsize$n$\normalsize}\end{minipage}\hspace{45pt}\Biggr\rangle_{3},\quad\Biggl\langle\hspace{5pt}\begin{minipage}{1\unitlength}\includegraphics[scale=0.08]{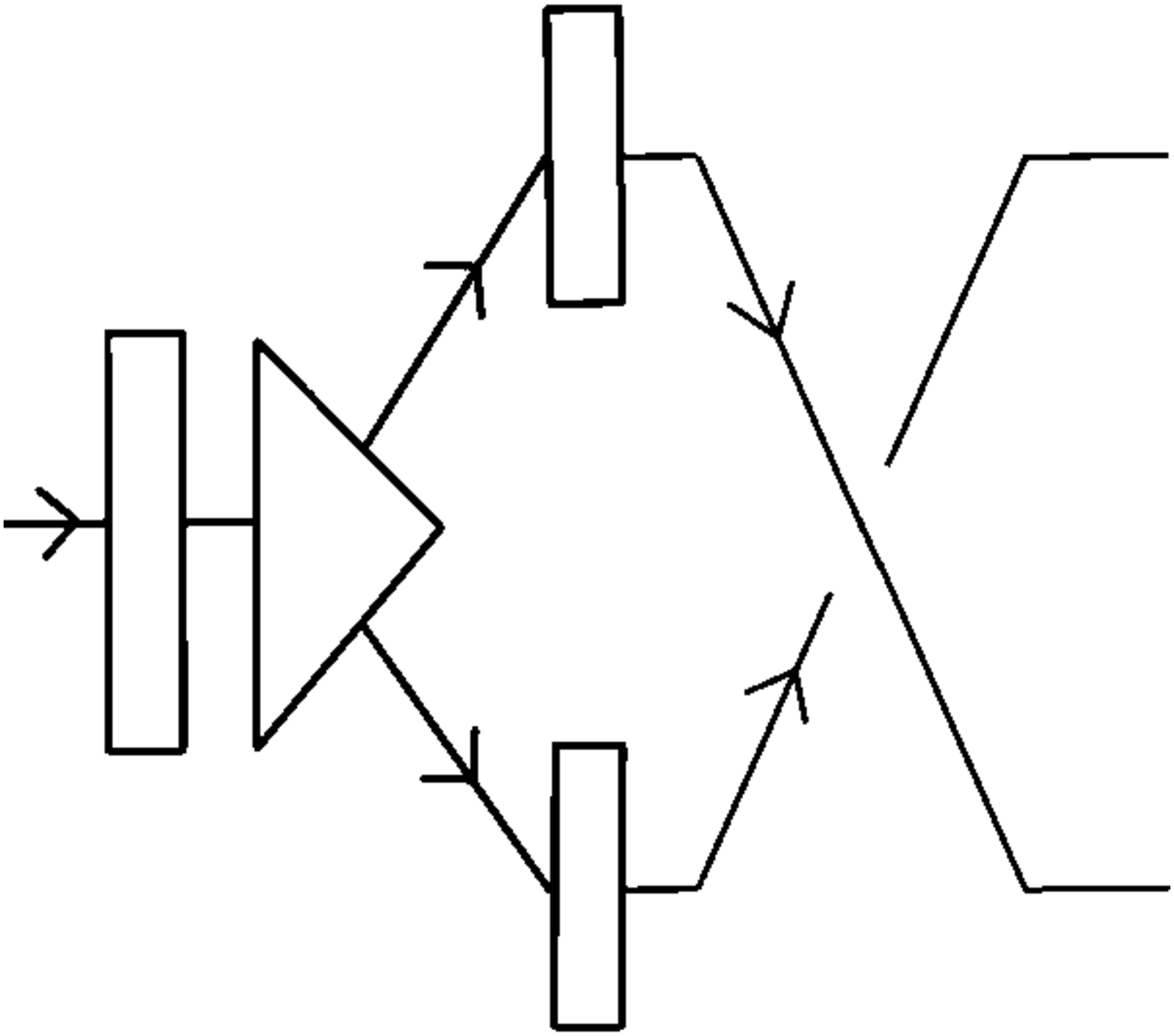}\put(-65,28){\scriptsize$n$\normalsize}\end{minipage}\hspace{65pt}\Biggr\rangle_{3}=(-1)^{n}q^{\frac{n^{2}+3n}{6}}\Biggl\langle\hspace{10pt}\begin{minipage}{1\unitlength}\includegraphics[scale=0.08]{pic/DeltaTwistafter.eps}\put(-55,30){\scriptsize$n$\normalsize}\put(-30,52){\scriptsize$n$\normalsize}\put(-30,12){\scriptsize$n$\normalsize}\end{minipage}\hspace{45pt}\Biggr\rangle_{3},\\
\label{al:delta8}
&\Biggl\langle\hspace{10pt}\begin{minipage}{1\unitlength}\includegraphics[scale=0.08]{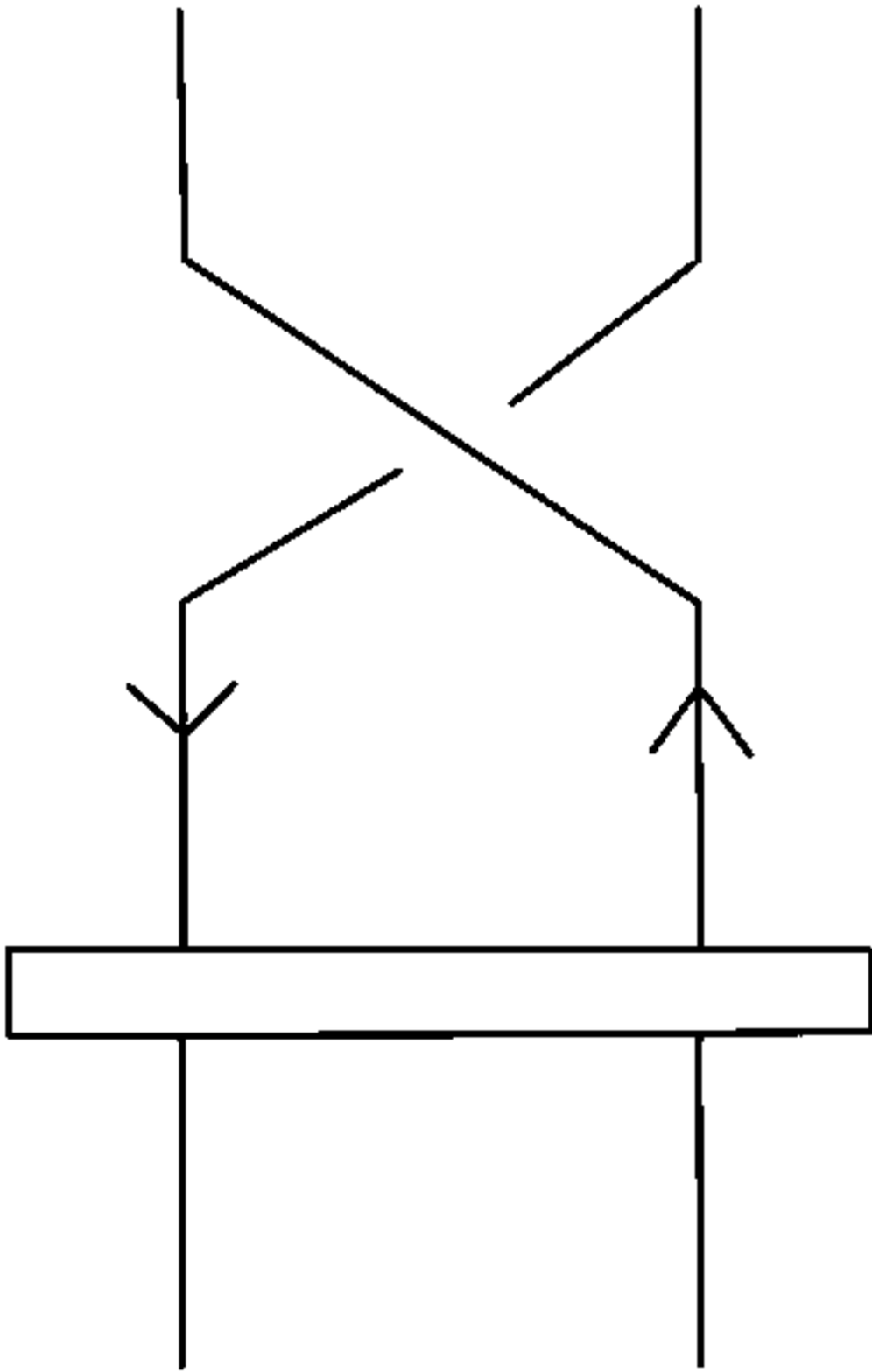}\put(-50,58){\scriptsize$m$\normalsize}\put(-15,58){\scriptsize$n$\normalsize}\end{minipage}\hspace{65pt}\Biggr\rangle_{3}=(-1)^{mn}q^{-\frac{nm}{6}}\Biggl\langle\hspace{10pt}\begin{minipage}{1\unitlength}\includegraphics[scale=0.08]{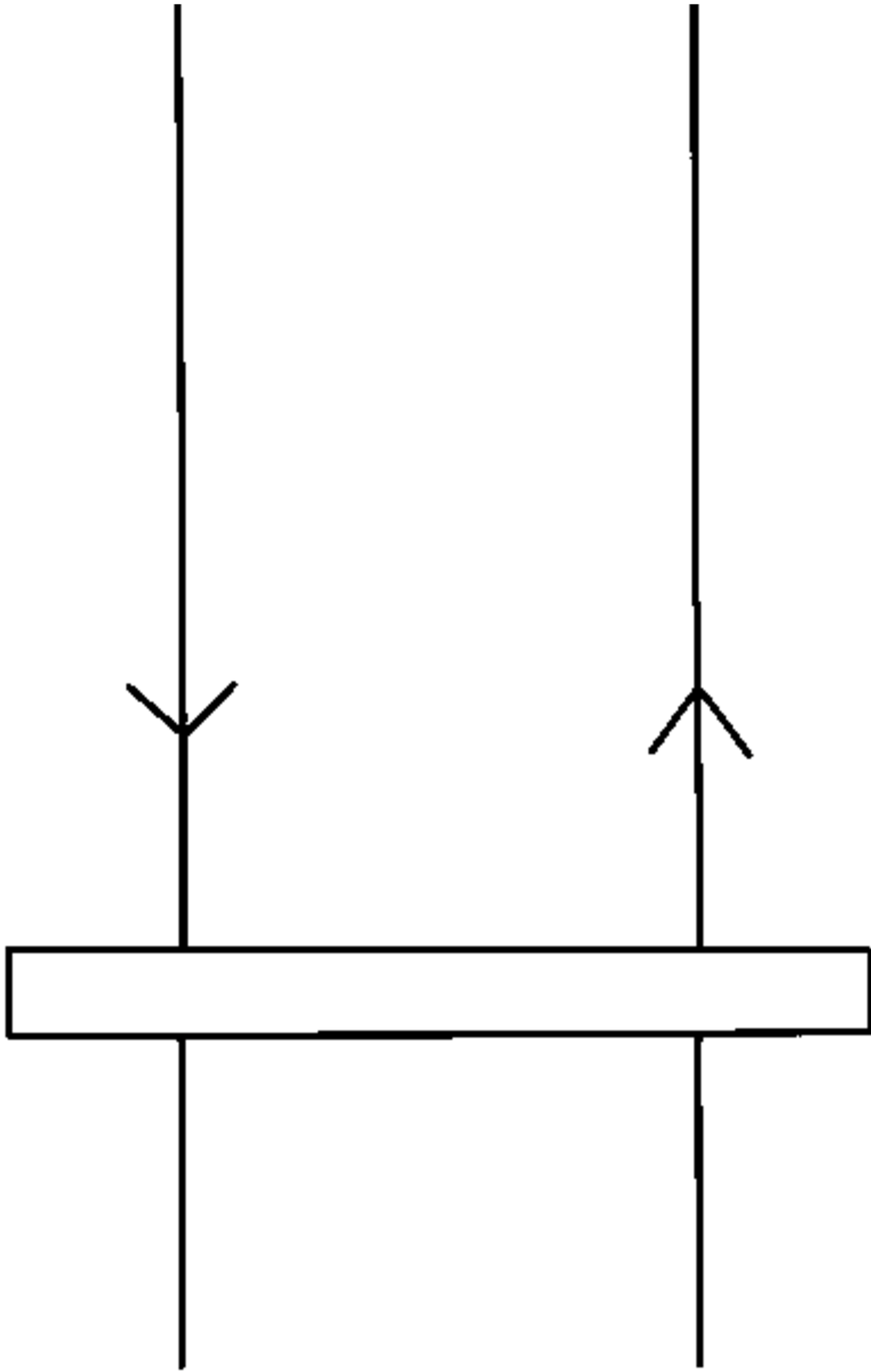}\put(-50,58){\scriptsize$m$\normalsize}\put(-15,58){\scriptsize$n$\normalsize}\end{minipage}\hspace{55pt}\Biggr\rangle_{3},\quad\Biggl\langle\hspace{10pt}\begin{minipage}{1\unitlength}\includegraphics[scale=0.08]{pic/A2claspTwistP.eps}\put(-50,58){\scriptsize$m$\normalsize}\put(-15,58){\scriptsize$n$\normalsize}\end{minipage}\hspace{65pt}\Biggr\rangle_{3}=(-1)^{mn}q^{\frac{nm}{6}}\Biggl\langle\hspace{10pt}\begin{minipage}{1\unitlength}\includegraphics[scale=0.08]{pic/A2claspTwistafter.eps}\put(-50,58){\scriptsize$m$\normalsize}\put(-15,58){\scriptsize$n$\normalsize}\end{minipage}\hspace{55pt}\Biggr\rangle_{3},\\
\label{al:delta9}
&\Biggl\langle\hspace{10pt}\begin{minipage}{1\unitlength}\includegraphics[scale=0.08]{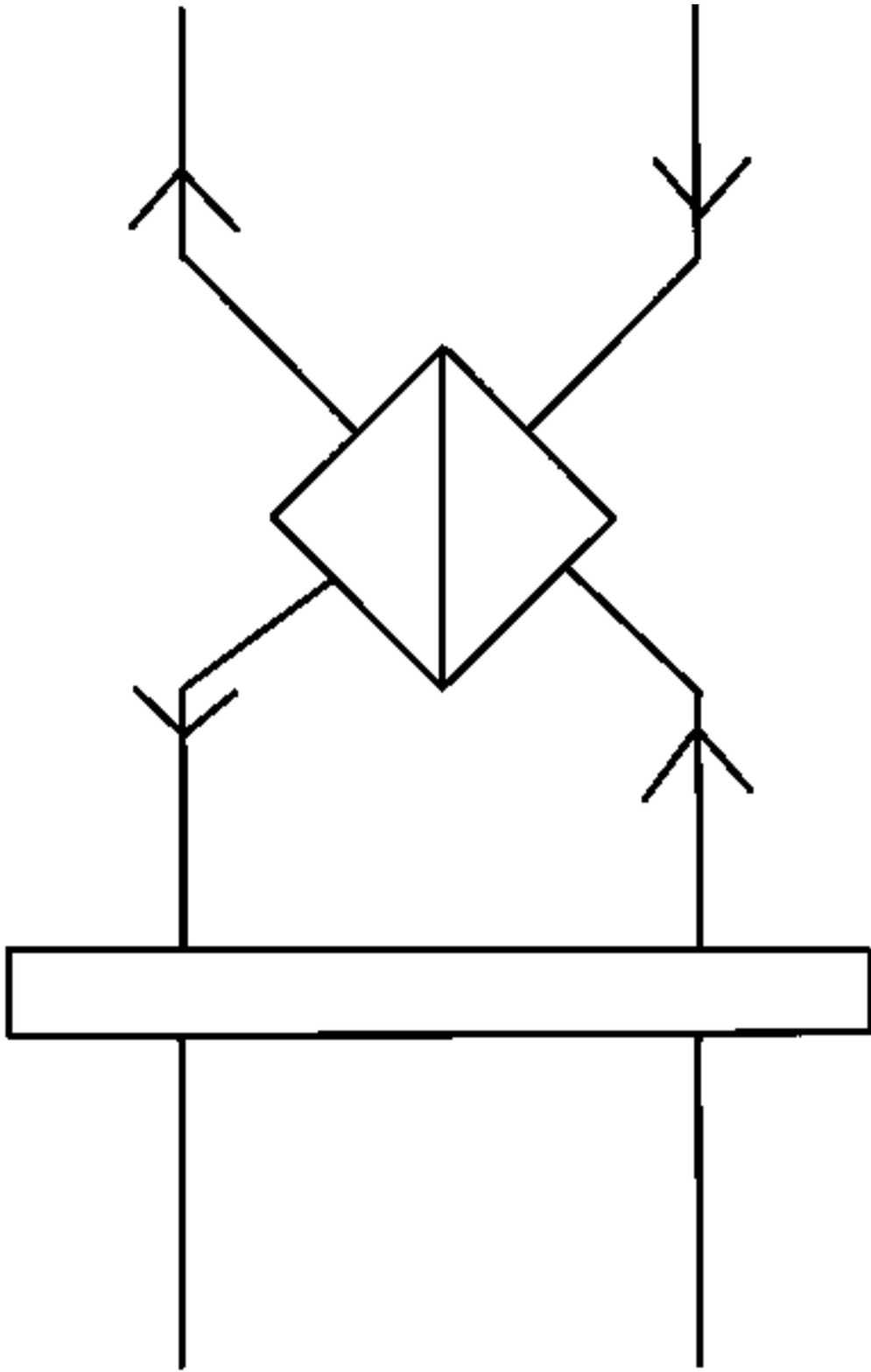}\put(-50,58){\scriptsize$m$\normalsize}\put(-15,58){\scriptsize$n$\normalsize}\end{minipage}\hspace{65pt}\Biggr\rangle_{3}=(-1)^{mn}q^{-\frac{nm}{6}}\Biggl\langle\hspace{10pt}\begin{minipage}{1\unitlength}\includegraphics[scale=0.08]{pic/A2claspTwistafter.eps}\put(-50,58){\scriptsize$m$\normalsize}\put(-15,58){\scriptsize$n$\normalsize}\end{minipage}\hspace{55pt}\Biggr\rangle_{3},\quad\Biggl\langle\hspace{10pt}\begin{minipage}{1\unitlength}\includegraphics[scale=0.08]{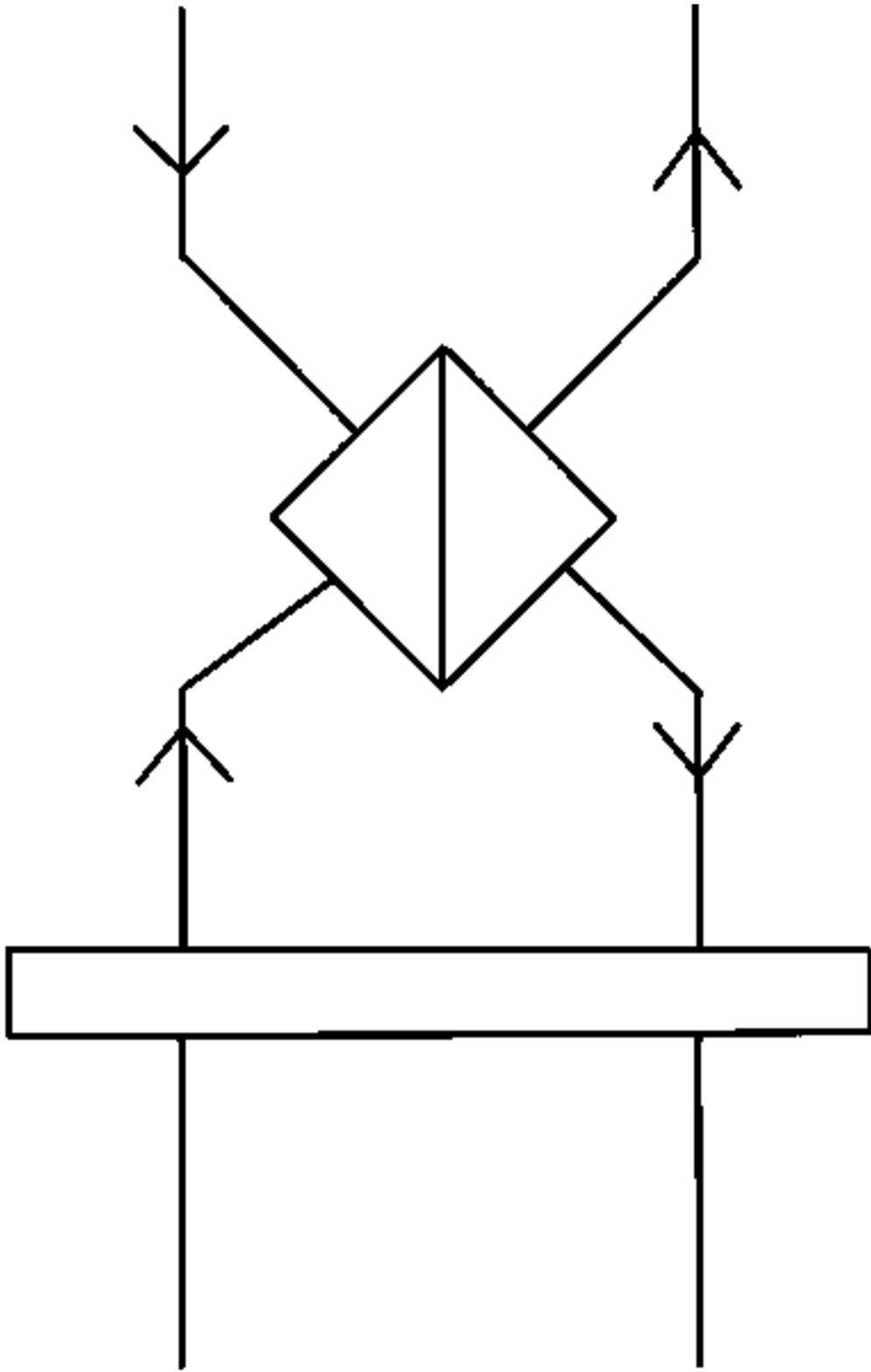}\put(-50,58){\scriptsize$m$\normalsize}\put(-15,58){\scriptsize$n$\normalsize}\end{minipage}\hspace{65pt}\Biggr\rangle_{3}=(-1)^{mn}q^{\frac{nm}{6}}\Biggl\langle\hspace{10pt}\begin{minipage}{1\unitlength}\includegraphics[scale=0.08]{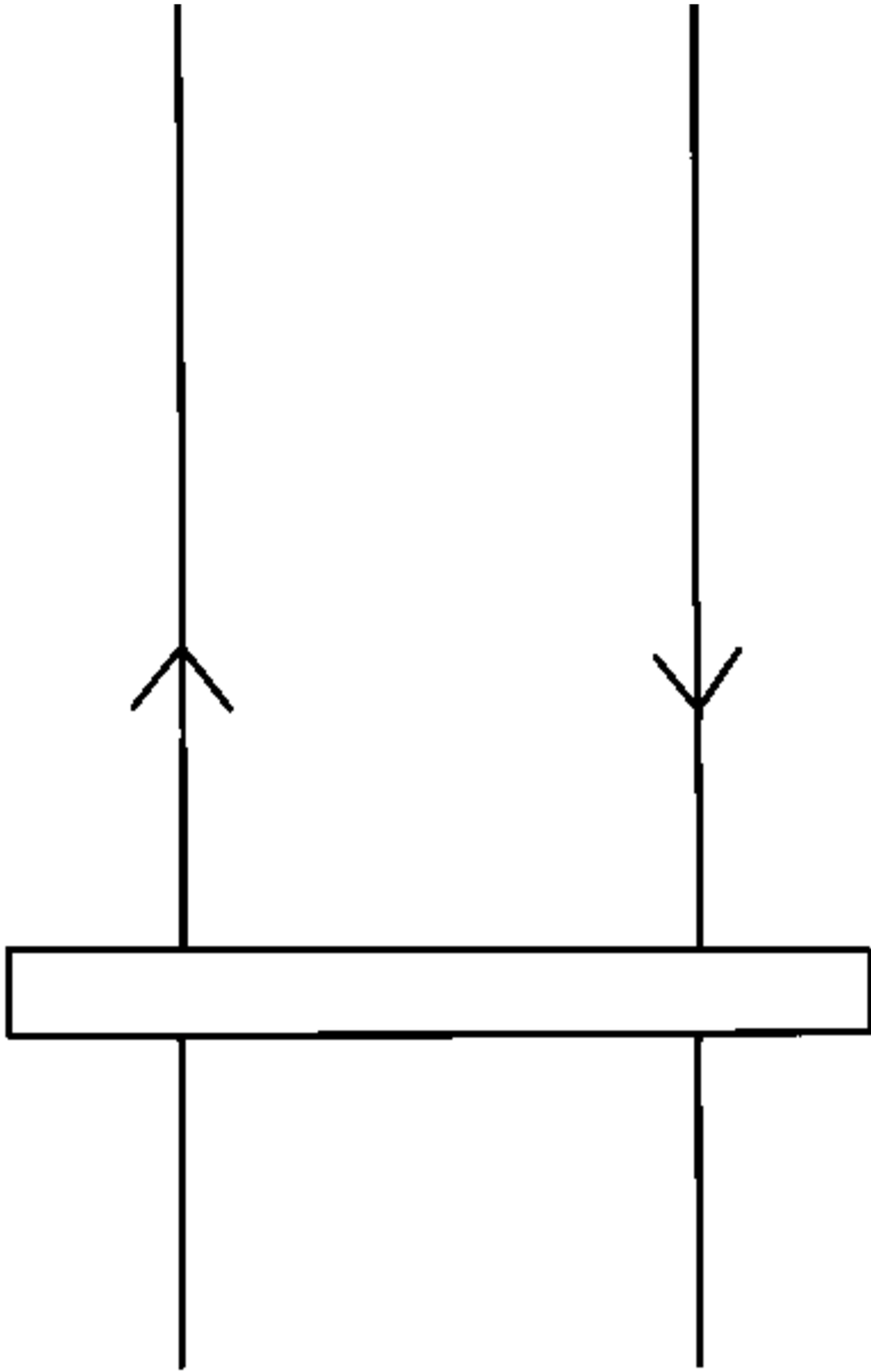}\put(-50,58){\scriptsize$m$\normalsize}\put(-15,58){\scriptsize$n$\normalsize}\end{minipage}\hspace{55pt}\Biggr\rangle_{3}.
\end{align}
\end{Lemma}

In \cite{Yua17}, Yuasa gave the $m$ times full twists formula for two strands with opposite directions and the $A_{2}$ bracket bubble skein expansion formula. We use these formulae to calculate the one-row $\mathfrak{sl}_{3}$ colored Jones polynomial for pretzel links in Section $\ref{sec:Computing}$.
\begin{Theorem}[$m$ times full twist formula \cite{Yua17}]
For a positive integer $n$, we have
\label{Thm:yuasaCP}
\begin{align*}
\Biggl\langle \begin{minipage}{1\unitlength}\includegraphics[scale=0.1]{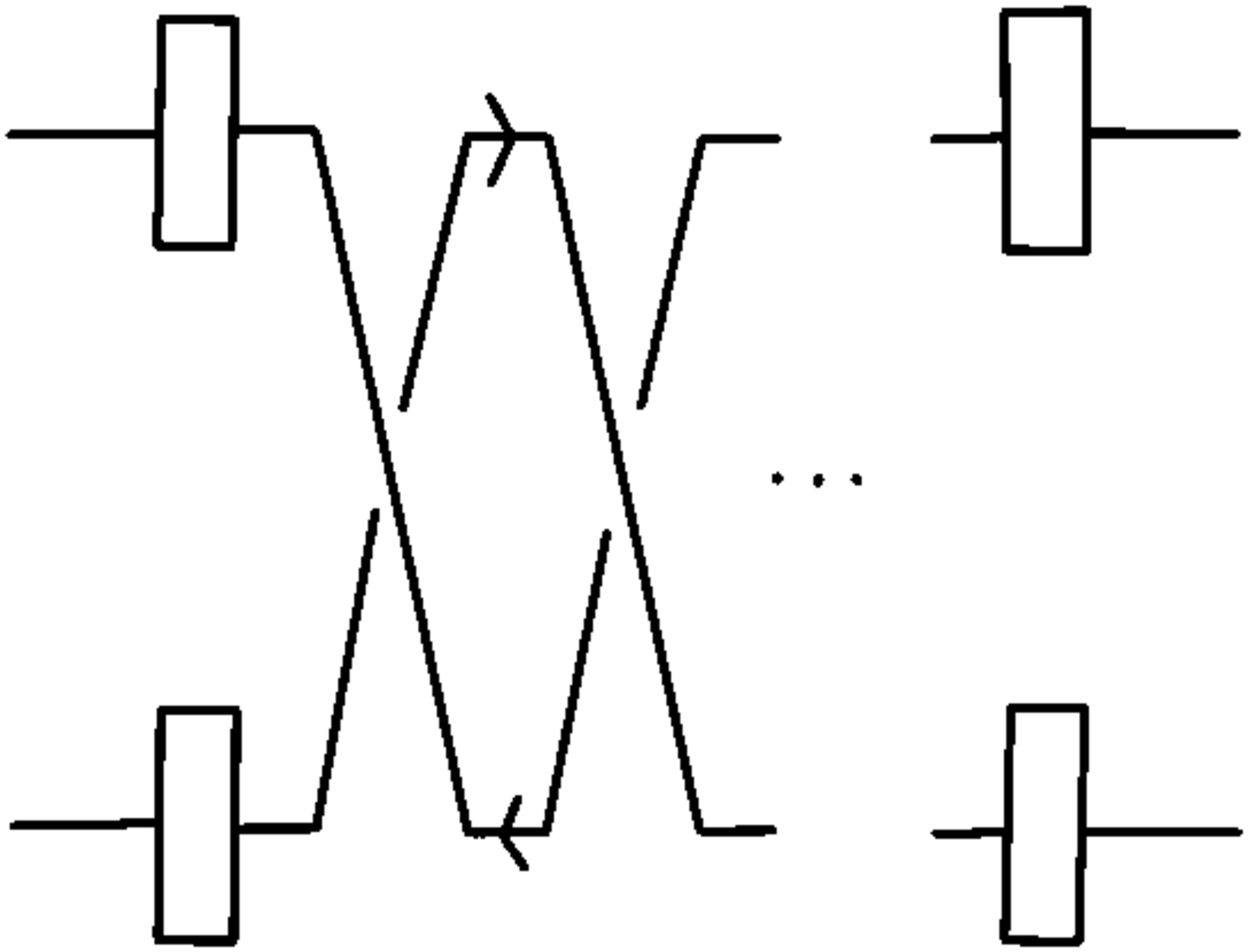}\put(-60,0){\scriptsize$m$ times full twists\normalsize}\put(-72,60){\scriptsize$n$\normalsize}\put(-72,10){\scriptsize$n$\normalsize}\end{minipage}\hspace{78pt}\Biggr\rangle_{3}=&\sum_{0\leq k_{m}\leq k_{m-1}\leq\cdots\leq k_{1}\leq n}\phi(n,k_{1},k_{2},...,k_{m})_{q^{\epsilon_{m}}}\Biggl\langle\hspace{10pt} \begin{minipage}{1\unitlength}\includegraphics[scale=0.1]{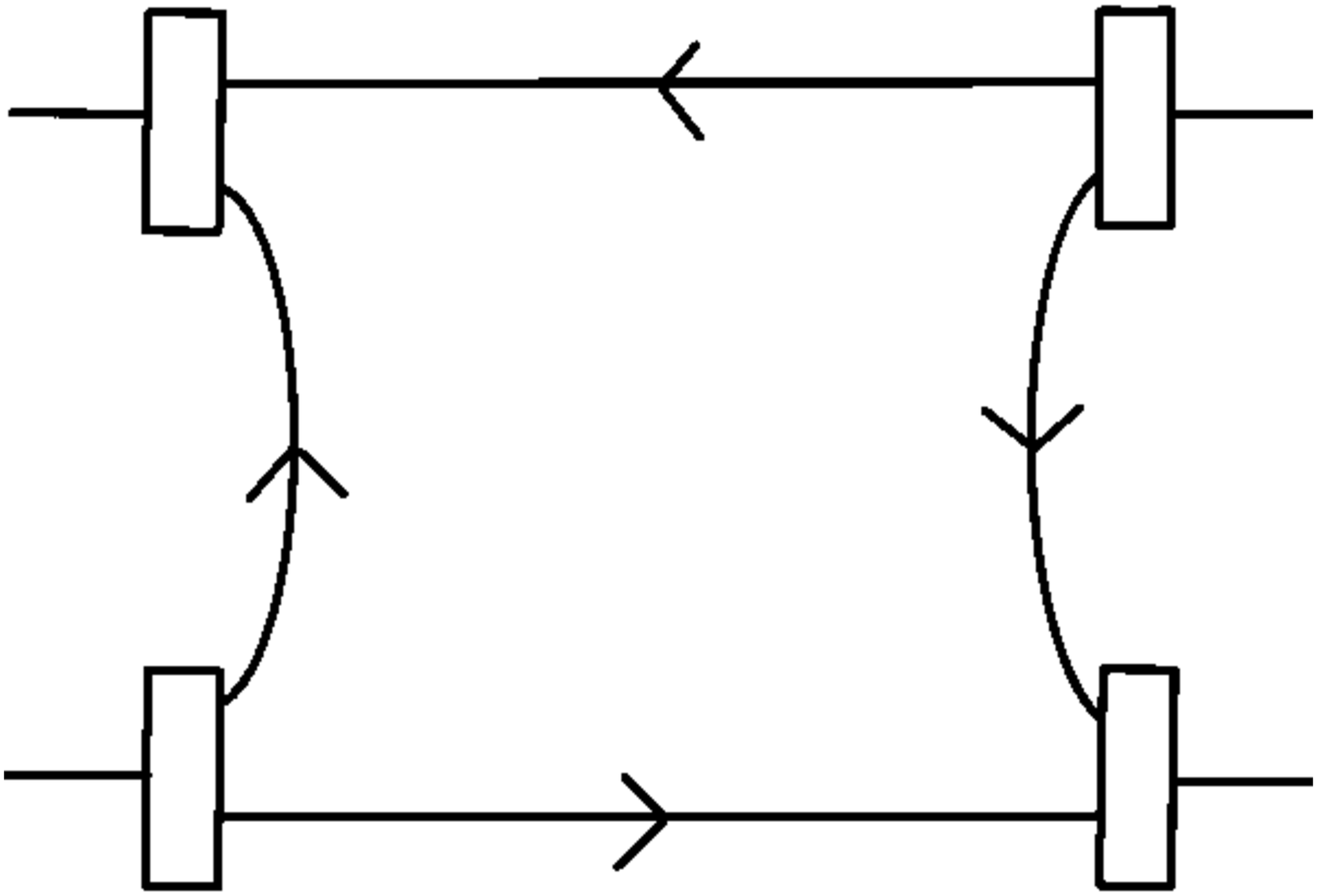}\put(-74,60){\scriptsize$n$\normalsize}\put(-74,10){\scriptsize$n$\normalsize}\put(-40,8){\scriptsize$k_{m}$\normalsize}\put(-40,62){\scriptsize$k_{m}$\normalsize}\put(-80,35){\scriptsize$n-k_{m}$\normalsize}\put(-13,35){\scriptsize$n-k_{m}$\normalsize}\end{minipage}\hspace{82pt}\Biggr\rangle_{3}
\end{align*}
where
\begin{align*}
\phi(n,k_{1},k_{2},...,k_{m})_{q^{\epsilon_m}}=\frac{(q^{\epsilon_m})^{-\frac{2m}{3}(n^{2}+3n)}(q^{\epsilon_m})^{n-k_{m}}(q^{\epsilon_m})^{\sum_{i=1}^{m}(k_{i}^{2}+2k_{i})}(q^{\epsilon_m})^{2}_{n}}{(q^{\epsilon_m})_{n-k_{1}}(q^{\epsilon_m})_{k_{1}-k_{2}}\cdots(q^{\epsilon_m})_{k_{m-1}-k_{m}}(q^{\epsilon_m})^{2}_{k_{m}}}.
\end{align*}
\end{Theorem}

\begin{Theorem}[the $A_{2}$ bracket bubble skein expansion formula\cite{Yua17}]
Let $n$ be a positive integer  and $k$, $l$ non-negative integers. For $n\ge k,l$ , we have 
\label{Thm:yuasa6}
\begin{align*}
\Biggl\langle \hspace{5pt}\begin{minipage}{1\unitlength}\includegraphics[scale=0.1]{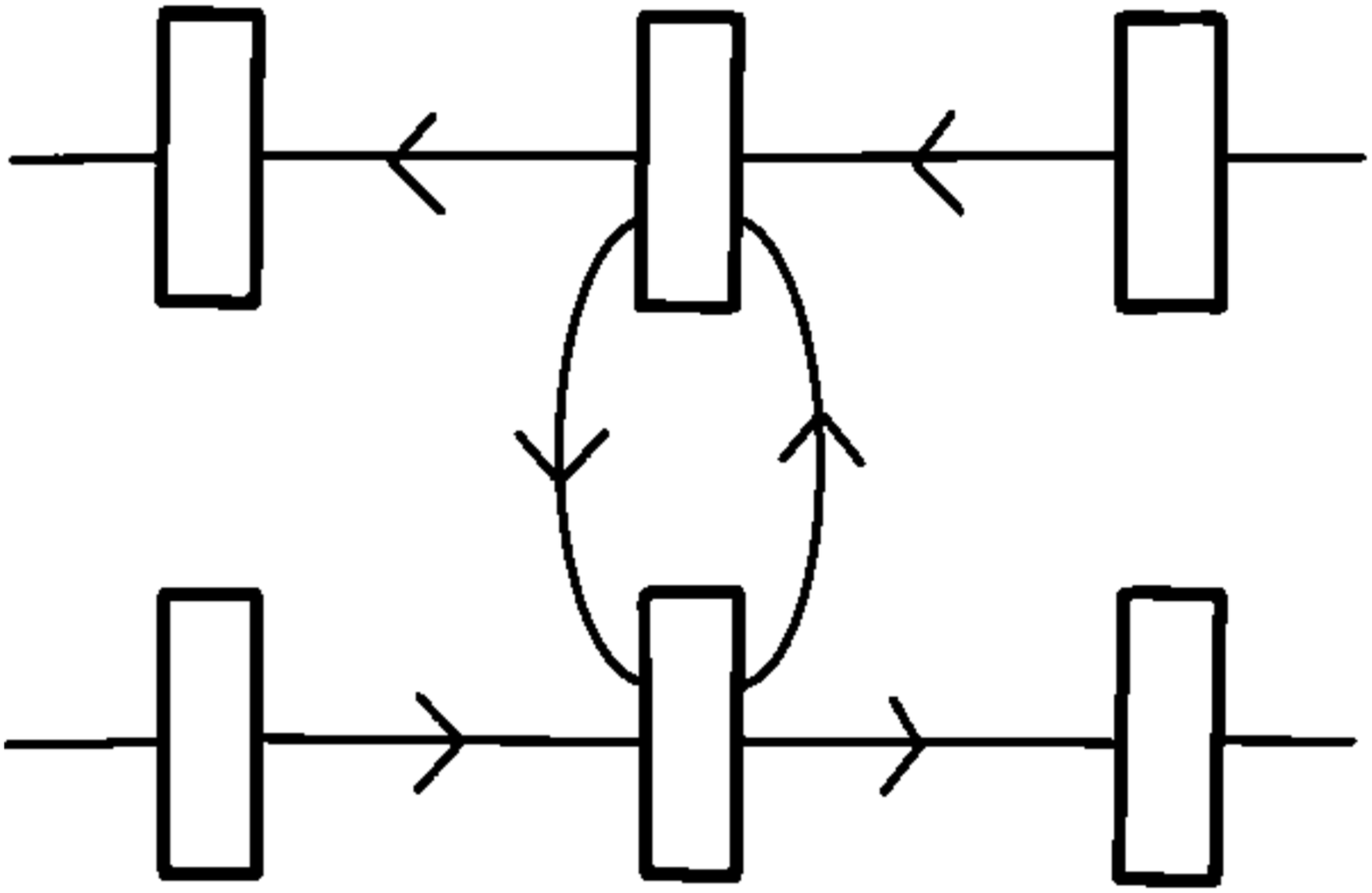}\put(-87,58){\scriptsize$n-k$\normalsize}\put(-87,12){\scriptsize$n-k$\normalsize}\put(-40,7){\scriptsize$n$\normalsize}\put(-40,63){\scriptsize$n$\normalsize}\put(-6,58){\scriptsize$n-l$\normalsize}\put(-6,12){\scriptsize$n-l$\normalsize}\end{minipage}\hspace{80pt}\Biggr\rangle_{3}=&\sum_{t=\max\{k,l\}}^{\min\{k+l,n\}}\psi(n,t,k,l)\Biggl\langle\hspace{10pt} \begin{minipage}{1\unitlength}\includegraphics[scale=0.1]{pic/twist_after.eps}\put(-75,60){\scriptsize$n$\normalsize}\put(-75,12){\scriptsize$n$\normalsize}\put(-44,8){\scriptsize$n-t$\normalsize}\put(-44,62){\scriptsize$n-t$\normalsize}\put(-70,35){\scriptsize$t$\normalsize}\put(-13,35){\scriptsize$t$\normalsize}\put(-6,60){\scriptsize$n$\normalsize}\put(-6,12){\scriptsize$n$\normalsize}\end{minipage}\hspace{82pt}\Biggr\rangle_{3}
\end{align*}
where
\begin{align*}
\psi(n,t,k,l)=\frac{q^{(t+1)(t-k-l)+kl}(q)_{k}(q)_{l}(q)^{2}_{n-k}(q)^{2}_{n-l}(q)_{2n-t+2}}{(q)^{2}_{n}(q)^{2}_{n-t}(q)_{t-k}(q)_{t-l}(q)_{2n-k-l+2}(q)_{-t+k+l}}.
\end{align*}
\end{Theorem}

\section{The formulae for the $A_{2}$ web space}
\label{sec:The formulas of the clasped }

We give $m$ times half twists formulae for two strands with the same directions for $A_{2}$ bracket. 
\begin{Proposition}[$m$ times half twists formula]
\label{Pro:halfTwist1}
Let $n$ be a positive integer and $k_{0}=n$. For a positeve integer $m$, we have
\begin{align}
\label{eq:halfTwist1}
\Biggl\langle
\hspace{5pt}\begin{minipage}{1\unitlength}\includegraphics[scale=0.1]{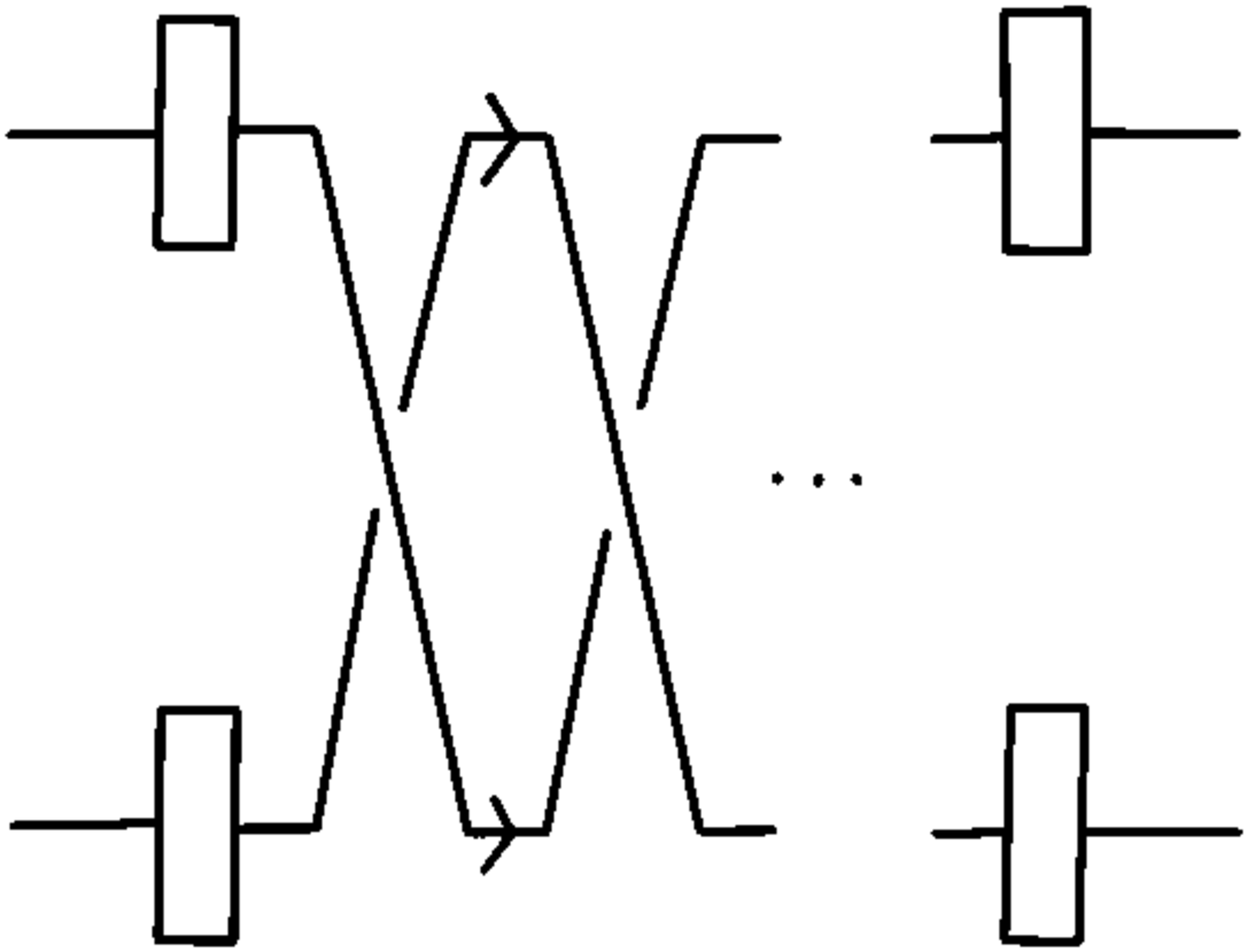}\put(-64,0){\scriptsize$m$
times half
twists\normalsize}\put(-72,60){\scriptsize$n$\normalsize}\put(-72,10){\scriptsize$n$\normalsize}\end{minipage}\hspace{80pt}\Biggr\rangle_{3}=\sum_{0\le
k_{m}\le k_{m-1}\le\cdots\le k_{1}\le
n}\chi_{1}(n,k_{1},k_{2},...,k_{m})\Biggl\langle
\hspace{5pt}\begin{minipage}{1\unitlength}\includegraphics[scale=0.1]{pic/delta.eps}\put(-50,64){\scriptsize$k_{m}$\normalsize}\put(-50,8){\scriptsize$k_{m}$\normalsize}\put(-67,42){\scriptsize$n-k_{m}$\normalsize}\put(-74,60){\scriptsize$n$\normalsize}\put(-74,10){\scriptsize$n$\normalsize}\end{minipage}\hspace{80pt}\Biggr\rangle_{3}
\end{align}
where
\begin{align*}
&\chi_{1}(n,k_{1},k_{2},...,k_{m})=(-1)^{nm}\frac{q^{-\frac{1}{6}(n^{2}+3n)m}q^{\frac{1}{2}(n-k_{m})}(-1)^{\sum_{i=1}^{m}k_{i}}q^{\sum_{i=1}^{m}\frac{1}{2}(k_{i}^{2}+k_{i})}(q)_{n}}{\prod_{i=1}^{m}(q)_{k_{i-1}-k_{i}}(q)_{k_{m}}}.
\end{align*}
For a negative integer $m$, we have
\begin{align}
\label{Pro:halfTwist2}
\Biggl\langle \hspace{5pt}\begin{minipage}{1\unitlength}\includegraphics[scale=0.1]{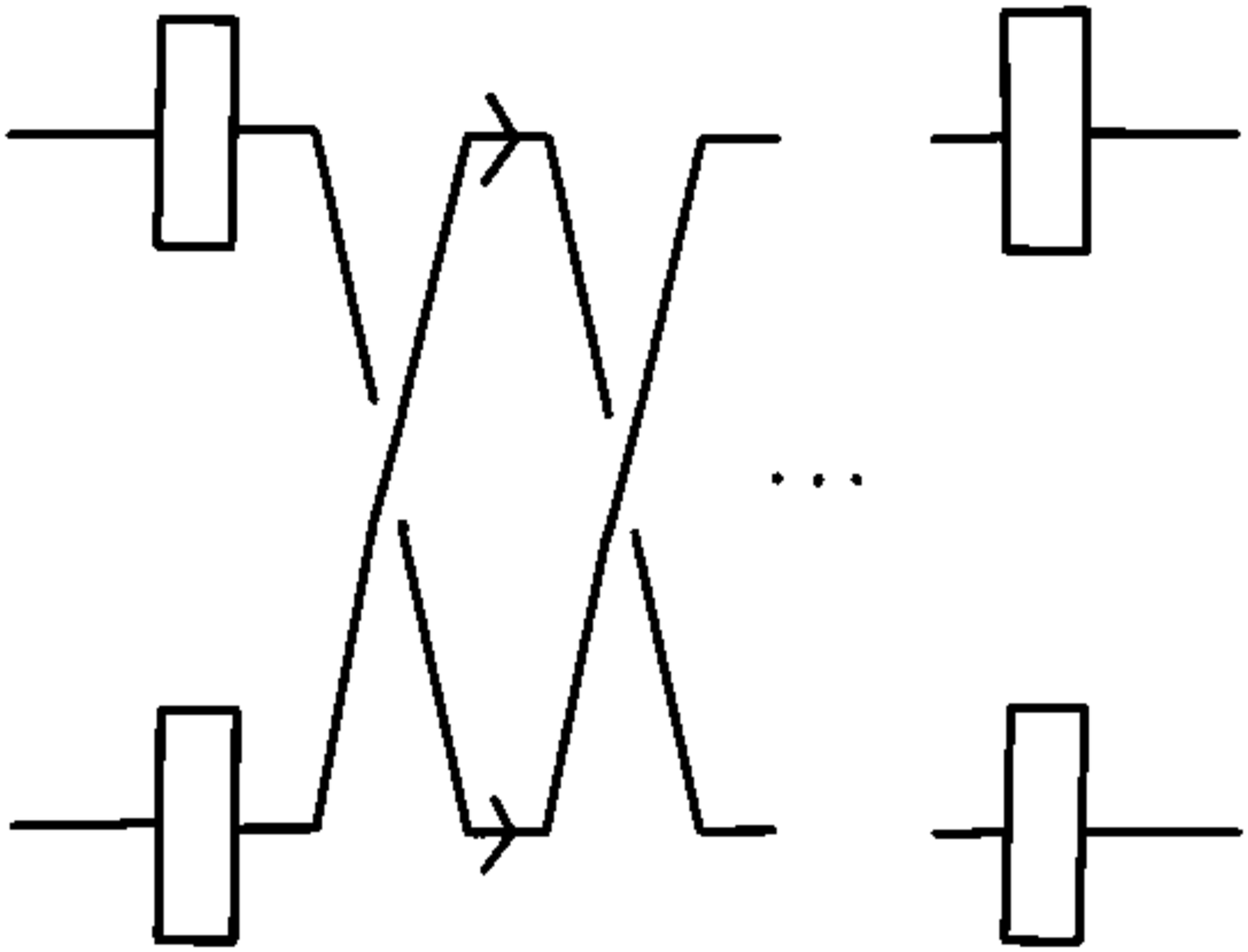}\put(-64,0){\scriptsize$m$ times-half twists\normalsize}\put(-72,60){\scriptsize$n$\normalsize}\put(-72,10){\scriptsize$n$\normalsize}\end{minipage}\hspace{80pt}\Biggr\rangle_{3}=&\sum_{0\le k_{m}\le k_{m-1}\le\cdots\le k_{1}\le n}\chi_{-1}(n,k_{1},k_{2},...,k_{m})\Biggl\langle \hspace{5pt}\begin{minipage}{1\unitlength}\includegraphics[scale=0.1]{pic/delta.eps}\put(-50,64){\scriptsize$k_{m}$\normalsize}\put(-50,8){\scriptsize$k_{m}$\normalsize}\put(-67,42){\scriptsize$n-k_{m}$\normalsize}\put(-74,60){\scriptsize$n$\normalsize}\put(-74,10){\scriptsize$n$\normalsize}\end{minipage}\hspace{80pt}\Biggr\rangle_{3}
\end{align}
where 
\begin{align*}
&\chi_{-1}(n,k_{1},k_{2},...,k_{m})=(-1)^{nm}\frac{q^{\frac{1}{6}(n^{2}+3n)m}q^{-\frac{1}{2}(n-k_{m})}(-1)^{\sum_{i=1}^{m}k_{i}}q^{\sum_{i=1}^{m}\frac{1}{2}(k_{i}^{2}-k_{i})}q^{\sum_{i=1}^{m}k_{i-1}k_{i}}(q)_{n}}{\prod_{i=1}^{m}(q)_{k_{i-1}-k_{i}}(q)_{k_{m}}}.
\end{align*}
\end{Proposition}

To prove Proposition $\ref{Pro:halfTwist1}$, we prepare the following lemma.
\begin{Lemma}
Let $n$ be a positive integer and $k$ non-negative integer. For $0\le k\le n$, we have
\label{Lem:DeltaClasp}
\begin{align}
\Biggl\langle \hspace{5pt}\begin{minipage}{1\unitlength}\includegraphics[scale=0.1]{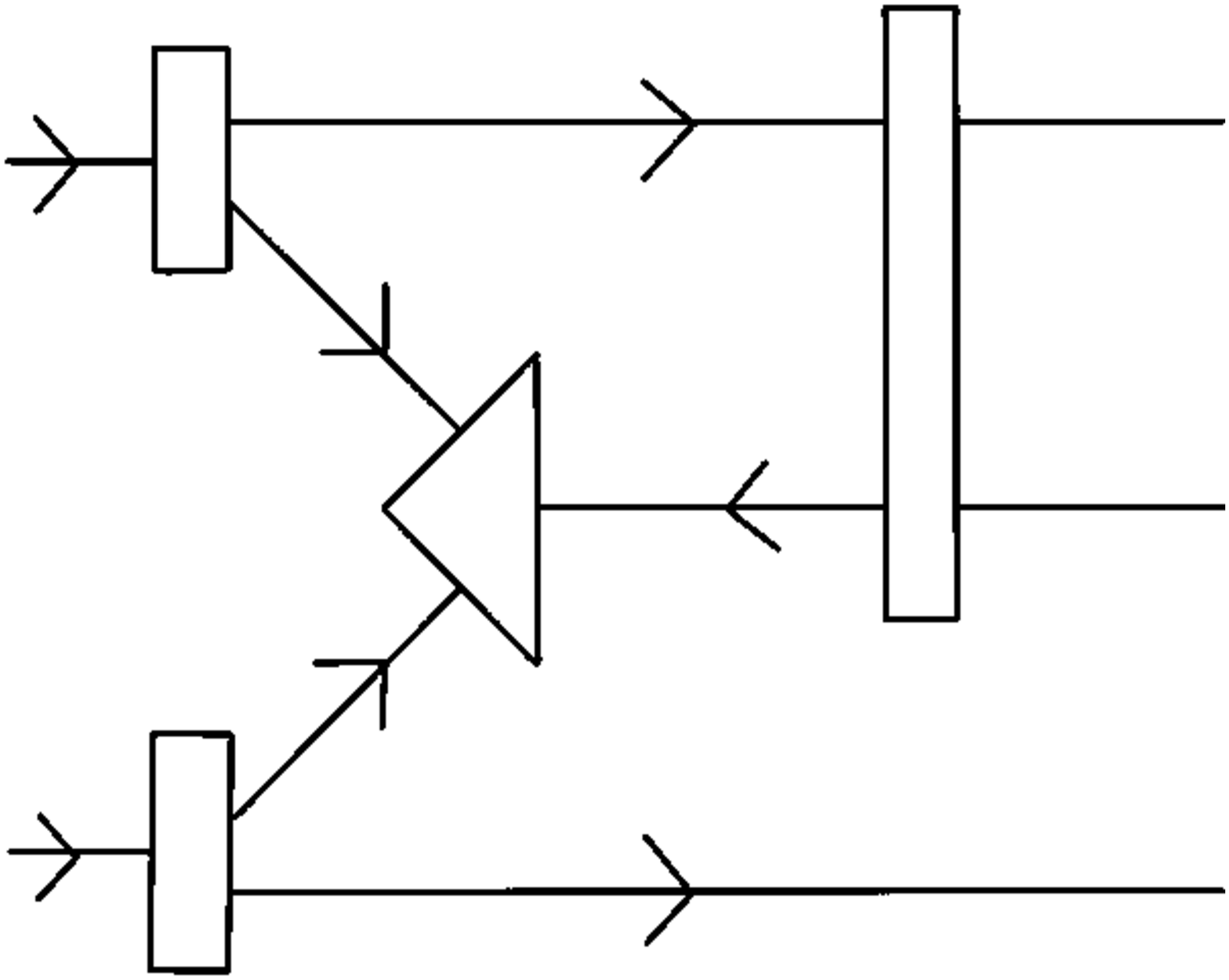}\put(-75,62){\scriptsize$n$\normalsize}\put(-65,38){\scriptsize$n-k$\normalsize}\put(-45,65){\scriptsize$k$\normalsize}\put(-45,4){\scriptsize$k$\normalsize}\put(-75,8){\scriptsize$n$\normalsize}\end{minipage}\hspace{80pt}\Biggr\rangle_{3}=\Biggl\langle \hspace{5pt}\begin{minipage}{1\unitlength}\includegraphics[scale=0.1]{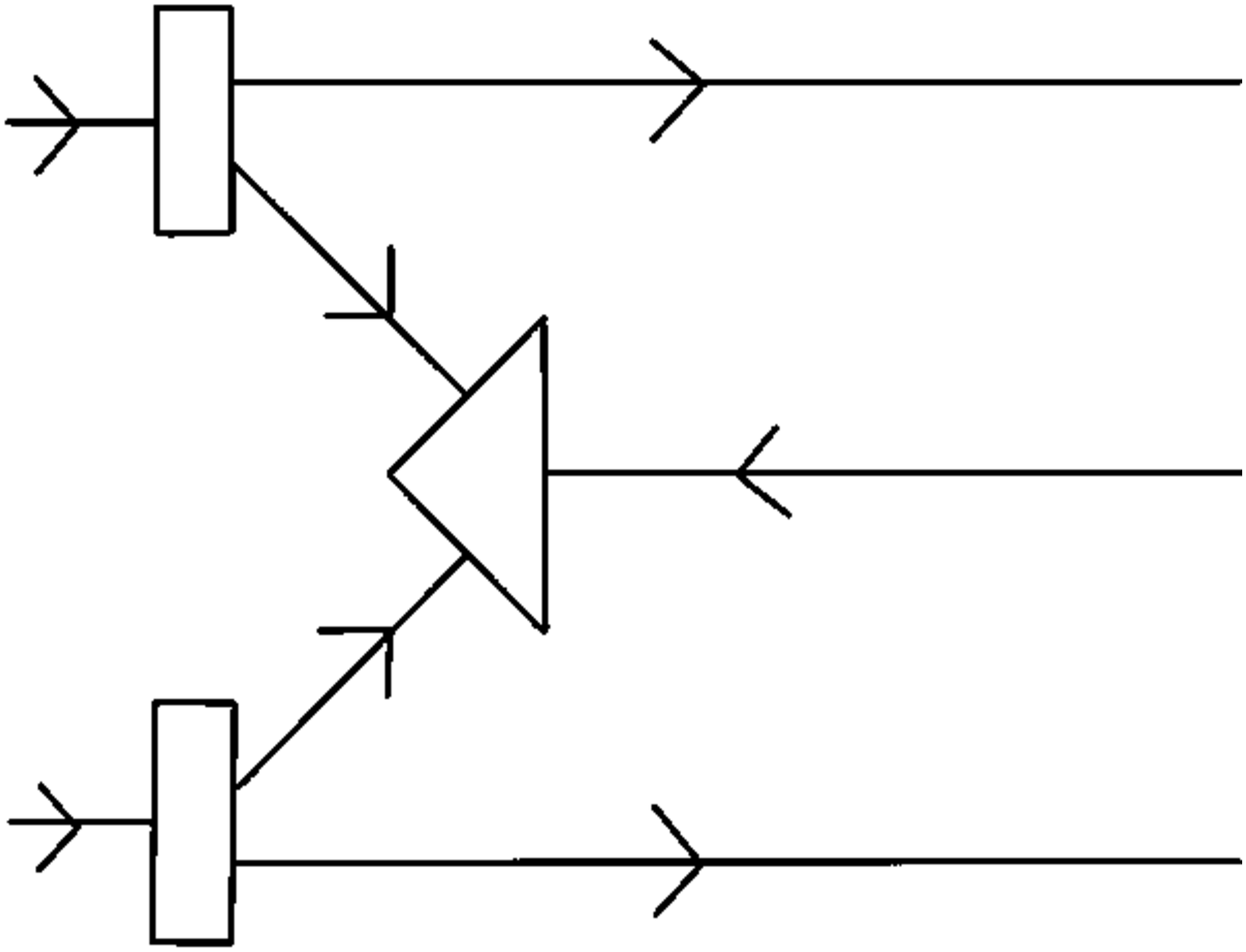}\put(-75,62){\scriptsize$n$\normalsize}\put(-65,38){\scriptsize$n-k$\normalsize}\put(-45,65){\scriptsize$k$\normalsize}\put(-45,4){\scriptsize$k$\normalsize}\put(-75,8){\scriptsize$n$\normalsize}\end{minipage}\hspace{80pt}\Biggr\rangle_{3}.
\end{align}
\end{Lemma}
\begin{proof}
By definition of $A_{2}$ clasp of type $(n_{1},n_{2})$, we obtain
\begin{align*}
\Biggl\langle \hspace{5pt}\begin{minipage}{1\unitlength}\includegraphics[scale=0.1]{pic/DeltaClasp2.eps}\put(-75,62){\scriptsize$n$\normalsize}\put(-65,38){\scriptsize$n-k$\normalsize}\put(-45,65){\scriptsize$k$\normalsize}\put(-45,4){\scriptsize$k$\normalsize}\put(-75,8){\scriptsize$n$\normalsize}\end{minipage}\hspace{80pt}\Biggr\rangle_{3}=&\sum_{i=0}^{\min\{n-k, k\}}(-1)^{i}\frac{\begin{bmatrix}
n-k  \\
i \\
\end{bmatrix}_{q}\begin{bmatrix}
k  \\
i \\
\end{bmatrix}_{q}}{\begin{bmatrix}
n+i  \\
i \\
\end{bmatrix}_{q}}\Biggl\langle \hspace{5pt}\begin{minipage}{1\unitlength}\includegraphics[scale=0.1]{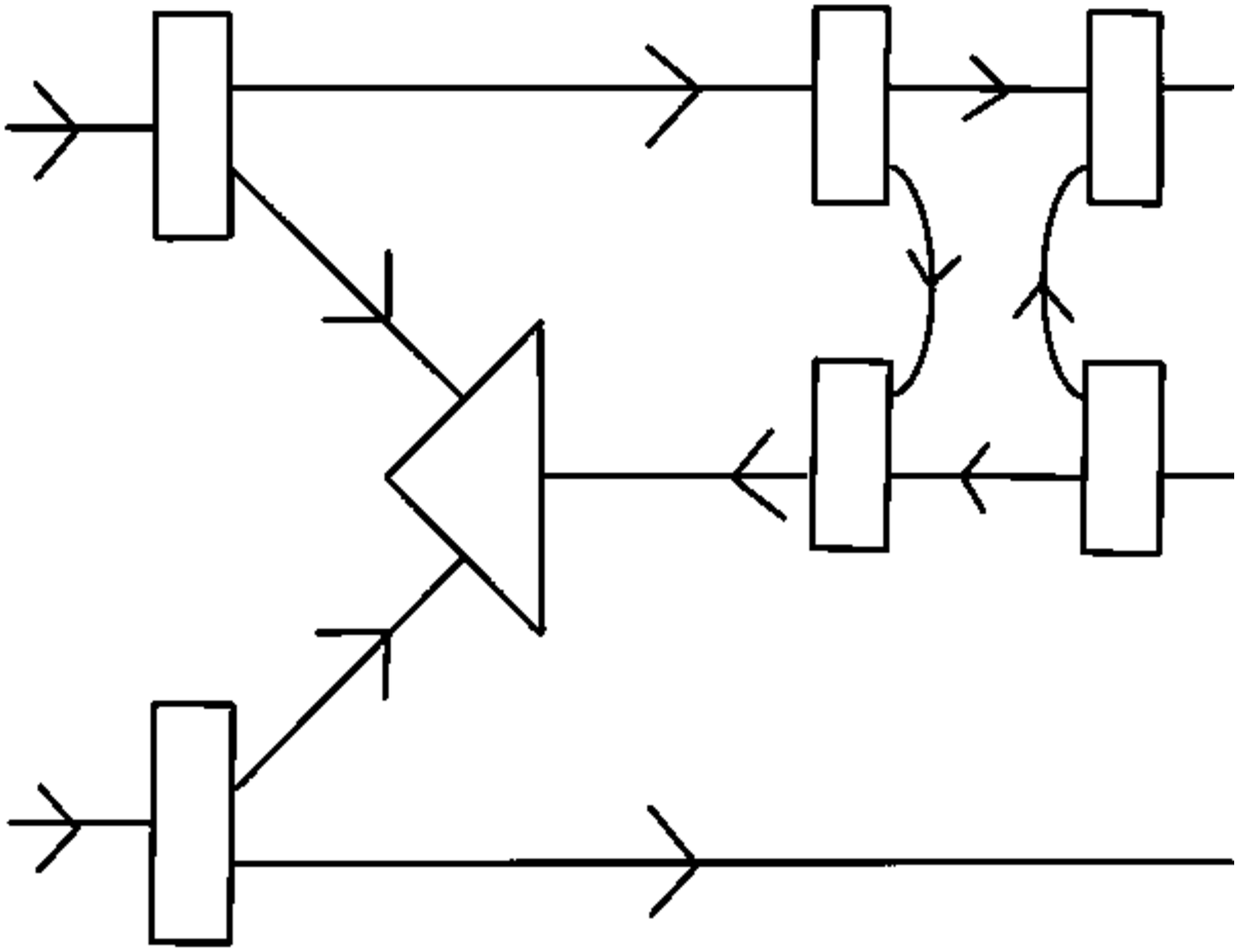}\put(-75,62){\scriptsize$n$\normalsize}\put(-65,38){\scriptsize$n-k$\normalsize}\put(-45,65){\scriptsize$k$\normalsize}\put(-25,48){\scriptsize$i$\normalsize}\put(-5,48){\scriptsize$i$\normalsize}\put(-45,4){\scriptsize$k$\normalsize}\put(-75,8){\scriptsize$n$\normalsize}\end{minipage}\hspace{80pt}\Biggr\rangle_{3}\\
=&\Biggl\langle \hspace{5pt}\begin{minipage}{1\unitlength}\includegraphics[scale=0.1]{pic/DeltaClasp1.eps}\put(-75,62){\scriptsize$n$\normalsize}\put(-65,38){\scriptsize$n-k$\normalsize}\put(-45,65){\scriptsize$k$\normalsize}\put(-45,4){\scriptsize$k$\normalsize}\put(-75,8){\scriptsize$n$\normalsize}\end{minipage}\hspace{80pt}\Biggr\rangle_{3}.
\end{align*}
The right-hand side of the above equation is zero by $(\ref{al:delta3})$, $(\ref{al:double1})$, and $(\ref{al:double2})$ except for $i=0$.
\end{proof}

\begin{proof}[\mbox{Proof of Proposition $\ref{Pro:halfTwist1}$}]
We proceed by induction on $m$. We can see that $(\ref{eq:halfTwist1})$ holds by $(\ref{al:delta1})$ for $m=1$. Assume $(\ref{eq:halfTwist1})$ holds when $m=l$ for some integers $l\ge 1$.
\begin{align*}
&\Biggl\langle \hspace{5pt}\begin{minipage}{1\unitlength}\includegraphics[scale=0.1]{pic/2m+1_twists.eps}\put(-64,0){\scriptsize$m$ times-half twists\normalsize}\put(-72,60){\scriptsize$n$\normalsize}\put(-72,10){\scriptsize$n$\normalsize}\end{minipage}\hspace{80pt}\Biggr\rangle_{3}\\
&=\sum_{0\le k_{l}\le k_{l-1}\le\cdots\le k_{1}\le n}\chi_{1}(n,k_{1},k_{2},...,k_{l})\Biggl\langle \hspace{5pt}\begin{minipage}{1\unitlength}\includegraphics[scale=0.1]{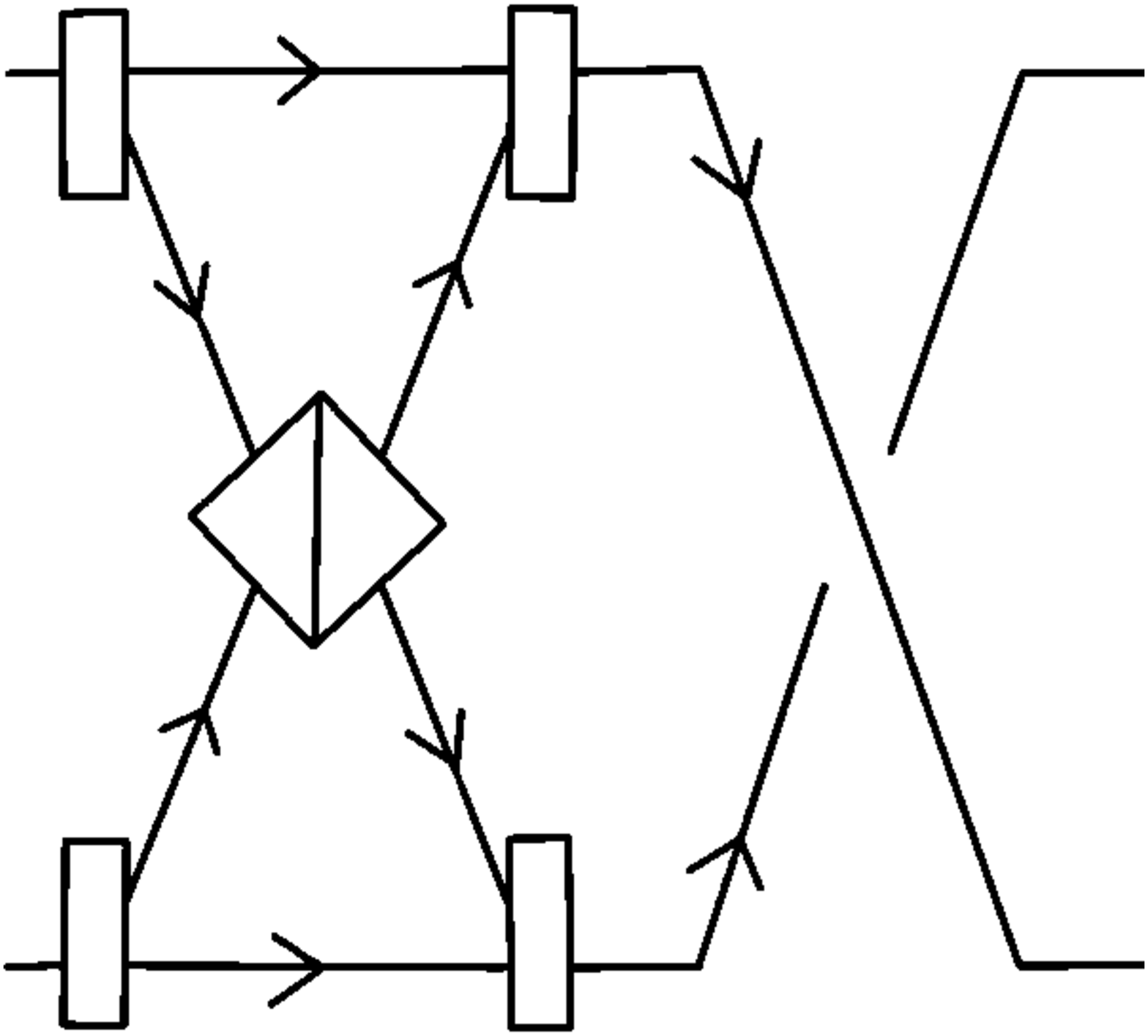}\put(-80,45){\scriptsize$n-k_{l}$\normalsize}\put(-80,30){\scriptsize$n-k_{l}$\normalsize}\put(-60,72){\scriptsize$k_{l}$\normalsize}\put(-60,13){\scriptsize$k_{l}$\normalsize}\put(-77,68){\scriptsize$n$\normalsize}\put(-77,10){\scriptsize$n$\normalsize}\end{minipage}\hspace{80pt}\Biggr\rangle_{3}=\sum_{0\le k_{l}\le k_{l-1}\le\cdots\le k_{1}\le n}\chi_{1}(n,k_{1},k_{2},...,k_{l})\Biggl\langle \hspace{5pt}\begin{minipage}{1\unitlength}\includegraphics[scale=0.1]{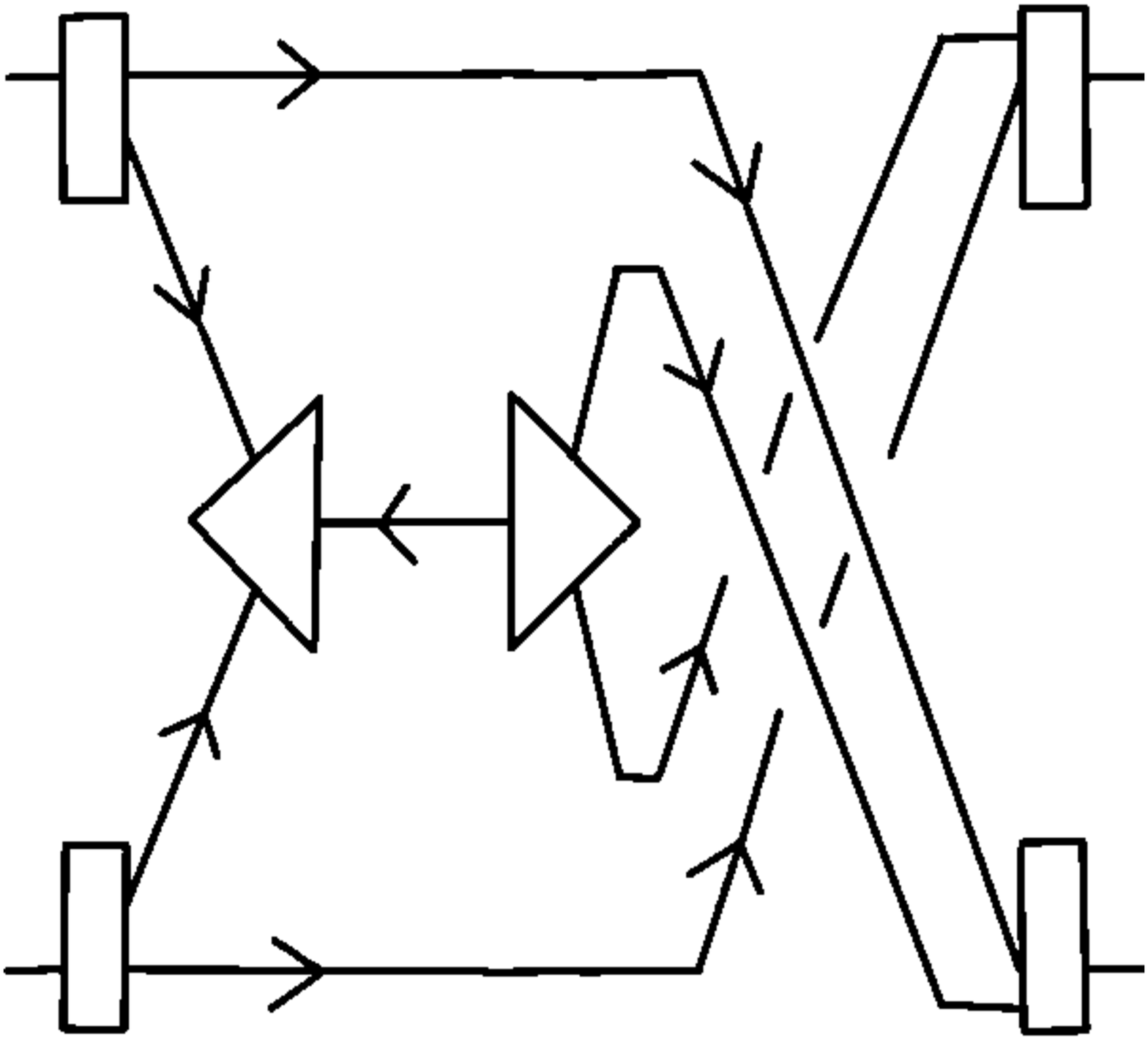}\put(-80,45){\scriptsize$n-k_{l}$\normalsize}\put(-80,30){\scriptsize$n-k_{l}$\normalsize}\put(-60,72){\scriptsize$k_{l}$\normalsize}\put(-60,13){\scriptsize$k_{l}$\normalsize}\put(-77,68){\scriptsize$n$\normalsize}\put(-77,10){\scriptsize$n$\normalsize}\end{minipage}\hspace{80pt}\Biggr\rangle_{3}\\
&=\sum_{0\le k_{l}\le k_{l-1}\le\cdots\le k_{1}\le n}\chi_{1}(n,k_{1},k_{2},...,k_{l})(-1)^{n-k_{l}}q^{-\frac{1}{6}(n-k_{l})^{2}-\frac{1}{2}(n-k_{l})}q^{-\frac{2}{3}(n-k_{l})k_{l}}\Biggl\langle \hspace{5pt}\begin{minipage}{1\unitlength}\includegraphics[scale=0.1]{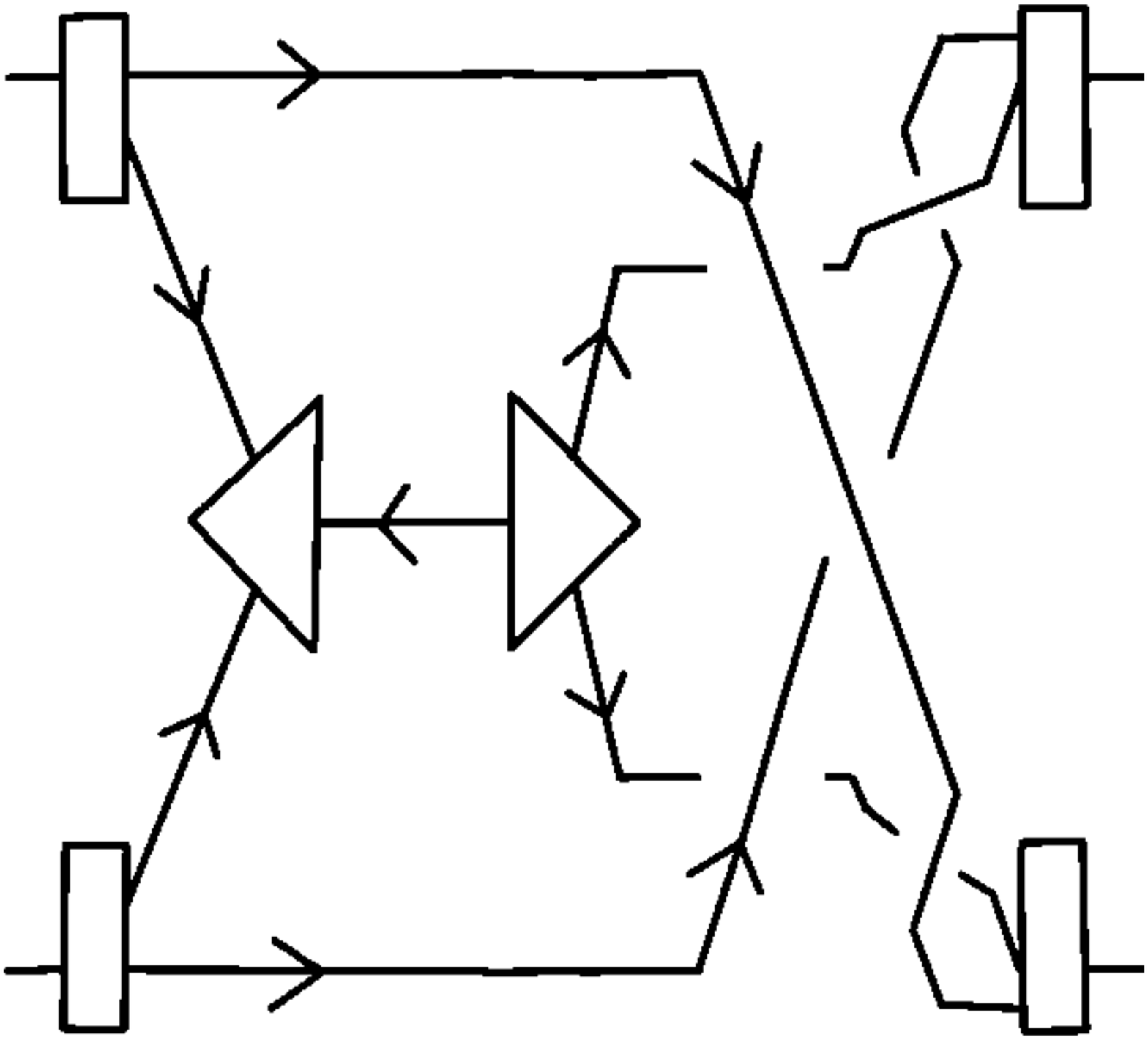}\put(-80,45){\scriptsize$n-k_{l}$\normalsize}\put(-80,30){\scriptsize$n-k_{l}$\normalsize}\put(-60,72){\scriptsize$k_{l}$\normalsize}\put(-60,13){\scriptsize$k_{l}$\normalsize}\put(-77,68){\scriptsize$n$\normalsize}\put(-77,10){\scriptsize$n$\normalsize}\end{minipage}\hspace{80pt}\Biggr\rangle_{3}\\
&\underset{\tiny(\mbox{Lemma $\ref{Lem:DeltaClasp}$} )\normalsize}{=}\sum_{0\le k_{l}\le k_{l-1}\le\cdots\le k_{1}\le n}\chi_{1}(n,k_{1},k_{2},...,k_{l})(-1)^{n-k_{l}}q^{-\frac{1}{6}(n-k_{l})^{2}-\frac{1}{2}(n-k_{l})}q^{-\frac{2}{3}(n-k_{l})k_{l}}\Biggl\langle \hspace{5pt}\begin{minipage}{1\unitlength}\includegraphics[scale=0.1]{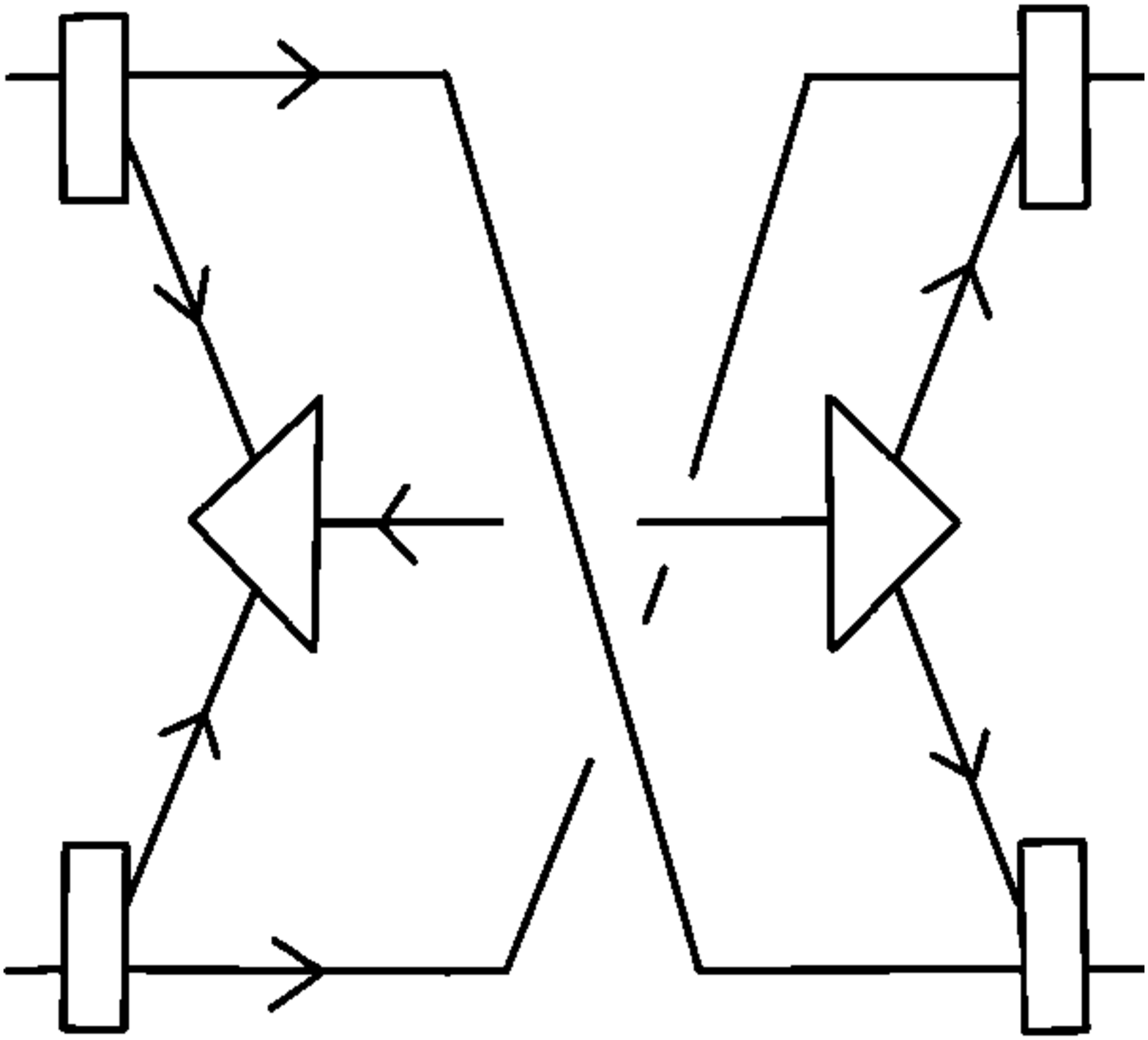}\put(-80,45){\scriptsize$n-k_{l}$\normalsize}\put(-80,30){\scriptsize$n-k_{l}$\normalsize}\put(-60,72){\scriptsize$k_{l}$\normalsize}\put(-60,13){\scriptsize$k_{l}$\normalsize}\put(-77,68){\scriptsize$n$\normalsize}\put(-77,10){\scriptsize$n$\normalsize}\end{minipage}\hspace{80pt}\Biggr\rangle_{3}\\
&\underset{\tiny(\mbox{Lemma $\ref{Lem:DeltaClasp}$, $(\ref{al:delta8})$}  )\normalsize}{=}\sum_{0\le k_{l}\le k_{l-1}\le\cdots\le k_{1}\le n}\chi_{1}(n,k_{1},k_{2},...,k_{l})(-1)^{n-k_{l}}q^{-\frac{1}{6}(n-k_{l})^{2}-\frac{1}{2}(n-k_{l})}q^{-\frac{2}{3}(n-k_{l})k_{l}}q^{\frac{1}{3}(n-k_{l})k_{l}}\Biggl\langle \hspace{5pt}\begin{minipage}{1\unitlength}\includegraphics[scale=0.1]{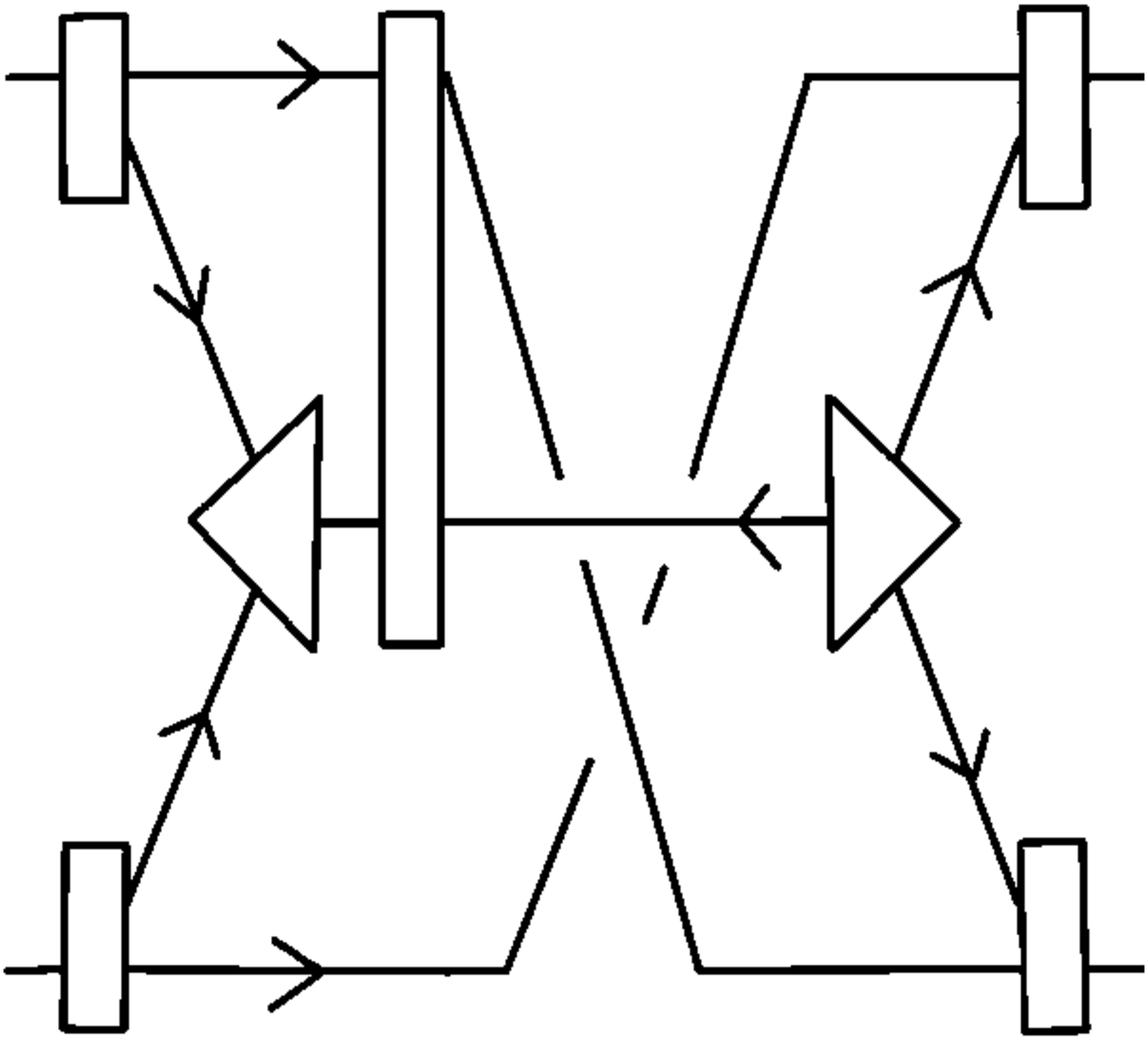}\put(-80,45){\scriptsize$n-k_{l}$\normalsize}\put(-80,30){\scriptsize$n-k_{l}$\normalsize}\put(-60,72){\scriptsize$k_{l}$\normalsize}\put(-60,13){\scriptsize$k_{l}$\normalsize}\put(-77,68){\scriptsize$n$\normalsize}\put(-77,10){\scriptsize$n$\normalsize}\end{minipage}\hspace{80pt}\Biggr\rangle_{3}\\
&\underset{\tiny(\mbox{Lemma $\ref{Lem:DeltaClasp}$}  )\normalsize}{=}\sum_{0\le k_{l}\le k_{l-1}\le\cdots\le k_{1}\le n}\chi_{1}(n,k_{1},k_{2},...,k_{l})(-1)^{n-k_{l}}q^{-\frac{1}{6}(n-k_{l})^{2}-\frac{1}{2}(n-k_{l})}q^{-\frac{1}{3}(n-k_{l})k_{l}}\Biggl\langle \hspace{5pt}\begin{minipage}{1\unitlength}\includegraphics[scale=0.1]{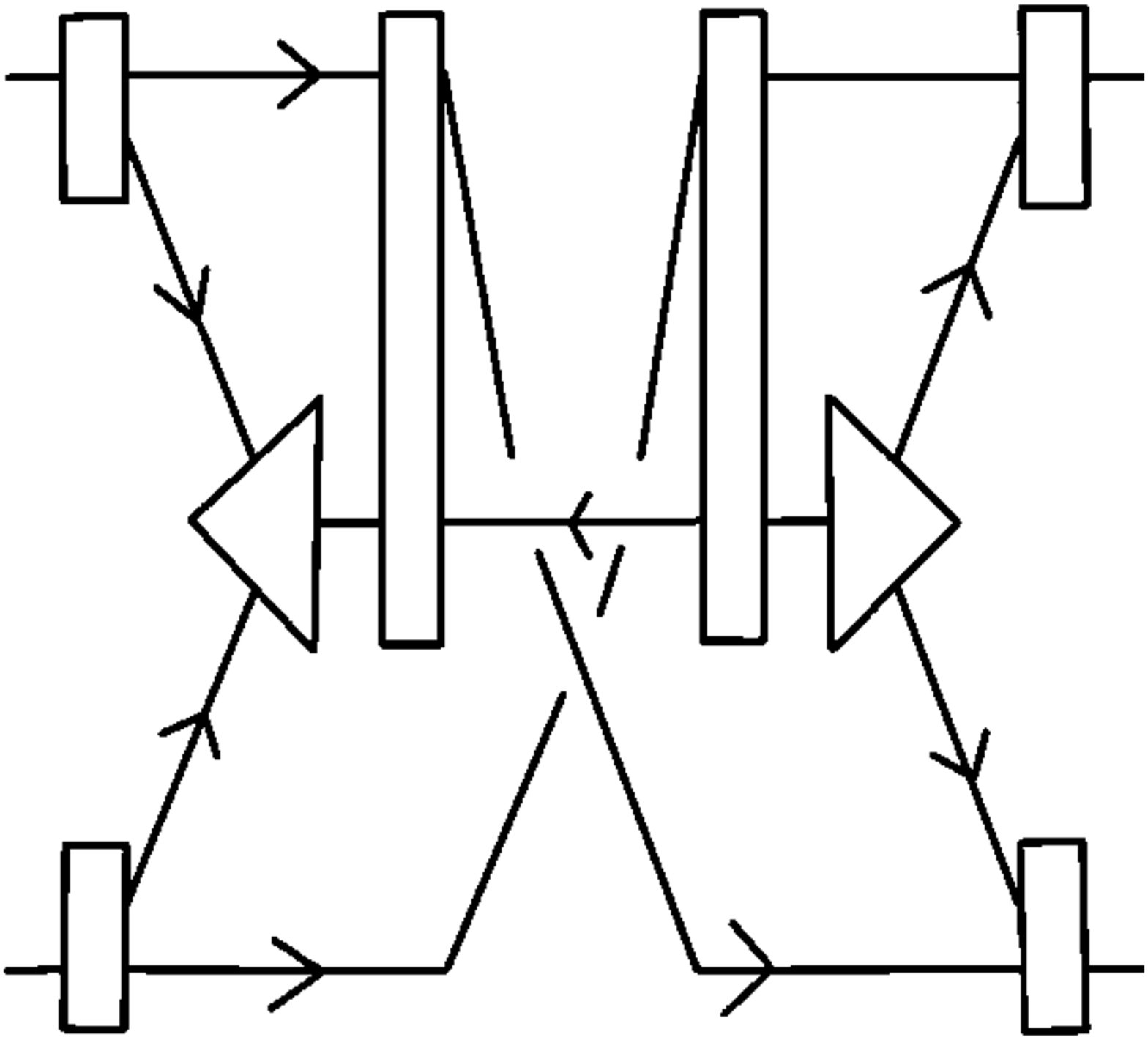}\put(-80,45){\scriptsize$n-k_{l}$\normalsize}\put(-80,30){\scriptsize$n-k_{l}$\normalsize}\put(-60,72){\scriptsize$k_{l}$\normalsize}\put(-60,13){\scriptsize$k_{l}$\normalsize}\put(-77,68){\scriptsize$n$\normalsize}\put(-77,10){\scriptsize$n$\normalsize}\end{minipage}\hspace{80pt}\Biggr\rangle_{3}\\
&\underset{\tiny(\mbox{$(\ref{al:delta8})$})\normalsize}{=}\sum_{0\le k_{l}\le k_{l-1}\le\cdots\le k_{1}\le n}\chi_{1}(n,k_{1},k_{2},...,k_{l})(-1)^{n-k_{l}}q^{-\frac{1}{6}(n-k_{l})^{2}-\frac{1}{2}(n-k_{l})}q^{-\frac{2}{3}(n-k_{l})k_{l}}\Biggl\langle \hspace{5pt}\begin{minipage}{1\unitlength}\includegraphics[scale=0.1]{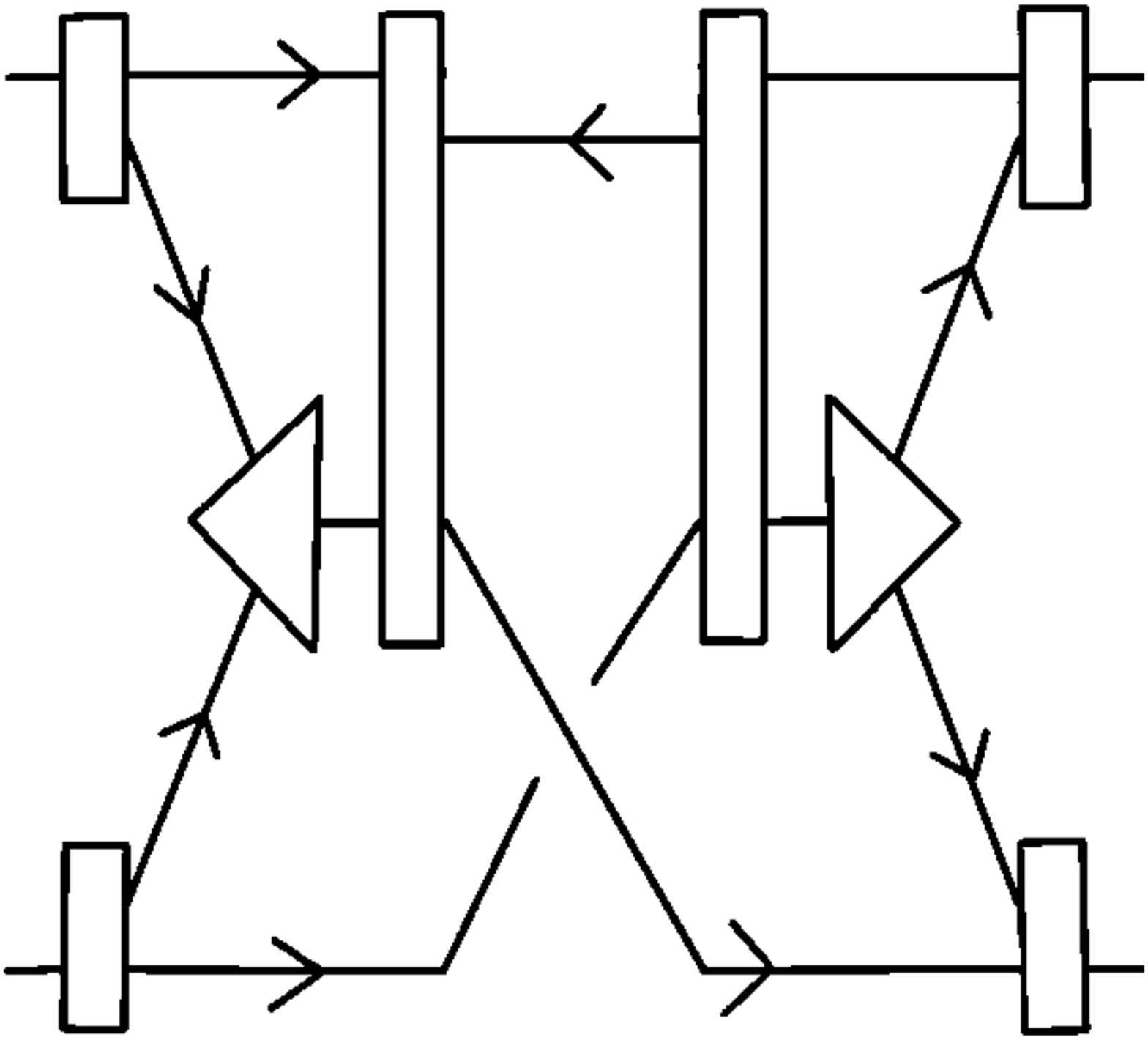}\put(-80,45){\scriptsize$n-k_{l}$\normalsize}\put(-80,30){\scriptsize$n-k_{l}$\normalsize}\put(-60,72){\scriptsize$k_{l}$\normalsize}\put(-60,13){\scriptsize$k_{l}$\normalsize}\put(-77,68){\scriptsize$n$\normalsize}\put(-77,10){\scriptsize$n$\normalsize}\end{minipage}\hspace{80pt}\Biggr\rangle_{3}\\
&\underset{\tiny((\ref{al:delta1}))\normalsize}{=}\sum_{0\le k_{l+1}\le k_{l}\le\cdots\le k_{1}\le n}\chi_{1}(n,k_{1},k_{2},...,k_{l+1})\Biggl\langle \hspace{5pt}\begin{minipage}{1\unitlength}\includegraphics[scale=0.1]{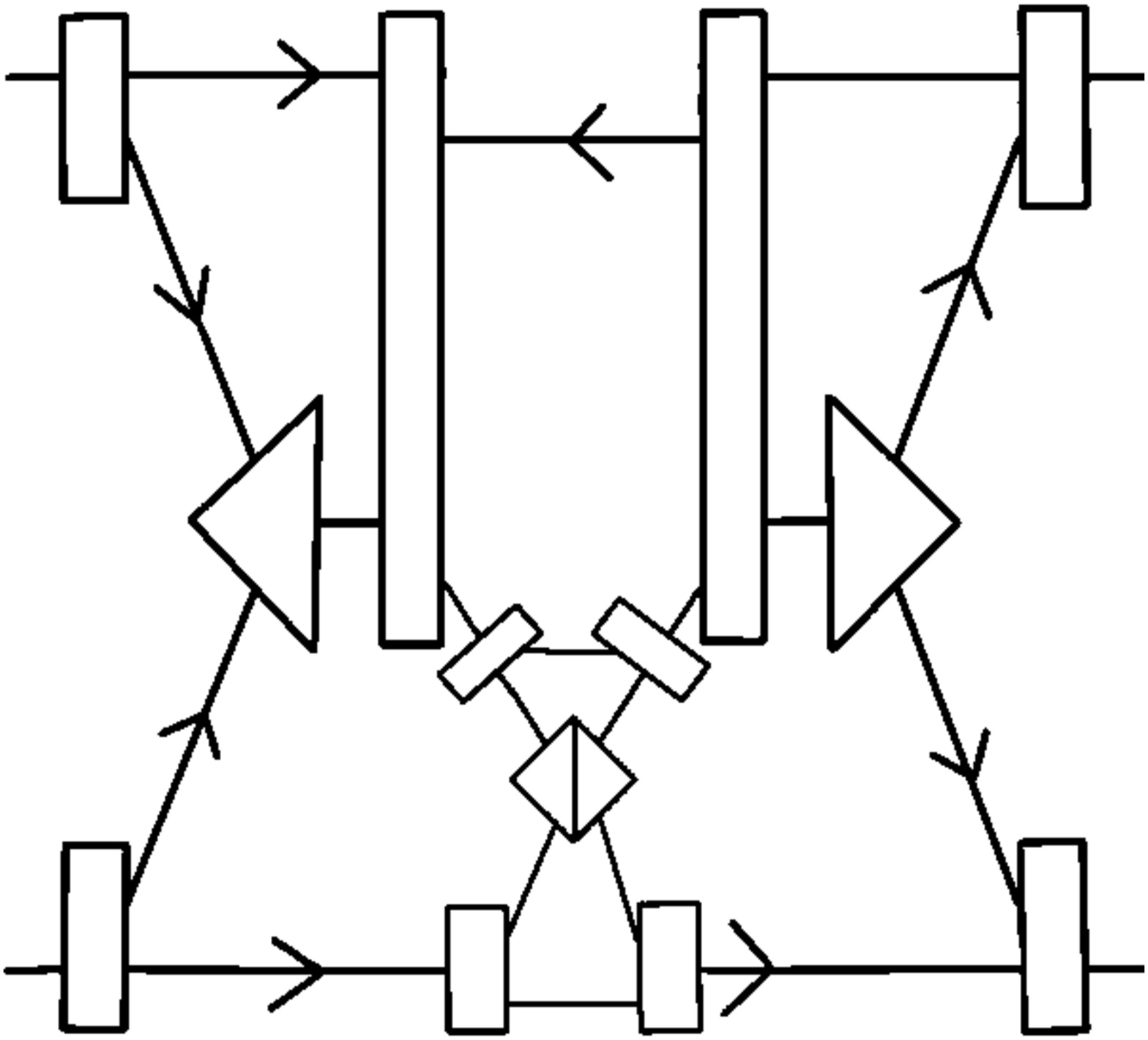}\put(-80,45){\scriptsize$n-k_{l}$\normalsize}\put(-80,30){\scriptsize$n-k_{l}$\normalsize}\put(-60,72){\scriptsize$k_{l}$\normalsize}\put(-60,13){\scriptsize$k_{l}$\normalsize}\put(-77,68){\scriptsize$n$\normalsize}\put(-77,10){\scriptsize$n$\normalsize}\end{minipage}\hspace{80pt}\Biggr\rangle_{3}
\end{align*}
We can see that
\begin{align}
\label{ali:deltaharf1}
\begin{aligned}
\Biggl\langle \hspace{5pt}\begin{minipage}{1\unitlength}\includegraphics[scale=0.1]{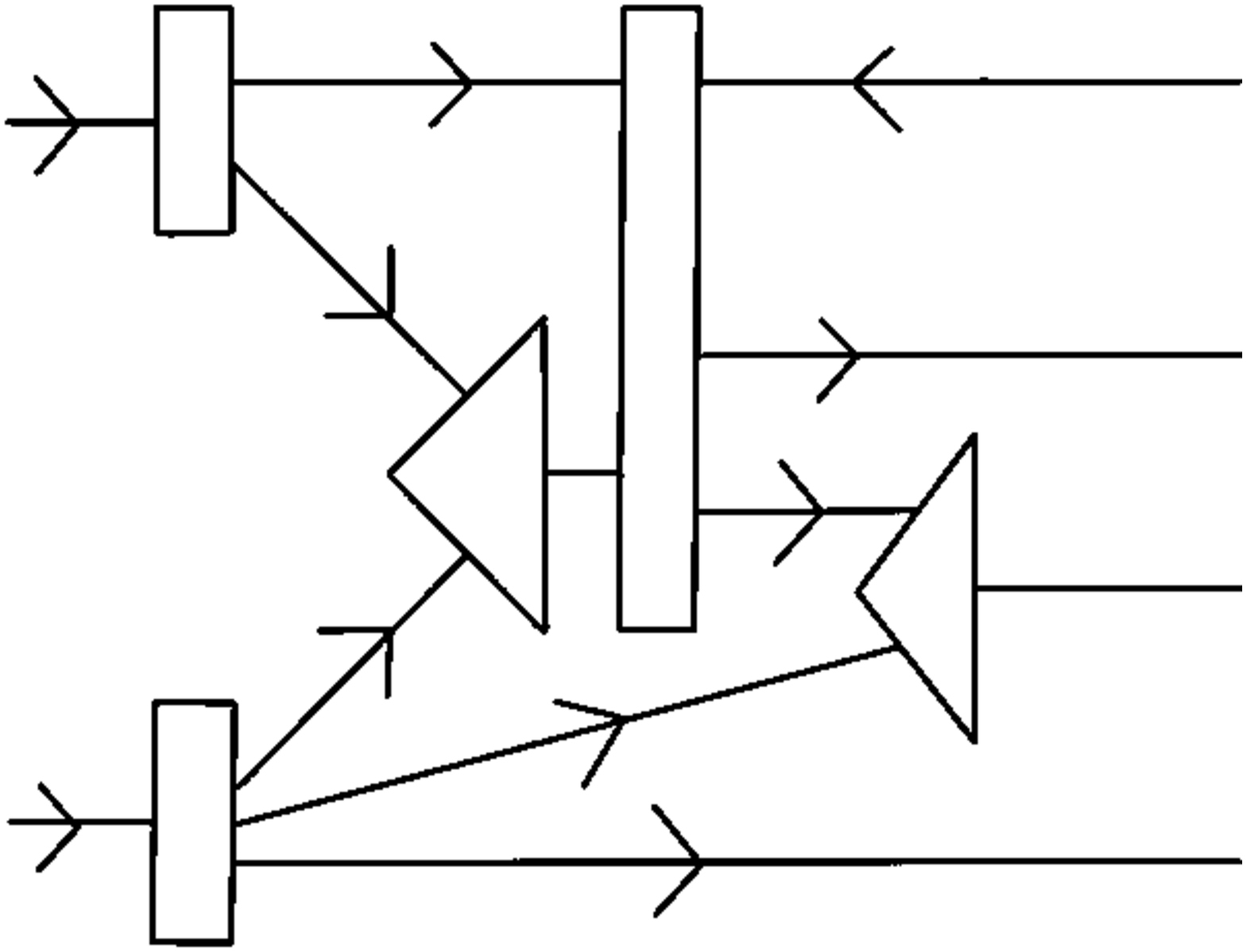}\put(-10,20){\scriptsize$k_{l}-k_{l+1}$\normalsize}\put(-30,48){\scriptsize$k_{l+1}$\normalsize}\put(-71,45){\scriptsize$n-k_{l}$\normalsize}\put(-71,30){\scriptsize$n-k_{l}$\normalsize}\put(-30,65){\scriptsize$n-k_{l}$\normalsize}\put(-60,65){\scriptsize$k_{l}$\normalsize}\put(-60,8){\scriptsize$k_{l}$\normalsize}\put(-77,68){\scriptsize$n$\normalsize}\put(-77,10){\scriptsize$n$\normalsize}\end{minipage}\hspace{80pt}\Biggr\rangle_{3}\underset{\tiny(\mbox{Lemma $\ref{Lem:DeltaClasp}$, (\ref{al:delta9})} )\normalsize}{=}&\Biggl\langle \hspace{5pt}\begin{minipage}{1\unitlength}\includegraphics[scale=0.1]{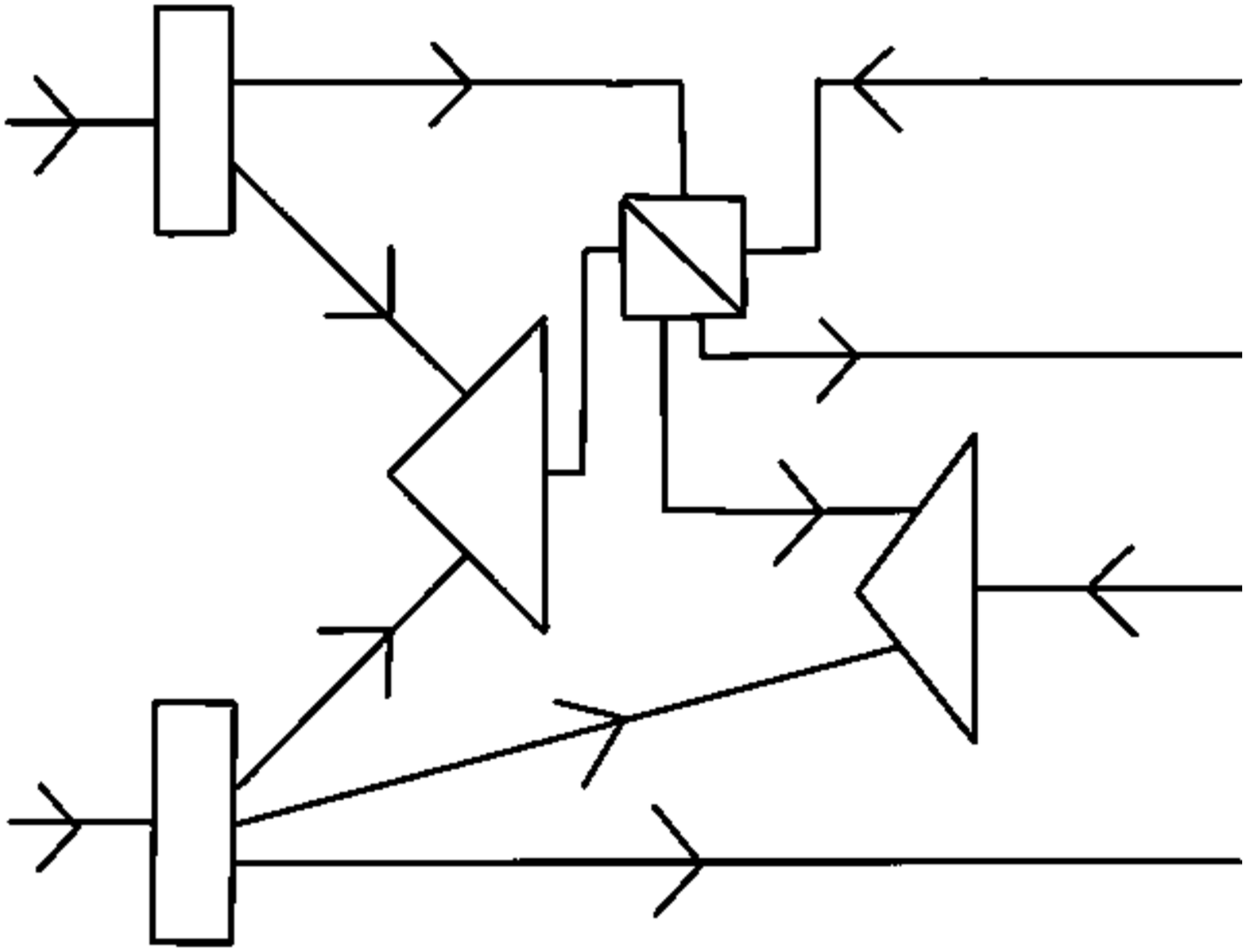}\put(-10,20){\scriptsize$k_{l}-k_{l+1}$\normalsize}\put(-20,48){\scriptsize$k_{l+1}$\normalsize}\put(-71,45){\scriptsize$n-k_{l}$\normalsize}\put(-71,30){\scriptsize$n-k_{l}$\normalsize}\put(-30,65){\scriptsize$n-k_{l}$\normalsize}\put(-60,65){\scriptsize$k_{l}$\normalsize}\put(-60,8){\scriptsize$k_{l}$\normalsize}\put(-77,68){\scriptsize$n$\normalsize}\put(-77,10){\scriptsize$n$\normalsize}\end{minipage}\hspace{80pt}\Biggr\rangle_{3}\underset{\tiny(\mbox{Definition $\ref{stairStep}$} )\normalsize}{=}\Biggl\langle \hspace{5pt}\begin{minipage}{1\unitlength}\includegraphics[scale=0.1]{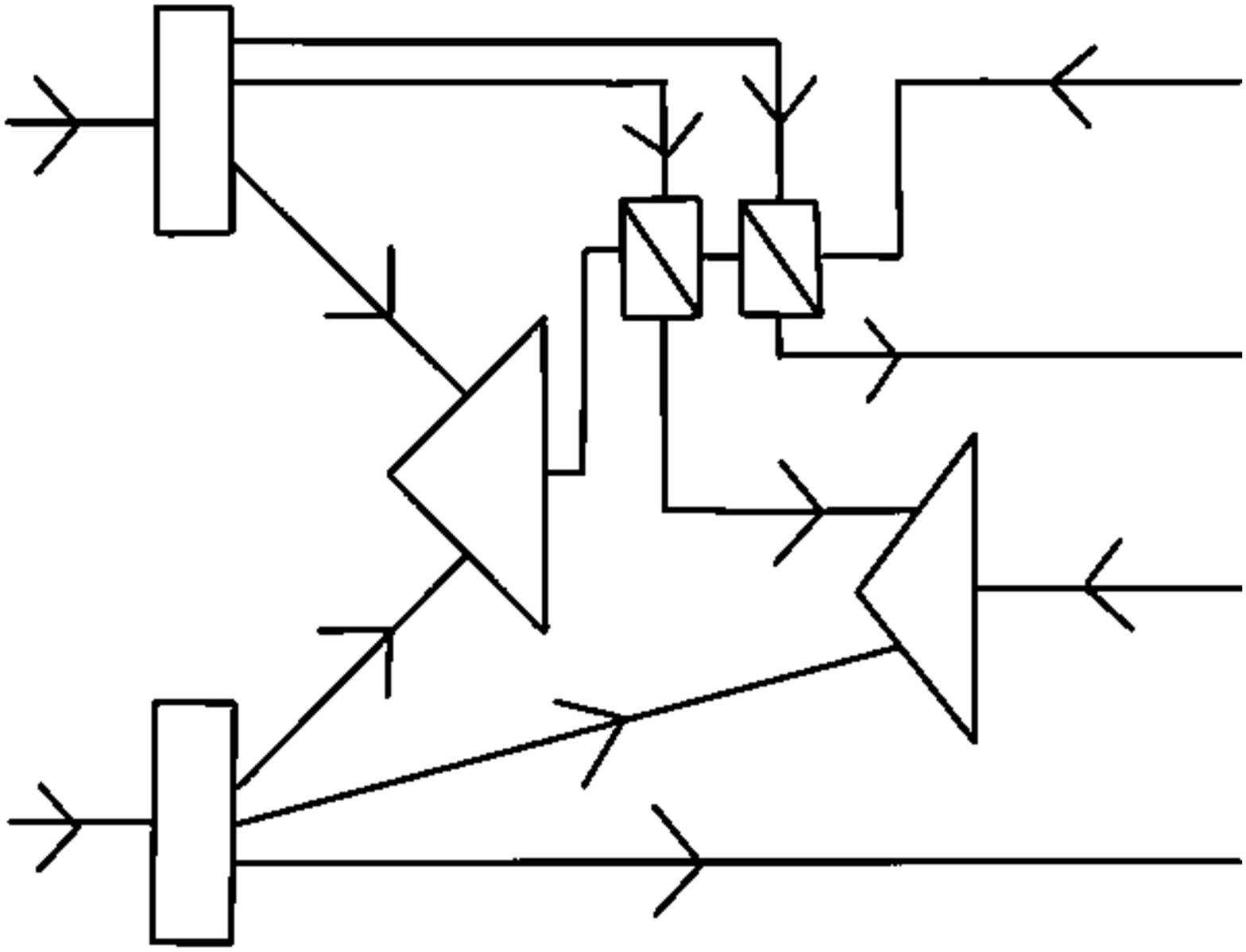}\put(-10,20){\scriptsize$k_{l}-k_{l+1}$\normalsize}\put(-20,48){\scriptsize$k_{l+1}$\normalsize}\put(-71,45){\scriptsize$n-k_{l}$\normalsize}\put(-71,30){\scriptsize$n-k_{l}$\normalsize}\put(-20,65){\scriptsize$n-k_{l}$\normalsize}\put(-60,65){\scriptsize$k_{l+1}$\normalsize}\put(-60,8){\scriptsize$k_{l}$\normalsize}\put(-77,68){\scriptsize$n$\normalsize}\put(-77,10){\scriptsize$n$\normalsize}\end{minipage}\hspace{80pt}\Biggr\rangle_{3}\\
\underset{\tiny(\mbox{$(\ref{al:delta12})$})\normalsize}{=}&\Biggl\langle \hspace{5pt}\begin{minipage}{1\unitlength}\includegraphics[scale=0.1]{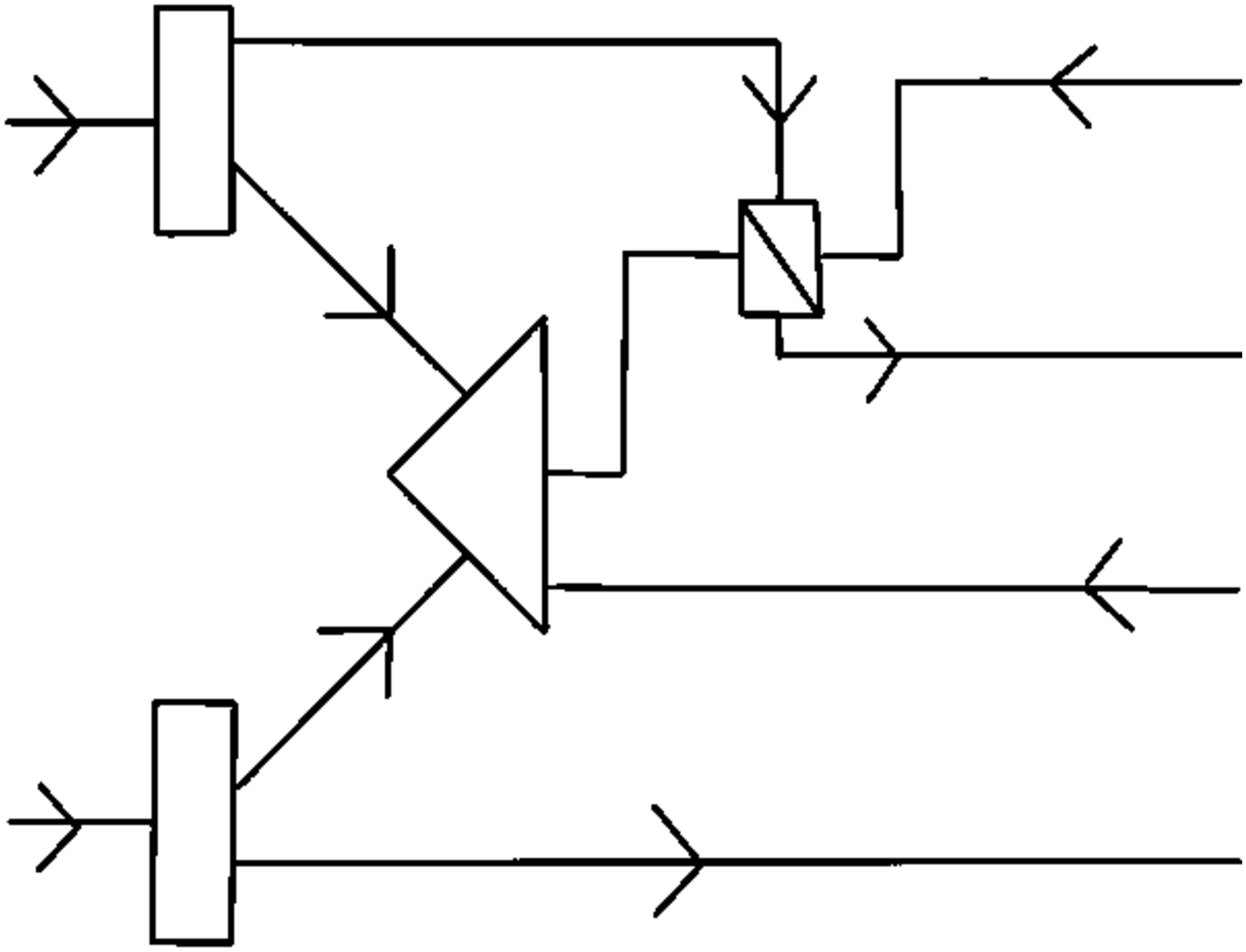}\put(-10,20){\scriptsize$k_{l}-k_{l+1}$\normalsize}\put(-20,48){\scriptsize$k_{l+1}$\normalsize}\put(-77,45){\scriptsize$n-k_{l+1}$\normalsize}\put(-77,30){\scriptsize$n-k_{l+1}$\normalsize}\put(-20,65){\scriptsize$n-k_{l}$\normalsize}\put(-60,65){\scriptsize$k_{l+1}$\normalsize}\put(-60,8){\scriptsize$k_{l}$\normalsize}\put(-77,68){\scriptsize$n$\normalsize}\put(-77,10){\scriptsize$n$\normalsize}\end{minipage}\hspace{80pt}\Biggr\rangle_{3}\underset{\tiny(\mbox{Lemma $\ref{Lem:DeltaClasp}$, $(\ref{al:delta9})$})\normalsize}{=}\Biggl\langle \hspace{5pt}\begin{minipage}{1\unitlength}\includegraphics[scale=0.1]{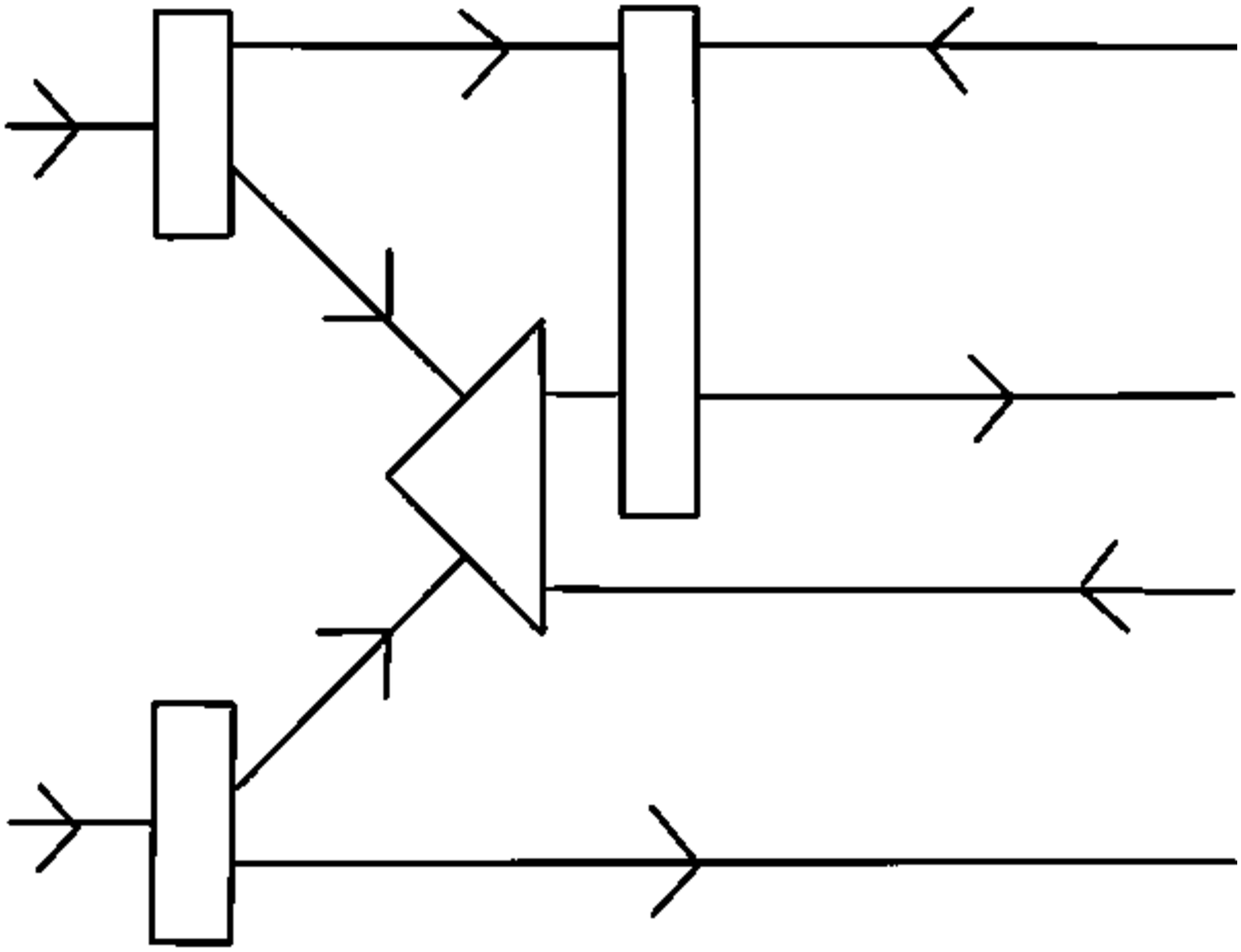}\put(-10,20){\scriptsize$k_{l}-k_{l+1}$\normalsize}\put(-20,48){\scriptsize$k_{l+1}$\normalsize}\put(-77,45){\scriptsize$n-k_{l+1}$\normalsize}\put(-77,30){\scriptsize$n-k_{l+1}$\normalsize}\put(-20,65){\scriptsize$n-k_{l}$\normalsize}\put(-60,65){\scriptsize$k_{l+1}$\normalsize}\put(-60,8){\scriptsize$k_{l}$\normalsize}\put(-77,68){\scriptsize$n$\normalsize}\put(-77,10){\scriptsize$n$\normalsize}\end{minipage}\hspace{80pt}\Biggr\rangle_{3}
\end{aligned}
\end{align}
By using $(\ref{ali:deltaharf1})$, we have
\begin{align}
\label{al:remove}
\Biggl\langle \hspace{15pt}\begin{minipage}{1\unitlength}\includegraphics[scale=0.1]{pic/halfTwist8.eps}\put(-40,-2){\scriptsize$k_{l+1}$\normalsize}\put(-61,18){\tiny$k_{l}-k_{l+1}$\normalsize}\put(-82,45){\scriptsize$n-k_{l}$\normalsize}\put(-82,30){\scriptsize$n-k_{l}$\normalsize}\put(-45,67){\scriptsize$n-k_{l}$\normalsize}\put(-60,72){\scriptsize$k_{l}$\normalsize}\put(-60,0){\scriptsize$k_{l}$\normalsize}\put(-77,68){\scriptsize$n$\normalsize}\put(-77,10){\scriptsize$n$\normalsize}\end{minipage}\hspace{80pt}\Biggr\rangle_{3}=\Biggl\langle \hspace{15pt}\begin{minipage}{1\unitlength}\includegraphics[scale=0.1]{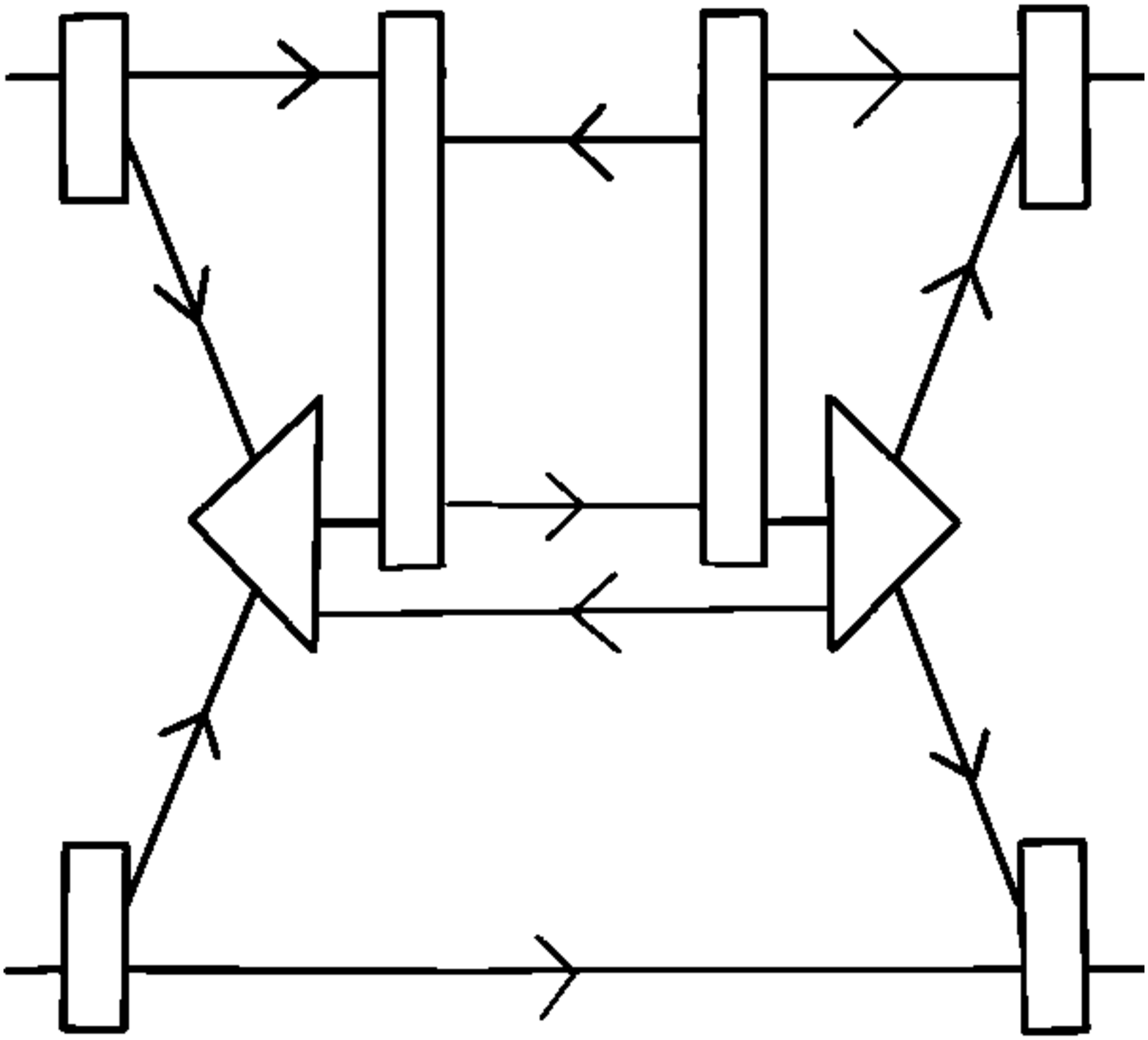}\put(-40,-2){\scriptsize$k_{l+1}$\normalsize}\put(-45,20){\scriptsize$k_{l}-k_{l+1}$\normalsize}\put(-42,42){\scriptsize$k_{l+1}$\normalsize}\put(-82,45){\scriptsize$n-k_{l+1}$\normalsize}\put(-82,30){\scriptsize$n-k_{l+1}$\normalsize}\put(-45,67){\scriptsize$n-k_{l}$\normalsize}\put(-62,72){\scriptsize$k_{l+1}$\normalsize}\put(-77,68){\scriptsize$n$\normalsize}\put(-77,10){\scriptsize$n$\normalsize}\end{minipage}\hspace{80pt}\Biggr\rangle_{3}
\end{align}
We remove the $A_{2}$ clasps by using Lemma $\ref{Lem:DeltaClasp}$ and $(\ref{ali:deltaharf1})$ in $(RHS$ of $(\ref{al:remove}))$. Therefore, $(\ref{Pro:halfTwist1})$ holds for $n=l+1$. We can see $(\ref{Pro:halfTwist2})$ in a similar way.
\end{proof}

 We use the same definition for the one-row $\mathfrak{sl}_{3}$ colored Jones polynomial as we do in \cite{Yua18}.
 \begin{Definition}
 \label{Def:coloredJones}
 Let $L$ be an oriented link, and $\bar{L}$ a diagram of $L$. The one-row $\mathfrak{sl}_{3}$ colored Jones polynomial for $L$ is defined by
\begin{align*}
J_{(n,0)}^{\mathfrak{sl}_{3}}(L;q)=(q^{\frac{n^{2}+3n}{3}})^{-w(\bar{L})}&\langle\, L(n,0) \,\rangle_{3}/\Delta(n,0)
\end{align*}
where $w(\bar{L})$ is the writhing number of $\bar{L}$ and $L(n,0)$ replaces a part of $\bar{L}$ with $A_{2}$ clasp of type $(n,0)$. For example, 
\begin{align*}
T(2,3)(n,0)=\begin{minipage}{1\unitlength}\includegraphics[scale=0.09]{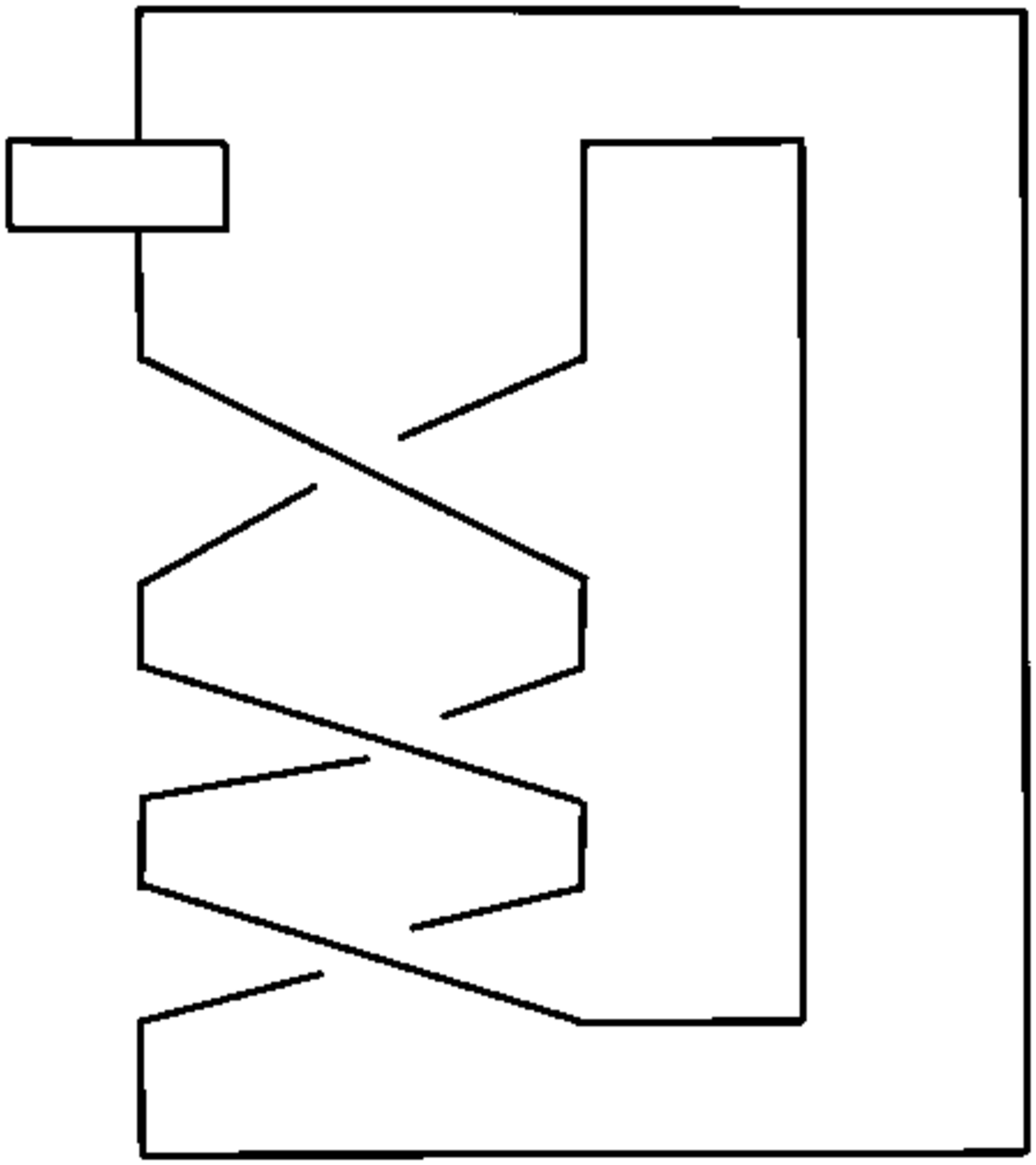}\put(-40,63){\scriptsize$n$\normalsize}\end{minipage}
\end{align*} 
\end{Definition}

Let a map $sign:\mathbb{Z}\rightarrow \{-1,0,1\}$ be defined by 
\begin{align*}
sign(x) = 
\begin{cases}
1 & (x > 0)\\
0 & (x= 0)\\
-1 & (x < 0)
\end{cases}.
\end{align*}
\begin{Corollary}
\label{Cor:halfTwist1}
The one-row colored $\mathfrak{sl}_{3}$ Jones polynomials for $(2,m)$-torus knots $T(2,m)$ are the following:
\begin{align}
\label{ali:torusknot}
J_{(n,0)}^{\mathfrak{sl}_{3}}(T(2,m);q)=(q^{\frac{n^{2}+3n}{3}})^{-(m)}\sum_{0\le k_{|m|}\le k_{|m|-1}\le\cdots\le k_{1}\le n}\chi_{sign(m)}q^{\frac{n+k_{|m|}}{2}}\frac{(1-q^{n+1})(1-q^{n+2})}{(1-q)(1-q^{n-k_{|m|}+2})}.
\end{align}
\end{Corollary}
\begin{proof}
we can obtain $(\ref{ali:torusknot})$ from Definiton $\ref{Def:coloredJones}$, Proposition $\ref{Pro:halfTwist1}$ and $(\ref{al:delta10})$  easily. 
\end{proof}

\begin{Remark}
For $m=3,5,7$, n=$1,2,...,10$, Corollary $\ref{Cor:halfTwist1}$ are equal to Theorem $2.1$ in \cite{Gar13} and Theorem $5.7$  in \cite{Yua17} for $T(2,m)$ by using Mathematica. For example,
\begin{align*}
J_{(4,0)}^{\mathfrak{sl}_{3}}(T(2,3);q)=&q^{-42}-q^{-40}-q^{-39}-q^{38}+q^{-36}+q^{-35}+2q^{-31}+q^{-30}-q^{-28}-q^{-27}+q^{-25}-q^{-23}-2q^{-22}-q^{-21}\\
&+q^{-20}+q^{-19}-q^{-17}+q^{-14}+q^{-13}+q^{-8}
\end{align*}
\end{Remark}

\begin{Proposition}
\label{pro:Omega}
Let $n$ be a positive integer and $k$, $l$ non-negative integers. For $n\ge k,l$ , we have 
\begin{align*}
\Biggl\langle \hspace{5pt}\begin{minipage}{1\unitlength}\includegraphics[scale=0.1]{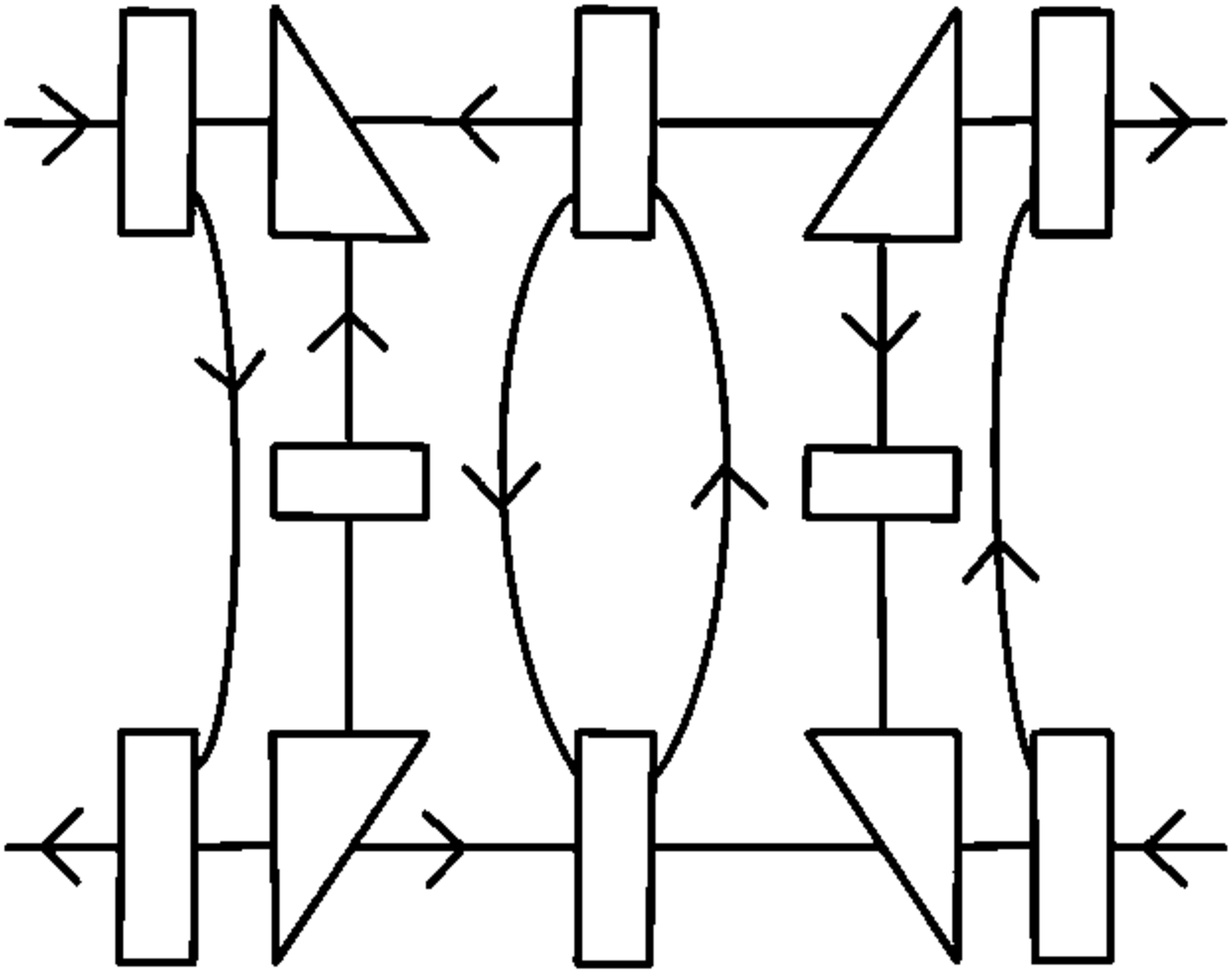}\put(-35,66){\scriptsize$n\hspace{-2pt}-\hspace{-2pt}l$\normalsize}\put(-35,4){\scriptsize$n\hspace{-2pt}-\hspace{-2pt}l$\normalsize}\put(-60,66){\scriptsize$n\hspace{-2pt}-\hspace{-2pt}k$\normalsize}\put(-60,4){\scriptsize$n\hspace{-2pt}-\hspace{-2pt}k$\normalsize}\put(-45,45){\scriptsize$k$\normalsize}\put(-38,25){\scriptsize$l$\normalsize}\put(-75,36){\scriptsize$k$\normalsize}\put(-8,36){\scriptsize$l$\normalsize}\put(-80,62){\scriptsize$n$\normalsize}\put(-80,8){\scriptsize$n$\normalsize}\end{minipage}\hspace{80pt}\Biggr\rangle_{3}=\sum_{t=\max\{k,l\}}^{\min\{k+l,n\}}\sum_{a=t}^{n}\Omega(n,t,k,l)\Biggl\langle\hspace{10pt} \begin{minipage}{1\unitlength}\includegraphics[scale=0.1]{pic/Omega_after.eps}\put(-74,60){\scriptsize$n$\normalsize}\put(-74,8){\scriptsize$n$\normalsize}\put(-45,6){\scriptsize$n\hspace{-2pt}-\hspace{-2pt}t$\normalsize}\put(-45,62){\scriptsize$n\hspace{-2pt}-\hspace{-2pt}t$\normalsize}\put(-65,35){\scriptsize$t$\normalsize}\put(-13,35){\scriptsize$t$\normalsize}\end{minipage}\hspace{82pt}\Biggr\rangle_{3}
\end{align*}
where
\begin{align*}
\Omega(n,k,l,t)=\frac{q^{-\frac{k+l}{2}+t}q^{(t+1)(t-k-l)+kl}(1-q^{n+1-k})(1-q^{n+1-l})(q)_{k}(q)_{l}(q)^{2}_{n-k}(q)^{2}_{n-l}(q)_{2n-t+2}}{(1-q^{n+1-t})^{2}(q)^{2}_{n}(q)^{2}_{n-t}(q)_{t-k}(q)_{t-l}(q)_{2n-k-l+2}(q)_{-t+k+l}}.
\end{align*}
\end{Proposition}

We first prove the following lemma.
\begin{Lemma}
\label{Lem:Omega2}
Let $n$ and $a$ be positeve integers. We have
\begin{align*}
\Biggl\langle \hspace{5pt}\begin{minipage}{1\unitlength}\includegraphics[scale=0.1]{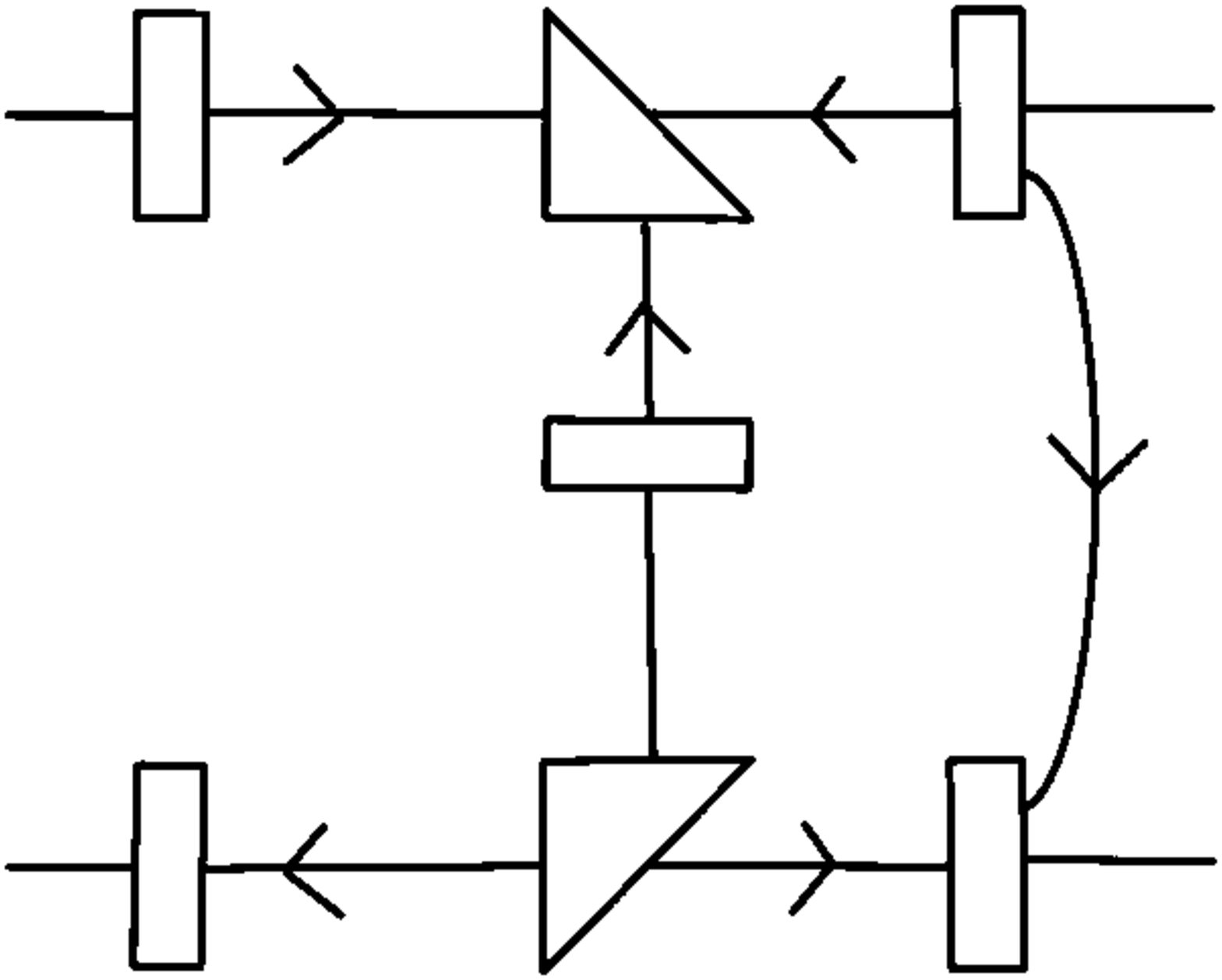}\put(-74,65){\scriptsize$n$\normalsize}\put(-74,17){\scriptsize$n$\normalsize}\put(-5,45){\scriptsize$a$\normalsize}\end{minipage}\hspace{80pt}\Biggr\rangle_{3}=\frac{[n+1]_{q}}{[n-a+1]_{q}}\Biggl\langle \hspace{5pt}\begin{minipage}{1\unitlength}\includegraphics[scale=0.1]{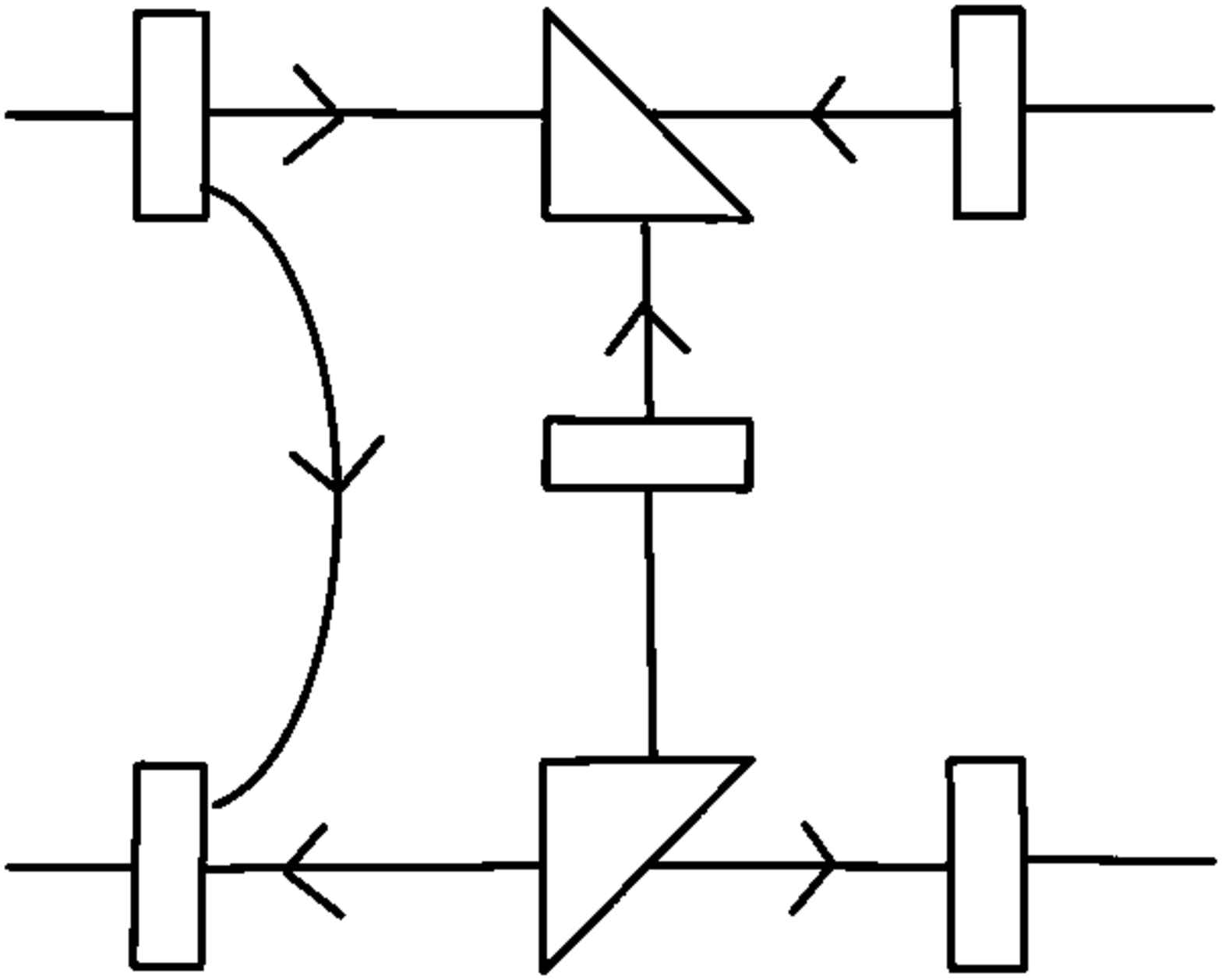}\put(-74,65){\scriptsize$n$\normalsize}\put(-74,17){\scriptsize$n$\normalsize}\put(-65,45){\scriptsize$a$\normalsize}\end{minipage}\hspace{80pt}\Biggr\rangle_{3}
\end{align*}
\end{Lemma}
\begin{proof}
First, we prove that it holds for $a=1$. 
\begin{align}
\label{eq:Omega1}
\Biggl\langle \hspace{5pt}\begin{minipage}{1\unitlength}\includegraphics[scale=0.1]{pic/OmegaBefore.eps}\put(-74,65){\scriptsize$n$\normalsize}\put(-74,17){\scriptsize$n$\normalsize}\put(-5,45){\scriptsize$1$\normalsize}\end{minipage}\hspace{80pt}\Biggr\rangle_{3}=\frac{[n+1]_{q}}{[n]_{q}}\Biggl\langle \hspace{5pt}\begin{minipage}{1\unitlength}\includegraphics[scale=0.1]{pic/OmegaAfter.eps}\put(-74,65){\scriptsize$n$\normalsize}\put(-74,17){\scriptsize$n$\normalsize}\put(-65,45){\scriptsize$1$\normalsize}\end{minipage}\hspace{80pt}\Biggr\rangle_{3}
\end{align}
The equation $(\ref{eq:Omega1})$ is the same lemma in \cite{Yua212}, but we prove it in a different way. We proceed by induction on $n$. We can see that $(\ref{eq:Omega1})$ holds by an easy calculation for $n=2$. Assume $(\ref{eq:Omega1})$ holds when $n\le k-1$ for some integer $k\ge3$. We obtain 
\begin{align*}
&\Biggl\langle \hspace{5pt}\begin{minipage}{1\unitlength}\includegraphics[scale=0.1]{pic/OmegaBefore.eps}\put(-74,63){\scriptsize$k$\normalsize}\put(-74,15){\scriptsize$k$\normalsize}\put(-3,35){\scriptsize$1$\normalsize}\end{minipage}\hspace{80pt}\Biggr\rangle_{3}\\
\underset{\tiny(\mbox{(\ref{al:delta4})}\normalsize}{=}&\Biggl\langle \hspace{5pt}\begin{minipage}{1\unitlength}\includegraphics[scale=0.1]{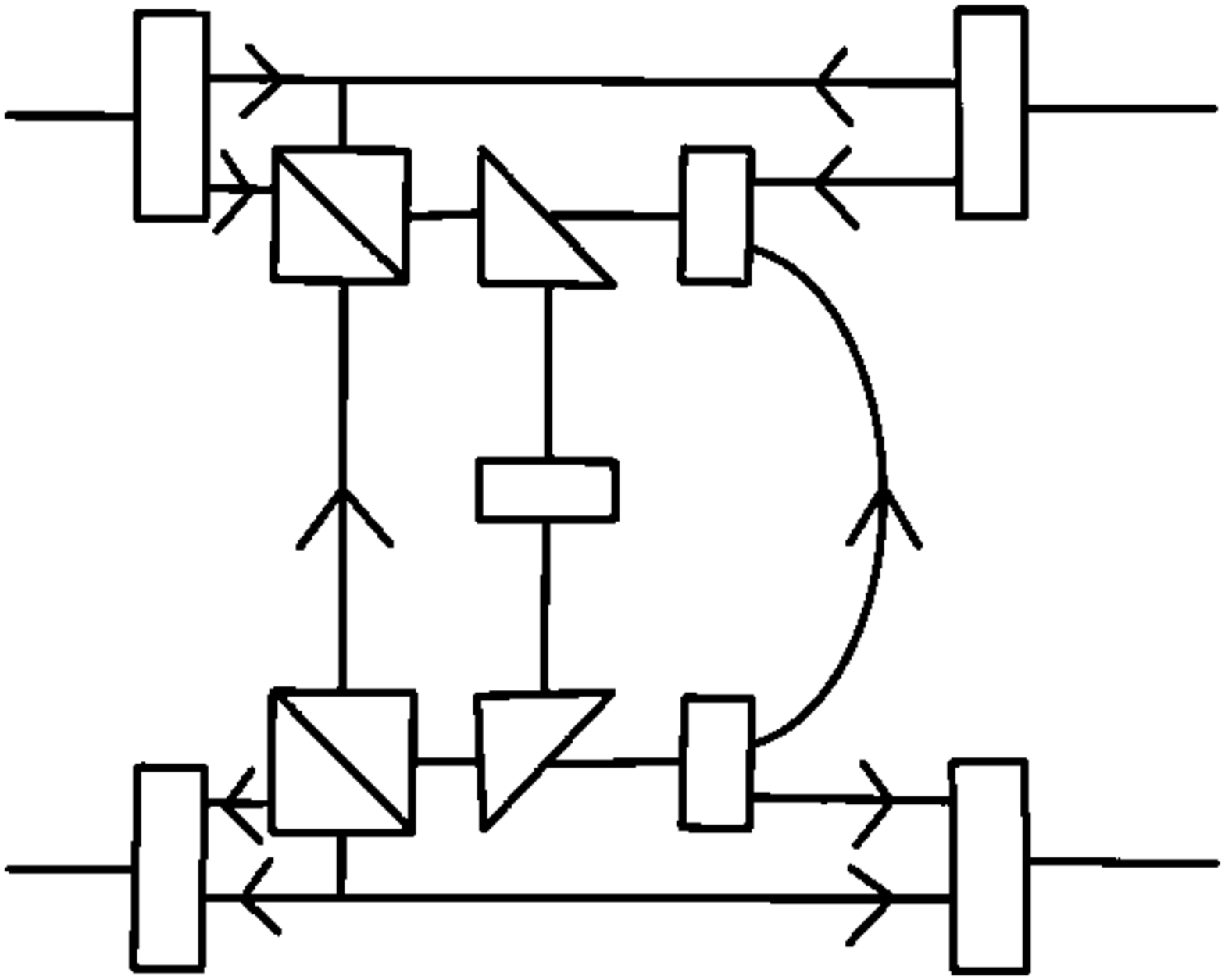}\put(-74,65){\scriptsize$k$\normalsize}\put(-74,15){\scriptsize$k$\normalsize}\put(-18,35){\scriptsize$1$\normalsize}\put(-72,49){\scriptsize$k-1$\normalsize}\put(-72,22){\scriptsize$k-1$\normalsize}\put(-60,65){\scriptsize$1$\normalsize}\put(-23,49){\scriptsize$k-2$\normalsize}\put(-23,22){\scriptsize$k-2$\normalsize}\end{minipage}\hspace{80pt}\Biggr\rangle_{3}-\frac{[k-1]_{q}}{[k]_{q}}\Biggl\langle \hspace{5pt}\begin{minipage}{1\unitlength}\includegraphics[scale=0.1]{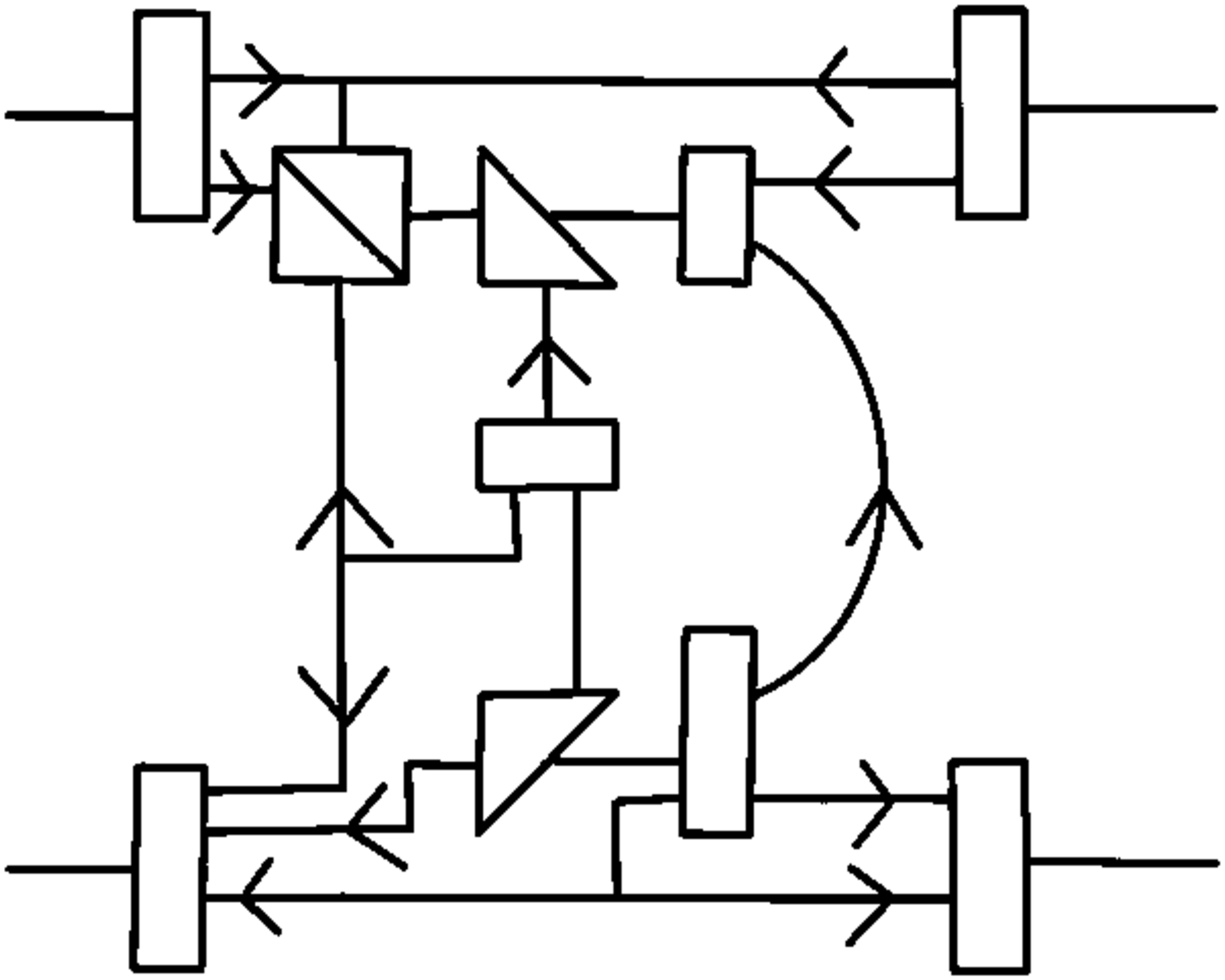}\put(-74,65){\scriptsize$k$\normalsize}\put(-74,15){\scriptsize$k$\normalsize}\put(-18,35){\scriptsize$1$\normalsize}\put(-72,49){\scriptsize$k-1$\normalsize}\put(-72,22){\scriptsize$k-1$\normalsize}\put(-60,65){\scriptsize$1$\normalsize}\put(-23,49){\scriptsize$k-2$\normalsize}\put(-23,22){\scriptsize$k-2$\normalsize}\end{minipage}\hspace{80pt}\Biggr\rangle_{3}\\
\underset{\tiny(\mbox{Definition $\ref{stairStep}$, $(\ref{al:delta20})$})\normalsize}{=}&\frac{[k]_{q}}{[k-1]_{q}}\Biggl\langle \hspace{5pt}\begin{minipage}{1\unitlength}\includegraphics[scale=0.1]{pic/Omega1.eps}\put(-74,65){\scriptsize$k$\normalsize}\put(-74,15){\scriptsize$k$\normalsize}\put(-18,35){\scriptsize$1$\normalsize}\put(-72,49){\scriptsize$k-1$\normalsize}\put(-72,22){\scriptsize$k-1$\normalsize}\put(-60,65){\scriptsize$1$\normalsize}\put(-23,49){\scriptsize$k-2$\normalsize}\put(-23,22){\scriptsize$k-2$\normalsize}\end{minipage}\hspace{80pt}\Biggr\rangle_{3}-(-1)^{k-1}\frac{1}{[k]_{q}}\Biggl\langle \hspace{5pt}\begin{minipage}{1\unitlength}\includegraphics[scale=0.1]{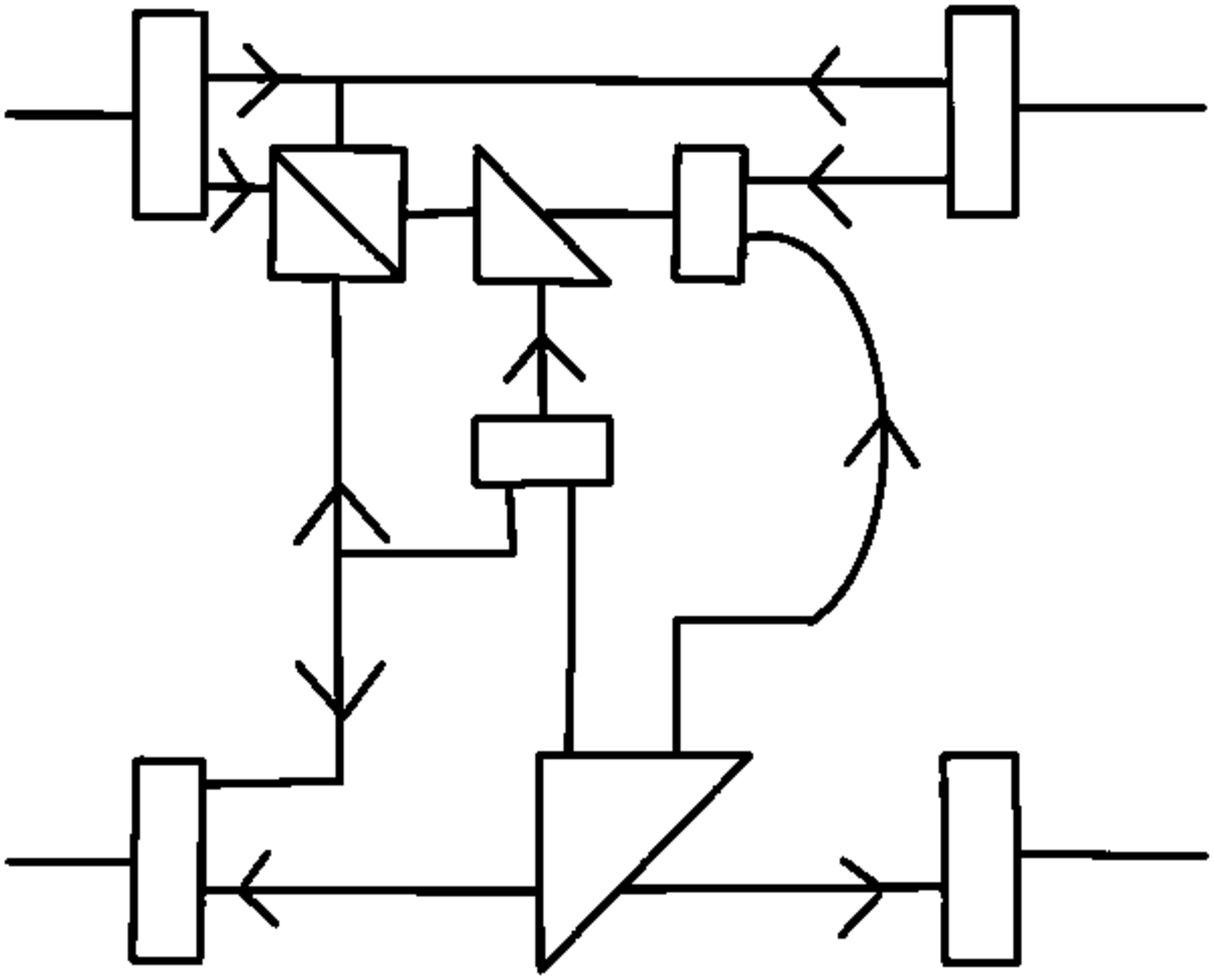}\put(-72,65){\scriptsize$k$\normalsize}\put(-72,15){\scriptsize$k$\normalsize}\put(-18,35){\scriptsize$1$\normalsize}\put(-72,49){\scriptsize$k-1$\normalsize}\put(-60,5){\scriptsize$k-1$\normalsize}\put(-60,65){\scriptsize$1$\normalsize}\put(-22,49){\scriptsize$k-2$\normalsize}\put(-30,5){\scriptsize$k-1$\normalsize}\end{minipage}\hspace{80pt}\Biggr\rangle_{3}\\
\underset{\tiny(\mbox{induction hypothesis})\normalsize}{=}&\frac{[k]_{q}}{[k-1]_{q}}\Biggl\langle \hspace{5pt}\begin{minipage}{1\unitlength}\includegraphics[scale=0.1]{pic/OmegaAfter.eps}\put(-72,63){\scriptsize$k$\normalsize}\put(-72,15){\scriptsize$k$\normalsize}\put(-62,35){\scriptsize$1$\normalsize}\end{minipage}\hspace{80pt}\Biggr\rangle_{3}-(-1)^{k-1}\frac{1}{[k]_{q}}\Biggl\langle \hspace{5pt}\begin{minipage}{1\unitlength}\includegraphics[scale=0.1]{pic/Omega3.eps}\put(-72,65){\scriptsize$k$\normalsize}\put(-72,15){\scriptsize$k$\normalsize}\put(-18,35){\scriptsize$1$\normalsize}\put(-72,49){\scriptsize$k-1$\normalsize}\put(-60,5){\scriptsize$k-1$\normalsize}\put(-60,65){\scriptsize$1$\normalsize}\put(-22,49){\scriptsize$k-2$\normalsize}\put(-30,5){\scriptsize$k-1$\normalsize}\end{minipage}\hspace{80pt}\Biggr\rangle_{3}
\end{align*}
For the second term, 
\begin{align*}
&\Biggl\langle \hspace{5pt}\begin{minipage}{1\unitlength}\includegraphics[scale=0.1]{pic/Omega3.eps}\put(-72,65){\scriptsize$k$\normalsize}\put(-72,15){\scriptsize$k$\normalsize}\put(-18,35){\scriptsize$1$\normalsize}\put(-72,49){\scriptsize$k-1$\normalsize}\put(-60,5){\scriptsize$k-1$\normalsize}\put(-60,65){\scriptsize$1$\normalsize}\put(-22,49){\scriptsize$k-2$\normalsize}\put(-30,5){\scriptsize$k-1$\normalsize}\end{minipage}\hspace{80pt}\Biggr\rangle_{3}\\
\underset{(\tiny\mbox{$(\ref{al:delta4})$}\normalsize)}{=}&\Biggl\langle \hspace{5pt}\begin{minipage}{1\unitlength}\includegraphics[scale=0.1]{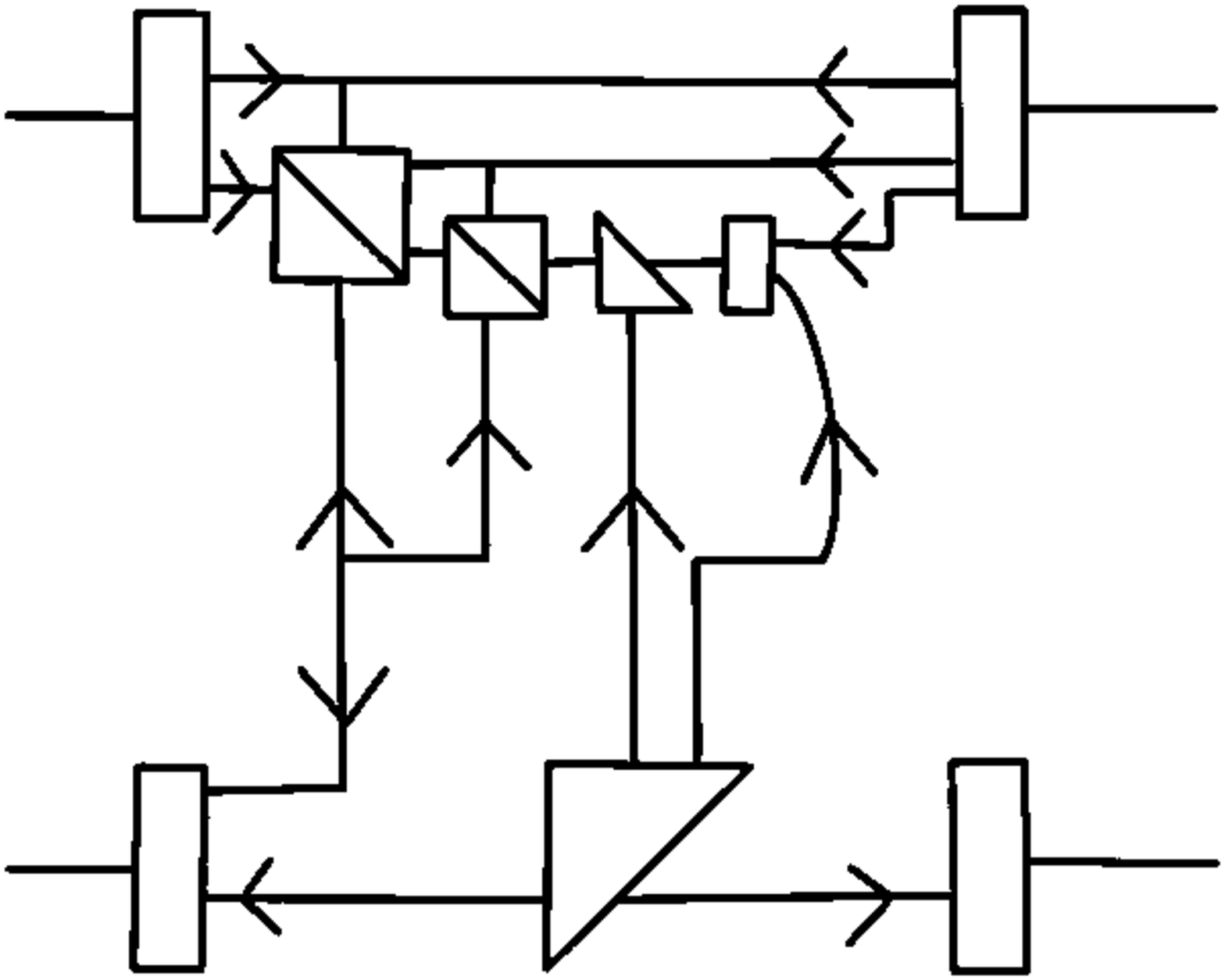}\put(-72,65){\scriptsize$k$\normalsize}\put(-72,15){\scriptsize$k$\normalsize}\put(-18,35){\scriptsize$1$\normalsize}\put(-72,49){\scriptsize$k-1$\normalsize}\put(-60,5){\scriptsize$k-1$\normalsize}\put(-60,65){\scriptsize$1$\normalsize}\put(-20,49){\scriptsize$k-3$\normalsize}\put(-30,5){\scriptsize$k-1$\normalsize}\end{minipage}\hspace{80pt}\Biggr\rangle_{3}-\frac{[k-2]_{q}}{[k-1]_{q}}\Biggl\langle \hspace{5pt}\begin{minipage}{1\unitlength}\includegraphics[scale=0.1]{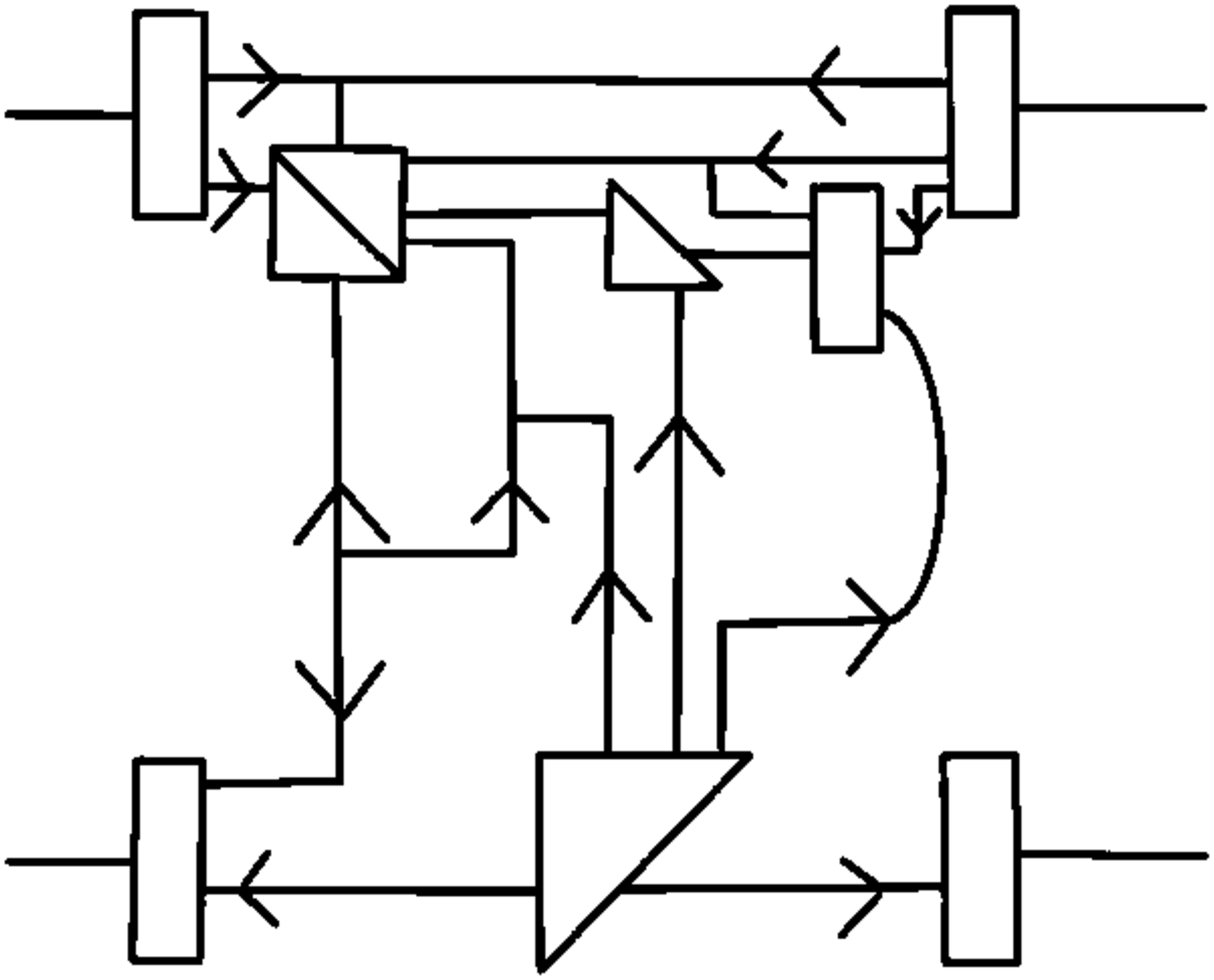}\put(-72,65){\scriptsize$k$\normalsize}\put(-72,15){\scriptsize$k$\normalsize}\put(-15,35){\scriptsize$1$\normalsize}\put(-72,49){\scriptsize$k-1$\normalsize}\put(-60,5){\scriptsize$k-1$\normalsize}\put(-60,65){\scriptsize$1$\normalsize}\put(-18,49){\scriptsize$k-3$\normalsize}\put(-30,5){\scriptsize$k-1$\normalsize}\end{minipage}\hspace{80pt}\Biggr\rangle_{3}\\
=&-\frac{[k-2]_{q}}{[k-1]_{q}}\Biggl\langle \hspace{5pt}\begin{minipage}{1\unitlength}\includegraphics[scale=0.1]{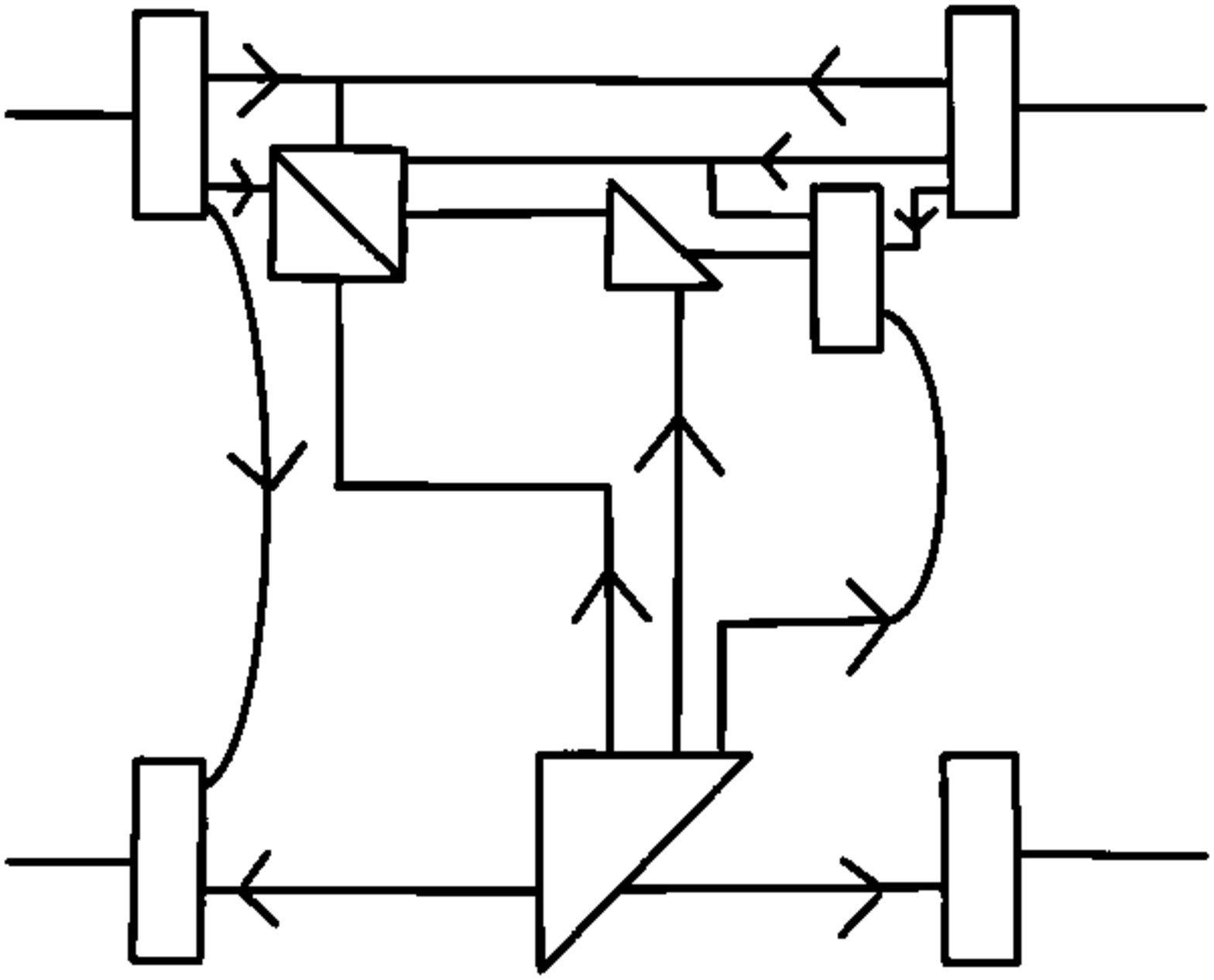}\put(-72,65){\scriptsize$k$\normalsize}\put(-72,15){\scriptsize$k$\normalsize}\put(-15,35){\scriptsize$1$\normalsize}\put(-65,35){\scriptsize$1$\normalsize}\put(-60,5){\scriptsize$k-1$\normalsize}\put(-51,45){\scriptsize$k-3$\normalsize}\put(-60,65){\scriptsize$1$\normalsize}\put(-18,49){\scriptsize$k-3$\normalsize}\put(-30,5){\scriptsize$k-1$\normalsize}\end{minipage}\hspace{80pt}\Biggr\rangle_{3}\underset{\tiny(\mbox{Definition $\ref{stairStep}$})\normalsize}{=}-\frac{[k-2]_{q}}{[k-1]_{q}}\Biggl\langle \hspace{5pt}\begin{minipage}{1\unitlength}\includegraphics[scale=0.1]{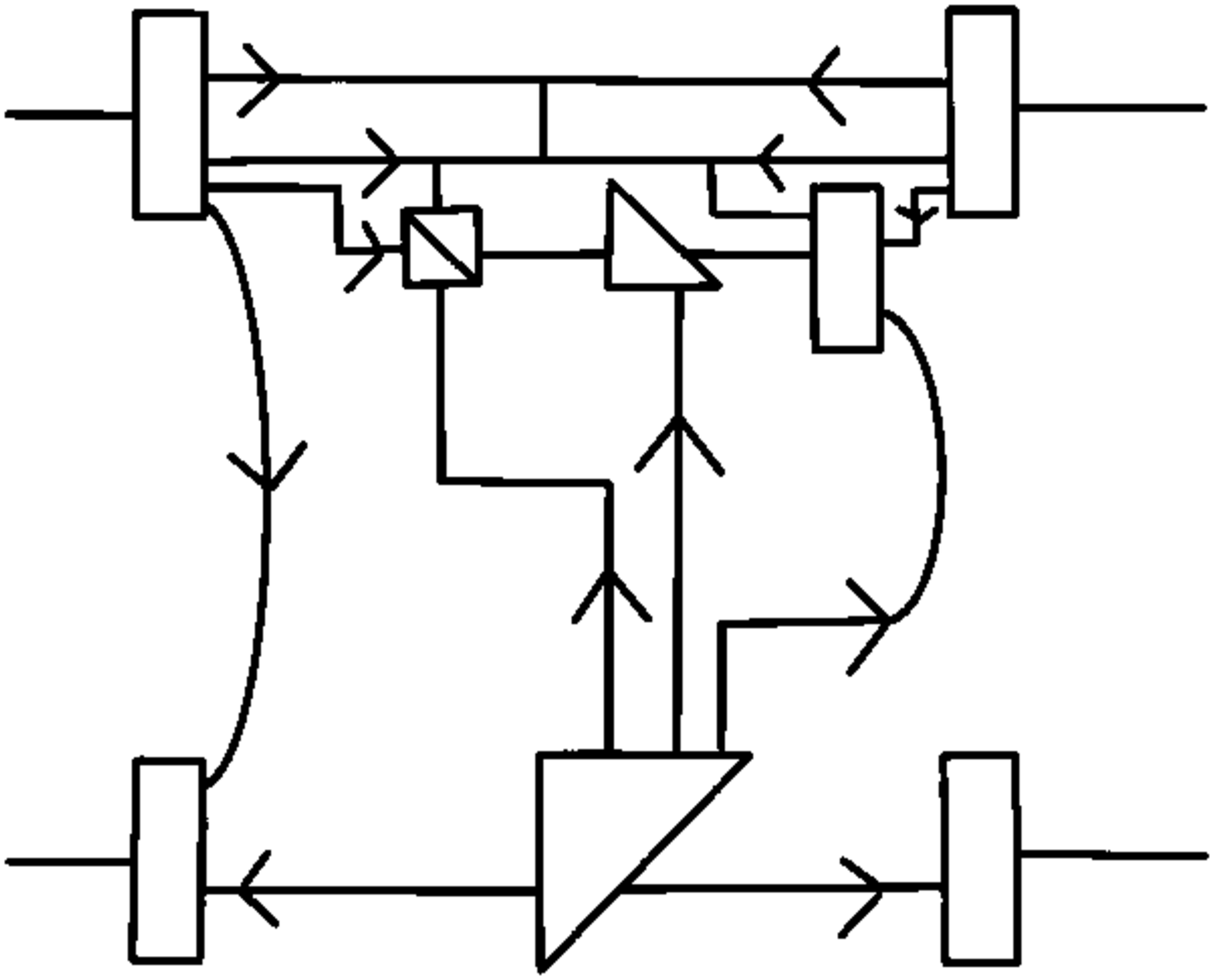}\put(-72,65){\scriptsize$k$\normalsize}\put(-72,15){\scriptsize$k$\normalsize}\put(-15,35){\scriptsize$1$\normalsize}\put(-65,35){\scriptsize$1$\normalsize}\put(-60,5){\scriptsize$k-1$\normalsize}\put(-31,35){\scriptsize$k-3$\normalsize}\put(-42,30){\scriptsize$1$\normalsize}\put(-60,65){\scriptsize$1$\normalsize}\put(-18,49){\scriptsize$k-3$\normalsize}\put(-30,5){\scriptsize$k-1$\normalsize}\end{minipage}\hspace{80pt}\Biggr\rangle_{3}\\
\underset{\tiny(\mbox{Definition $\ref{stairStep}$})\normalsize}{=}&-\frac{[k-2]_{q}}{[k-1]_{q}}\Biggl\langle \hspace{5pt}\begin{minipage}{1\unitlength}\includegraphics[scale=0.1]{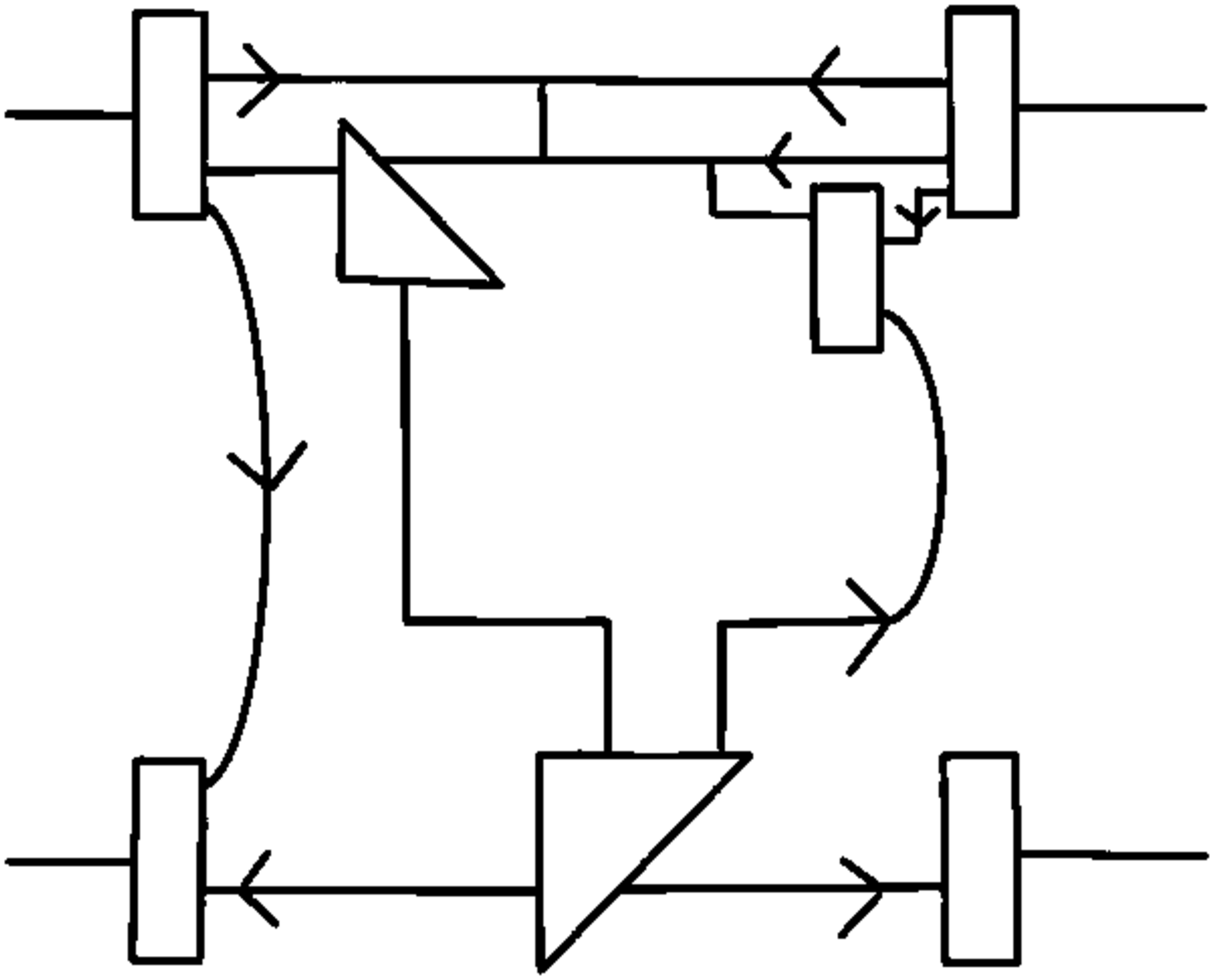}\put(-72,65){\scriptsize$k$\normalsize}\put(-72,15){\scriptsize$k$\normalsize}\put(-15,35){\scriptsize$1$\normalsize}\put(-65,35){\scriptsize$1$\normalsize}\put(-60,5){\scriptsize$k-1$\normalsize}\put(-18,49){\scriptsize$k-3$\normalsize}\put(-47,30){\scriptsize$k-2$\normalsize}\put(-60,65){\scriptsize$1$\normalsize}\put(-30,5){\scriptsize$k-1$\normalsize}\end{minipage}\hspace{80pt}\Biggr\rangle_{3}\underset{\tiny(\mbox{Definition $\ref{stairStep}$, $(\ref{al:delta20})$})\normalsize}{=}-(-1)^{k-2}\frac{1}{[k-1]_{q}}\Biggl\langle \hspace{5pt}\begin{minipage}{1\unitlength}\includegraphics[scale=0.1]{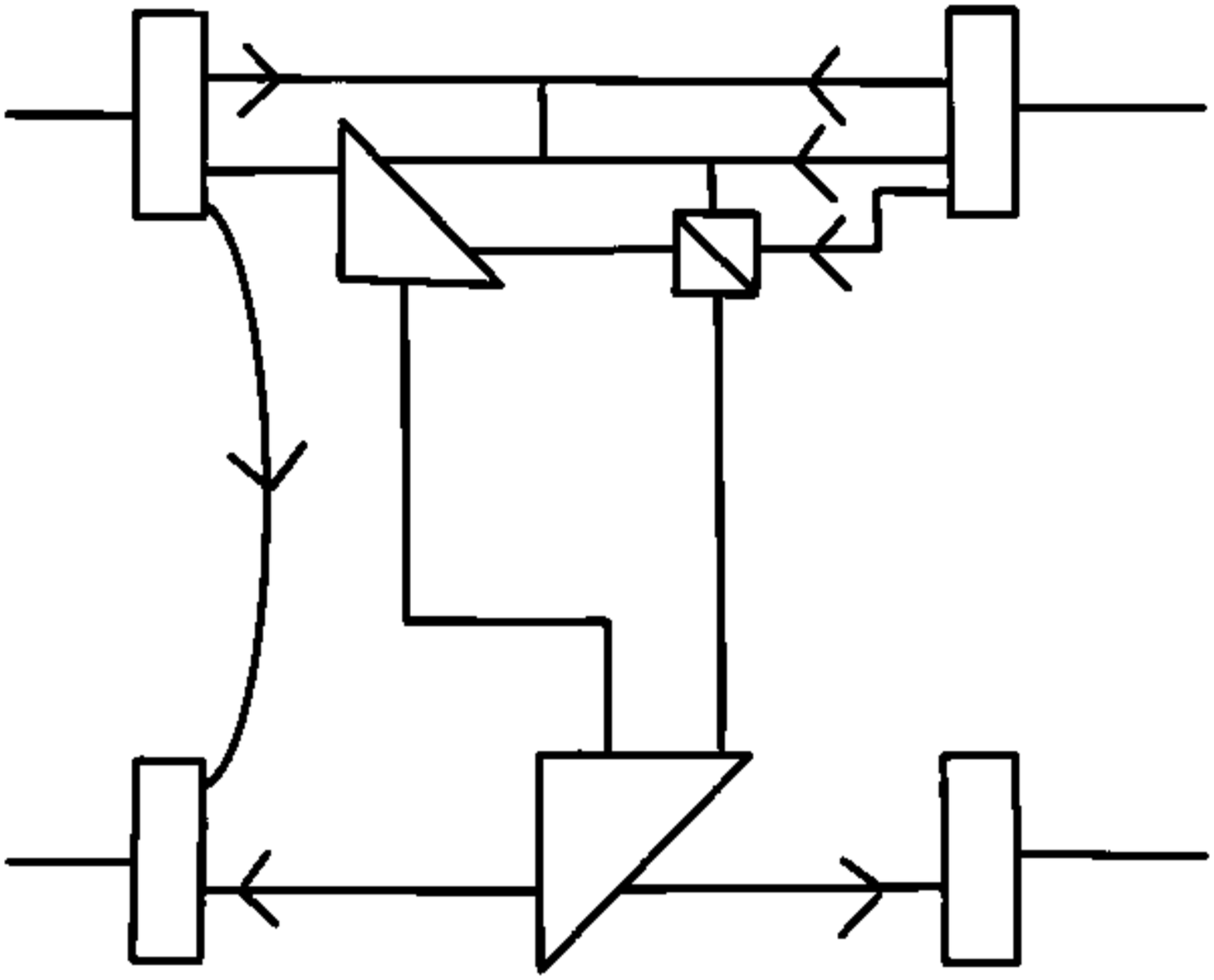}\put(-72,65){\scriptsize$k$\normalsize}\put(-72,15){\scriptsize$k$\normalsize}\put(-30,35){\scriptsize$1$\normalsize}\put(-65,35){\scriptsize$1$\normalsize}\put(-60,5){\scriptsize$k-1$\normalsize}\put(-18,49){\scriptsize$k-3$\normalsize}\put(-47,30){\scriptsize$k-2$\normalsize}\put(-60,65){\scriptsize$1$\normalsize}\put(-30,5){\scriptsize$k-1$\normalsize}\end{minipage}\hspace{80pt}\Biggr\rangle_{3}\\
\underset{\tiny(\mbox{Definition $\ref{stairStep}$})\normalsize}{=}&-(-1)^{k-2}\frac{1}{[k-1]_{q}}\Biggl\langle \hspace{5pt}\begin{minipage}{1\unitlength}\includegraphics[scale=0.1]{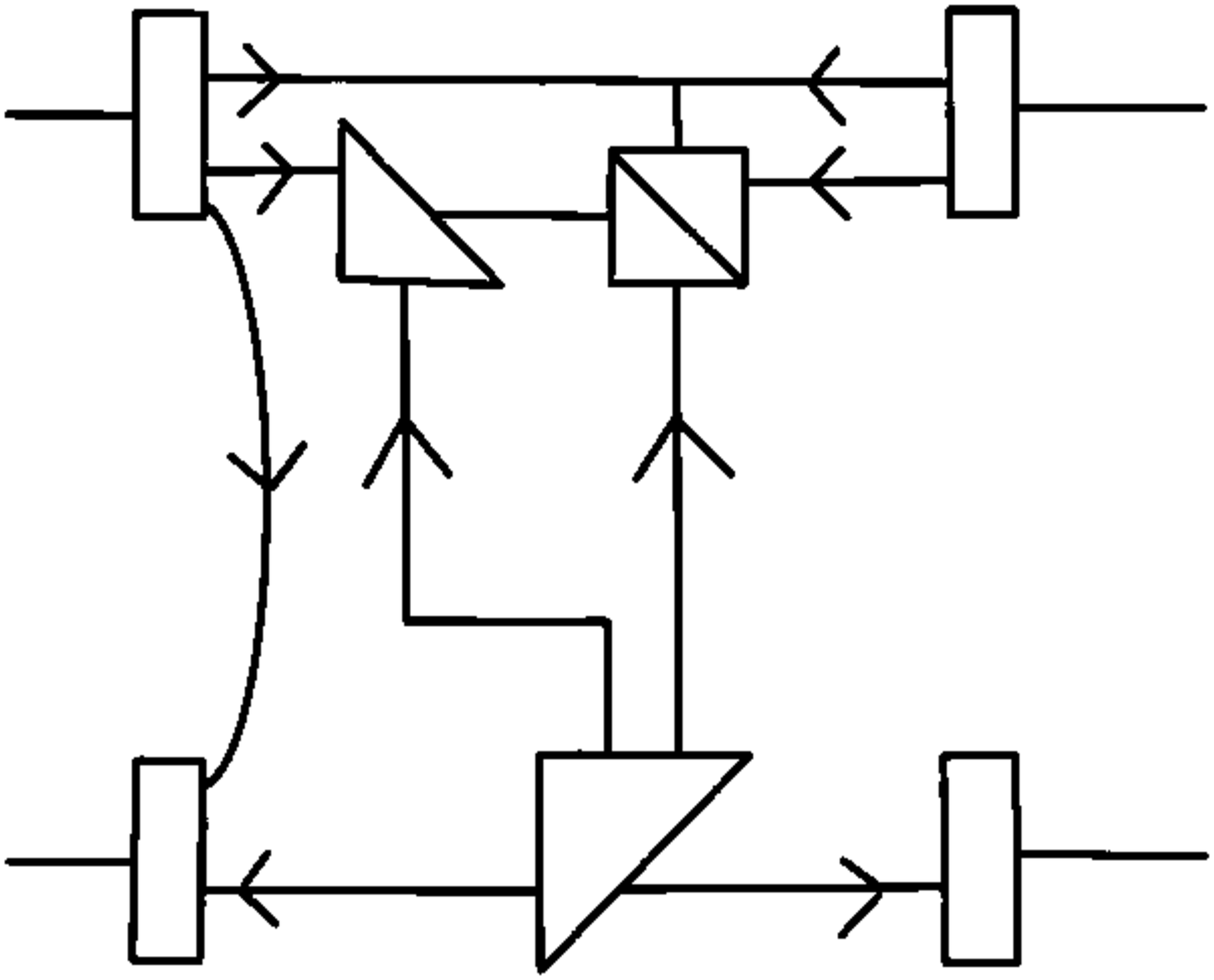}\put(-72,65){\scriptsize$k$\normalsize}\put(-72,15){\scriptsize$k$\normalsize}\put(-30,35){\scriptsize$1$\normalsize}\put(-65,35){\scriptsize$1$\normalsize}\put(-60,5){\scriptsize$k-1$\normalsize}\put(-26,48){\scriptsize$k-2$\normalsize}\put(-47,30){\scriptsize$k-2$\normalsize}\put(-60,65){\scriptsize$1$\normalsize}\put(-30,5){\scriptsize$k-1$\normalsize}\end{minipage}\hspace{80pt}\Biggr\rangle_{3}\underset{\tiny(\mbox{induction hypothesis})\normalsize}{=}-(-1)^{k-2}\frac{1}{[k-1]_{q}}\Biggl\langle \hspace{5pt}\begin{minipage}{1\unitlength}\includegraphics[scale=0.1]{pic/OmegaAfter.eps}\put(-72,63){\scriptsize$k$\normalsize}\put(-72,15){\scriptsize$k$\normalsize}\put(-62,35){\scriptsize$1$\normalsize}\end{minipage}\hspace{80pt}\Biggr\rangle_{3}
\end{align*}
Regarding second equality, we have
\begin{align*}
\Biggl\langle \hspace{5pt}\begin{minipage}{1\unitlength}\includegraphics[scale=0.1]{pic/Omega4.eps}\put(-72,65){\scriptsize$k$\normalsize}\put(-72,15){\scriptsize$k$\normalsize}\put(-18,35){\scriptsize$1$\normalsize}\put(-72,49){\scriptsize$k-1$\normalsize}\put(-60,5){\scriptsize$k-1$\normalsize}\put(-60,65){\scriptsize$1$\normalsize}\put(-20,49){\scriptsize$k-3$\normalsize}\put(-30,5){\scriptsize$k-1$\normalsize}\end{minipage}\hspace{80pt}\Biggr\rangle_{3}=0
\end{align*}
by using Definition $\ref{stairStep}$ and (\ref{al:double2}). Hence,
\begin{align*}
\Biggl\langle \hspace{5pt}\begin{minipage}{1\unitlength}\includegraphics[scale=0.1]{pic/OmegaBefore.eps}\put(-74,65){\scriptsize$k$\normalsize}\put(-74,17){\scriptsize$k$\normalsize}\put(-5,35){\scriptsize$1$\normalsize}\end{minipage}\hspace{80pt}\Biggr\rangle_{3}=\frac{[k]^{2}_{q}-1}{[k]_{q}[k-1]_{q}}\Biggl\langle \hspace{5pt}\begin{minipage}{1\unitlength}\includegraphics[scale=0.1]{pic/OmegaAfter.eps}\put(-74,65){\scriptsize$k$\normalsize}\put(-74,17){\scriptsize$k$\normalsize}\put(-60,35){\scriptsize$1$\normalsize}\end{minipage}\hspace{80pt}\Biggr\rangle_{3}
\end{align*}
We can check easily the following equation by Lemma \ref{quantumIntger}.
\begin{align}
\label{eq:induction1}
\frac{[k]^{2}_{q}-1}{[k]_{q}[k-1]_{q}}=\frac{[k+1]_{q}}{[k]_{q}}
\end{align}
Therefore, $(\ref{eq:Omega1})$ holds for $n=k$. $(\ref{eq:induction1})$ is not defined for $n=1$. However  $(\ref{eq:Omega1})$ holds for $n=1$ by definition of $A_{2}$ bracket. We can prove easlry   that if we use $(\ref{eq:induction1})$ repeatedly, then Lemma$\ref{Lem:Omega2}$ holds.
\end{proof}

\begin{proof}[$\mbox{Proof of Proposition}\ref{pro:Omega}$]
\begin{align*}
\Biggl\langle \hspace{5pt}\begin{minipage}{1\unitlength}\includegraphics[scale=0.1]{pic/Omega.eps}\put(-35,66){\scriptsize$n-l$\normalsize}\put(-35,4){\scriptsize$n-l$\normalsize}\put(-60,66){\scriptsize$n-k$\normalsize}\put(-60,4){\scriptsize$n-\hspace{-2pt}k$\normalsize}\put(-49,45){\scriptsize$k$\normalsize}\put(-36,25){\scriptsize$l$\normalsize}\put(-77,36){\scriptsize$k$\normalsize}\put(-6,36){\scriptsize$l$\normalsize}\put(-80,62){\scriptsize$n$\normalsize}\put(-80,8){\scriptsize$n$\normalsize}\end{minipage}\hspace{80pt}\Biggr\rangle_{3}\underset{\tiny(\mbox{Theorem$\ref{Thm:yuasa6}$})\normalsize}{=}&\sum_{t=\max\{k,l\}}^{\min\{k+l,n\}}\psi(n,t,k,l)\Biggl\langle \hspace{5pt}\begin{minipage}{1\unitlength}\includegraphics[scale=0.1]{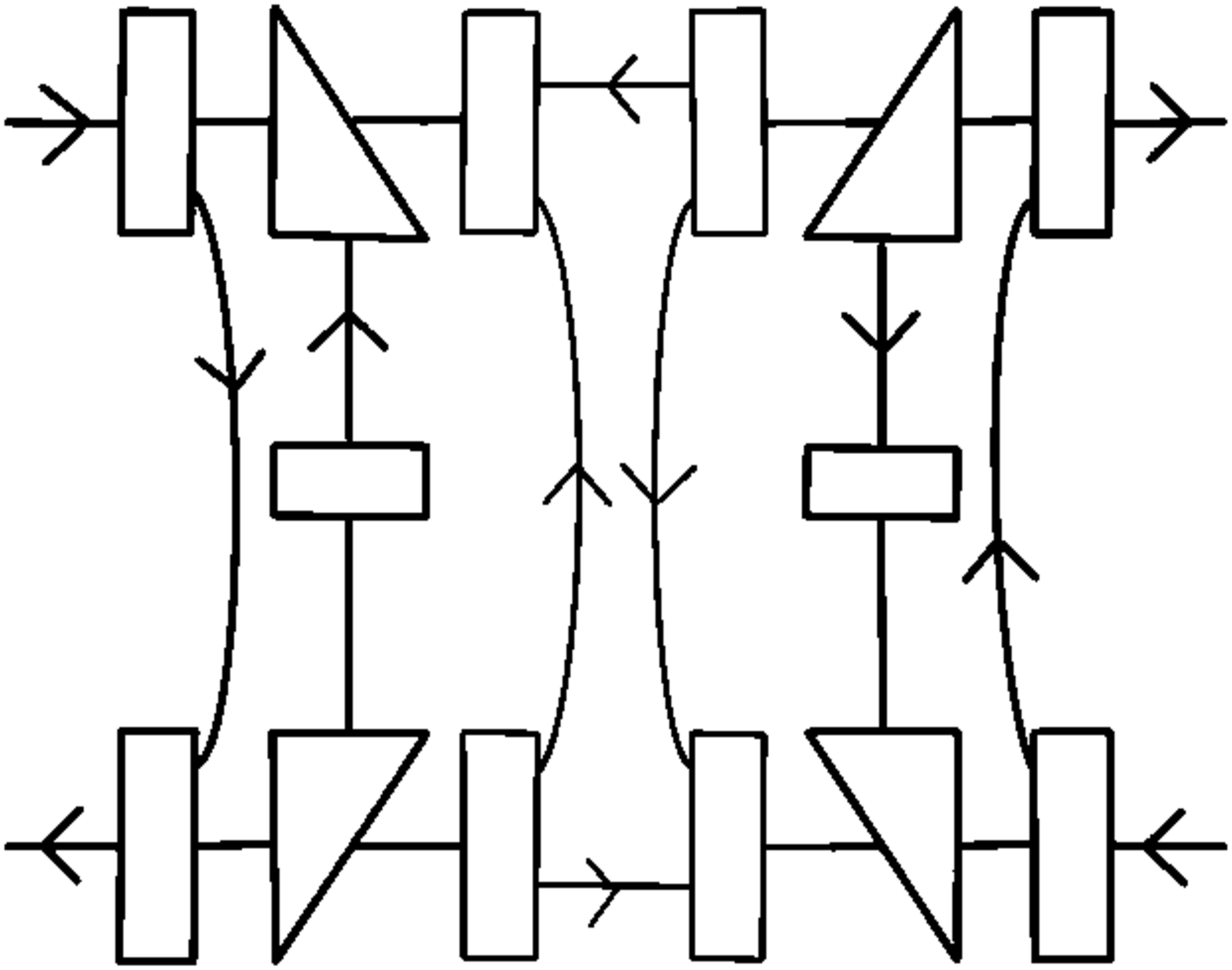}\put(-33,68){\scriptsize$n-l$\normalsize}\put(-33,2){\scriptsize$n-l$\normalsize}\put(-60,68){\scriptsize$n-k$\normalsize}\put(-60,2){\scriptsize$n-k$\normalsize}\put(-55,40){\scriptsize$t-k$\normalsize}\put(-35,25){\scriptsize$t-l$\normalsize}\put(-75,36){\scriptsize$k$\normalsize}\put(-8,36){\scriptsize$l$\normalsize}\put(-80,62){\scriptsize$n$\normalsize}\put(-80,8){\scriptsize$n$\normalsize}\end{minipage}\hspace{80pt}\Biggr\rangle_{3}\\
\underset{\tiny(\mbox{$(\ref{al:double1})$, Lemma $\ref{Lem:Omega2}$})\normalsize}{=}&\sum_{t=\max\{k,l\}}^{\min\{k+l,n\}}\frac{[n-k+1]_{q}[n-l+1]_{q}}{[n+1-t]_{q}^{2}}\psi(n,t,k,l)\Biggl\langle \hspace{5pt}\begin{minipage}{1\unitlength}\includegraphics[scale=0.1]{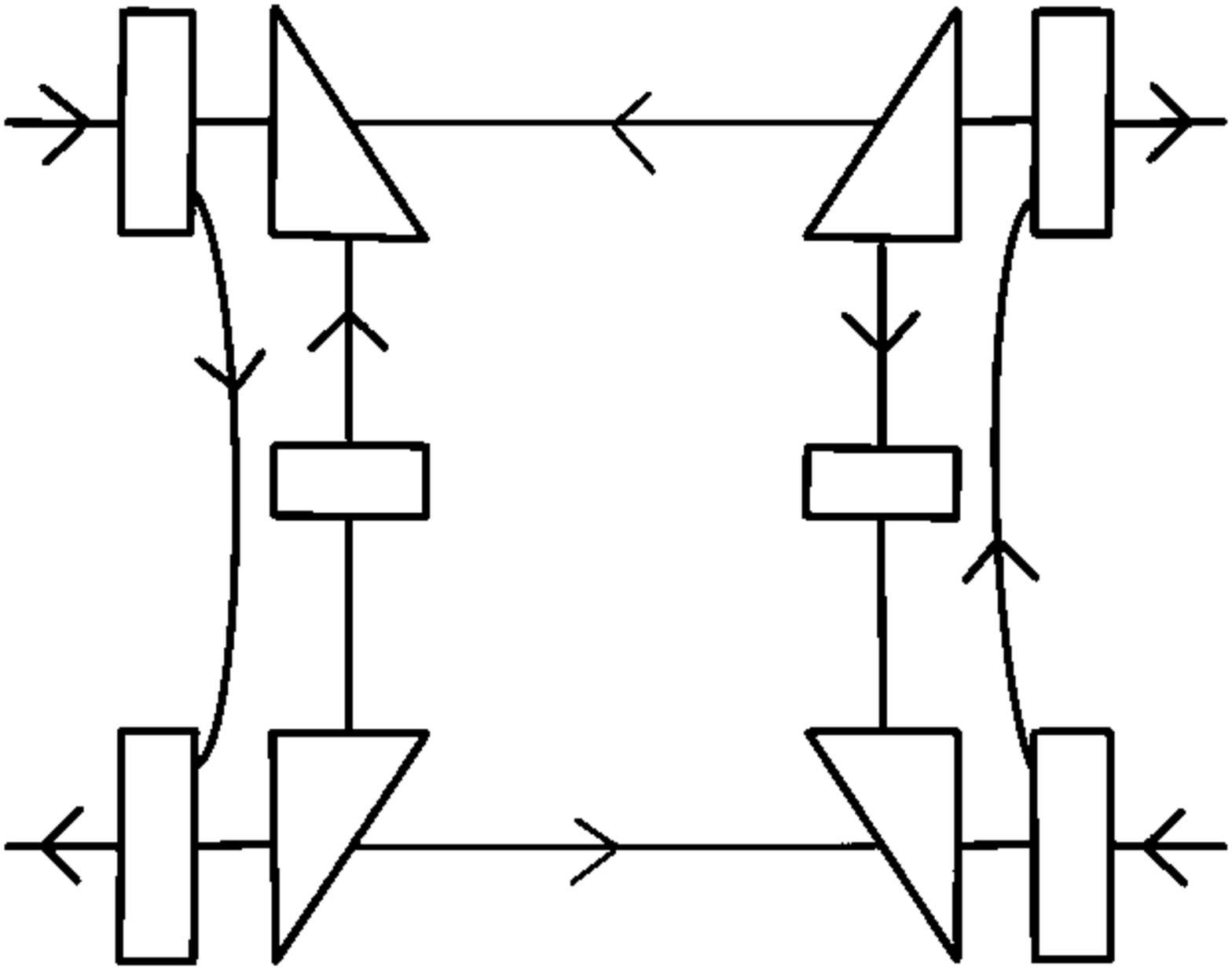}\put(-45,66){\scriptsize$n-t$\normalsize}\put(-45,4){\scriptsize$n-t$\normalsize}\put(-75,36){\scriptsize$t$\normalsize}\put(-8,36){\scriptsize$t$\normalsize}\put(-80,62){\scriptsize$n$\normalsize}\put(-80,8){\scriptsize$n$\normalsize}\end{minipage}\hspace{80pt}\Biggr\rangle_{3}\\
\underset{(\tiny\mbox{(\ref{al:delta6})}\normalsize)}{=}&\sum_{t=\max\{k,l\}}^{\min\{k+l,n\}}\sum_{a=t}^{n}\Omega(n,t,k,l)\Biggl\langle\hspace{10pt} \begin{minipage}{1\unitlength}\includegraphics[scale=0.1]{pic/Omega_after.eps}\put(-74,60){\scriptsize$n$\normalsize}\put(-74,8){\scriptsize$n$\normalsize}\put(-45,6){\scriptsize$n-a$\normalsize}\put(-45,62){\scriptsize$n-a$\normalsize}\put(-65,35){\scriptsize$a$\normalsize}\put(-13,35){\scriptsize$a$\normalsize}\end{minipage}\hspace{82pt}\Biggr\rangle_{3}
\end{align*}

\end{proof}


\section{Computing the one-row $\mathfrak{sl}_{3}$ colored  Jones polynomial for pretzel links}
\label{sec:Computing}
	In this section, we compute the one-row $\mathfrak{sl}_{3}$ colored Jones polynomials for pretzel links. We consider the three-parameter family of pretzel links $P(\alpha,\beta,\gamma)$  for integers $\alpha,\beta,\gamma$.\\
\\
\begin{align*}
   P(\alpha,\beta,\gamma)=\begin{minipage}{10\unitlength}\includegraphics[scale=0.1]{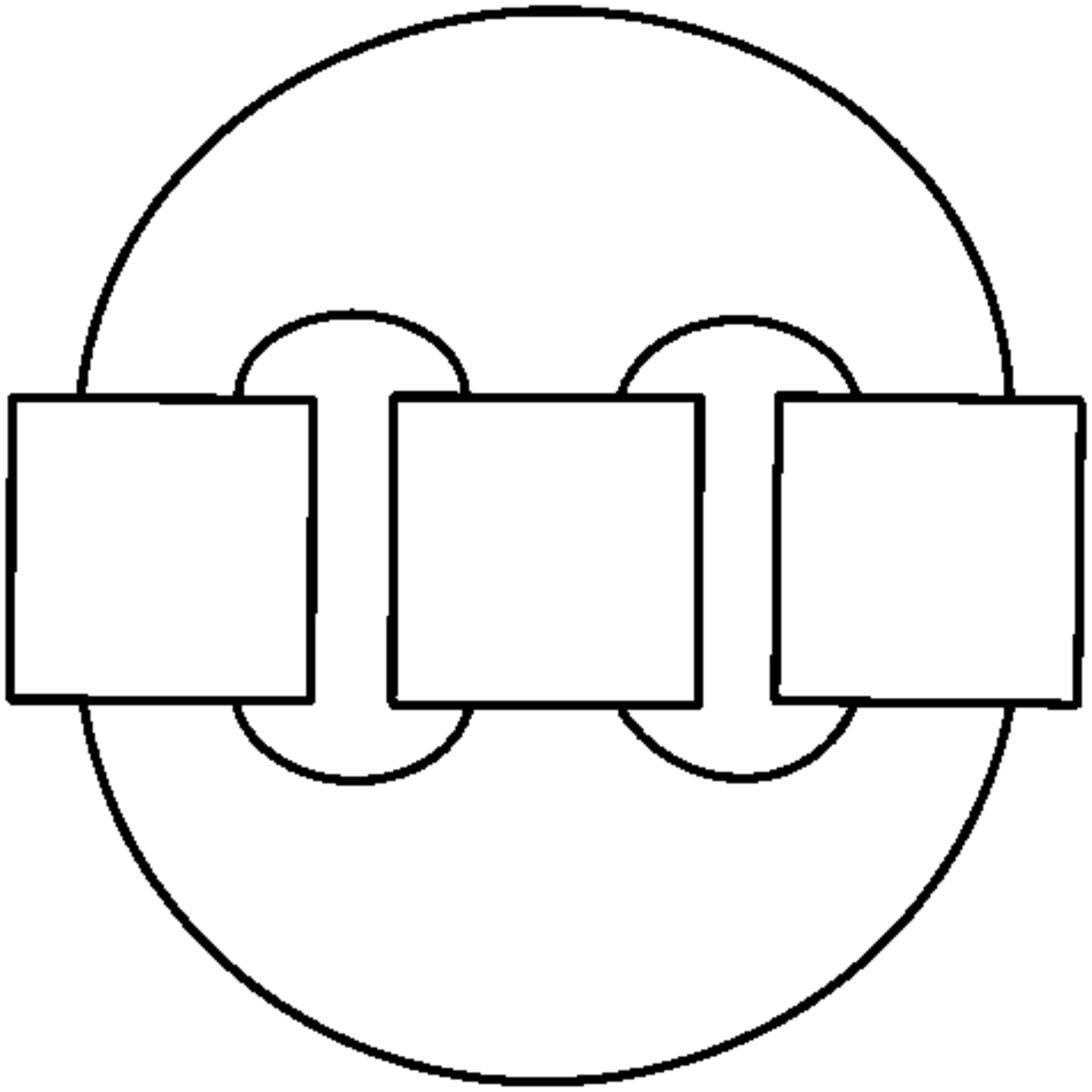}\put(-62,35){$\alpha$}\put(-40,35){$\beta$}\put(-19,35){$\gamma$}\end{minipage}\qquad\qquad
   \end{align*}
where a box marked with the letter $\alpha$ represents a right-handed $($resp. left-handed$)$ $\left| \alpha \right|$-twist if $\alpha>0$ $($resp. $\alpha<0$$)$. Boxes with other letters represent the same thing. The one-row colored $\mathfrak{sl}_{3}$ Jones polynomial of a link depends on the direction of the link. Therefore, it is necessary to distinguish links with different directions. For example, we consider a knot with the directions using the arrow symbol as follows.
\begin{align*}
P(\downarrow 3 \downarrow,\uparrow3\uparrow,\downarrow2\uparrow)=\begin{minipage}{10\unitlength}\includegraphics[scale=0.1]{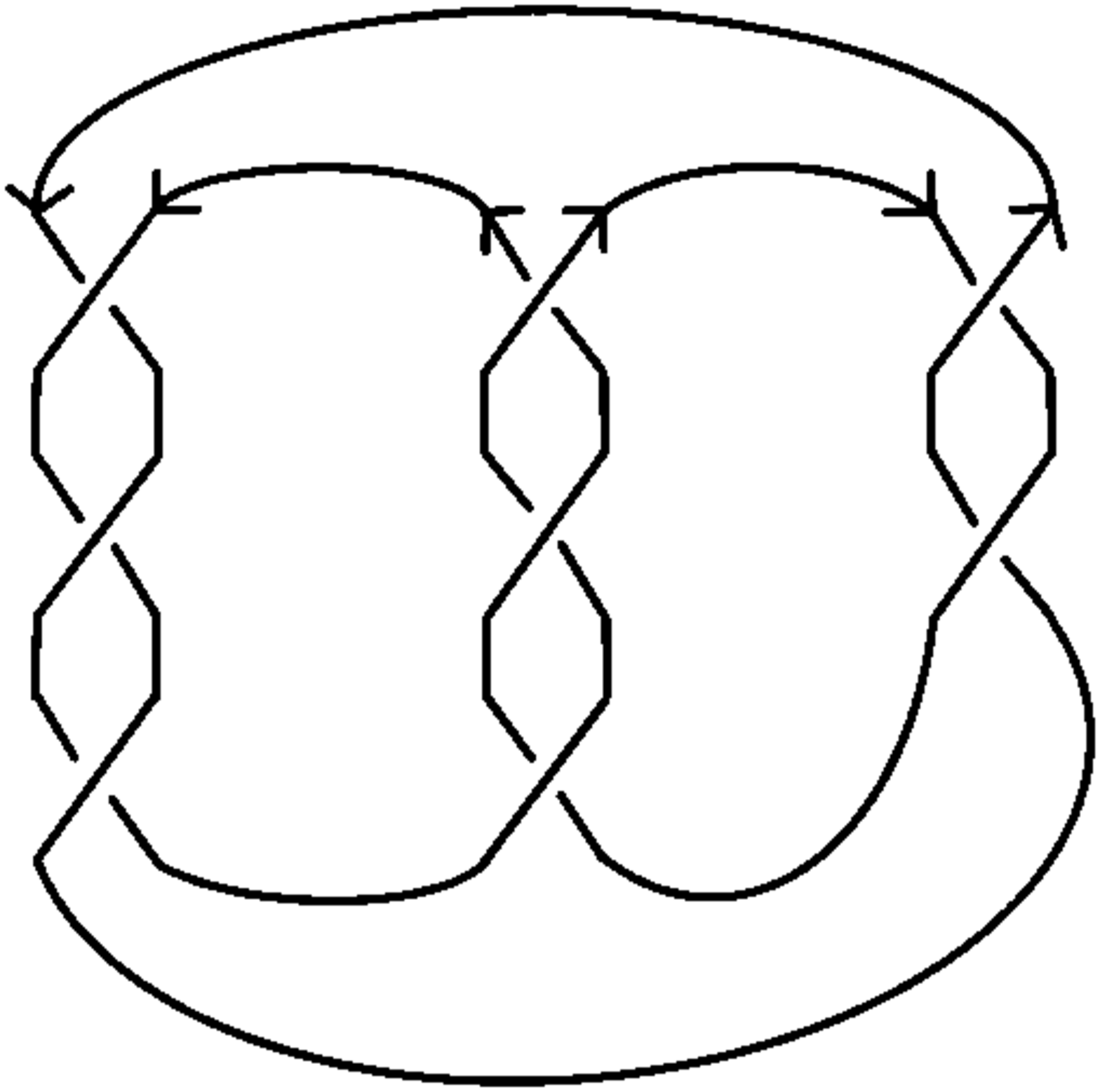}\end{minipage}\qquad\qquad
\end{align*}

\begin{Theorem}
Let $\alpha, \beta,\gamma$ be integers. The one-row colored $\mathfrak{sl}_{3}$  Jones polynomials  for pretzel links $P(\alpha,\beta,\gamma)$ are the following:
 \label{Thm:main2}
 	\begin{align}
 	\label{ali:Jones1}
	\begin{aligned}
	&J_{(n,0)}^{\mathfrak{sl}_{3}}(P(\downarrow2\alpha\uparrow, \downarrow2\beta\uparrow,\downarrow 2\gamma\uparrow);q)\\
=&\sum_{0\leq k_{|\alpha|}\leq k_{|\alpha |-1}\leq\cdots\leq k_{1}\leq n}\quad\sum_{0\leq l_{|\beta|}\leq l_{|\beta| -1}\leq\cdots\leq l_{1}\leq n}\quad\sum_{0\leq m_{|\gamma|}\leq m_{|\gamma| -1}\leq\cdots\leq m_{1}\leq n} \quad\sum_{s=\max\{k_{|\alpha|}, l_{|\beta|}\}}^{\min\{k_{|\alpha|}+l_{|\beta|},n\}}\quad\sum_{t=\max\{s, m_{|\gamma|}\}}^{\min\{s+m_{|\gamma|},n\}}(q^{\frac{n^{2}+3n}{3}})^{-(2\alpha+2\beta+2\gamma)}\\
&\phi(n,k_{1},k_{2},...,k_{|\alpha|})_{q^{\epsilon_{\alpha}}}\phi(n,l_{1},l_{2},...,l_{|\beta|})_{q^{\epsilon_{\beta}}}\phi(n,m_{1},m_{2},...,m_{|\gamma|})_{q^{\epsilon_{\gamma}}}	\psi(n,s,k_{|\alpha|},l_{|\beta|})\psi(n,t,s,m_{|\gamma|})q^{-(n-t)}\frac{(1-q^{n+1})(1-q^{n+2})}{(1-q^{t+1})(1-q^{t+2})},
\end{aligned}
\end{align}
\begin{align}
\label{ali:Jones2}
\begin{aligned}
&J_{(n,0)}^{\mathfrak{sl}_{3}}(P(\downarrow2\alpha\downarrow, \uparrow2\beta\uparrow,\downarrow 2\gamma\uparrow);q)\\
=&\sum_{0\leq k_{2|\alpha|}\leq k_{2|\alpha| -1}\leq\cdots\leq k_{1}\leq n}\quad\sum_{0\leq l_{2|\beta|}\leq l_{2|\beta| -1}\leq\cdots\leq l_{1}\leq n}\quad\sum_{0\leq m_{\gamma}\leq m_{\gamma -1}\leq\cdots\leq m_{1}\leq n} \quad\sum_{s=\max\{k_{2|\alpha|}, l_{2|\beta|}\}}^{\min\{k_{2|\alpha|}+l_{2|\beta|},n\}}\quad\sum_{a=0}^{n-s}\quad\sum_{t=\max\{a, m_{|\gamma|}\}}^{\min\{a+t+m_{|\gamma|},n\}}(q^{\frac{n^{2}+3n}{3}})^{2\alpha+2\beta-2\gamma}\\&\chi_{sign(2\alpha)}(n,k_{1},k_{2},...,k_{|\alpha|})\chi_{sign(2|\beta|)}(n,l_{1},l_{2},...,l_{2|\beta|})\phi(n,m_{1},m_{2},...,m_{|\gamma|})_{q^{\epsilon_{\gamma}}}\Omega(n,s,k_{2|\alpha|},l_{2|\beta|})\psi(n,t,a,m_{|\gamma|})\\
&q^{-(n-t)}\frac{(1-q^{n+1})(1-q^{n+2})}{(1-q^{t+1})(1-q^{t+2})},
\end{aligned}
\end{align}
\begin{align}
\label{ali:Jones3}
	\begin{aligned}
&J_{(n,0)}^{\mathfrak{sl}_{3}}(P(\downarrow2\alpha+1\downarrow,\uparrow2\beta+1\uparrow,\downarrow2\gamma\uparrow);q)\\
=&\sum_{0\leq k_{|2\alpha+1|}\leq k_{|2\alpha| }\leq\cdots\leq k_{1}\leq n}\quad\sum_{0\leq l_{|2\beta+1|}\leq l_{|2\beta| }\leq\cdots\leq l_{1}\leq n}\quad\sum_{0\leq m_{|\gamma|}\leq m_{|\gamma|-1}\leq\cdots \leq m_{1}\leq n}\quad\sum^{\min\{k_{|2\alpha+1|}+l_{|2\beta+1|},n\}}_{s=\max\{k_{|2\alpha+1|},l_{|2\beta+1|}\}}\quad\sum_{a=s}^{n}\quad\sum^{\min\{a+s,n\}}_{t=\max\{a,s\}}\\
&(q^{\frac{n^{2}+3n}{3}})^{-(2\alpha+2\beta-2\gamma+2)}\chi_{sign(2\alpha+1)}(n,k_{1},k_{2},...,k_{|2\alpha+1|})\chi_{sing(2\beta+1)}(n,l_{1},l_{2},...,l_{|2\beta+1|})\phi(n,m_{1},m_{2},...,m_{|\gamma|})_{q^{\epsilon_{\gamma}}}\\
&\Omega(n,s,k_{|2\alpha+1|},l_{|2\beta+1|})\psi(n,a,m_{|\gamma|},t)q^{-(n-t)}\frac{(1-q^{n+1})(1-q^{n+2})}{(1-q^{t+1})(1-q^{t+2})},
\end{aligned}
\end{align}
\begin{align}
	\begin{aligned}
 	&J^{\mathfrak{sl}_{3}}_{(n,0)}(P(\downarrow2\alpha+1\downarrow,\uparrow2\beta\uparrow,\downarrow2\gamma\uparrow);q)\\
 	=&\sum_{0\leq k_{|2\alpha+1|}\leq k_{|2\alpha|}\leq\cdots\leq k_{1}\leq n}\quad\sum_{0\leq l_{2|\beta|}\leq l_{2|\beta| -1}\leq\cdots\leq l_{1}\leq n}\quad\sum_{0\leq m_{|\gamma|}\leq m_{|\gamma|-1}\leq\cdots \leq m_{1}\leq n}\quad\sum_{s=\max\{k_{|2\alpha+1|},l_{2|\beta|}\}}^{\min\{k_{|2\alpha+1|}+l_{2|\beta|}\}}\quad\sum_{a=0}^{n-s}\quad\sum_{t=\{a,s\}}^{\min\{a+s,n\}}(q^{\frac{n^{2}+3n}{3}})^{-(-2\alpha-2\beta+2\gamma-1)}\\
&\chi_{sign(2\alpha+1)}(n,k_{1},k_{2},...,k_{|2\alpha+1|}) \chi_{sign(2\beta)}(n,l_{1},l_{2},...,l_{2|\beta|})\phi(n,m_{1},m_{2},...,m_{|\gamma|})_{q^{\epsilon_{|\gamma|}}}\Omega(n,k_{|2\alpha+1|},l_{2|\beta|},t)\\
&\psi(n,t,a,m_{|\gamma|})q^{-(n-t)}\frac{(1-q^{n+1})(1-q^{n+2})}{(1-q^{t+1})(1-q^{t+2})},
\end{aligned}
\end{align}
\begin{align}
\label{ali:Jones4}
	\begin{aligned}
 	&J^{\mathfrak{sl}_{3}}_{(n,0)}(P(\downarrow2\alpha+1\downarrow,\uparrow2\beta\downarrow,\uparrow2\gamma\uparrow);q)\\
 	=&\sum_{0\leq k_{|2\alpha+1|}\leq k_{|2\alpha|}\leq\cdots\leq k_{1}\leq n}\quad\sum_{0\leq l_{|\beta|}\leq l_{|\beta| -1}\leq\cdots\leq l_{1}\leq n}\quad\sum_{0\leq m_{|\gamma|}\leq m_{|\gamma|-1}\leq\cdots \leq m_{1}\leq n}\quad\sum_{s=\max\{k_{2|\alpha|},l_{|\beta|}\}}^{\min\{k_{2|\alpha|}+l_{|\beta|}\}}\quad\sum_{a=0}^{n-s}\quad\sum_{t=\{a,s\}}^{\min\{a+s,n\}}(q^{\frac{n^{2}+3n}{3}})^{-(-2\alpha-2\beta+2\gamma-1)}\\
&\chi_{sign(2\alpha+1)}(n,k_{1},k_{2},...,k_{|2\alpha+1|}) \phi(n,l_{1},l_{2},...,l_{|\beta|})_{q^{\epsilon_{\beta}}}\chi_{sing(2\gamma)}(n,m_{1},m_{2},...,m_{|\gamma|})_{q^{\epsilon_{|\gamma|}}}\Omega(n,k_{2|\alpha|},l_{|\beta|},t)\\
&\psi(n,t,a,m_{|\gamma|})q^{-(n-t)}\frac{(1-q^{n+1})(1-q^{n+2})}{(1-q^{t+1})(1-q^{t+2})}.
\end{aligned}
\end{align}
\end{Theorem}
\begin{proof}
We first prove $(\ref{ali:Jones2})$. 
\begin{align*}
&\qquad\qquad J_{(n,0)}^{\mathfrak{sl}_{3}}(P(\downarrow2\alpha+1\downarrow,\uparrow2\beta+1\uparrow,\downarrow2\gamma\uparrow);q)\\
&\underset{\tiny(\mbox{Definition $\ref{Def:coloredJones}$})\normalsize}{=}(q^{\frac{n^{2}+3n}{3}})^{-(2\alpha+2\beta-2\gamma+2)}\Biggl\langle\begin{minipage}{1\unitlength}\includegraphics[scale=0.12]{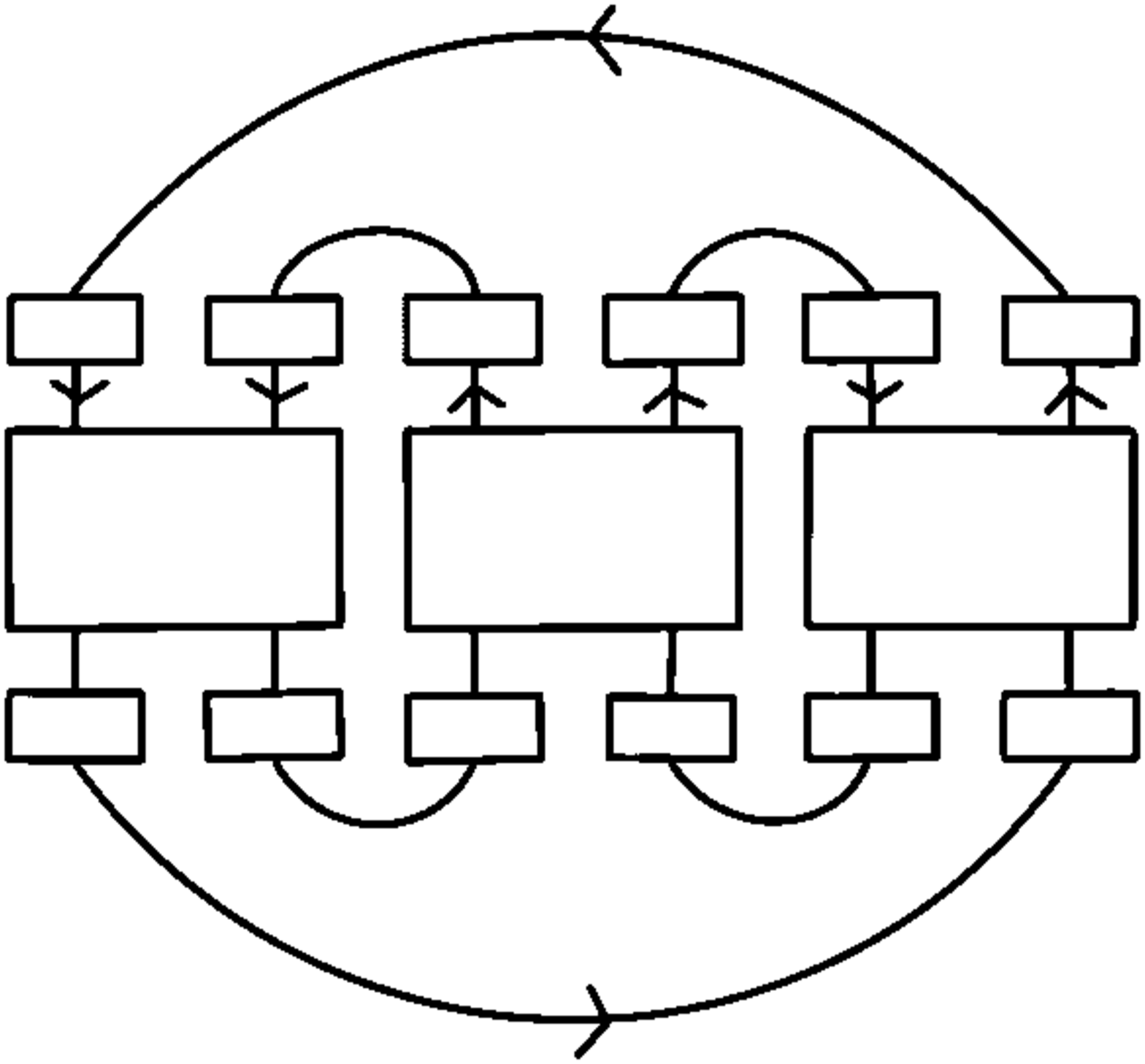}\put(-85,68){\scriptsize$n$\normalsize}\put(-88,40){\scriptsize$2\alpha+1$\normalsize}\put(-54,40){\scriptsize$2\beta+1$\normalsize}\put(-22,40){\scriptsize$2\gamma$\normalsize}\end{minipage}\hspace{95pt} \Biggr\rangle_{3}/\Delta(n,0)\\
&\underset{\tiny(\mbox{Proposition $\ref{Pro:halfTwist1}$})\normalsize}{=}(q^{\frac{n^{2}+3n}{3}})^{-(2\alpha+2\beta-2\gamma+2)}\sum_{0\leq k_{|2\alpha+1|}\leq k_{|2\alpha|}\leq\cdots\leq k_{1}\leq n}\quad\chi_{sign(2\alpha+1)}(n,k_{1},k_{2},...,k_{|2\alpha+1|})\\
&\qquad\qquad \Biggl\langle\hspace{15pt}\begin{minipage}{1\unitlength}\includegraphics[scale=0.12]{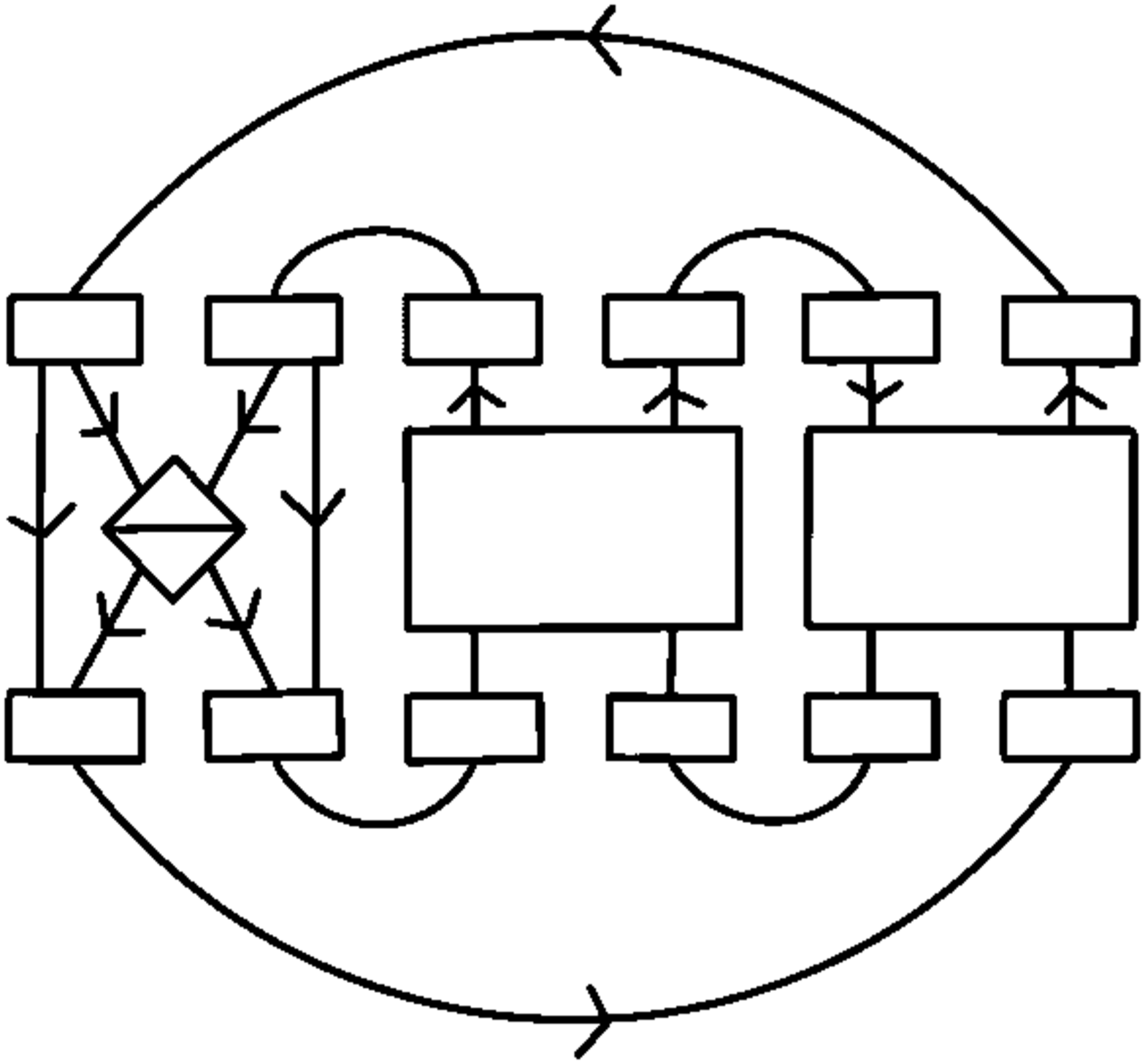}\put(-85,68){\scriptsize$n$\normalsize}\put(-110,45){\scriptsize$k_{|2\alpha+1|}$\normalsize}\put(-54,40){\scriptsize$2\beta+1$\normalsize}\put(-22,40){\scriptsize$2\gamma$\normalsize}\end{minipage}\hspace{95pt} \Biggr\rangle_{3}/\Delta(n,0)\\
&\underset{\tiny(\mbox{Proposition $\ref{Pro:halfTwist1}$,Theorem $\ref{Thm:yuasaCP}$})\normalsize}{=}(q^{\frac{n^{2}+3n}{3}})^{-(2\alpha+2\beta-2\gamma+2)}\sum_{0\leq k_{|2\alpha+1|}\leq k_{|2\alpha|}\leq\cdots\leq k_{1}\leq n}\quad\sum_{0\leq l_{|2\beta+1|}\leq l_{|2\beta|}\leq\cdots\leq l_{1}\leq n}\quad\sum_{0\leq m_{\gamma}\leq m_{\gamma-1}\leq\cdots \leq m_{1}\leq n}\\
&\qquad\qquad\times\chi_{sign(2\alpha+1)}(n,k_{1},k_{2},...,k_{|2\alpha+1|})\chi_{sign(2\beta+1)}(n,l_{1},l_{2},...,l_{|2\beta+1|})\phi(n,m_{1},m_{2},...,m_{\gamma})_{q^{\epsilon_{\gamma}}}\\
&\qquad\qquad\times\Biggl\langle\hspace{15pt}\begin{minipage}{1\unitlength}\includegraphics[scale=0.17]{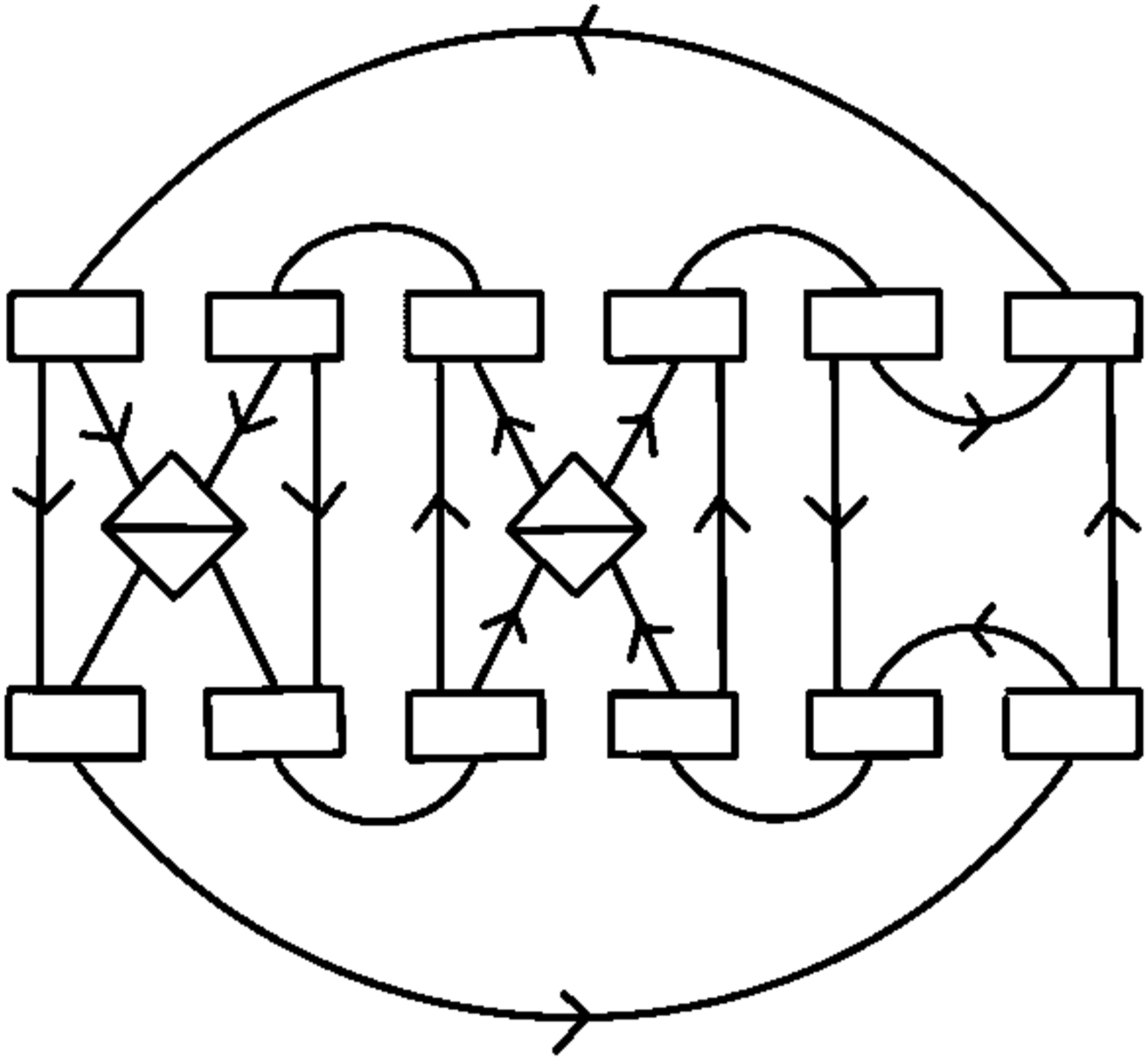}\put(-100,110){\scriptsize$n$\normalsize}\put(-145,65){\scriptsize$k_{|2\alpha+1|}$\normalsize}\put(-94,55){\fontsize{4.8pt}{0cm}{$l_{|2\beta+1|}$}\normalsize}\put(-5,60){\scriptsize$m_{|\gamma|}$\normalsize}\end{minipage}\hspace{140pt} \Biggr\rangle_{3}/\Delta(n,0)\\
&\underset{\tiny(\mbox{$(\ref{al:double1})$})\normalsize}{=}(q^{\frac{n^{2}+3n}{3}})^{-(2\alpha+2\beta-2\gamma+2)}\sum_{0\leq k_{|2\alpha+1|}\leq k_{|2\alpha|}\leq\cdots\leq k_{1}\leq n}\quad\sum_{0\leq l_{|2\beta+1|}\leq l_{|2\beta|}\leq\cdots\leq l_{1}\leq n}\quad\sum_{0\leq m_{\gamma}\leq m_{\gamma-1}\leq\cdots \leq m_{1}\leq n}\\
&\qquad\qquad\times\chi_{sign(2\alpha+1)}(n,k_{1},k_{2},...,k_{|2\alpha+1|})\chi_{sign(2\beta+1)}(n,l_{1},l_{2},...,l_{|2\beta+1|})\phi(n,m_{1},m_{2},...,m_{\gamma})_{q^{\epsilon_{\gamma}}}\\
&\qquad\times\Biggl\langle\begin{minipage}{1\unitlength}\includegraphics[scale=0.12]{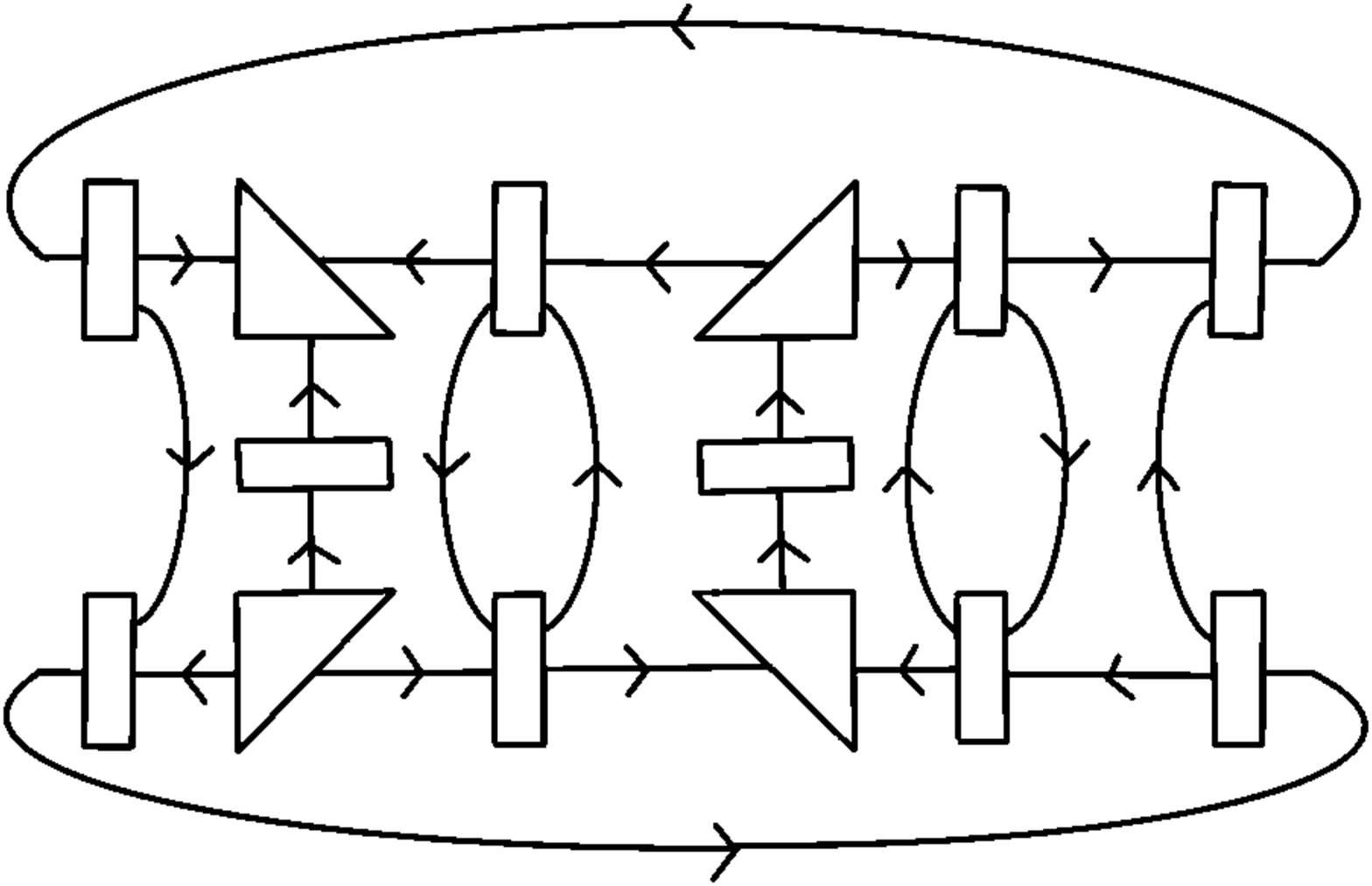}\put(-80,67){\scriptsize$n-l_{|2\beta+1|}$\normalsize}\put(-35,65){\scriptsize$n-m_{|\gamma|}$\normalsize}\put(-125,80){\scriptsize$n$\normalsize}\put(-135,40){\tiny$k_{|2\alpha+1|}$\normalsize}\end{minipage}\hspace{130pt} \Biggr\rangle_{3}/\Delta(n,0)\\
&\underset{\tiny(\mbox{Propositon $\ref{pro:Omega}$})\normalsize}{=}(q^{\frac{n^{2}+3n}{3}})^{-(2\alpha+2\beta-2\gamma+2)}\sum_{0\leq k_{|2\alpha+1|}\leq k_{|2\alpha|}\leq\cdots\leq k_{1}\leq n}\quad\sum_{0\leq l_{|2\beta+1|}\leq l_{|2\beta|}\leq\cdots\leq l_{1}\leq n}\quad\sum_{0\leq m_{\gamma}\leq m_{\gamma-1}\leq\cdots \leq m_{1}\leq n}\quad\sum^{\min\{k_{|2\alpha+1|}+l_{|2\beta+1|},n\}}_{s=\max\{k_{|2\alpha+1|},l_{|2\beta+1|}\}}\quad\sum_{a=s}^{n}\\
&\qquad\qquad\times\chi_{sign(2\alpha+1)}(n,k_{1},k_{2},...,k_{|2\alpha+1|})\chi_{sign(2\beta+1)}(n,l_{1},l_{2},...,l_{|2\beta+1|})\phi(n,m_{1},m_{2},...,m_{\gamma})_{q^{\epsilon_{\gamma}}}\Biggl\langle\begin{minipage}{1\unitlength}\includegraphics[scale=0.1]{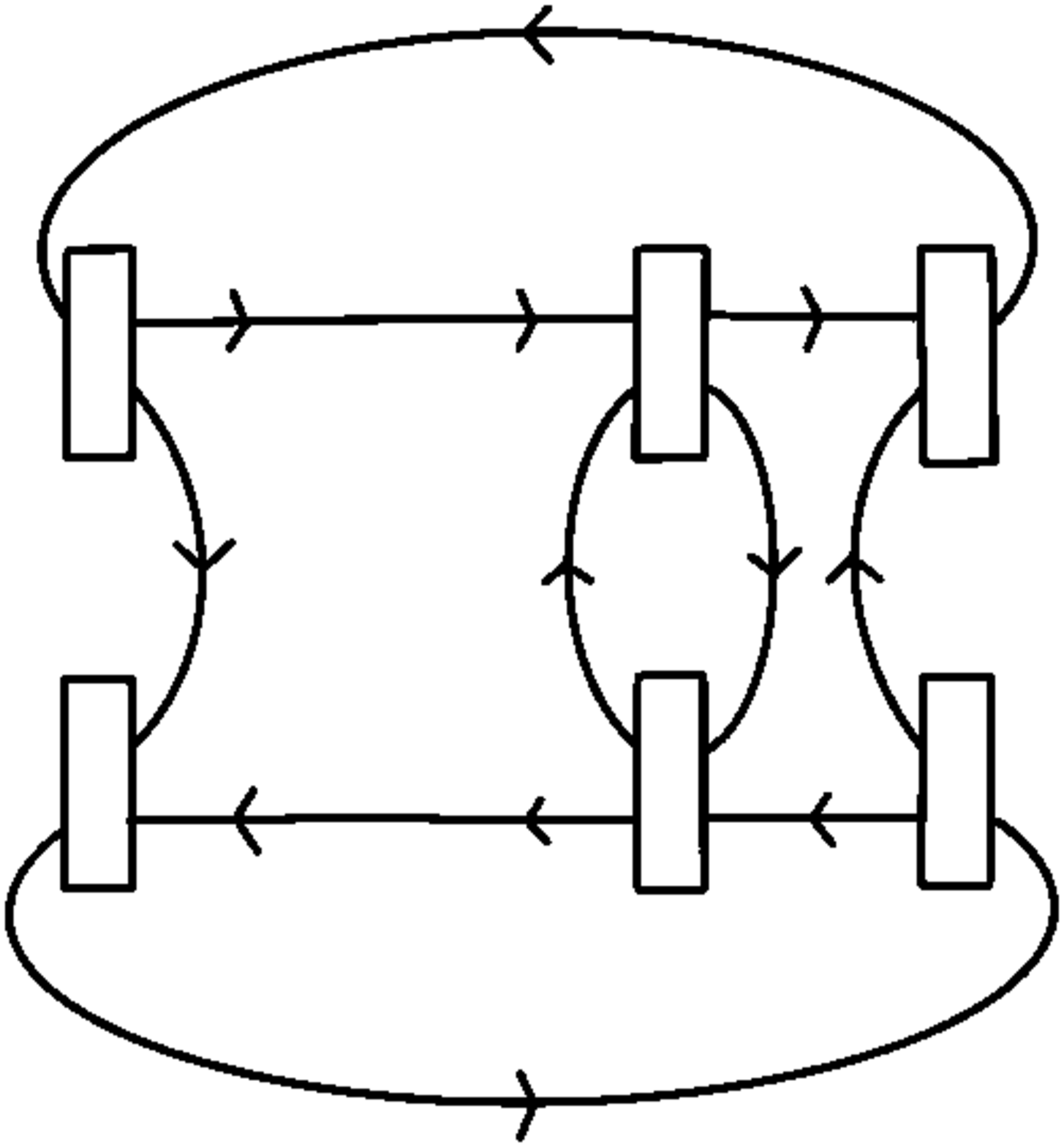}\put(-65,65){\scriptsize$n$\normalsize}\put(-65,35){\scriptsize$a$\normalsize}\put(-15,35){\scriptsize$m_{|\gamma|}$\normalsize}\put(-45,55){\scriptsize$n-a$\normalsize}\put(-25,15){\tiny$n-m_{|\gamma|}$\normalsize}\end{minipage}\hspace{70pt} \Biggr\rangle_{3}/\Delta(n,0)\\
&\underset{\tiny(\mbox{(\ref{al:double1})})\normalsize}{=}(q^{\frac{n^{2}+3n}{3}})^{-(2\alpha+2\beta-2\gamma+2)}\sum_{0\leq k_{|2\alpha+1|}\leq k_{|2\alpha|}\leq\cdots\leq k_{1}\leq n}\quad\sum_{0\leq l_{|2\beta+1|}\leq l_{|2\beta|}\leq\cdots\leq l_{1}\leq n}\quad\sum_{0\leq m_{\gamma}\leq m_{\gamma-1}\leq\cdots \leq m_{1}\leq n}\quad\sum^{\min\{k_{|2\alpha+1|}+l_{|2\beta+1|},n\}}_{s=\max\{k_{|2\alpha+1|},l_{|2\beta+1|}\}}\quad\sum_{a=s}^{n}\\
&\qquad\times\chi_{sign(2\alpha+1)}(n,k_{1},k_{2},...,k_{|2\alpha+1|})\chi_{sign(2\beta+1)}(n,l_{1},l_{2},...,l_{|2\beta+1|})\phi(n,m_{1},m_{2},...,m_{\gamma})_{q^{\epsilon_{\gamma}}}\\
&\qquad\times\Omega(n,s,k_{|2\alpha+1},l_{|2\beta+1|})\Biggl\langle\begin{minipage}{1\unitlength}\includegraphics[scale=0.1]{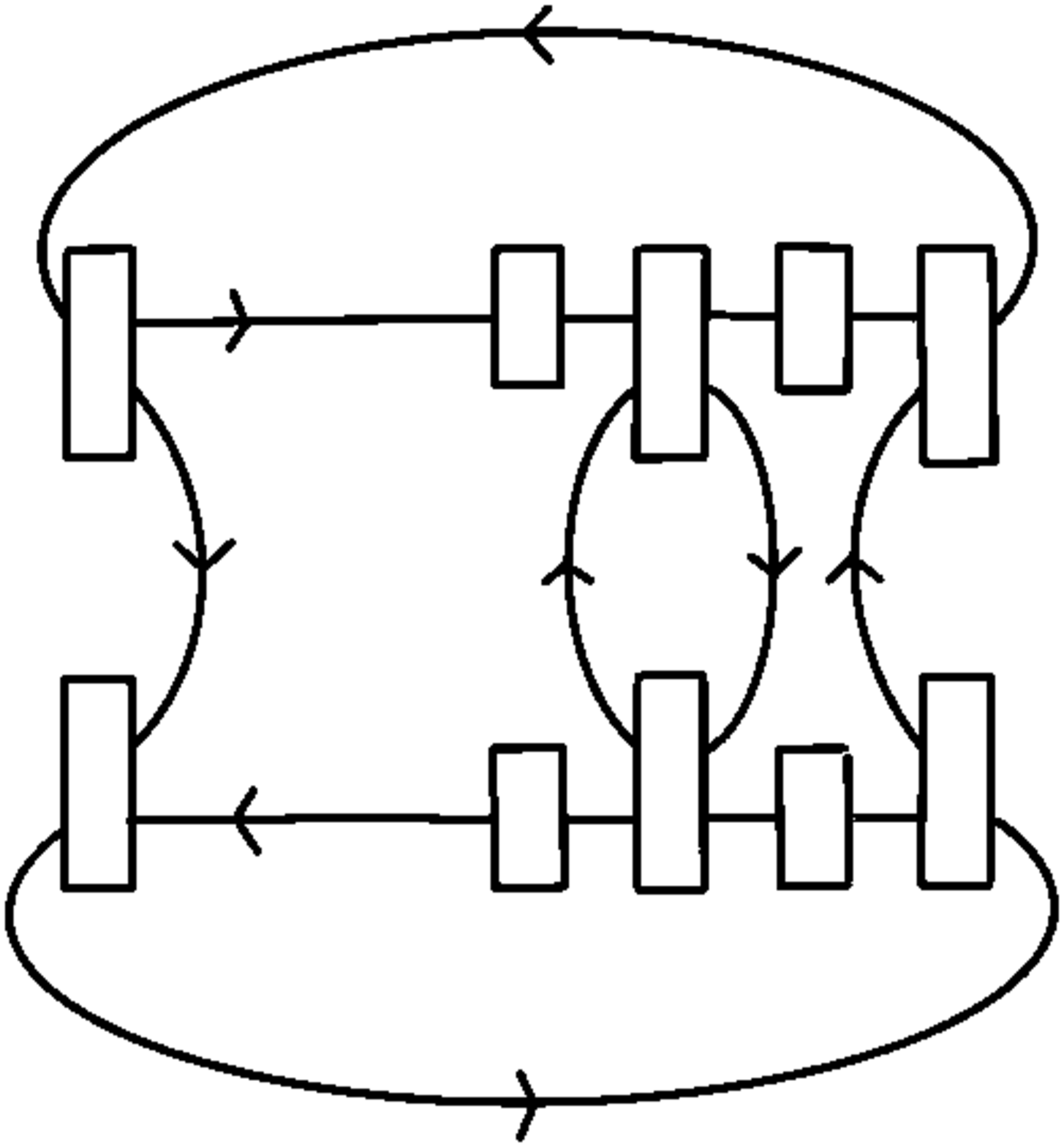}\put(-65,65){\scriptsize$n$\normalsize}\put(-65,35){\scriptsize$a$\normalsize}\put(-15,35){\scriptsize$m_{|\gamma|}$\normalsize}\put(-45,60){\scriptsize$n-a$\normalsize}\put(-25,15){\tiny$n-m_{|\gamma|}$\normalsize}\end{minipage}\hspace{70pt} \Biggr\rangle_{3}/\Delta(n,0)\\
&\underset{\tiny(\mbox{Theorem $\ref{Thm:yuasa6}$})\normalsize}{=}(q^{\frac{n^{2}+3n}{3}})^{-(2\alpha+2\beta-2\gamma+2)}\sum_{0\leq k_{|2\alpha+1|}\leq k_{|2\alpha|}\leq\cdots\leq k_{1}\leq n}\quad\sum_{0\leq l_{|2\beta+1|}\leq l_{|2\beta|}\leq\cdots\leq l_{1}\leq n}\quad\sum_{0\leq m_{\gamma}\leq m_{\gamma-1}\leq\cdots \leq m_{1}\leq n}\quad\sum^{\min\{k_{|2\alpha+1|}+l_{|2\beta+1|},n\}}_{s=\max\{k_{|2\alpha+1|},l_{|2\beta+1|}\}}\\
&\qquad\qquad\sum_{a=s}^{n}\quad\sum^{\min\{s+m_{\gamma},n\}}_{t=\max\{s,m_{\gamma}\}}\chi_{sign(2\alpha+1)}(n,k_{1},k_{2},...,k_{|2\alpha+1|})\chi_{sign(2\beta+1)}(n,l_{1},l_{2},...,l_{|2\beta+1|})\phi(n,m_{1},m_{2},...,m_{\gamma})_{q^{\epsilon_{\gamma}}}\\
&\qquad\qquad\times\Omega(n,s,k_{|2\alpha+1|},l_{|2\beta+1|})\psi(n,a,m_{\gamma},t)\Biggl\langle\begin{minipage}{1\unitlength}\includegraphics[scale=0.1]{pic/P3.eps}\put(-65,65){\scriptsize$n$\normalsize}\put(-65,35){\scriptsize$t$\normalsize}\put(-15,35){\scriptsize$t$\normalsize}\put(-45,57){\scriptsize$n-t$\normalsize}\put(-45,12){\scriptsize$n-t$\normalsize}\end{minipage}\hspace{70pt} \Biggr\rangle_{3}/\Delta(n,0)\\
&\underset{\tiny(\mbox{$(\ref{al:double7})$})\normalsize}{=}\sum_{0\leq k_{2|\alpha|}\leq k_{2|\alpha| -1}\leq\cdots\leq k_{1}\leq n}\quad\sum_{0\leq l_{2|\beta|}\leq l_{2|\beta| -1}\leq\cdots\leq l_{1}\leq n}\quad\sum_{0\leq m_{\gamma}\leq m_{\gamma -1}\leq\cdots\leq m_{1}\leq n} \quad\sum_{s=\max\{k_{2|\alpha|}, l_{2|\beta|}\}}^{\min\{k_{2|\alpha|}+l_{2|\beta|},n\}}\quad\sum_{a=0}^{n-s}\quad\sum_{t=\max\{a, m_{|\gamma|}\}}^{\min\{a+t+m_{|\gamma|},n\}}(q^{\frac{n^{2}+3n}{3}})^{2\alpha+2\beta-2\gamma}\\
&\qquad\qquad\times\chi_{sign(2\alpha)}(n,k_{1},k_{2},...,k_{|\alpha|})\chi_{sign(2\beta)}(n,l_{1},l_{2},...,l_{2|\beta|})\phi(n,m_{1},m_{2},...,m_{|\gamma|})_{q^{\epsilon_{\gamma}}}\Omega(n,s,k_{2|\alpha|},l_{2|\beta|})\psi(n,t,a,m_{|\gamma|})\\
&\qquad\qquad\times q^{-(n-t)}\frac{(1-q^{n+1})(1-q^{n+2})}{(1-q^{t+1})(1-q^{t+2})}
\end{align*}
We can prove $(\ref{ali:Jones1})$, $(\ref{ali:Jones3})$ and $(\ref{ali:Jones4})$ in a similar way to $(\ref{ali:Jones2})$ .
\end{proof}

\begin{Remark}
For n=$1,2,...,10$, Theorem $\ref{Pro:halfTwist1}$ is equal to Theorem $5.7$ in  \cite{Yua17} for $4_{1}, 6_{2}$ and $8_{2}$  by using Mathematica. For example, 
\begin{align*}
J_{(3,0)}^{\mathfrak{sl}_{3}}(4_{1};q)=&q^{15}-q^{13}-q^{12}-q^{11}+2 q^{10}+q^9-2 q^7-q^6+4 q^5+2 q^4-2 q^3-4 q^2+5-4q^{-2}-2q^{-3}\\
&+2q^{-4}+4q^{-5}-q^{-6}-2q^{-7}+q^{-9}+2q^{-10}-q^{-11}-q^{-12}-q^{-13}+q^{-15}.
\end{align*}
\end{Remark}

\section{the tails of the one-row colored $\mathfrak{sl}_{3}$ Jones polynomials for  $P(2\alpha +1, 2\beta +1, 2\gamma)$ pretzel knots}
In this section, we use Theorem $\ref{Thm:main2}$ to show the existence of the tails of the one-row colored $\mathfrak{sl}_{3}$ Jones polynomials for alternating pretzel knots $P(2\alpha +1, 2\beta +1, 2\gamma)$.

\begin{Definition}
For any formal Laurent series $f(q)$, we define $
\hat{f}(q)$ by
\begin{align*}
\hat{f}(q)=\pm q^{-\mbox{mindeg}(f(q))}f(q)=\sum_{i=0}^{\infty}a_{i}q^{i}\in \mathbb{Z}[[q]].
\end{align*}
In the above normalization, we determine $a_{0}>0$.
\end{Definition}

\begin{Corollary}
For positive ingegers $\alpha$, $\beta$ and $\gamma$,
\begin{align*}
 \mbox{mindeg}(J_{(n,0)}^{\mathfrak{sl}_{3}}(P(\downarrow2\alpha+1\downarrow,\uparrow2\beta+1\uparrow,\downarrow2\gamma\uparrow);q))=-(\alpha+\beta+1)n^{2}-(3\alpha + 3\beta +2)n
 \end{align*}
\end{Corollary}
\begin{proof}
\begin{align*}
 &\mbox{mindeg}(J_{(n,0)}^{\mathfrak{sl}_{3}}(P(\downarrow2\alpha+1\downarrow,\uparrow2\beta+1\uparrow,\downarrow2\gamma\uparrow);q))\\
=&\mbox{mindeg}(\sum_{0\leq k_{2\alpha+1}\leq k_{2\alpha }\leq\cdots\leq k_{1}\leq n}\quad\sum_{0\leq l_{2\beta+1}\leq l_{2\beta }\leq\cdots\leq l_{1}\leq n}\quad\sum_{0\leq m_{\gamma}\leq m_{\gamma-1}\leq\cdots \leq m_{1}\leq n}\quad\sum^{\min\{k_{2\alpha+1}+l_{2\beta+1},n\}}_{s=\max\{k_{2\alpha+1},l_{2\beta+1}\}}\quad\sum_{a=s}^{n}\quad\sum^{\min\{a+m_{\gamma},n\}}_{t=\max\{a,m_{\gamma}\}}\\
&(q^{\frac{n^{2}+3n}{3}})^{-(2\alpha+2\beta-2\gamma+2)}\frac{q^{-\frac{1}{6}(n^{2}+3n)(2\alpha +1)}q^{\frac{1}{2}(n-k_{2\alpha +1})}q^{\sum_{i=1}^{2\alpha +1}\frac{1}{2}(k_{i}^{2}+k_{i})}(q)_{n}}{({\displaystyle \prod_{i=1}^{2\alpha +1}}(q)_{k_{i-1}-k_{i}})(q)_{k_{2\alpha +1}}}\times\frac{q^{-\frac{1}{6}(n^{2}+3n)(2\beta +1)}q^{\frac{1}{2}(n-l_{2\beta +1})}q^{\sum_{i=1}^{2\beta +1}\frac{1}{2}(l_{i}^{2}+l_{i})}(q)_{n}}{({\displaystyle \prod_{i=1}^{2\beta +1}}(q)_{l_{i-1}-l_{i}})(q)_{k_{2\beta +1}}}\\
&\times\frac{(q)^{-\frac{2\gamma}{3}(n^{2}+3n)}(q)^{n-m_{\gamma}}(q)^{\sum_{i=1}^{\gamma}(m_{i}^{2}+2m_{i})}(q)^{2}_{n}}{({\displaystyle \prod_{i=1}^{\gamma}}(q)_{m_{i-1}-m_{i}})(q)^{2}_{m_{\gamma}}}\\
&\times\frac{q^{-\frac{k_{2\alpha +1}+l_{2\beta +1}}{2}+s}q^{(s+1)(s-k_{2\alpha +1}-l_{2\beta +1})+k_{2\alpha +1}l_{2\beta +1}}(1-q^{n+1-k_{2\alpha +1}})(1-q^{n+1-l_{2\beta +1}})(q)_{k_{2\alpha +1}}(q)_{l_{2\beta +1}}(q)^{2}_{n-k_{2\alpha +1}}(q)^{2}_{n-l_{2\beta +1}}(q)_{2n-s+2}}{(1-q^{n+1-s})^{2}(q)^{2}_{n}(q)^{2}_{n-s}(q)_{s-k_{2\alpha +1}}(q)_{s-l_{2\beta +1}}(q)_{2n-k_{2\alpha +1}-l_{2\beta +1}+2}(q)_{-s+k_{2\alpha +1}+l_{2\beta +1}}}\\
&\times\frac{q^{(t+1)(t-a-m_{\gamma})+am_{\gamma}}(q)_{a}(q)_{m_{\gamma}}(q)^{2}_{n-a}(q)^{2}_{n-m_{\gamma}}(q)_{2n-t+2}}{(q)^{2}_{n}(q)^{2}_{n-t}(q)_{t-a}(q)_{t-m_{\gamma}}(q)_{2n-a-m_{\gamma}+2}(q)_{-t+a+m_{\gamma}}}q^{-(n-t)}\frac{(1-q^{n+1})(1-q^{n+2})}{(1-q^{t+1})(1-q^{t+2})})\\
=&\mbox{mindeg}(q^{-(\alpha+\beta+1)n^{2}-(3\alpha + 3\beta +2)n}\sum_{0\leq k_{2\alpha+1}\leq k_{2\alpha }\leq\cdots\leq k_{1}\leq n}\quad\sum_{0\leq l_{2\beta+1}\leq l_{2\beta }\leq\cdots\leq l_{1}\leq n}\quad\sum_{0\leq m_{\gamma}\leq m_{\gamma-1}\leq\cdots \leq m_{1}\leq n}\quad\sum^{\min\{k_{2\alpha+1}+l_{2\beta+1},n\}}_{s=\max\{k_{2\alpha+1},l_{2\beta+1}\}}\\
&\quad\sum_{a=s}^{n}\quad\sum^{\min\{a+m_{\gamma},n\}}_{t=\max\{a,m_{\gamma}\}}q^{\frac{1}{2}(-k_{2\alpha +1})}q^{\sum_{i=1}^{2\alpha +1}\frac{1}{2}(k_{i}^{2}+k_{i})}q^{\frac{1}{2}(-l_{2\beta +1})}q^{\sum_{i=1}^{2\beta +1}\frac{1}{2}(l_{i}^{2}+l_{i})}q^{-m_{\gamma}}q^{\sum_{i=1}^{\gamma}(m_{i}^{2}+2m_{i})} q^{-\frac{k_{2\alpha +1}+l_{2\beta +1}}{2}+s}\\
&\times q^{(s+1)(s-k_{2\alpha +1}-l_{2\beta +1})+k_{2\alpha +1}l_{2\beta +1}}q^{(t+1)(t-a-m_{\gamma})+am_{\gamma}}q^{t})\\
&=\mbox{mindeg}(q^{-(\alpha+\beta+1)n^{2}-(3\alpha + 3\beta +2)n}\sum_{0\leq k_{2\alpha+1}\leq k_{2\alpha }\leq\cdots\leq k_{1}\leq n}\quad\sum_{0\leq l_{2\beta+1}\leq l_{2\beta }\leq\cdots\leq l_{1}\leq n}\quad\sum_{0\leq m_{\gamma}\leq m_{\gamma-1}\leq\cdots \leq m_{1}\leq n}\quad\sum^{\min\{k_{2\alpha+1}+l_{2\beta+1},n\}}_{s=\max\{k_{2\alpha+1},l_{2\beta+1}\}}\\
&\quad\sum_{a=s}^{n}\quad\sum^{\min\{a+m_{\gamma},n\}}_{t=\max\{a,m_{\gamma}\}}q^{\sum_{i=1}^{2\alpha }\frac{1}{2}(k_{i}^{2}+k_{i})}q^{\sum_{i=1}^{2\beta }\frac{1}{2}(l_{i-1}^{2}+l_{i})}q^{m_{\gamma}^{2}+m_{\gamma}}q^{\sum_{i=1}^{\gamma-1}(m_{i}^{2}+2m_{i})} q^{\frac{k^{2}_{2\alpha +1}+l^{2}_{2\beta +1}-(k_{2\alpha +1}+l_{2\beta +1})}{2}}\\
&\times q^{s+(s+1)(s-k_{2\alpha +1}-l_{2\beta +1})+k_{2\alpha +1}l_{2\beta +1}}q^{t+(t+1)(t-a-m_{\gamma})+am_{\gamma}})\\
=&-(\alpha+\beta+1)n^{2}-(3\alpha + 3\beta +2)n
\end{align*}
We use $\frac{1}{1-q}=1+q+q^{2}+\cdots$ in the second equality. If $l_{2\beta +1}\le k_{2\alpha +1} \le s$, then
\begin{align}
\label{ali:a}
\begin{aligned}
&s+(s+1)(s-(k_{2\alpha +1}+l_{2\beta +1}))+k_{2\alpha +1}l_{2\beta +1}\\
=& s^{2}+s(2-(k_{2\alpha +1}+l_{2\beta +1})-k_{2\alpha +1}-l_{2\beta +1}+k_{2\alpha +1}l_{2\beta +1}\\
=&(s-k_{2\alpha +1}+1)(s-l_{2\beta +1}+1)-1\\
\ge& (k_{2\alpha +1}-l_{2\beta +1}+1)-1\ge0
\end{aligned}
\end{align}
The same is true for $k_{2\alpha +1}\le l_{2\beta +1}  \le s$. In the above last inequality, we have equality if and only if  $l_{2\beta +1}= k_{2\alpha +1} = s$ holds.  By the same method as the above, if $a\le m_{\gamma} \le t$ and $m_{\gamma} \le a\le  t$, then we can obtain 
\begin{align}
\label{ali:b}
t+(t+1)(t-(a+m_{\gamma}))+am_{\gamma}\ge 0.
\end{align}
The equality sign is valid for $a= m_{\gamma} = t$ and $m_{\gamma} = a= t$. By $(\ref{ali:a})$, $(\ref{ali:b})$, $k^{2}_{2\alpha +1}-k_{2\alpha +1}\ge 0$ and $^{2}_{2\beta +1}-l_{2\beta +1}\ge 0$, the fourth equality holds.
\end{proof}

\begin{Theorem}
\label{Tm:tails}
Let  $\alpha$, $\beta$, and $\gamma$ be positive integers. For oriented pretzel knots $P(\downarrow2\alpha+1\downarrow,\uparrow2\beta+1\uparrow,\downarrow2\gamma\uparrow)$, there exists $\mathcal{T}^{\mathfrak{sl}_{3}}(P(\downarrow2\alpha+1\downarrow,\uparrow2\beta+1\uparrow,\downarrow2\gamma\uparrow);q)$ in $\mathbb{Z}[[q]]$ such that 
\begin{align*}
\mathcal{T}^{\mathfrak{sl}_{3}}(P(\downarrow2\alpha+1\downarrow,\uparrow2\beta+1\uparrow,\downarrow2\gamma\uparrow);q)-\hat{J_{(n,0)}^{\mathfrak{sl}_{3}}}(P(\downarrow2\alpha+1\downarrow,\uparrow2\beta+1\uparrow,\downarrow2\gamma\uparrow);q))\in q^{n+1}\mathbb{Z}[[q]].
\end{align*}
\end{Theorem}

\begin{Remark}
For a simple Lie algebra $\mathfrak{g}$, Le\cite{Le00} proved integrality theorem for a quantum $\mathfrak{g}$ invariant. This says $\hat{J}_{(n,0)}^{\mathfrak{sl}_{3}}(L;q)$ belongs to $\mathbb{Z}[[q]]$.
\end{Remark}

\begin{Definition}
For two Laurent series $\hat{f}(q)$, $\hat{g}(q)\in \mathbb{Z}[[q]]$, we define
\begin{align*}
\hat{f(q)}\equiv_{n}\hat{g(q)}
\end{align*}
by $\hat{f(q)}=\hat{g(q)}$ in $ \mathbb{Z}[[q]]/ q^{n+1}\mathbb{Z}[[q]]$ for all $n$.
\end{Definition}

\begin{proof}[Proof of Theorem$\ref{Tm:tails}$]
There exists $\mathcal{T}^{\mathfrak{sl}_{3}}(P(\downarrow2\alpha+1\downarrow,\uparrow2\beta+1\uparrow,\downarrow2\gamma\uparrow);q)$ if and only if 
\begin{align}
\hat{J}_{(n,0)}^{\mathfrak{sl}_{3}}(P(\downarrow2\alpha+1\downarrow,\uparrow2\beta+1\uparrow,\downarrow2\gamma\uparrow);q))\equiv_{n+1}\hat{J_{(n+1,0)}^{\mathfrak{sl}_{3}}}(P(\downarrow2\alpha+1\downarrow,\uparrow2\beta+1\uparrow,\downarrow2\gamma\uparrow);q))
\end{align}
for all $n$. First, we can easily obtain the following.
\begin{align}
\label{ali:minn}
\begin{aligned}
&\hat{J}_{(n,0)}^{\mathfrak{sl}_{3}}(P(\downarrow2\alpha+1\downarrow,\uparrow2\beta+1\uparrow,\downarrow2\gamma\uparrow);q))\\
\equiv_{n+1}&\sum_{0\leq k_{2\alpha}+1\leq k_{2\alpha }\leq\cdots\leq k_{1}\leq n}\quad\sum_{0\leq l_{2\beta}+1\leq l_{2\beta }\leq\cdots\leq l_{1}\leq n}\quad\sum_{0\leq m_{\gamma}\leq m_{\gamma-1}\leq\cdots \leq m_{1}\leq n}\quad\sum^{\min\{k_{2\alpha+1}+l_{2\beta+1},n\}}_{s=\max\{k_{2\alpha+1},l_{2\beta+1}\}}\quad\sum_{a=s}^{n}\quad\sum^{\min\{a+m_{\gamma},n\}}_{t=\max\{a,m_{\gamma}\}}\\
&\frac{(-1)^{\sum_{i=1}^{2\alpha +1}k_{i}}q^{\sum_{i=1}^{2\alpha}\frac{1}{2}(k_{i}^{2}+k_{i})}}{({\displaystyle \prod_{i=2}^{2\alpha +1}}(q)_{k_{i-1}-k_{i}})(q)_{k_{2\alpha +1}}}\times\frac{(-1)^{\sum_{i=1}^{2\beta +1}k_{i}}q^{\sum_{i=1}^{2\beta }\frac{1}{2}(l_{i}^{2}+l_{i})}}{({\displaystyle \prod_{i=2}^{2\beta +1}}(q)_{l_{i-1}-l_{i}})(q)_{l_{2\beta +1}}}\times\frac{(q)^{m_{\gamma}^{2}+m_{\gamma}}(q)^{\sum_{i=1}^{\gamma-1}(m_{i}^{2}+2m_{i})}(q)_{n}}{({\displaystyle \prod_{i=2}^{\gamma}}(q)_{m_{i-1}-m_{i}})(q)^{2}_{m_{\gamma}}}\\
&\times\frac{q^{\frac{(k^{2}_{2\alpha +1}+l^{2}_{2\beta +1})-(k_{2\alpha +1}+l_{2\beta +1})}{2}}q^{s+(s+1)(s-k_{2\alpha +1}-l_{2\beta +1})+k_{2\alpha +1}l_{2\beta +1}}(q)_{k_{2\alpha +1}}(q)_{l_{2\beta +1}}}{(q)_{s-k_{2\alpha +1}}(q)_{s-l_{2\beta +1}}(q)_{-s+k_{2\alpha +1}+l_{2\beta +1}}}\\
&\times\frac{q^{t+(t+1)(t-a-m_{\gamma})+am_{\gamma}}(q)_{a}(q)_{m_{\gamma}}}{(q)_{t-a}(q)_{t-m_{\gamma}}(q)_{-t+a+m_{\gamma}}}\frac{1}{(1-q^{t+1})(1-q^{t+2})}
\end{aligned}
\end{align} 
Here,
\begin{align}
\begin{aligned}
\frac{q^{s}}{1-q^{n+2-s}}=&q^{s}(1+q^{n+2-s}+q^{2(n+2-s)}+\cdots)\\
									  =&q^{s}+q^{n+2}+(\mbox{higher order terms})\\
									  \equiv_{n+1}&0
\end{aligned}
\end{align}
We use a deformation similar to the above equation repeatedly. We next consider $\hat{J_{(n+1,0)}^{\mathfrak{sl}_{3}}}(P(\downarrow2\alpha+1\downarrow,\uparrow2\beta+1\uparrow,\downarrow2\gamma\uparrow);q))$.
\begin{align*}
&\hat{J}_{(n+1,0)}^{\mathfrak{sl}_{3}}(P(\downarrow2\alpha+1\downarrow,\uparrow2\beta+1\uparrow,\downarrow2\gamma\uparrow);q))\\
=&\sum_{0\leq k_{2\alpha+1}\leq k_{2\alpha }\leq\cdots\leq k_{1}\leq n+1}\quad\sum_{0\leq l_{2\beta+1}\leq l_{2\beta }\leq\cdots\leq l_{1}\leq n+1}\quad\sum_{0\leq m_{\gamma}\leq m_{\gamma-1}\leq\cdots \leq m_{1}\leq n+1}\quad\sum^{\min\{k_{2\alpha+1}+l_{2\beta+1},n+1\}}_{s=\max\{k_{2\alpha+1},l_{2\beta+1}\}}\quad\sum_{a=s}^{n+1}\quad\sum^{\min\{a+m_{\gamma},n+1\}}_{t=\max\{a,m_{\gamma}\}}\\
&\times\frac{(-1)^{\sum_{i=1}^{2\alpha +1}k_{i}}q^{\frac{1}{2}k^{2}_{2\alpha +1}}q^{\sum_{i=1}^{2\alpha}\frac{1}{2}(k_{i}^{2}+k_{i})}(q)_{n+1}}{({\displaystyle \prod_{i=1}^{2\alpha +1}}(q)_{k_{i-1}-k_{i}})(q)_{n+1-k_{1}}(q)_{k_{2\alpha +1}}}\times\frac{(-1)^{\sum_{i=1}^{2\beta +1}l_{i}}q^{\frac{1}{2}l^{2}_{2\beta +1}}q^{\sum_{i=1}^{2\beta }\frac{1}{2}(l_{i}^{2}+l_{i})}(q)_{n+1}}{({\displaystyle \prod_{i=1}^{2\beta +1}}(q)_{l_{i-1}-l_{i}})(q)_{n+1-l_{1}}(q)_{k_{2\beta +1}}}\times\frac{(q)^{m_{\gamma}^{2}+m_{\gamma}}(q)^{\sum_{i=1}^{\gamma-1}(m_{i}^{2}+2m_{i})}(q)^{2}_{n+1}}{({\displaystyle \prod_{i=1}^{\gamma}}(q)_{m_{i-1}-m_{i}})(q)_{n+1-m_{1}}(q)^{2}_{m_{\gamma}}}\\
&\times\frac{q^{-\frac{k_{2\alpha +1}+l_{2\beta +1}}{2}+s}q^{(s+1)(s-k_{2\alpha +1}-l_{2\beta +1})+k_{2\alpha +1}l_{2\beta +1}}(1-q^{n+2-k_{2\alpha +1}})(1-q^{n+2-l_{2\beta +1}})(q)_{k_{2\alpha +1}}(q)_{l_{2\beta +1}}(q)^{2}_{n+1-k_{2\alpha +1}}(q)^{2}_{n+1-l_{2\beta +1}}}{(1-q^{n+2-s})^{2}(q)^{2}_{n+1}(q)^{2}_{n+1-s}(q)_{s-k_{2\alpha +1}}(q)_{s-l_{2\beta +1}}(q)_{2(n+1)-k_{2\alpha +1}-l_{2\beta +1}+2}(q)_{-s+k_{2\alpha +1}+l_{2\beta +1}}}\\
&\times(q)_{2(n+1)-s+2}\frac{q^{(t+1)(t-a-m_{\gamma})+am_{\gamma}}(q)_{a}(q)_{m_{\gamma}}(q)^{2}_{n+1-a}(q)^{2}_{n+1-m_{\gamma}}(q)_{2(n+1)-t+2}}{(q)^{2}_{n+1}(q)^{2}_{n+1-t}(q)_{t-a}(q)_{t-m_{\gamma}}(q)_{2(n+1)-a-m_{\gamma}+2}(q)_{-t+a+m_{\gamma}}}q^{-(n+1-t)}\frac{(1-q^{n+2})(1-q^{n+3})}{(1-q^{t+1})(1-q^{t+2})}\\
\equiv_{n+1}&\sum_{0\leq k_{2\alpha+1}\leq k_{2\alpha }\leq\cdots\leq k_{1}\leq n+1}\quad\sum_{0\leq l_{2\beta+1}\leq l_{2\beta }\leq\cdots\leq l_{1}\leq n+1}\quad\sum_{0\leq m_{\gamma}\leq m_{\gamma-1}\leq\cdots \leq m_{1}\leq n+1}\quad\sum^{\min\{k_{2\alpha+1}+l_{2\beta+1},n+1\}}_{s=\max\{k_{2\alpha+1},l_{2\beta+1}\}}\quad\sum_{a=s}^{n+1}\quad\sum^{\min\{a+m_{\gamma},n+1\}}_{t=\max\{a,m_{\gamma}\}}\\
&\frac{(-1)^{\sum_{i=1}^{2\alpha +1}k_{i}}q^{\frac{1}{2}k^{2}_{2\alpha +1}}q^{\sum_{i=1}^{2\alpha}\frac{1}{2}(k_{i}^{2}+k_{i})}}{({\displaystyle \prod_{i=2}^{2\alpha +1}}(q)_{k_{i-1}-k_{i}})(q)_{k_{2\alpha +1}}}\times\frac{(-1)^{\sum_{i=1}^{2\beta +1}l_{i}}q^{\frac{1}{2}l^{2}_{2\beta +1}}q^{\sum_{i=1}^{2\beta }\frac{1}{2}(l_{i}^{2}+l_{i})}}{({\displaystyle \prod_{i=2}^{2\beta +1}}(q)_{l_{i-1}-l_{i}})(q)_{l_{2\beta +1}}}\times\frac{(q)^{m_{\gamma}^{2}+m_{\gamma}}(q)^{\sum_{i=1}^{\gamma-1}(m_{i}^{2}+2m_{i})}(q)_{n}}{({\displaystyle \prod_{i=2}^{\gamma}}(q)_{m_{i-1}-m_{i}})(q)^{2}_{m_{\gamma}}}\\
&\times\frac{q^{-\frac{k_{2\alpha +1}+l_{2\beta +1}}{2}+s}q^{(s+1)(s-k_{2\alpha +1}-l_{2\beta +1})+k_{2\alpha +1}l_{2\beta +1}}(q)_{k_{2\alpha +1}}(q)_{l_{2\beta +1}}}{(q)_{s-k_{2\alpha +1}}(q)_{s-l_{2\beta +1}}(q)_{-s+k_{2\alpha +1}+l_{2\beta +1}}}\times\frac{q^{t+(t+1)(t-a-m_{\gamma})+am_{\gamma}}(q)_{a}(q)_{m_{\gamma}}}{(q)_{t-a}(q)_{t-m_{\gamma}}(q)_{-t+a+m_{\gamma}}}\frac{1}{(1-q^{t+1})(1-q^{t+2})}\\
=&\sum_{0\leq k_{2\alpha+1}\leq k_{2\alpha }\leq\cdots\leq k_{1}\leq n}\quad\sum_{0\leq l_{2\beta}+1\leq l_{2\beta }\leq\cdots\leq l_{1}\leq n+1}\quad\sum_{0\leq m_{\gamma}\leq m_{\gamma-1}\leq\cdots \leq m_{1}\leq n+1}\quad\sum^{\min\{k_{2\alpha+1}+l_{2\beta+1},n+1\}}_{s=\max\{k_{2\alpha+1},l_{2\beta+1}\}}\quad\sum_{a=s}^{n+1}\quad\sum^{\min\{a+m_{\gamma},n+1\}}_{t=\max\{a,m_{\gamma}\}}\\
&\frac{(-1)^{\sum_{i=1}^{2\alpha +1}k_{i}}q^{\frac{1}{2}k^{2}_{2\alpha +1}}q^{\sum_{i=1}^{2\alpha}\frac{1}{2}(k_{i}^{2}+k_{i})}}{({\displaystyle \prod_{i=2}^{2\alpha +1}}(q)_{k_{i-1}-k_{i}})(q)_{k_{2\alpha +1}}}\times\frac{(-1)^{\sum_{i=1}^{2\beta +1}l_{i}}q^{\frac{1}{2}l^{2}_{2\beta +1}}q^{\sum_{i=1}^{2\beta }\frac{1}{2}(l_{i}^{2}+l_{i})}}{({\displaystyle \prod_{i=2}^{2\beta +1}}(q)_{l_{i-1}-l_{i}})(q)_{l_{2\beta +1}}}\times\frac{(q)^{m_{\gamma}^{2}+m_{\gamma}}(q)^{\sum_{i=1}^{\gamma^1}(m_{i}^{2}+2m_{i})}(q)_{n}}{({\displaystyle \prod_{i=2}^{\gamma}}(q)_{m_{i-1}-m_{i}})(q)^{2}_{m_{\gamma}}}\\
&\times\frac{q^{-\frac{k_{2\alpha +1}+l_{2\beta +1}}{2}+s}q^{(s+1)(s-k_{2\alpha +1}-l_{2\beta +1})+k_{2\alpha +1}l_{2\beta +1}}(q)_{k_{2\alpha +1}}(q)_{l_{2\beta +1}}}{(q)_{s-k_{2\alpha +1}}(q)_{s-l_{2\beta +1}}(q)_{-s+k_{2\alpha +1}+l_{2\beta +1}}}\times\frac{q^{t+(t+1)(t-a-m_{\gamma})+am_{\gamma}}(q)_{a}(q)_{m_{\gamma}}}{(q)_{t-a}(q)_{t-m_{\gamma}}(q)_{-t+a+m_{\gamma}}}\frac{1}{(1-q^{t+1})(1-q^{t+2})}\\
&+q^{n+1+\frac{1}{2}(n^{2}+n)}\sum_{0\leq k_{2\alpha+1}\leq k_{2\alpha }\leq\cdots\le k_{2}\leq n+1}\quad\sum_{0\leq l_{2\beta+1}\leq l_{2\beta }\leq\cdots\leq l_{1}\leq n+1}\quad\sum_{0\leq m_{\gamma}\leq m_{\gamma-1}\leq\cdots \leq m_{1}\leq n+1}\quad\sum^{\min\{k_{2\alpha+1}+l_{2\beta+1},n+1\}}_{s=\max\{k_{2\alpha+1},l_{2\beta+1}\}}\quad\\
&\sum_{a=s}^{n+1}\quad\sum^{\min\{a+m_{\gamma},n+1\}}_{t=\max\{a,m_{\gamma}\}}\frac{(-1)^{\sum_{i=2}^{2\alpha +1}k_{i}}q^{\sum_{i=2}^{2\alpha}\frac{1}{2}(k_{i}^{2}+k_{i})}}{({\displaystyle \prod_{i=2}^{2\alpha +1}}(q)_{k_{i-1}-k_{i}}(q)_{k_{2\alpha +1}})}\times\frac{(-1)^{\sum_{i=1}^{2\beta +1}l_{i}}q^{\sum_{i=1}^{2\beta }\frac{1}{2}(l_{i}^{2}+l_{i})}}{({\displaystyle \prod_{i=2}^{2\beta +1}}(q)_{l_{i-1}-l_{i}})(q)_{l_{2\beta +1}})}\times\frac{(q)^{m_{\gamma}^{2}+m_{\gamma}}(q)^{\sum_{i=1}^{\gamma-1}(m_{i}^{2}+2m_{i})}(q)_{n}}{({\displaystyle \prod_{i=2}^{\gamma}}(q)_{m_{i-1}-m_{i}}(q)^{2}_{m_{\gamma}})}\\
&\times\frac{q^{\frac{(k^{2}_{2\alpha +1}+l^{2}_{2\beta +1})-(k_{2\alpha +1}+l_{2\beta +1})}{2}}q^{s+(s+1)(s-k_{2\alpha +1}-l_{2\beta +1})+k_{2\alpha +1}l_{2\beta +1}}(q)_{k_{2\alpha +1}}(q)_{l_{2\beta +1}}}{(q)_{s-k_{2\alpha +1}}(q)_{s-l_{2\beta +1}}(q)_{-s+k_{2\alpha +1}+l_{2\beta +1}}}\\
&\times\frac{q^{t+(t+1)(t-a-m_{\gamma})+am_{\gamma}}(q)_{a}(q)_{m_{\gamma}}}{(q)_{t-a}(q)_{t-m_{\gamma}}(q)_{-t+a+m_{\gamma}}}\frac{1}{(1-q^{t+1})(1-q^{t+2})}
\end{align*}

By $(\ref{ali:a})$, $(\ref{ali:b})$, $k^{2}_{2\alpha +1}-k_{2\alpha +1}\ge 0$ and $^{2}_{2\beta +1}-l_{2\beta +1}\ge 0$, 

\begin{align*}
&q^{n+1+\frac{1}{2}(n^{2}+n)}\sum_{0\leq k_{2\alpha+1}\leq k_{2\alpha }\leq\cdots\le k_{2}\leq n+1}\quad\sum_{0\leq l_{2\beta+1}\leq l_{2\beta }\leq\cdots\leq l_{1}\leq n+1}\quad\sum_{0\leq m_{\gamma}\leq m_{\gamma-1}\leq\cdots \leq m_{1}\leq n+1}\quad\sum^{\min\{k_{2\alpha+1}+l_{2\beta+1},n+1\}}_{s=\max\{k_{2\alpha+1},l_{2\beta+1}\}}\quad\\
&\sum_{a=s}^{n+1}\quad\sum^{\min\{a+m_{\gamma},n+1\}}_{t=\max\{a,m_{\gamma}\}}\frac{(-1)^{\sum_{i=1}^{2\alpha +1}k_{i}}q^{\sum_{i=1}^{2\alpha}\frac{1}{2}(k_{i}^{2}+k_{i})}}{({\displaystyle \prod_{i=2}^{2\alpha +1}}(q)_{k_{i-1}-k_{i}})(q)_{k_{2\alpha +1}}}\times\frac{(-1)^{\sum_{i=1}^{2\beta +1}l_{i}}q^{\sum_{i=1}^{2\beta}\frac{1}{2}(l_{i}^{2}+l_{i})}}{({\displaystyle \prod_{i=2}^{2\beta +1}}(q)_{l_{i-1}-l_{i}})(q)_{l_{2\beta +1}}}\times\frac{(q)^{m_{\gamma}^{2}+m_{\gamma}}(q)^{\sum_{i=1}^{\gamma-1}(m_{i}^{2}+2m_{i})}(q)_{n}}{({\displaystyle \prod_{i=2}^{\gamma}}(q)_{m_{i-1}-m_{i}})(q)^{2}_{m_{\gamma}}}\\
&\times\frac{q^{\frac{(k^{2}_{2\alpha +1}+l^{2}_{2\beta +1})-(k_{2\alpha +1}+l_{2\beta +1})}{2}}q^{s+(s+1)(s-k_{2\alpha +1}-l_{2\beta +1})+k_{2\alpha +1}l_{2\beta +1}}(q)_{k_{2\alpha +1}}(q)_{l_{2\beta +1}}}{(q)_{s-k_{2\alpha +1}}(q)_{s-l_{2\beta +1}}(q)_{2(n+1)-k_{2\alpha +1}-l_{2\beta +1}+2}(q)_{-s+k_{2\alpha +1}+l_{2\beta +1}}}\\
&\times\frac{q^{t+(t+1)(t-a-m_{\gamma})+am_{\gamma}}(q)_{a}(q)_{m_{\gamma}}}{(q)_{t-a}(q)_{t-m_{\gamma}}(q)_{-t+a+m_{\gamma}}}\frac{1}{(1-q^{t+1})(1-q^{t+2})}\\
=&q^{n+1+\frac{1}{2}(n^{2}+n)}+(\mbox{high order terms})\\
\equiv_{n+1}&0\\
\end{align*}
We consider $l_{1}=n+1$ and $m_{1}=n+1$ in the same way. Therefore, 
\begin{align}
\label{ali:min1}
\begin{aligned}
&\hat{J}_{(n+1,0)}^{\mathfrak{sl}_{3}}(P(\downarrow2\alpha+1\downarrow,\uparrow2\beta+1\uparrow,\downarrow2\gamma\uparrow);q))\\
\equiv_{n+1}&\sum_{0\leq k_{2\alpha+1}\leq k_{2\alpha }\leq\cdots\leq k_{1}\leq n}\quad\sum_{0\leq l_{2\beta+1}\leq l_{2\beta }\leq\cdots\leq l_{1}\leq n}\quad\sum_{0\leq m_{\gamma}\leq m_{\gamma-1}\leq\cdots \leq m_{1}\leq n}\quad\sum^{\min\{k_{2\alpha+1}+l_{2\beta+1},n+1\}}_{s=\max\{k_{2\alpha+1},l_{2\beta+1}\}}\quad\sum_{a=s}^{n+1}\quad\sum^{\min\{a+m_{\gamma},n+1\}}_{t=\max\{a,m_{\gamma}\}}\\
&\frac{(-1)^{\sum_{i=1}^{2\alpha +1}}q^{\sum_{i=1}^{2\alpha}\frac{1}{2}(k_{i}^{2}+k_{i})}}{({\displaystyle \prod_{i=2}^{2\alpha +1}}(q)_{k_{i-1}-k_{i}})(q)_{k_{2\alpha +1}}}\times\frac{(-1)^{\sum_{i=1}^{2\beta +1}}q^{\sum_{i=1}^{2\beta }\frac{1}{2}(l_{i}^{2}+l_{i})}(q)_{n}}{({\displaystyle \prod_{i=2}^{2\beta +1}}(q)_{l_{i-1}-l_{i}})(q)_{l_{2\beta +1}}}\times\frac{(q)^{m_{\gamma}^{2}+m_{\gamma}}(q)^{\sum_{i=1}^{\gamma-1}(m_{i}^{2}+2m_{i})}(q)_{n}}{({\displaystyle \prod_{i=2}^{\gamma}}(q)_{m_{i-1}-m_{i}})(q)^{2}_{m_{\gamma}}}\\
&\times\frac{q^{\frac{(k^{2}_{2\alpha +1}+l^{2}_{2\beta +1})-(k_{2\alpha +1}+l_{2\beta +1})}{2}}q^{s+(s+1)(s-k_{2\alpha +1}-l_{2\beta +1})+k_{2\alpha +1}l_{2\beta +1}}(q)_{k_{2\alpha +1}}(q)_{l_{2\beta +1}}}{(q)_{s-k_{2\alpha +1}}(q)_{s-l_{2\beta +1}}(q)_{-s+k_{2\alpha +1}+l_{2\beta +1}}}\\
&\times\frac{q^{t+(t+1)(t-a-m_{\gamma})+am_{\gamma}}(q)_{a}(q)_{m_{\gamma}}}{(q)_{t-a}(q)_{t-m_{\gamma}}(q)_{-t+a+m_{\gamma}}}\frac{1}{(1-q^{t+1})(1-q^{t+2})}
\end{aligned}
\end{align}
If $k_{2\alpha+1}+l_{2\beta+1}\ge n+1$, then
\begin{align*}
&(\mbox{RHS of }(\ref{ali:min1}))\\
=&\sum_{0\leq k_{2\alpha+1}\leq k_{2\alpha }\leq\cdots\leq k_{1}\leq n}\quad\sum_{0\leq l_{2\beta+1}\leq l_{2\beta }\leq\cdots\leq l_{1}\leq n}\quad\sum_{0\leq m_{\gamma}\leq m_{\gamma-1}\leq\cdots \leq m_{1}\leq n}\quad\sum^{n}_{s=\max\{k_{2\alpha+1},l_{2\beta+1}\}}\quad\sum_{a=s}^{n+1}\quad\sum^{\min\{a+m_{\gamma},n+1\}}_{t=\max\{a,m_{\gamma}\}}\\
&\times\frac{(-1)^{\sum_{i=1}^{2\alpha +1}}q^{\sum_{i=1}^{2\alpha}\frac{1}{2}(k_{i}^{2}+k_{i})}}{({\displaystyle \prod_{i=2}^{2\alpha +1}}(q)_{k_{i-1}-k_{i}})(q)_{k_{2\alpha +1}}}\times\frac{(-1)^{\sum_{i=1}^{2\beta +1}}q^{\sum_{i=1}^{2\beta }\frac{1}{2}(l_{i}^{2}+l_{i})}}{({\displaystyle \prod_{i=2}^{2\beta +1}}(q)_{l_{i-1}-l_{i}})(q)_{l_{2\beta +1}}}\times\frac{(q)^{m_{\gamma}^{2}+m_{\gamma}}(q)^{\sum_{i=1}^{\gamma-1}(m_{i}^{2}+2m_{i})}(q)_{n}}{({\displaystyle \prod_{i=2}^{\gamma}}(q)_{m_{i-1}-m_{i}})(q)^{2}_{m_{\gamma}}}\\
&\times\frac{q^{\frac{(k^{2}_{2\alpha +1}+l^{2}_{2\beta +1})-(k_{2\alpha +1}+l_{2\beta +1})}{2}}q^{s+(s+1)(s-k_{2\alpha +1}-l_{2\beta +1})+k_{2\alpha +1}l_{2\beta +1}}(q)_{k_{2\alpha +1}}(q)_{l_{2\beta +1}}}{(q)_{s-k_{2\alpha +1}}(q)_{s-l_{2\beta +1}}(q)_{-s+k_{2\alpha +1}+l_{2\beta +1}}}\\
&\times\frac{q^{t+(t+1)(t-a-m_{\gamma})+am_{\gamma}}(q)_{a}(q)_{m_{\gamma}}}{(q)_{t-a}(q)_{t-m_{\gamma}}(q)_{-t+a+m_{\gamma}}}\frac{1}{(1-q^{t+1})(1-q^{t+2})}\\
&+\sum_{0\leq k_{2\alpha+1}\leq k_{2\alpha }\leq\cdots\leq k_{1}\leq n}\quad\sum_{0\leq l_{2\beta+1}\leq l_{2\beta }\leq\cdots\leq l_{1}\leq n}\quad\sum_{0\leq m_{\gamma}\leq m_{\gamma-1}\leq\cdots \leq m_{1}\leq n}\frac{(-1)^{\sum_{i=1}^{2\alpha +1}}q^{\sum_{i=1}^{2\alpha}\frac{1}{2}(k_{i}^{2}+k_{i})}}{({\displaystyle \prod_{i=2}^{2\alpha +1}}(q)_{k_{i-1}-k_{i}})(q)_{k_{2\alpha +1}}}\times\frac{(-1)^{\sum_{i=1}^{2\beta +1}}q^{\sum_{i=1}^{2\beta }\frac{1}{2}(l_{i}^{2}+l_{i})}}{({\displaystyle \prod_{i=2}^{2\beta +1}}(q)_{l_{i-1}-l_{i}})(q)_{l_{2\beta +1}}}\\
&\times\frac{(q)^{m_{\gamma}^{2}}(q)^{\sum_{i=1}^{\gamma-1}(m_{i}^{2}+2m_{i})}(q)_{n}}{({\displaystyle \prod_{i=2}^{\gamma}}(q)_{m_{i-1}-m_{i}})(q)^{2}_{m_{\gamma}}}\times\frac{q^{\frac{(k^{2}_{2\alpha +1}+l^{2}_{2\beta +1})-(k_{2\alpha +1}+l_{2\beta +1})}{2}}q^{(n+1-k_{2\alpha +1}+1)(n+1-l_{2\beta +1}+1)-1}(q)_{k_{2\alpha +1}}(q)_{l_{2\beta +1}}}{(q)_{n+1-k_{2\alpha +1}}(q)_{n+1-l_{2\beta +1}}(q)_{-(n+1)+k_{2\alpha +1}+l_{2\beta +1}}}\times\frac{q^{(n+1)}(q)_{n+1}}{(q)_{n+1-m_{\gamma}}}\\
&\times \frac{1}{(1-q^{n+2})(1-q^{n+3})}\\
\end{align*}
Here, by $(\ref{ali:a})$, $(\ref{ali:b})$, $k^{2}_{2\alpha +1}-k_{2\alpha +1}\ge 0$ and $^{2}_{2\beta +1}-l_{2\beta +1}\ge 0$,
\begin{align*}
&\sum_{0\leq k_{2\alpha+1}\leq k_{2\alpha }\leq\cdots\leq k_{1}\leq n}\quad\sum_{0\leq l_{2\beta+1}\leq l_{2\beta }\leq\cdots\leq l_{1}\leq n}\quad\sum_{0\leq m_{\gamma}\leq m_{\gamma-1}\leq\cdots \leq m_{1}\leq n}\frac{(-1)^{\sum_{i=1}^{2\alpha +1}}q^{\sum_{i=1}^{2\alpha}\frac{1}{2}(k_{i}^{2}+k_{i})}}{({\displaystyle \prod_{i=2}^{2\alpha +1}}(q)_{k_{i-1}-k_{i}})(q)_{k_{2\alpha +1}}}\times\frac{(-1)^{\sum_{i=1}^{2\beta +1}}q^{\sum_{i=1}^{2\beta }\frac{1}{2}(l_{i}^{2}+l_{i})}}{({\displaystyle \prod_{i=2}^{2\beta +1}}(q)_{l_{i-1}-l_{i}})(q)_{l_{2\beta +1}}}\\
&\times\frac{(q)^{m_{\gamma}^{2}}(q)^{\sum_{i=1}^{\gamma-1}(m_{i}^{2}+2m_{i})}(q)_{n}}{({\displaystyle \prod_{i=2}^{\gamma}}(q)_{m_{i-1}-m_{i}})(q)^{2}_{m_{\gamma}}}\times\frac{q^{\frac{(k^{2}_{2\alpha +1}+l^{2}_{2\beta +1})-(k_{2\alpha +1}+l_{2\beta +1})}{2}}q^{(n+1-k_{2\alpha +1}+1)(n+1-l_{2\beta +1}+1)-1}(q)_{k_{2\alpha +1}}(q)_{l_{2\beta +1}}}{(q)_{n+1-k_{2\alpha +1}}(q)_{n+1-l_{2\beta +1}}(q)_{-(n+1)+k_{2\alpha +1}+l_{2\beta +1}}}\times\frac{q^{(n+1)}(q)_{n+1}}{(q)_{n+1-m_{\gamma}}}\\
&\times \frac{1}{(1-q^{n+2})(1-q^{n+3})}\\
=&q^{n+1}+((\mbox{high order terms}))\\
\equiv_{n+1}&=0
\end{align*}
If $a=n+1$, then terms greater than the ${n}$-th power of $q$ vanish in the above way. Thus, 
\begin{align}
\label{ali:min2}
\begin{aligned}
&\hat{J}_{(n+1,0)}^{\mathfrak{sl}_{3}}(P(\downarrow2\alpha+1\downarrow,\uparrow2\beta+1\uparrow,\downarrow2\gamma\uparrow);q))\\
\equiv_{n+1}&\sum_{0\leq k_{2\alpha}+1\leq k_{2\alpha }\leq\cdots\leq k_{1}\leq n}\quad\sum_{0\leq l_{2\beta}+1\leq l_{2\beta }\leq\cdots\leq l_{1}\leq n}\quad\sum_{0\leq m_{\gamma}\leq m_{\gamma-1}\leq\cdots \leq m_{1}\leq n}\quad\sum^{\min\{k_{2\alpha+1}+l_{2\beta+1},n\}}_{s=\max\{k_{2\alpha+1},l_{2\beta+1}\}}\quad\sum_{a=s}^{n}\quad\sum^{\min\{a+m_{\gamma},n+1\}}_{t=\max\{a,m_{\gamma}\}}\\
&\frac{q^{\sum_{i=1}^{2\alpha}\frac{1}{2}(k_{i}^{2}+k_{i})}}{({\displaystyle \prod_{i=2}^{2\alpha +1}}(q)_{k_{i-1}-k_{i}})(q)_{k_{2\alpha +1}}}\times\frac{q^{\sum_{i=1}^{2\beta }\frac{1}{2}(l_{i}^{2}+l_{i})}(q)_{n}}{({\displaystyle \prod_{i=2}^{2\beta +1}}(q)_{l_{i-1}-l_{i}})(q)_{l_{2\beta +1}}}\times\frac{(q)^{m_{\gamma}^{2}+m_{\gamma}}(q)^{\sum_{i=1}^{\gamma-1}(m_{i}^{2}+2m_{i})}}{({\displaystyle \prod_{i=2}^{\gamma}}(q)_{m_{i-1}-m_{i}})(q)^{2}_{m_{\gamma}}}\\
&\times\frac{q^{\frac{(k^{2}_{2\alpha +1}+l^{2}_{2\beta +1})-(k_{2\alpha +1}+l_{2\beta +1})}{2}}q^{s+(s+1)(s-k_{2\alpha +1}-l_{2\beta +1})+k_{2\alpha +1}l_{2\beta +1}}(q)_{k_{2\alpha +1}}(q)_{l_{2\beta +1}}}{(q)_{s-k_{2\alpha +1}}(q)_{s-l_{2\beta +1}}(q)_{-s+k_{2\alpha +1}+l_{2\beta +1}}}\\
&\times\frac{q^{t+(t+1)(t-a-m_{\gamma})+am_{\gamma}}(q)_{a}(q)_{m_{\gamma}}}{(q)_{t-a}(q)_{t-m_{\gamma}}(q)_{-t+a+m_{\gamma}}}\frac{1}{(1-q^{t+1})(1-q^{t+2})}
\end{aligned}
\end{align}
If $n+1\le a+m_{\gamma}$, then
\begin{align*}
&(\mbox{RHS of $(\ref{ali:min2})$})\\
=&\sum_{0\leq k_{2\alpha}+1\leq k_{2\alpha }\leq\cdots\leq k_{1}\leq n}\quad\sum_{0\leq l_{2\beta}+1\leq l_{2\beta }\leq\cdots\leq l_{1}\leq n}\quad\sum_{0\leq m_{\gamma}\leq m_{\gamma-1}\leq\cdots \leq m_{1}\leq n}\quad\sum^{\min\{k_{2\alpha+1}+l_{2\beta+1},n\}}_{s=\max\{k_{2\alpha+1},l_{2\beta+1}\}}\quad\sum_{a=s}^{n}\quad\sum^{n}_{t=\max\{a,m_{\gamma}\}}\\
&\frac{q^{\sum_{i=1}^{2\alpha}\frac{1}{2}(k_{i}^{2}+k_{i})}}{({\displaystyle \prod_{i=2}^{2\alpha +1}}(q)_{k_{i-1}-k_{i}})(q)_{k_{2\alpha +1}}}\times\frac{q^{\sum_{i=1}^{2\beta }\frac{1}{2}(l_{i}^{2}+l_{i})}}{({\displaystyle \prod_{i=2}^{2\beta +1}}(q)_{l_{i-1}-l_{i}})(q)_{l_{2\beta +1}}}\times\frac{(q)^{m_{\gamma}^{2}+m_{\gamma}}(q)^{\sum_{i=1}^{\gamma-1}(m_{i}^{2}+2m_{i})}(q)_{n}}{({\displaystyle \prod_{i=2}^{\gamma}}(q)_{m_{i-1}-m_{i}})(q)^{2}_{m_{\gamma}}}\\
&\times\frac{q^{\frac{(k^{2}_{2\alpha +1}+l^{2}_{2\beta +1})-(k_{2\alpha +1}+l_{2\beta +1})}{2}}q^{s+(s+1)(s-k_{2\alpha +1}-l_{2\beta +1})+k_{2\alpha +1}l_{2\beta +1}}(q)_{k_{2\alpha +1}}(q)_{l_{2\beta +1}}}{(q)_{s-k_{2\alpha +1}}(q)_{s-l_{2\beta +1}}(q)_{-s+k_{2\alpha +1}+l_{2\beta +1}}}\\
&\times\frac{q^{t+(t+1)(t-a-m_{\gamma})+am_{\gamma}}(q)_{a}(q)_{m_{\gamma}}}{(q)_{t-a}(q)_{t-m_{\gamma}}(q)_{-t+a+m_{\gamma}}}\frac{1}{(1-q^{t+1})(1-q^{t+2})}\\
&+\sum_{0\leq k_{2\alpha}+1\leq k_{2\alpha }\leq\cdots\leq k_{1}\leq n}\quad\sum_{0\leq l_{2\beta}+1\leq l_{2\beta }\leq\cdots\leq l_{1}\leq n}\quad\sum_{0\leq m_{\gamma}\leq m_{\gamma-1}\leq\cdots \leq m_{1}\leq n}\quad\sum^{\min\{k_{2\alpha+1}+l_{2\beta+1},n\}}_{s=\max\{k_{2\alpha+1},l_{2\beta+1}\}}\quad\sum_{a=s}^{n}\\
&\frac{q^{\sum_{i=1}^{2\alpha}\frac{1}{2}(k_{i}^{2}+k_{i})}}{({\displaystyle \prod_{i=2}^{2\alpha +1}}(q)_{k_{i-1}-k_{i}})(q)_{k_{2\alpha +1}}}\times\frac{q^{\sum_{i=1}^{2\beta }\frac{1}{2}(l_{i}^{2}+l_{i})}}{({\displaystyle \prod_{i=2}^{2\beta +1}}(q)_{l_{i-1}-l_{i}})(q)_{l_{2\beta +1}}}\times\frac{(q)^{m_{\gamma}^{2}+m_{\gamma}}(q)^{\sum_{i=1}^{\gamma-1}(m_{i}^{2}+2m_{i})}(q)_{n}}{({\displaystyle \prod_{i=2}^{\gamma}}(q)_{m_{i-1}-m_{i}})(q)^{2}_{m_{\gamma}}}\\
&\times\frac{q^{\frac{(k^{2}_{2\alpha +1}+l^{2}_{2\beta +1})-(k_{2\alpha +1}+l_{2\beta +1})}{2}}q^{s+(s+1)(s-k_{2\alpha +1}-l_{2\beta +1})+k_{2\alpha +1}l_{2\beta +1}}(q)_{k_{2\alpha +1}}(q)_{l_{2\beta +1}}}{(q)_{s-k_{2\alpha +1}}(q)_{s-l_{2\beta +1}}(q)_{-s+k_{2\alpha +1}+l_{2\beta +1}}}\\
&\times\frac{q^{n+1+(n+1+1)(n+1-a-m_{\gamma})+am_{\gamma}}(q)_{a}(q)_{m_{\gamma}}}{(q)_{t-a}(q)_{t-m_{\gamma}}(q)_{-t+a+m_{\gamma}}}\frac{1}{(1-q^{n+2})(1-q^{n+3})}\\
\equiv_{n+1}&\sum_{0\leq k_{2\alpha}+1\leq k_{2\alpha }\leq\cdots\leq k_{1}\leq n}\quad\sum_{0\leq l_{2\beta}+1\leq l_{2\beta }\leq\cdots\leq l_{1}\leq n}\quad\sum_{0\leq m_{\gamma}\leq m_{\gamma-1}\leq\cdots \leq m_{1}\leq n}\quad\sum^{\min\{k_{2\alpha+1}+l_{2\beta+1},n\}}_{s=\max\{k_{2\alpha+1},l_{2\beta+1}\}}\quad\sum_{a=s}^{n}\quad\sum^{n}_{t=\max\{a,m_{\gamma}\}}\\
&\frac{q^{\sum_{i=1}^{2\alpha}\frac{1}{2}(k_{i}^{2}+k_{i})}}{({\displaystyle \prod_{i=2}^{2\alpha +1}}(q)_{k_{i-1}-k_{i}})(q)_{k_{2\alpha +1}}}\times\frac{q^{\sum_{i=1}^{2\beta }\frac{1}{2}(l_{i}^{2}+l_{i})}}{({\displaystyle \prod_{i=2}^{2\beta +1}}(q)_{l_{i-1}-l_{i}})(q)_{l_{2\beta +1}}}\times\frac{(q)^{m_{\gamma}^{2}+m_{\gamma}}(q)^{\sum_{i=1}^{\gamma-1}(m_{i}^{2}+2m_{i})}(q)_{n}}{({\displaystyle \prod_{i=2}^{\gamma}}(q)_{m_{i-1}-m_{i}})(q)^{2}_{m_{\gamma}}}\\
&\times\frac{q^{\frac{(k^{2}_{2\alpha +1}+l^{2}_{2\beta +1})-(k_{2\alpha +1}+l_{2\beta +1})}{2}}q^{s+(s+1)(s-k_{2\alpha +1}-l_{2\beta +1})+k_{2\alpha +1}l_{2\beta +1}}(q)_{k_{2\alpha +1}}(q)_{l_{2\beta +1}}}{(q)_{s-k_{2\alpha +1}}(q)_{s-l_{2\beta +1}}(q)_{-s+k_{2\alpha +1}+l_{2\beta +1}}}\\
&\times\frac{q^{t+(t+1)(t-a-m_{\gamma})+am_{\gamma}}(q)_{a}(q)_{m_{\gamma}}}{(q)_{t-a}(q)_{t-m_{\gamma}}(q)_{-t+a+m_{\gamma}}}\frac{1}{(1-q^{t+1})(1-q^{t+2})}
\end{align*}
The second congruence is obtained by the fact that
\begin{align*}
&n+1+(n+1+1)(n+1-a-m_{\gamma})+am_{\gamma}+m_{\gamma}\\
=&(n+1-a+1)(n+1-m_{\gamma}+1)-1+m_{\gamma}\\
\ge &2(n+1-m_{\gamma}+1)-1+m_{\gamma}\\
= &n+1+n-m_{\gamma}+4\ge n+5
\end{align*}
Therefore,
\begin{align}
\label{ali:kansei}
\begin{aligned}
&\hat{J}_{(n+1,0)}^{\mathfrak{sl}_{3}}(P(\downarrow2\alpha+1\downarrow,\uparrow2\beta+1\uparrow,\downarrow2\gamma\uparrow);q))\\
\equiv_{n+1}&\sum_{0\leq k_{2\alpha}+1\leq k_{2\alpha }\leq\cdots\leq k_{1}\leq n}\quad\sum_{0\leq l_{2\beta}+1\leq l_{2\beta }\leq\cdots\leq l_{1}\leq n}\quad\sum_{0\leq m_{\gamma}\leq m_{\gamma-1}\leq\cdots \leq m_{1}\leq n}\quad\sum^{\min\{k_{2\alpha+1}+l_{2\beta+1},n\}}_{s=\max\{k_{2\alpha+1},l_{2\beta+1}\}}\quad\sum_{a=s}^{n}\quad\sum^{\min\{a+m_{\gamma},n\}}_{t=\max\{a,m_{\gamma}\}}\\
&\frac{q^{\sum_{i=1}^{2\alpha}\frac{1}{2}(k_{i}^{2}+k_{i})}}{({\displaystyle \prod_{i=2}^{2\alpha +1}}(q)_{k_{i-1}-k_{i}})(q)_{k_{2\alpha +1}}}\times\frac{q^{\sum_{i=1}^{2\beta }\frac{1}{2}(l_{i}^{2}+l_{i})}(q)_{n}}{({\displaystyle \prod_{i=2}^{2\beta +1}}(q)_{l_{i-1}-l_{i}})(q)_{l_{2\beta +1}}}\times\frac{(q)^{m_{\gamma}^{2}+m_{\gamma}}(q)^{\sum_{i=1}^{\gamma-1}(m_{i}^{2}+2m_{i})}}{({\displaystyle \prod_{i=2}^{\gamma}}(q)_{m_{i-1}-m_{i}})(q)^{2}_{m_{\gamma}}}\\
&\times\frac{q^{\frac{(k^{2}_{2\alpha +1}+l^{2}_{2\beta +1})-(k_{2\alpha +1}+l_{2\beta +1})}{2}}q^{s+(s+1)(s-k_{2\alpha +1}-l_{2\beta +1})+k_{2\alpha +1}l_{2\beta +1}}(q)_{k_{2\alpha +1}}(q)_{l_{2\beta +1}}}{(q)_{s-k_{2\alpha +1}}(q)_{s-l_{2\beta +1}}(q)_{-s+k_{2\alpha +1}+l_{2\beta +1}}}\\
&\times\frac{q^{t+(t+1)(t-a-m_{\gamma})+am_{\gamma}}(q)_{a}(q)_{m_{\gamma}}}{(q)_{t-a}(q)_{t-m_{\gamma}}(q)_{-t+a+m_{\gamma}}}\frac{1}{(1-q^{t+1})(1-q^{t+2})}
\end{aligned}
\end{align}
According to $(\ref{ali:minn})$ and $(\ref{ali:kansei})$,
\begin{align*}
\hat{J}_{(n,0)}^{\mathfrak{sl}_{3}}(P(\downarrow2\alpha+1\downarrow,\uparrow2\beta+1\uparrow,\downarrow2\gamma\uparrow);q))\equiv_{n+1}\hat{J}_{(n+1,0)}^{\mathfrak{sl}_{3}}(P(\downarrow2\alpha+1\downarrow,\uparrow2\beta+1\uparrow,\downarrow2\gamma\uparrow);q))
\end{align*}
\end{proof}

\section{Appendix}

The knots $8_{10},8_{15},8_{20}$ and $8_{21}$ are four or five parameters pretzel knots. For these knots, we can calculate the one-row colored $\mathfrak{sl}_{3}$ Jones polynomials by combining the similar method as three-parameter pretzel knots and links. 
\begin{Theorem}
 \label{Th:Appendix}\quad\\
 The one-row colored $\mathfrak{sl}_{3}$  Jones polynomials  for  $P(-3,-2,3,-1)=8_{10}$, $P(3,-1,-2,-1,3)=8_{15}$, $P(3,-2,-3,1)=8_{20}$ and $P(3,3,-1,2)=8_{21}$  are the following:
\begin{align*}
(1)\quad&J^{\mathfrak{sl}_{3}}_{(n,0)}(P(-3,-2,3,-1);q)\\
=&\sum_{0\leq k_{3}\leq k_{2}\leq k_{1}\leq n}\quad\sum_{0\leq l_{2} \leq l_{1}\leq n}\quad\sum_{0\leq p_{3}\leq p_{2}\leq p_{1}\leq n}\quad\sum_{m=0}^{n}\quad\sum_{t=\max\{k_{3},l_{2}\}}^{\min\{k_{3}+l_{2},n\}}\quad\sum_{a=t}^{n}\quad\sum_{s=\max\{p_{3},m\}}^{\min\{p_{3}+m,n\}}\sum_{b=s}^{n}\quad\sum_{u=\max\{a,b\}}^{\min\{a+b,n\}}(q^{\frac{n^{2}+3n}{3}})^{3}\\
&\chi_{-1}(n,k_{0},k_{1},k_{2})\chi_{-1}(n,l_{0},l_{1})_{q^{-1}}\chi_{1}(n,p_{1},p_{2},p_{3})\chi_{-1}(n,m)\Omega(n,t,k_{3},l_{2})\Omega(n,s,p_{3},m)\psi(n,u,a,b)\\
&q^{-(n-u)}\frac{(1-q^{n+1})(1-q^{n+2})}{(1-q^{u+1})(1-q^{u+2})},\\
(2)\quad&J^{\mathfrak{sl}_{3}}_{(n,0)}(P(3,-1,-2,-1,3);q)\\
=&\sum_{0\leq k_{3}\leq k_{2}\leq k_{1}\le n}\quad\sum_{0\leq l \leq n}\quad\sum_{0\leq m_{2}\leq m_{1}}\sum_{0\leq p \leq n}\quad\sum_{0\leq r_{3}\leq r_{2}\leq r_{1}\le n}\quad\sum_{t=\max\{k_{3},l\}}^{\min\{k_{3}+l,n\}}\quad\sum_{a=t}^{n}\sum_{s=\max\{p,r_{3}\}}^{\min\{p+r_{3},n\}}\quad\sum_{b=s}^{n}\\
&\quad\sum_{u=\max\{a,m_{1}\}}^{\min\{a+m_{2},n\}}\quad\sum_{v=\max\{b,m_{2}\}}^{\min\{b+m_{2},n\}}(q^{\frac{n^{2}+3n}{3}})^{-2}\chi_{1}(n,k_{1},k_{2},k_{3})\chi_{-1}(n,l)\phi(n,m_{1},m_{2})\chi_{-1}(n,p)\chi_{1}(n,r_{1},r_{2},r_{3})\Omega(n,t,k_{3},l)\\
&\Omega(n,s,p,r_{3})\psi(n,u,a,m_{2})\psi(n,v,u,b)q^{-(n-v)}\frac{(1-q^{n+1})(1-q^{n+2})}{(1-q^{v+1})(1-q^{v+2})}\\
(3)\quad&J^{\mathfrak{sl}_{3}}_{(n,0)}(P(3,-2,-3,1);q),\\
=&\sum_{0\leq k_{3}\leq k_{2}\leq k_{1}\leq n}\quad\sum_{0\leq l_{2} \leq l_{1}\leq n}\quad\sum_{0\leq p_{3}\leq p_{2}\leq p_{1}\leq n}\quad\sum_{m=0}^{n}\quad\sum_{t=\max\{k_{3},l_{2}\}}^{\min\{k_{3}+l_{2},n\}}\quad\sum_{a=t}^{n}\quad\sum_{s=\max\{p_{3},m\}}^{\min\{p_{3}+m,n\}}\quad\sum_{b=s}^{n}\sum_{u=\max\{a,b\}}^{\min\{a+b,n\}}q^{\frac{n^{2}+3n}{3}}\\
&\chi_{1}(n,k_{1},k_{2},k_{3})\chi_{-1}(n,l_{1},l_{2})_{q^{-1}}\chi_{-1}(n,p_{1},p_{2},p_{3})\chi_{1}(n,m)\Omega(n,t,n-k_{3},l_{2})\Omega(n,s,n-p_{2},n-m)\\
&\psi(n,u,t+a,s+b)q^{-(n-u)}\frac{(1-q^{n+1})(1-q^{n+2})}{(1-q^{u+1})(1-q^{u+2})}\\
(4)\quad&J^{\mathfrak{sl}_{3}}_{(n,0)}(P(3,3,-1,-2);q),\\
=&\sum_{0\leq k_{3}\leq k_{2}\leq k_{1}\leq n}\quad\sum_{0\leq l_{3}\leq l_{2} \leq l_{1}\leq n}\quad\sum_{0\leq p\leq n}\quad\sum_{0\leq m_{2}\leq m_{1} \leq n}\quad\sum_{t=\max\{k_{3},l_{3}\}}^{\min\{k_{3}+l_{3},n\}}\quad\sum_{a=t}^{n}\quad\sum_{s=\max\{p,m\}}^{\min\{p+m,n\}}\sum_{b=s}^{n}\quad\sum_{u=\max\{a,b\}}^{\min\{t+s,n\}}\\
&(q^{\frac{n^{2}+3n}{3}})^{-3}\chi_{1}(n,k_{1},k_{2},k_{3})\chi_{1}(n,l_{1},l_{2},l_{3})\chi_{-1}(n,p)\chi_{-1}(n,m_{1},m_{2})\\
&\Omega(n,t,k_{3},l_{3})\Omega(n,s,p,m_{2})\psi(n,u,a,b)q^{-(n-u)}\frac{(1-q^{n+1})(1-q^{n+2})}{(1-q^{u+1})(1-q^{u+2})}.
\end{align*}

\end{Theorem}
\begin{Remark}
According to Theorem\ref{Th:Appendix}, Theorem\ref{Thm:main2} and known results,  the one-row $\mathfrak{sl}_{3}$ colored Jones polynomial is determined for all knots with eight or  fewer crossings except $8_{16},8_{17},8_{18}$.
\end{Remark}

\subsubsection*{Acknowledgment}
I would like to express my sincere gratitude to my advisor, Professor Masaaki Suzuki, for his constructive suggestions.

\end{document}